%% file: motivic-kq-v2.tex
\documentclass[10pt]{amsart}

\input{packages}

\input{theorem-styles.tex}

\input{commands.tex}

\input{toc-tweak.tex}

\input{bib-format.tex}

\title{A motivic spectrum representing hermitian K-theory}

\author[Calmès]{Baptiste Calmès}
\address{Université d'Artois, Laboratoire de Mathématiques de Lens (LML), UR 2462, Lens, France}
\email{baptiste.calmes@univ-artois.fr}

\author[Harpaz]{Yonatan Harpaz}
\address{Université de Paris 13; Institut Galilée; Villetaneuse, France}
\email{harpaz@math.univ-paris13.fr}

\author[Nardin]{Denis Nardin}
\address{Universität Regensburg; Mathematisches Institut; Regensburg, Germany}
\email{Denis.Nardin@mathematik.uni-regensburg.de}

\begin{document}

\begin{abstract}
We establish fundamental motivic results about hermitian K-theory without assuming that 2 is invertible on the base scheme. In particular, we prove that both quadratic and symmetric Grothendieck-Witt theory satisfy Nisnevich descent, and that symmetric Grothendieck-Witt theory further satisfies a projective bundle formula, as well as dévissage and $\Aone$-invariance over a regular Noetherian base of finite Krull dimension. We use this to show that over a regular Noetherian base, symmetric Grothendieck-Witt theory is represented by a motivic $\Einf$-ring spectrum, which we then show is an absolutely pure spectrum, answering a question of Déglise. As with algebraic K-theory, we show that over a general base, one can also construct a hermitian K-theory motivic spectrum, representing this time a suitable homotopy invariant and Karoubi-localising version of Grothendieck-Witt theory. 
\end{abstract}

\maketitle

\tableofcontents

\section{Introduction}

Originally justified by Voevodsky's proof of the Milnor conjecture, motivic homotopy theory is now generally recognized as a deep and efficient framework to study cohomological invariants of rings and schemes. A key insight of this approach is that such invariants can be studied in a manner closely resembling that taken up by classical algebraic topology in the study of cohomological invariants of topological spaces. More precisely, the latter can be axiomatised via the notion of generalized cohomology theories, which in turn can be classified by the notion of a spectrum. Similarly, cohomological invariants of schemes can be classified by a suitable notion of a motivic spectrum. For example, motivic cohomology, algebraic K-theory, and algebraic cobordism, which are algebraic analogues of singular cohomology, complex K-theory, and cobordism, are each represented by a corresponding motivic spectrum.

On the side of topological spaces, another generalized cohomology theory which plays an important role in the study of smooth manifolds, is real K-theory. Its analogue on the side of schemes is hermitian K-theory, also called Grothendieck-Witt theory, which is the K-theory of unimodular quadratic forms. This somewhat less obvious analogy is due to the inclusions $\Orth_n(\CC) \supseteq \Orth_n(\RR) \subseteq \GL_n(\RR)$ both being homotopy equivalences, so that the $\K$-theory of real vector bundles coincides with that of complex vector bundles equipped with non-degenerate quadratic forms.
 
The behaviour of hermitian $\K$-theory is notoriously sensitive to the prime 2. In particular, its behaviour simplifies considerably when considering schemes over which 2 is invertible. As a result, until recently a significant part of the relevant theory was developed under this assumption, and in particular, the realization of hermitian K-theory as a motivic spectrum, as was first done by Hornbostel~\cite{hornbostel-motivic}, with geometric models later constructed by Schlichting and Tripathi~\cite{Schlichting-Tripathi}, was limited to $\ZZ[\tfrac{1}{2}]$-schemes.

In a series of papers~\cites{9-authors-I,9-authors-II,9-authors-III,9-authors-IV,9-authors-V} by a group of authors including the authors of the present paper, the foundations of hermitian K-theory were revisited using the powerful recent technology of higher category theory, in a manner which enables one to dispense with the invertibility of 2 assumption. The goal of the present paper is to harness these results in the service of the motivic aspects of hermitian K-theory, allowing a base scheme $S$ in which 2 is not assumed invertible. The first and primary outcome of our work is the following result, which is itself a prerequisite to any further investigations in the motivic direction:

\begin{theorem}
\label{theorem:main-intro}%
Let $S$ be a regular Noetherian scheme of finite Krull dimension. Then symmetric Grothendieck-Witt theory is represented by a motivic spectrum over $S$.
\end{theorem}

Let us take a moment to clarify what the term ``symmetric'' means in this context. First, note that when 2 is invertible, the notions of quadratic form and symmetric bilinear form are essentially equivalent: every symmetric bilinear form admits a unique quadratic refinement. This is no longer true when 2 is not invertible, so in general, one needs to specify with respect to which flavour of forms Grothendieck-Witt theory is taken. For example, in the above theorem, we make reference to symmetric bilinear forms.

There are, however, more then one variant of those as well. Recall that by the Gillet-Waldhausen theorem, the algebraic K-theory of a ring can be calculated either by considering finitely generated projective modules and then taking group completion, or by considering all perfect complexes, in which case one needs to enforce relations for every exact sequence of such. Passing to the hermitian setting, projective modules can be endowed with symmetric bilinear forms in the usual manner, while perfect complexes can be endowed with a homotopical analogue of symmetric forms, where the symmetry property is interpreted as a homotopy fixed point structure. 
When 2 is not invertible, these two constructions lead to different Grothendieck-Witt groups (or spectra) in general.

In the paper series mentioned above, we use the term symmetric Grothendieck-Witt to refer to the latter construction using homotopy fixed points, and introduce the term \emph{genuine symmetric} to refer to the former. While the genuine symmetric is the one more directly related to classical constructions in hermitian K-theory, the (non-genuine) symmetric admits better formal properties. Most notably from the present viewpoint, it is $\Aone$-invariant over regular Noetherian schemes of finite Krull dimension, which is the main reason why it is this flavour of Grothendieck-Witt theory that features in Theorem~\ref{theorem:main-intro}. It also underlines a possible reason for why the motivic hermitian K-theory spectrum was previously only constructed when 2 is invertible: without this assumption, no flavour of Grothendieck-Witt groups defined using a classical notion of forms is $\Aone$-invariant (even over fields); see Example~\ref{example:non-homotopy-invariance}.

While Theorem~\ref{theorem:main-intro} is a necessary first step, it still leaves several obvious questions unanswered:
\begin{enumerate}
\item
What happens if $S$ is not regular Noetherian?
\item
What about Witt theory, obtained from Grothendieck-Witt theory in principle by quotienting out the hyperbolic forms?
\item
What about the multiplicative structure of Grothendieck-Witt theory?
\item
What about all the natural maps relating Grothendieck-Witt theory, Witt theory, and algebraic K-theory?
\item
What happens to some of the expected properties of motivic Grothendieck-Witt theory, such as purity or the projective bundle formula, when 2 is not invertible?
\end{enumerate}
In the present paper, we proceed beyond Theorem~\ref{theorem:main-intro} by giving complete answers to all the above questions. To describe how this is done, let us further recall the theoretical framework of~\cite{9-authors-I} and its sequels. The main point is that Grothendieck-Witt theory should be considered as an invariant defined on \emph{Poincaré $\infty$-categories}. These are stable $\infty$-categories equipped with a spectrum valued functor $\QF\colon \C\op \to \Spa$ satisfying suitable axioms. These axioms guarantee in particular the existence of a uniquely determined duality $\Dual\colon \C\op \to \C$ such that $\QF(x \oplus y) = \QF(x) \oplus \QF(y) \oplus \map(x,\Dual y)$. The functor $\QF$ is then called a \emph{Poincaré structure} on $\C$. For an object $x \in \C$, the underlying infinite loop space $\Om^{\infty}\QF(x)$ should be thought of as the space of forms on $x$. In particular, we call a point $\bet \in \Om^{\infty}\QF(x)$ a \emph{form} on $x$. Such a form determines a map $\bet_{\sharp}\colon x \to \Dual x$, and we say that $\bet$ is a \emph{Poincaré form} if $\bet_{\sharp}$ is an equivalence. Poincaré forms are what corresponds in this setting to non-degenerate forms. A pair $(x,\bet)$ of an object equipped with a Poincaré form $\bet \in \Om^{\infty}\QF(x)$ is called a Poincaré object.

The fact that $\QF$ is an extra piece of structure (and not an inherent property of $\C$) is what allows one to specify the flavour of forms one is interested in. 
For example, if $R$ is a commutative ring, then we may consider its perfect derived $\infty$-category $\Dperf(R)$, whose objects are the perfect $R$-complexes, and consider the functor
\[
\QF^{\sym}_R(M) := \map(M \otimes_R M, R)^{\hC} ,
\]
where $(-)^{\hC}$ denotes homotopy fixed points with respect to the $\Ct$-action induced by swapping the two $M$ factors. Then $\QF^{\sym}$ sends $M$ to the spectrum of what can be considered as the homotopy theoretical analogue of symmetric bilinear forms on $M$. Alternatively, we may take the Poincaré structure $\QF^{\qdr}_R(M) := \map(M \otimes_R M,R)_{\hC}$ obtained by using homotopy orbits instead of fixed points. This is the homotopy theoretical analogue of quadratic forms. As mentioned above, these are not exactly the same as working with symmetric and quadratic forms on projective modules. The latter notions can also be encoded via suitable Poincaré structures $\QF^{\gs}$ and $\QF^{\gq}$, called the genuine symmetric and genuine quadratic structures. All these Poincaré structures are generally different from each other, and are related by a sequence of natural transformations
\[
\QF_R^{\qdr} \Rightarrow \QF_R^{\gq} \Rightarrow \QF_R^{\gs} \Rightarrow \QF_R^{\sym} .
\]
When $2$ is invertible in $R$, these maps are all equivalences.

To a Poincaré $\infty$-category $(\C,\QF)$, one may associate an invariant $\GW(\C,\QF)$, called the Grothendieck-Witt spectrum, whose homotopy groups are the Grothendieck-Witt groups of $(\C,\QF)$. By~\cite{Hebestreit-Steimle}, in the case of $(\Dperf(R),\QF^{\gs})$ and $(\Dperf(R),\QF^{\gq})$, these reproduce the classical symmetric and quadratic Grothendieck-Witt groups of $R$, though, as mentioned earlier, the Grothendieck-Witt groups of the Poincaré $\infty$-category $(\Dperf(R),\QF^{\sym}_R)$ have better formal properties, and play a more important role in the current paper. More generally, one may replace the commutative ring $R$ by a scheme $X$, and consider the perfect derived $\infty$-category $\Dperf(X)$ of perfect complexes on $X$. As above, one can define the symmetric Poincaré structure $\QF^{\sym}_X$ using the same formulas. 
More generally, given a line bundle $L$ on $X$, we may consider the Poincaré structure $\QF^{\sym}_L$ defined by $\QF^{\sym}_L(M) = \map(M \otimes_{\cO_X} M, L)^{\hC}$. It is also useful to consider the slightly more general case where, on the one hand, $L$ is endowed with an involution, and, on the other hand, is allowed to be a shift of a line bundle (or more generally, any tensor invertible perfect complex). The Grothendieck-Witt groups of $(\Dperf(X),\QF^{\sym}_L)$ are the main invariants of schemes we consider in this paper. To keep the introduction easily readable, we mention only line bundles with constant involution when describing our main results here, though we provide references to the body of the paper where the reader can find the full version of each result.

The collection of Poincaré $\infty$-categories can be organized into a (large) category $\Catp$, which is presentable and compactly generated (see~\cite{9-authors-V}). The functor $\GW(-)$ is Verdier localising, that is, it sends fibre-cofibre sequences of Poincaré $\infty$-categories to exact sequences of spectra, and is furthermore the universal such functor equipped with a natural transformation $\Poinc \Rightarrow \Om^{\infty}\GW$, where $\Poinc$ is the functor which associates to a Poincaré $\infty$-category its space of Poincaré objects. It turns out that the property that is more relevant from the motivic point of view is not being Verdier localising, but being Karoubi-localising, which is equivalent to being Verdier localising and invariant under idempotent completion. While the functor $\GW$ is not Karoubi-localising, it can be turned into one in a universal manner, yielding a functor $\KGW$ we call the \emph{Karoubi-Grothendieck-Witt} functor. It comes equipped with a canonical natural transformation $\GW \Rightarrow \KGW$, which is generally not equivalence. This transformation is closely related to the transformation $\K \Rightarrow \KK$ between algebraic $\K$-theory and Bass $\K$-theory (also known as non-connective $\K$-theory). In particular, if $\C$ is such that $\K(\C) \to \KK(\C)$ is an equivalence (e.g., $\C=\Dperf(X)$ for $X$ a regular Noetherian scheme) then $\GW(\C,\QF) \to \KGW(\C,\QF)$ is an equivalence, see~\cite[\S 1]{9-authors-IV}. 

The importance of the Karoubi-localising property in the process of passing from invariants of Poincaré $\infty$-categories to invariants of schemes is expressed by the following key result of the present paper (see Corollary~\ref{corollary:nisnevich-descent} below):

\begin{proposition}[Nisnevich descent]
\label{proposition:nisnevich-descent-intro}%
Let $S$ be a quasi-compact quasi-separated base scheme and $\F\colon \Catp \to \Spa$ a Karoubi-localising spectrum valued functor. Then, the functor
\[
\F^\sym_S\colon \qSch_{/S} \to \Spa \quad\quad X \mapsto \F(X,\QF^\sym_X)
\]
is a Nisnevich sheaf, where $\qSch_{/S}$ is the category of quasi-compact quasi-separated $S$-schemes.
\end{proposition}

In particular, the functor $\KGW^\sym_S(X) = \KGW(X,\QF^{\sym}_X)$ is a Nisnevich sheaf on $\qSch_{/S}$. We may also restrict attention to the full subcategory $\Sm_S \subseteq \Sch_{/S}$ spanned by the (quasi-compact quasi-separated) smooth $S$-schemes, and consider $\KGW^\sym_S$ as a Nisnevich sheaf on that site. If the base scheme $S$ itself is regular, then every smooth $S$-scheme is regular as well, and, as mentioned above, for a regular Noetherian scheme $X$, the map $\GW(X,\QF^\sym_X) \to \KGW(X,\QF^\sym_X)$ is an equivalence. This means that for a regular Noetherian $S$, the functor $\GW^\sym_S(X) = \GW(X,\QF^\sym)$ is a Nisnevich sheaf on $\Sm_S$. Alternatively, one may replace $\GW$ by the functor $\L$ which associates to a Poincaré $\infty$-category its $\L$-spectrum (whose homotopy groups correspond to higher Witt groups in algebraic geometry), and obtain using similar arguments that the symmetric $\L$-theory functor $\L^\sym_S$ is a Nisnevich sheaf on $\Sm_S$ for any regular Noetherian $S$.

Now a motivic spectrum over $S$ is not just a Nisnevich sheaf on $\Sm_S$: it possess more structure (a tower of $\Pone$-deloopings) and satisfies more axioms ($\Aone$-invariance). In order to incorporate the former into the construction of the motivic Grothendieck-Witt spectrum, we need to understand the behaviour of $\GW$ under taking products with $\Pone$. More generally, we consider projective line bundles which are not necessarily constant (and eventually also higher rank projective bundles; see Proposition~\ref{proposition:proj-bundle-I}, Theorem~\ref{theorem:projective-bundle-formula} and Lemma~\ref{lemma:pbf-boundary} in the body of the paper):

\begin{proposition}[Projective line formula]
Let $X$ be a quasi-compact and quasi-separated scheme equipped with a line bundle $L$ 
and let $V$ be a vector bundle over $X$ of rank $2$. Let $\F\colon \Catp\to \Spa$ be a Karoubi-localising functor. 
Then we have a split fibre sequence
\begin{equation*}
\F\big(\Dperf(X),\QF^\sym_L\big)\to \F\big(\Dperf(\PP_XV),\QF^\sym_{p^*L}\big) \to \F\big(\Dperf(X),\QF^\sym_{L \otimes \det V^\vee[-1]}\big) \ ,
\end{equation*}
where $p\colon\PP_X V=\Proj(\Sym^\bullet V)\to X$ is the corresponding projective bundle over $X$ and the first arrow is induced by pullback along $p$.
\end{proposition}

The last two propositions combined can be leveraged to obtain the following:
\begin{proposition}[Bott periodicity]
\label{proposition:P1-spectra-intro}%
Let $S$ be a quasi-compact quasi-separated base scheme. Then the association $\F \mapsto \F^\sym_S$ described in Proposition~\ref{proposition:nisnevich-descent-intro} refines to a functor
\[
\Fun^{\kloc}(\Catp,\Spa) \to \Spa_{\Pone}^{\Sig}(\Sh^{\Nis}(\Sm_S,\Spa)) \quad\quad \F \mapsto \R^{\sym}_S(\F), 
\]
where the left hand side is the $\infty$-category of (spectrum valued) Karoubi-localising functors on $\Catp$ and the right hand side the $\infty$-category of symmetric $\Pone$-spectrum objects in Nisnevich (spectrum valued) sheaves on $\Sm_S$. Furthermore, the $\Pone$-deloopings of $\R^{\sym}_S(\F)$ are governed by the rule 
\[
\Sig^n_{\Pone}\R^{\sym}_S(\F) \simeq \R^{\sym}_S(\F\qshift{n})
\]
where $\F\qshift{n}$ is the functor $\F\qshift{n}(\C,\QF) = \F(\C,\QF\qshift{n}) = \F(\C,\Sig^n\QF)$.
\end{proposition}

Applying the last result to $\KGW$ yields a $\Pone$-spectrum object 
$\R^{\sym}_S(\KGW)$ 
(infinitely) delooping the Nisnevich sheaf $\KGW^\sym_S$. In general 
$\R^{\sym}_S(\KGW)$
can fail to be a motivic spectrum, as it might not be $\Aone$-invariant. However, if $S$ is regular Noetherian of finite Krull dimension then the $\Aone$-invariance of 
$\R^{\sym}_S(\KGW)$
is assured, as we establish in the following key result (see Theorem~\ref{theorem:A1-invariance}):

\begin{theorem}
\label{theorem:homotopy-invariance-intro}%
Let $S$ be a regular Noetherian scheme of finite Krull dimension. Then the Nisnevich sheaf $\KGW^{\sym}_S$ is $\Aone$-invariant and $\R^{\sym}_S(\KGW)$ is a motivic spectrum. 
\end{theorem}

This means that for a regular Noetherian $S$ of finite Krull dimension symmetric Grothendieck-Witt theory is represented by a motivic spectrum, which we denote by $\KQ_S \in \SH(S)$. This establishes Theorem~\ref{theorem:main-intro} above. The proof of Theorem~\ref{theorem:homotopy-invariance-intro} is primarily based on another key property of independent interest of symmetric Grothendieck-Witt theory - dévissage - whose proof occupies \S\ref{section:devissage} in its entirety (see Theorem~\ref{theorem:global-devissage} for the full version of the result):

\begin{theorem}[Devissage]
Let $i\colon Z\hrar X$ be a closed embedding of finite dimensional regular Noetherian schemes 
and $L$ a line bundle 
on $X$. 
Then the map
\[
\GW\big(\Dperf(Z),\QF^\sym_{i^!L}\big)\to \GW\big(\Dperf_Z(X),\QF^\sym_L|_Z\big) \ .
\]
induced by push-forward along $i$, is an equivalence. 
\end{theorem}

If $S$ is only assumed quasi-compact and quasi-separated then the $\Pone$-spectrum $\R^{\sym}_S(\KGW)$
is not necessarily a motivic spectrum. We can however turn it into a motivic spectrum in a universal manner by applying to it the $\Aone$-localisation functor
\[
\Loc_{\Aone}\colon \Spa_{\Pone}^{\Sig}(\Sh^{\Nis}(S,\Spa)) \to \SH(S) ,
\]
yielding what we call the motivic realization 
\[
\M^{\sym}_S(\F) := \Loc_{\Aone}\R^{\sym}(\F)
\]
of $\F$. We hence obtain a Grothendieck-Witt motivic spectrum $\KQ_S := \M^{\sym}_S(\KGW) \in \SH(S)$ for an arbitrary quasi-compact quasi-separated $S$. These motivic spectra depend functorially on $S$, in the sense that if $f\colon T \to S$ is a map of quasi-compact quasi-separated schemes then we have an induced map
\[
\eta_f\colon f^*\KQ_S \to \KQ_T .
\]
We then verify the following property, which is of technical importance to future applications:

\begin{proposition}[Base change invariance]
The map $\eta_f$ is an equivalence. In particular, the motivic spectra $\KQ_S$ are all pulled back from the absolute motivic spectrum $\KQ := \KQ_{\spec(\ZZ)} \in \SH(\spec(\ZZ))$.
\end{proposition}

In addition to Grothendieck-Witt theory, we also apply our motivic realization functor $\M^{\sym}_S$ to Karoubi $\L$-theory $\KL$ (the Karoubi-localising approximation of $\L$-theory) and to Bass $\K$-theory $\KK$, thus yielding motivic spectra $\KW_S$ and $\KGL_S$ representing Karoubi Witt theory and Bass $\K$-theory, respectively. Though the latter was previously constructed in the literature, reproducing the construction in the present setting allows one to use all the resulting functoriality. 
In particular, all the structural maps relating Grothendieck-Witt theory, $\K$-theory and $\L$-theory induce corresponding maps on the level of motivic spectra, and we obtain for example, the Tate fibre sequence
\[
(\KGL_S)_{\hC} \to \KQ_S \to \KW_S
\]
and the Wood sequence
\[
\Om_{\Pone}\KQ_S \to \KGL_S \to \KW_S .
\]

Finally, we also consider the question of multiplicative structures on $\KQ, \KW$ and the maps relating them to each other and to $\KGL$. In fact, this accounts for most of the more technical parts of the paper, whose goal is to prove the following:

\begin{proposition}
Let $S$ be a quasi-compact quasi-separated base scheme. Then, the association $\F \mapsto \R^{\sym}_S(\F)$
described in Proposition~\ref{proposition:P1-spectra-intro}
refines to a lax symmetric monoidal functor. In particular, if $\F$ is a Karoubi-localising lax symmetric monoidal functor then $\M^{\sym}_S(\F)$ is a motivic commutative ring spectrum.
\end{proposition}

This means in particular that $\KQ_S$ and $\KW_S$ are motivic commutative ring spectra, and that the canonical map $\KQ_S \to \KW_S$ is a commutative ring map. It also implies that the forgetful map $\KQ_S \to \KGL_S$ is one of motivic commutative ring spectra.

\subsection*{Relation to other work}
In recent work~\cite{Arun-Kumar-arxiv}, based on his PhD thesis, Arun Kumar constructs a cellular motivic spectrum over $\spec(\ZZ)$ whose base change to $\spec(\ZZ[1/2])$ is the motivic hermitian K-theory spectrum of Hornbostel. Over $\spec(\ZZ)$, the invariant represented by this motivic spectrum is however not compared to any known one. This is remedied in an even more recent paper~\cite{Arun-Kumar-Roendigs}, in which Arun Kumar and Röndigs prove that Aron Kumar's construction coincides with the one of the present paper. In particular, this shows that the hermitian K-theory spectrum we construct here is cellular.

\subsection*{Acknowledgements}

We thank Aravind Asok, Jens Hornbostel, Marco Schlichting and Mura Yakerson for fruitful discussions. We wish to extend special thanks to Frédéric Déglise for explaining to us his ideas on purity and to Marc Hoyois for his invaluable insights and for pointing out a mistake in a previous version of the manuscript.

\subsection*{Financial Acknowledgements}

The first author was partially supported by the french ANR project HQDiag (ANR-21-CE40-0015).
The second author was supported by the European Research Council as part of the project ``Foundations of Motivic Real K-Theory'' (ERC grant no. 949583).

\section{Recollections on hermitian K-theory}

We use the framework developed in \cites{9-authors-I,9-authors-II}: hermitian K-theory is a functor sending a Poincaré $\infty$-category, i.e.\ a stable $\infty$-category endowed with a quadratic functor satisfying non-degeneracy conditions, to a spectrum. Let us walk the reader through some of the constructions and properties of this formalism.
 	
\subsection{Poincaré $\infty$-categories} 	
\label{subsection:poincare-categories}%
 	
In the framework developed in~\cite{9-authors-I} and~\cite{9-authors-II}, we view Grothendieck-Witt theory as an invariant of what we call a \emph{Poincaré} $\infty$-\emph{category}, a notion first introduced by Lurie~\cite{Lurie-L-theory} as a novel framework for Ranicki's L-theory. By definition, a Poincaré $\infty$-category is a pair $(\C,\QF)$ consisting of a (small) stable $\infty$-category $\C$ equipped with a \emph{Poincaré structure} $\QF$, that is, a functor $\QF\colon \C\op \to \Spa$ which is quadratic, that is, reduced and $2$-excisive in the sense of Goodwillie's functor calculus, and whose symmetric cross-effect $\Bil_{\QF}\colon \C\op\times \C\op\to\Spa$ is of the form $\Bil_{\QF}(x,y)=\map_{\C}(x,\Dual y)$ for some equivalence 
$\Dual\colon\C\op \tosimeq \C$.
In this case $\Dual$ is uniquely determined and endows $\C$ with the structure of a \emph{perfect duality}, that is, a lift of $\C$ to a $\Ct$-homotopy fixed point in $\Catx$ with respect to the $\mop$-action. 
The notion of a morphism of Poincaré $\infty$-categories $(\C,\QF)\to (\C',\QF')$ is that of a \emph{Poincaré functor}, 
which consists of a pair $(f,\eta)$ where $f\colon \C\to\C'$ is an exact functor and $\eta\colon \QF\Rightarrow \QF'\circ f\op$ is a natural transformation for which the induced arrow 
$
f\Dual\Rightarrow \Dual f\op
$
is an equivalence. In particular, Poincaré functors can be considered as a refinement of duality preserving functors. The collection of Poincaré $\infty$-categories and Poincaré functors between then assembles to form an $\infty$-category we denote by $\Catp$. 

We think of a Poincaré structure as providing an abstract notion of forms on the objects of $\C$. More precisely, we call points $q\in \Omega^{\infty}\QF(x)$ hermitian forms on $x$, and we say that such a form is Poincaré if the map $q_{\sharp}\colon x\to \Dual x$ determined by $q$ is an equivalence.

\begin{example}
\label{example:r-basic}%
\ 
\begin{enumerate}
\item
\label{item:sym-qdr}%
Let $(\C,\Dual)$ be a stable $\infty$-category with perfect duality.
Then the functors
\[
\QF^{\sym}(x)=\map_\C(x,\Dual x)^\hC\qquad \QF^\qdr(x)=\map_\C(x,\Dual x)_\hC
\]
are Poincaré structures with underlying duality $\Dual$. We call them the symmetric and quadratic Poincaré structures associated to the duality $\D$, respectively. The corresponding notions of hermitian forms they encode can be considered as the homotopy coherent avatars of the classical notions of symmetric bilinear and quadratic forms, respectively.
\item
Let $X$ be a scheme and $L$ a line bundle. Then $L$ determines a duality on the perfect derived $\infty$-category $\Dperf(X)$ of $X$, whose underlying duality $\Dperf(X) \to \Dperf(X)\op$ sends $M$ to the internal mapping object $\Dual_L(M) := \uline{\hom}_X(M,L)$. The symmetric and quadratic Poincaré structures associated to this duality are then denoted by $\QF^{\sym}_L$ and $\QF^{\qdr}_L$. More generally, we can introduce a $\Ct$-action on $L$, which affects the structure of $\Dual_L$ as a duality (though not the underlying equivalence). 
The Poincaré $\infty$-categories of the form $(\Dperf(X),\QF^{\sym}_L)$ constitute the main examples of interest for us in the present paper. 
\item
For $\C$ a stable $\infty$-category consider $\ovl{\Hyp}(\C):=\C\times\C\op$ equipped with the Poincaré structure
\[
\QF_\hyp(X,Y):=\map_\C(X,Y)\,.
\]
Then $\Hyp(\C) := (\ovl{\Hyp},\QF_{\hyp})$ is a Poincaré $\infty$-category with underlying duality $\Dual_\hyp(X,Y)=(Y,X)$. We refer to it as the hyperbolic category of $\C$. We note that $\QF_{\hyp}$ is both the symmetric and the quadratic Poincaré structure associated to the duality $\Dual_{\hyp}$ in the sense of~\eqref{item:sym-qdr}.
\item If $(\C,\QF)$ is a Poincaré $\infty$-category with underlying duality $\Dual\colon\C\op \to \C$, then the pair $(\C,\QF\qshift n) := (\C,\Sigma^n\QF)$ is a Poincaré $\infty$-category with underlying duality $\Sigma^n\Dual$. 
\item Let $R$ be a discrete commutative ring. Then, there exists a unique Poincaré structure $\QF^\gs$ on the perfect derived $\infty$-category $\Dperf(R)$ of $R$ whose value on any finitely generated projective module $P$ is the (Eilenberg-Mac Lane spectrum of the) abelian group of symmetric bilinear forms on $P$.
Similarly, there exists a unique Poincaré structure $\QF^\gq$ such that $\QF^\gq(P)$ is the abelian group of quadratic forms on $P$ for such $P$. 
We call these the \emph{genuine symmetric} and \emph{genuine quadratic} Poincaré structures, respectively. Their underlying duality is the usual duality $\Dual_R := \underline{\hom}(-,R)$, but $\QF^{\gs}$ and $\QF^{\gq}$ are generally \emph{not} the associated symmetric and quadratic Poincaré structures in the sense of~\ref{item:sym-qdr}.
\end{enumerate}
\end{example}

\subsection{Hermitian K-theory}
\label{subsection:recall-hermitian}%

Let $(\C,\QF)$ be a Poincaré $\infty$-category.

A \defi{Poincaré object} in $(\C,\QF)$ is a pair $(x,q)$ where $x$ is an object of $\C$ and $q\in \Omega^\infty\QF(x)$ is a Poincaré form on $x$, that is, a hermitian form such that
the map
$q_\sharp\colon X\to \Dual X$
determined by the image of $q$ under $\QF(x)\to \Bil(x,x) = \map_\C(x,\Dual x)$, is an equivalence. 
We then write 
\[
\Poinc(\C,\QF) \subseteq \int_{x \in \core\C}\Omega^{\infty}\QF(x)
\]
for the full subgroupoid spanned by the Poincaré objects, where $\core\C$ denotes the core $\infty$-groupoid of $\C$.
The association $(\C,\QF) \mapsto \Poinc(\C,\QF)$ assembles to a functor $\Poinc\colon \Catp \to \Sps$ by \cite{9-authors-I}*{Lemma~2.1.5}, analogous to the groupoid core functor $\core\colon \Catx \to \Sps$ yielding classes in K-theory, where $\Catx$ is the $\infty$-category of stable $\infty$-categories.
\medskip

In order to describe the context in which hermitian K-theory naturally lives, we need to recall the notions of \emph{additive}, \emph{localising} and \emph{Karoubi-localising} functors.

Both the $\infty$-categories $\Catx$ and $\Catp$ are pointed (in fact, semi-additive), and we can consider fibre and cofibre sequences in them. A \defi{Verdier} or \defi{Poincaré-Verdier} sequence is a sequence
\[
\C \to \D \to \E
\qquad
\text{or}
\qquad
(\C,\QF) \to (\D,\QFD) \to (\E,\QFE)
\]
that is both a fibre and a cofibre sequence in $\Catx$ or $\Catp$ respectively. A Verdier sequence is furthermore \defi{split} when both maps in the sequence have adjoints on both sides. Split Verdier sequences are equivalently described as stable recollements (see \cite{9-authors-II}*{Definition~A.2.9 and Proposition~A.2.10}). A Poincaré-Verdier sequence is \defi{split} when its underlying Verdier sequence is (see \cite{9-authors-II}*{Definition~1.1.1}). 

A functor $\Catp \to \Spa$ is \defi{additive} or \defi{localising} if it sends split Poincaré-Verdier or Poincaré-Verdier respectively to fibre sequences of spectra. Analogous definitions apply to a functor $\Catx \to \Spa$, removing ``Poincaré'' everywhere (see \cite{9-authors-II}*{Definition~1.5.4 and Proposition~1.5.5}).

Hermitian K-theory, also known as the \defi{Grothendieck-Witt spectrum} functor
\[
\GW\colon \Catp \to \Spa 
\]
is the universal (initial) \emph{additive} functor equipped with a transformation of functors $\Poinc \to \Ominfty \GW$ (see \cite{9-authors-II}*{Definition~4.2.1 and Corollary~4.2.2}). Analogously, the K-theory spectrum functor is the universal additive functor $\K\colon \Catx \to \Spa$ equipped with a map $\core \to \Ominfty\K$. A notable consequence of additivity is the Bott-Genauer sequence
\[
\GW(\C,\QF) \to \K(\C) \to \GW(\C,\QF\qshift{1})
\]
which exhibits K-theory as an extension of GW-theory by GW-theory of a shifted Poincaré structure.

Another important invariant of a Poincaré $\infty$-category is its L-theory spectrum $\L(\C,\QF)$, which, in the setting of Poincaré $\infty$-categories, was first defined in~\cite{Lurie-L-theory}, but was extensively studied in more classical contexts by Ranicki. For $n \in \ZZ$ the group $\L_n(\C,\QF) := \pi_n(\C,\QF)$ is the quotient of the monoid of $\pi_0\Poinc(\C,\QF\qshift{n})$ of equivalences classes of $n$-shifted Poincaré objects by the submonoid consisting of the \emph{metabolic Poincaré objects}, that is, those admitting a Lagrangian. As established in~\cite{9-authors-II}, The Grothendieck-Witt, L-theory and K-theory spectra are related via the \emph{fundamental fibre sequence} 
\[
\K_{\hC} \to \GW \to \L .
\]
One can show that both $\K$ and $\L$ are in fact Verdier-localising functors, and hence so is $\GW$
(see \cite{9-authors-II}*{Corollary~4.4.15}). In particular, $\GW$ is also universal as a Verdier-localising invariant.

\subsection{Karoubi-Grothendieck-Witt theory}
\label{subsection:KGW-recall}%

In this article, we actually mostly use the Karoubi-Grothendieck-Witt functor, i.e.\ a modified version of the above that is insensitive to idempotent completion of the underlying stable $\infty$-category. Let us now describe it.

A \defi{Karoubi} or a \defi{Poincaré-Karoubi} sequence is a sequence in $\Catx$ or $\Catp$, respectively, that becomes a fibre-cofibre sequence in idempotent complete stable $\infty$-categories or idempotent complete Poincaré categories, respectively, after idempotent completion of its terms (see \cite{9-authors-II}*{Definition~1.3.6}). 
A functor $\Catx \to \Spa$ or $\Catp \to \Spa$ is called \defi{Karoubi-localising} if it sends Karoubi or Poincaré-Karoubi sequences, respectively, to fibre sequences (see \cite{9-authors-II}*{Definition~2.7.1 and Proposition~1.5.5}). In~\cite[\S 1]{9-authors-IV}, we construct universal Karoubi-localising approximations
\[
\GW \Rightarrow \KGW \quad\text{and}\quad \L \Rightarrow \KL
\]
of Grothendieck-Witt- and L-theory, respectively. These are the analogues of the map $\K \Rightarrow \KK$ from algebraic to Bass K-theory, which is the universal Karoubi-localising approximation of K-theory, see~\cite{BGT}. These Karoubi-localising variants fit to form analogues of the Bott-Genauer sequence
\[
\KGW \Rightarrow \KK \Rightarrow \KGW((-)\qshift{1})
\]
and fundamental fibre sequence
\[
\KK_{\hC} \Rightarrow \KGW \Rightarrow \KL .
\]
In addition, we prove in~\cite[\S 1]{9-authors-IV} that the squares
\begin{equation}
\label{equation:cofinality}%
\begin{tikzcd}
\GW(\C,\QF)\ar[r]\ar[d] & \KGW(\C,\QF)\ar[d] \\
\K(\C)^\hC\ar[r] & \KK(\C)^\hC 
\end{tikzcd}
\quad\text{and}\quad
\begin{tikzcd}
\L(\C) \ar[r] \ar[d] & \KL(\C,\QF) \ar[d] \\
\K(\C)^\tC \ar[r] & \KK(\C)^{\tC} \ .
\end{tikzcd}
\end{equation}
are exact, a result we call \emph{Karoubi cofinality}.
This means, in particular, that the map $\GW(\C,\QF) \to \KGW(\C,\QF)$ is an equivalence on positive homotopy groups, and also on $\pi_0$ if $\C$ is idempotent complete. In addition, if $\C$ is such that the map $\K(\C) \to \KK(\C)$ is an equivalence then the maps $\GW(\C,\QF) \to \KGW(\C,\QF)$ and 
$\L(\C,\QF) \to \KL(\C,\QF)$ are equivalences as well.

For a scheme $X$ with a line bundle $L$, the spectra 
\[
\KGW^{\sym}(X,L) := \KGW(\Dperf(X),\QF^{\sym}_L) \quad\text{and}\quad \KL^{\sym}(X,L) := \KL(\Dperf(X),\QF^{\sym}_L)
\]
will be our principal object of interest in the present paper. When $X$ is regular the map $\K(X) \to \KK(X)$ is an equivalence and hence by the above these coincide with 
\[
\GW^{\sym}(X,L) = \GW(\Dperf(X),\QF^{\sym}_L) \quad\text{and}\quad \L^{\sym}(X,L) := \L(\Dperf(X),\QF^{\sym}_L),
\]
respectively.
The specific role played by Poincaré-Karoubi sequences is due to the fact that the localisation sequence of derived categories
\[
\Dperf_Z(X) \to \Dperf(X) \to \Dperf(U)
\]
associated to a closed subscheme $Z \subset X$ with open complement $U$ is only a Karoubi sequence and not a Verdier sequence in general. As was emphasized in \cite{thomason-trobaugh}, although all categories are idempotent complete, the map $\Dperf(X) \to \Dperf(U)$ is not essentially surjective in general: an object in the target is generally only a direct factor of an object in the image.

\subsection{Variants of Poincaré structures}
\label{subsection:variants}%

The notion of a Poincaré $\infty$-category admits several closely related variants. In this subsection we recall from~\cite{9-authors-I} a few of these that we will need in later parts of the paper. Some of these serve a primarily technical role, while others also a conceptual one, notably that of a stable $\infty$-category with just a duality (as opposed to a full Poincaré structure).

As recalled in \S\ref{subsection:poincare-categories}, a Poincaré $\infty$-category is a pair $(\C,\QF)$ consisting of a stable $\infty$-category and a quadratic functor $\QF$ whose associated symmetric cross effect $\Bil_{\QF}\colon \C\op \times \C\op \to \Spa$ is perfect, that is, is of the form $\Bil_{\QF}(x,y)=\map_{\C}(x,\Dual y)$ for some equivalence $\Dual\colon\C\op \tosimeq \C$. To allow for more flexibility in manipulating these objects, it is convenient to also consider the weaker notion where we just require $\QF$ to be quadratic, but do not enforce any further conditions on $\Bil_{\QF}$. We will refer to such a pair $(\C,\QF)$ as a \defi{hermitian $\infty$-category}. The notion of a morphism of hermitian $\infty$-categories is then just that of a pair $(f,\eta)$ where $f\colon \C\to\C'$ is an exact functor and $\eta\colon \QF\Rightarrow \QF'\circ f\op$ is a natural transformation.
We refer to these as hermitian functors. In particular, a Poincaré $\infty$-category is a special case of a hermitian $\infty$-category, and a Poincaré functor is a special case of a hermitian functor. Formally, if we write $\Funq(\C) \subseteq \Fun(\C\op,\Spa)$ for the full subcategory spanned by the quadratic functors, then the collection of hermitian $\infty$-categories can be organized into an $\infty$-category $\Cath$ arising as the unstraightening of the functor $\C \mapsto \Funq(\C)$, and $\Catp$ sits inside $\Cath$ as a certain non-full subcategory. 

One may go further and endow a stable $\infty$-category $\C$ with an even weaker structure, just that of a symmetric bilinear functor $\Bil\colon \C\op \times \C\op \to \Spa$, that is, a functor which is exact in each variable separately and that carries a $\Ct$-equivariant structure with respect to the swap action on the domain and the trivial action on the target. A pair $(\C,\Bil)$ consisting of a stable $\infty$-category and a symmetric bilinear functor is called a \defi{symmetric bilinear $\infty$-category}, see~\cite[\S 7.2]{9-authors-I}. Their collection can similarly be organized into an $\infty$-category $\Catsb$ where the morphisms are the pairs $(f,\beta)\colon (\C,\Bil) \to (\C',\Bil')$ where $f\colon \C \to \C'$ is an exact functor and $\beta\colon \Bil \Rightarrow \Bil \circ (f\op \times f\op)$ is a natural transformation. Formally, $\Catsb$ is the unstraightening of the functor which associates to $\C$ the full subcategory $\Funs(\C) \subseteq \Fun(\C\op \times \C\op,\Spa)^{\hC}$ spanned by the bilinear $\Ct$-equivariant functors. Extracting symmetric cross effects then assembles to give a functor
\[
\Cath \to \Catsb \quad\quad (\C,\QF) \mapsto (\C,\Bil_{\QF})
\]
which preserves the forgetful functor to $\Catx$ on both sides. In fact, these forgetful functors are both cartesian and cocartesian fibrations, and the above functor preserves both cartesian and cocartesian arrows.

Recall from~\cite[Definition 7.2.13]{9-authors-I} that a symmetric bilinear $\infty$-category $(\C,\Bil)$ is said to be non-degenerate if $\Bil$ is non-degenerate, that is, the functor $\C\op \to \Ind(\C)$ determines by $\Bil$ takes values in $\C \subseteq \Ind(\C)$, and hence determines a well-defined functor $\Dual_{\Bil}\colon \C\op \to \C$ equipped with an equivalence $\Bil(x,y) \simeq \map(x,\Dual_{\Bil}(y))$. Note that $\Dual_{\Bil}$ is generally not an equivalence. If $(f,\beta)\colon (\C,\Bil) \to (\C',\Bil')$ is morphism of symmetric bilinear $\infty$-categories which are both non-degenerate then $\beta\colon \Bil \Rightarrow f^*\Bil'$ determines a natural transformation $f\Dual_{\Bil} \Rightarrow \Dual_{\Bil'}f\op$, and we say that $(f,\eta)$ is duality preserving if this natural transformation is an equivalence. Finally, we say that a non-degenerate $(\C,\Bil)$ (or just $\Bil$) is perfect if $\Dual_{\Bil}$ is an equivalence. We also note that if $\Bil$ is non-degenerate then $\Dual_{\Bil}$ is canonically right adjoint to $\Dual_{\Bil}\op$, with the unit and counit both determined by the same double dual natural transformation $\id \Rightarrow \Dual_{\Bil}\Dual_{\Bil}\op$, so that $(\C,\Bil)$ is perfect if and only if this double dual transformation is an equivalence. We then write $\Catps \subseteq \Catsb$ for the (non-full) subcategory spanned by the perfect symmetric bilinear $\infty$-categories and duality preserving morphisms between them. 

The $\infty$-categories presented thus far fit in a cartesian square
\[
\begin{tikzcd}
\Catp \ar[r]\ar[d] & \Catps \ar[d] \\
\Cath \ar[r] & \Catsb
\end{tikzcd}
\]
whose vertical arrows are subcategory inclusions. As shown in~\cite[Corollary 7.2.16]{9-authors-I}, the forgetful functor $\Catps \to \Catx$ lifts to an equivalence $\Catps \simeq (\Catx)^{\hC}$ where the $\Ct$-homotopy fixed points are taken with respect to the $\mop$-action on $\Catx$. Put differently, the data of a perfect symmetric bilinear functor on $\C$ is equivalent to that of a duality on $\C$. 
The identification of $\Catps$ with $(\Catx)^{\hC}$ is done via the intermediary $\infty$-category $\Catb$, which is the unstraightening of the functor $\Catx \times \Catx \to \CAT$ sending $(\cA,\cB)$ to the $\infty$-category $\Funbil(\cA,\cB;\Spa)$ of bilinear functors $\cA\op \times \cB \to \Spa$. There is a natural $\Ct$-action on $\Catb$ sending $(\cA,\cB,\Bil)$ to $(\cB\op,\cA\op,\Bil)$, and one can identify $\Catsb$ with the $\Ct$-homotopy fixed points of this action. Similarly, $\Catps$ identifies with the $\Ct$-homotopy fixed points of a certain subcategory $\Catpb \subseteq \Catb$ spanned by those $(\cA,\cB,\Bil)$ such that $\Bil(a,b) = \map_{\cB}(f(a),b)$ for some equivalence 
$f\colon \cA \xrightarrow{\simeq} \cB$, and a suitable class of morphisms between them. At the same time, one can show that the projection $(\cA,\cB,\Bil) \mapsto \cA$ gives an equivalence $\Catpb \simeq \Catx$ under which the $\Ct$-action on $\Catpb$ becomes the $\mop$-action on $\Catx$, so that we get an equivalence $\Catps \simeq (\Catx)^{\hC}$. In the present paper we will generally not distinguish between $\Catps$ and $(\Catx)^{\hC}$, and consistently use the notation $(\C,\Dual)$ to describe objects in $\Catps$, thus making reference to the duality on $\C$ rather than to the associated perfect bilinear functor $\Bil_{\Dual} = \map_{\C}(-,\Dual(-))$. 

\section{Poincaré structures on derived categories of schemes}
\label{section:poinc-schemes}%

For a scheme $X$, we write $\Picspace(X) \subseteq \core\Dperf(X)$ for the full subgroupoid spanned by the tensor invertible objects, and $\Picspace(X)^{\BC} := \Fun(\BC,\Picspace)$ for the corresponding $\infty$-category of such objects equipped with a $\Ct$-action.
To an $L \in \Picspace(X)^{\BC}$, one may associate a Poincaré structure $\QF^{\sym}_L$ on the $\infty$-category $\Dperf(X)$ of perfect complexes as follows:
\begin{definition}
\label{definition:linear-part}%
Given $X$ and $L \in \Picspace(X)^{\BC}$ as above, we define $\QF^{\sym}_L\colon \Dperf(X)\op \to \Spa$ to be the Poincaré structure given by
\[
\QF^{\sym}_L(M) = \map_{\Dperf(X)}(M \otimes M,L)^{\hC} .
\]
\end{definition}

The fact that this is a Poincaré structure will be proven below, see Proposition~\ref{proposition:construction-gives-poincare}.
We refer to $\QF^{\sym}_L$ as the associated \defi{symmetric} Poincaré structure. We then write $\LT_L$ for the quasi-coherent complex characterized by the existence of a natural equivalence
\[
\map_X(P\otimes P,L)^\tC\cong \map_X(P,\LT_L)
\]
for $P \in \Dperf(X)$. In other words, $\LT_L$ is the object of $\Der^{\qc}(X) = \Ind(\Dperf(X))$ representing the linear part of $\QF^{\sym}_L$. Similarly, for every $m$, we may consider the truncated Poincaré structure $\QF^{\geq m}_L$ on $\Dperf(X)$, given as the fibre product $\QF^{\geq m}_L(M) = \QF^{\sym}_L(M) \times_{\map_X(P,\LT_L)} \map_X(P,\tau_{\geq m}\LT_L)$

\begin{remark}
\label{remark:shift-of-line-bundle}%
The difference between the notion of an invertible perfect complex and that of a line bundle is not huge.
By the main result of~\cite{fausk}, for any invertible perfect complex $L$ on $X$, there exist a disjoint union decomposition $X = \coprod_{n \in \ZZ} X_n$ (where each $X_n$ is potentially empty) such that $L|_{X_n}$ is a degree $n$ shift of a line bundle. In particular, if $X$ is connected then any invertible perfect complex $L$ is globally a shift of a line bundle. 
\end{remark}

While the definition itself is simple, some effort is required in order to formally set up all the functoriality and multiplicativity properties of this construction in $X$ and $L$ (see, e.g., the commutative diagram~\eqref{equation:functoriality-schemes-1} below and subsequent discussion), 
and the present section is dedicated to these somewhat technical considerations. We begin in \S\ref{subsection:mult-hermitian-bilinear} by recalling and elaborating on the symmetric monoidal structure on $\Catp$ and some of its variants recalled in \S\ref{subsection:variants}. 
In \S\ref{subsection:rigid-to-poinc}, we analyse the formation of the symmetric Poincaré structure associated to a rigid stable symmetric monoidal $\infty$-category, in \S\ref{subsection:scheme-to-rigid} we focus on the dependency on $X$ and $L$ and in \S\ref{subsection:genuine} we consider the genuine variants. 
We also note that while the present discussion is sufficient for our needs in \S\ref{section:nisnevich}-\ref{section:projective-bundle}, in \S\ref{section:motivic-realization}, we eventually need a stronger multiplicativity property, which we discuss in \S\ref{subsection:more-derived}.

\subsection{Multiplicative structures on hermitian and bilinear categories}
\label{subsection:mult-hermitian-bilinear}%

In what follows, it will be important to track the various multiplicative structures attached to the passage from a scheme equipped with a line bundle to the associated derived Poincaré $\infty$-category $(\Dperf(X),\QF^{\sym}_L)$. In view of this, in this subsection we recall and further pursue the construction of the symmetric monoidal structure on $\Catp$ and several of its variants. This content can be considered as supplementary material to~\cite[\S 5, \S 7.2]{9-authors-I}.

To begin, recall that the $\infty$-category $\Catx$ of stable $\infty$-categories and exact functors between them admits a symmetric monoidal structure, where for $\cA,\cB \in \Catx$, their tensor product $\cA \otimes \cB$ is the initial recipient of a functor $\cA \times \cB \to \cA \otimes \cB$ which is exact in each variable.
The mapping spectra in $\cA \otimes \cB$ between pure tensors can be described as follows. For any pair of tuples $(a,b),(a',b') \in \cA \times \cB$, the map 
\[
\map_{\cA}(a,a') \otimes \map_{\cB}(b,b') \to \map_{\cA \otimes \cB}(a \otimes b,a' \otimes b')
\]
induced by the exact functors $(-) \times b\colon \cA \to \cA \otimes \cB$ and $a'\otimes (-) \otimes \cB \to \cA \otimes \cB$, is an equivalence.

Recall from~\cite{HA}*{Definition 2.1.2.13} that a map of $\infty$-operads $q\colon \cP^{\otimes} \to \cO^{\otimes}$ is said to be a cocartesian fibration if it is a cocartesian fibration on the level of total categories (that is, when forgetting the structure map to $\Fin_*$). Such a map is also said to exhibit $\cP^{\otimes}$ as an $\cO$-monoidal $\infty$-category. In particular, $\cP^{\otimes}$ is a symmetric monoidal $\infty$-category as soon as $\cO^{\otimes}$ is such. In this case, we will say that $q$ is a \defi{symmetric monoidal cocartesian fibration}. In addition, recall from~\cite{HA}*{Remark 2.4.2.6} that for a symmetric monoidal $\infty$-category $\cB^{\otimes}$ there is an unstraightening equivalence between the $\infty$-category $\Alg_{\cB}(\Cat^{\times})$ of lax symmetric monoidal functors $\cB^{\otimes} \to \Cat^{\times}$ and the subcategory of $\mathrm{Op}_{/\cB^{\otimes}}$ spanned by the symmetric monoidal cocartesian fibrations $\E^{\otimes} \to \cB^{\otimes}$ and cocartesian edge preserving (hence in particular symmetric monoidal) functors between them.

This unstraightening equivalence can be used to construct the symmetric monoidal structure on $\Cath$ and $\Catp$. Since we will need to employ this procedure at several points, let us take the reader through some of the details. To begin, we note that by~\cite[Remark 4.8.1.8]{HA} applied for the collection of all small diagrams we have that the functor
\[
\Cat \to \PrL \quad\quad \C \mapsto \cP(\C) := \Fun(\C\op,\Sps)
\]
is symmetric monoidal. To avoid confusion, we point out that the functoriality in question is covariant in $\C$ via left Kan extension. Tensoring with the idempotent object $\Spa \in \PrL$, we similarly obtain a symmetric monoidal functor
\[
\Cat \to \PrL \quad\quad \C \mapsto \cP(\C)\otimes \Spa := \Fun(\C\op,\Spa).
\]
Pre-composing with the lax symmetric monoidal functor $(\Catx)^{\otimes} \to \Cat^{\times}$ extracting underlying $\infty$-categories and post-composing with the lax symmetric monoidal functor $\PrL \to \CAT$ extracting underlying large $\infty$-categories we may view the functor $\C \mapsto \Fun(\C\op,\Spa)$ as a lax symmetric monoidal functor $(\Catx)^{\otimes} \to \CAT^{\times}$. Now for $\C \in \Catx$, the full subcategory $\Funq(\C) \subseteq \Fun(\C\op,\Spa)$ spanned by the quadratic functors is reflective, with associated localisation functor 
\[
\Lqdr_{\C}\colon \Fun(\C\op,\Spa) \to \Funq(\C)
\]
given by reduced 2-excisive approximation. We refer to the maps sent to equivalences by $\Lqdr_{\C}$ as quadratic equivalences, and denote their collection by $W_{\qdr}$. Now for any exact $\C \to \D$ the induced restriction functor $\Fun(\D\op,\Spa) \to \Fun(\C\op,\Spa)$ sends $\Funq(\D)$ to $\Funq(\C)$ and hence the associated left Kan extension functor $\Fun(\C\op,\Spa) \to \Fun(\D\op,\Spa)$ preserves quadratic equivalences.
A similar argument shows that the structure maps of the lax symmetric monoidal structure on $\Fun(-\op,\Spa)$ preserves quadratic equivalences.
Since forming localisations is product preserving (\cite[Proposition 4.1.7.2]{HA}) we may consequently apply it to obtain a lax symmetric monoidal refinement of the functor $\Funq(-) = \Fun(\C\op,\Spa)[W_{\qdr}^{-1}]$, together with a refinement of $\Lqdr_{(-)}$ to a lax symmetric monoidal natural transformation
\[
\Lqdr_{(-)}\colon \Fun((-)\op,\Spa) \to \Funq(-)
\]
which exhibits $\Funq(-)$ as the initial lax symmetric monoidal recipient of a lax symmetric monoidal natural transformation whose components invert quadratic equivalences. Applying the symmetric monoidal unstraightening procedure above to $\Fun((-)\op,\Spa), \Funq(-)$, and $\Lqdr_{(-)}$ we now obtain a two symmetric monoidal cocartesian fibrations over $\Catx$ and a cocartesian edge preserving map between them:
\[
\begin{tikzcd}
(\Catx_{\mop//\Spa})^{\otimes} \ar[rr]\ar[dr] && (\Cath)^{\otimes} \ar[dl] \\
& (\Catx)^{\otimes} &
\end{tikzcd}
\]
yielding in particular the desired symmetric monoidal structure on $\Cath$. One may then verify that the collection of Poincaré $\infty$-categories is closed under tensor products (including the empty one) in $\Cath$, and similarly for the collection of Poincaré functors, so that the symmetric monoidal structure on $(\Cath)^{\otimes}$ restricts to one on $(\Catp)^{\otimes}$, see \cite[Theorem 5.2.7]{9-authors-I}.

\begin{remark}
By (the dual of) \cite[Proposition 7.3.2.6]{HA} the above map of cocartesian fibrations admits a relative right adjoint
\[
\begin{tikzcd}
(\Catx_{\mop//\Spa})^{\otimes} \ar[dr] && (\Cath)^{\otimes} \ar[ll]\ar[dl] \\
& (\Catx)^{\otimes} &
\end{tikzcd}
\]
which is given fibrewise by the inclusion of quadratic functors inside of all functors. This right adjoint is then lax symmetric monoidal, and is more precisely a full inclusion of $\infty$-operads. In particular, we could have also constructed the symmetric monoidal structure on $\Cath$ as a suitable full suboperad of $(\Catx_{\op//\Spa})^{\otimes}$, and then show that this full suboperad is in fact a symmetric monoidal $\infty$-category. This is essentially the approach taken up in~\cite[\S 5]{9-authors-I}.
\end{remark}

We now employ a similar approach in order to endow $\Catsb$ with a symmetric monoidal structure. For this variant, we compose the lax symmetric monoidal functor $\Fun((-)\op,\Spa) \colon \Cat^{\times} \to \CAT^{\times}$ we used above with the product preserving endofunctor $\C \mapsto \C \times \C$ on $\Cat$ to obtain a lax symmetric monoidal functor $\C \mapsto \Fun(\C\op \times \C\op,\Spa)$, which we then restrict along $(\Catx)^{\otimes} \to \Cat^{\times}$ to obtain a lax symmetric monoidal functor on $\Catx$.
Now for every $\C \in \Catx$, we have the full subcategory $\Funbil(\C) \subseteq \Fun(\C\op \times \C\op,\Spa)$ spanned by the bilinear functors, that is, the functors which are exact in each entry. This inclusion again admits a left adjoint
\[
\Lbil_{\C}\colon \Fun(\C\op \times \C\op,\Spa) \to \Funbil(\C) ,
\]
that is, $\Funbil(\C)$ is a left Bousfield localisation of $\Fun(\C\op \times \C\op,\Spa)$ by the collection of maps $W_{\bil}$ inverted by $\Lbil_{\C}$, which we call bilinear equivalences. As in the quadratic case, the collection of bilinear equivalences is preserved under left Kan extension along exact functors and under the lax symmetric monoidal structure maps, so that $\Funbil(\C) = \Fun(\C\op \times \C\op,\Spa)[W_{\bil}^{-1}]$ inherits from $\Fun((-)\op \times (-)\op,\Spa)$ a lax symmetric monoidal structure and the functors $\Lbil_{\C}$ assemble to give a lax symmetric monoidal natural transformation
\[
\Lbil_{(-)}\colon \Fun((-)\op,\Spa) \to \Funx((-)\op,\Spa)
\]
which exhibits its target as the initial recipient of a lax symmetric monoidal natural transformation inverting bilinear equivalences. Now note that the lax symmetric monoidal functor $\C \mapsto \C \times \C$ carries a natural $\Ct$-action given by swapping the terms. This induces a $\Ct$-action on the lax symmetric monoidal functor $\Fun((-)\op \times (-)\op,\Spa)$ which pointwise preserves bilinear equivalences, and hence descends to a $\Ct$-action on the lax symmetric monoidal functor $\Funbil(-)$. We then write $\Funs(-) := \Funbil(-)^{\hC}$ for the corresponding $\Ct$-homotopy fixed points, so that $\Funs(-)$ is the lax symmetric monoidal functor $(\Catx)^{\otimes} \to \CAT^{\times}$ associating to $\C \in \Catx$ the $\infty$-category of symmetric bilinear functors $\C\op \times \C\op \to \Spa$. Applying the symmetric monoidal unstraightening procedure above, we then obtain a symmetric monoidal cocartesian fibration
\[
(\Catsb)^{\otimes} \to (\Catx)^{\otimes}
\]
giving a symmetric monoidal refinement of the $\infty$-category $\Catsb$ of symmetric bilinear $\infty$-categories studied in~\cite[\S 7.2]{9-authors-I}. One may then verify that the collection of perfect symmetric bilinear $\infty$-categories is closed under tensor products (including the empty one) in $\Catsb$, and similarly for the collection of duality preserving functors, so that the symmetric monoidal structure on $(\Catsb)^{\otimes}$ restricts to one on $(\Catps)^{\otimes}$.

\begin{remark}
\label{remark:perfect-duality}%
The identification of $\Catps$ with $(\Catx)^{\hC}$ described in \S\ref{subsection:variants}
extends to the symmetric monoidal setting in a straightforward manner. In particular, identifying $\Funbil(-,-;\Spa)$ with a multiplicative localisation of the lax symmetric monoidal functor $(\cA,\cB) \mapsto \Fun(\cA\op \times \cB,\Spa)$ and using symmetric monoidal unstraightening as above, one can endow $\Catb$ with a symmetric monoidal structure whose $\Ct$-homotopy fixed points is $(\Catsb)^{\otimes}$, and which restricts to a symmetric monoidal structure on $\Catpb$ whose $\Ct$-homotopy fixed points is $(\Catps)^{\otimes}$. By construction, the projection $\Catpb \to \Catx$ is symmetric monoidal, and is hence an equivalence of symmetric monoidal $\infty$-categories, so that one obtains the equivalence $\Catps \simeq (\Catx)^{\hC}$ multiplicatively. 
\end{remark}

The $\infty$-categories $\Cath$ and $\Catsb$ are related by extracting from a hermitian structure its symmetric bilinear part. To write this in a multiplicative manner, note that the diagonal natural transformation $\C \to \C \times \C$, canonically viewed as a lax symmetric monoidal transformation of functors $\Cat^{\times} \to \Cat^{\times}$, induces a symmetric monoidal natural transformation
\[
\Del_!\colon \Fun((-)\op,\Spa) \to \Fun((-)\op \times (-)\op,\Spa)
\]
whose component at $\C$ is given pointwise by left Kan extension along the diagonal $\C \to \C \times \C$. In addition, using the swap $\Ct$-action on the functor $\C \mapsto \C \times \C$ and the corresponding $\Ct$-invariance of the diagonal, we may view the above transformation as a $\Ct$-equivariant symmetric monoidal natural transformation between two symmetric monoidal functors with $\Ct$-action, the latter being trivial on the domain. Pre-composing with the lax symmetric monoidal forgetful functor $(\Catx)^{\otimes} \to \Cat^{\times}$, we may view the above transformation as a $\Ct$-equivariant lax symmetric monoidal one between two lax symmetric monoidal functors $(\Catx)^{\otimes} \to \CAT^{\times}$. Now for $\C \in \Catx$, restriction along $\Del_{\C}\colon \C \to \C \times \C$ sends bilinear functors to quadratic functors, and so left Kan extension along $\Del_{\C}$ sends quadratic equivalences to bilinear equivalences. The natural transformation $\Del_!$ then descends to a $\Ct$-equivariant lax symmetric monoidal transformation
\[
\Bil_{(-)}\colon \Funq(-) \Rightarrow \Funbil(-) ,
\]
where the $\Ct$-action on the domain is again trivial. We hence obtain a lax symmetric monoidal transformation
\[
\Bil_{(-)}\colon \Funq(-) \Rightarrow \Funs(-) ,
\]
which we denote identically. Unwinding the definitions, the component $\Bil_{\C}$ sends a quadratic functor $\QF$ to its underlying bilinear part, equipped with its induced symmetric structure.
Passing to symmetric monoidal unstraightenings, we then obtain a symmetric monoidal functor
\[
\begin{tikzcd}
(\Cath)^{\otimes}\ar[dr]\ar[rr] && (\Catsb)^{\otimes}\ar[dl] \\
& (\Catx)^{\otimes} &
\end{tikzcd}
\]
preserving cocartesian edges over $(\Catx)^{\otimes}$, thus yielding a multiplicative refinement of the extraction of underlying symmetric bilinear $\infty$-categories as constructed in~\cite[\S 7.2]{9-authors-I}. We also note that by the dual of \cite[Proposition 7.3.2.6]{HA}, the above map of cocartesian fibrations admits a right adjoint
\[
\begin{tikzcd}
(\Cath)^{\otimes} \ar[dr] && (\Catsb)^{\otimes} \ar[ll]\ar[dl] \\
& (\Catx)^{\otimes} &
\end{tikzcd}
\]
relative to $(\Catx)^{\otimes}$. On the level of underlying $\infty$-categories, this right adjoint is given by $(\C,\Bil) \mapsto (\C,\QF^{\sym}_{\Bil})$, that is, it associates to a given symmetric bilinear form on $\C$ the corresponding symmetric Poincaré structure. This right adjoint then inherits a lax symmetric monoidal structure. We also show in~\cite[Proposition 7.2.17]{9-authors-I} that it is fully-faithful. Noting that by definition, the symmetric monoidal functor $(\Cath)^{\otimes} \to (\Catsb)^{\otimes}$ sends Poincaré $\infty$-categories to perfect symmetric bilinear $\infty$-categories and Poincaré functors to duality preserving functors, and conversely for the right adjoint $(\Catsb)^{\otimes} \to (\Cath)^{\otimes}$, we consequently obtain:

\begin{corollary}
\label{corollary:adj-catp-catps}%
The above adjunction descends to an adjunction
\[
(\Catp)^{\otimes} \adj (\Catps)^{\otimes}
\]
in which the left adjoint $(\C,\QF) \mapsto (\C,\Bil_{\QF})$ is symmetric monoidal and the right adjoint $(\C,\Dual) \mapsto (\C,\QF^{\sym}_{\Dual})$ is lax symmetric monoidal and fully-faithful.
\end{corollary}

Finally, let us explain another instance of the above construction, where instead of quadratic or bilinear functors we consider linear (that is, exact) functors.
More precisely, for a stable $\infty$-category $\C$ the full subcategory $\Funx(\C\op,\Spa) \subseteq \Fun(\C\op,\Spa)$ spanned by the exact functors is again reflective, with corresponding localisation functor 
\[
\Llin_{\C}\colon \Fun(\C\op,\Spa) \to \Funx(\C\op,\Spa)
\]
given by reduced excisive (or exact) approximation functor. We will refer to the collection $W_{\lin}$ of maps it inverts as linear equivalences. As with quadratic equivalences, the collection of linear equivalences is preserved under left Kan extension along exact functors and under the structure maps of the lax symmetric monoidal structure on $\Fun((-)\op,\Spa)$, so that $\Funx(\C\op,\Spa) = \Fun(\C\op,\Spa)[W_{\lin}^{-1}]$ inherits from $\Fun((-)\op,\Spa)$ a lax symmetric monoidal structure and the functors $\Llin_{\C}$ assemble to give a lax symmetric monoidal natural transformation
\[
\Llin_{(-)}\colon \Fun((-)\op,\Spa) \to \Funx((-)\op,\Spa)
\]
which exhibits its target as the initial recipient of a lax symmetric monoidal natural transformation inverting linear equivalences.

\begin{remark}
\label{remark:two-approach}%
The lax symmetric monoidal forgetful functor $(\Catx)^{\otimes} \to \Cat^{\times}$ determines an oplax square
\[
\begin{tikzcd}
(\Catx)^{\otimes} \ar[d]\ar[rr,"\C \mapsto \C \otimes \C"] && (\Catx)^{\otimes} \ar[d] \\
\Cat^{\times}\ar[urr,Rightarrow, shorten <= 3em, shorten >= 2.2em, pos=0.45] \ar[rr,"\C \mapsto \C \times \C"] && \Cat^{\times}
\end{tikzcd}
\]
in which the horizontal arrows are symmetric monoidal and the vertical arrows are lax symmetric monoidal. The non-invertible 2-cell in this square then encodes the lax symmetric monoidal transformation
\[
(-) \times (-) \Rightarrow (-) \otimes (-)
\]
where both sides are viewed as lax symmetric monoidal functors $\Catx \to \Cat$. This natural transformation then determined a lax symmetric monoidal natural transformation
\[
\Fun((-)\op \times (-)\op) \Rightarrow \Fun((-)\op \otimes (-)\op,\Spa)
\]
whose component at $\C$ are given by left Kan extension along $\C\op \times \C\op \to \C\op \otimes \C\op$. Since the last functor is bilinear the left Kan extension in question sends bilinear equivalences to linear equivalences and hence the above transformation descends to a lax symmetric monoidal natural transformation 
\[
\Funbil(-) \Rightarrow \Funx((-)\op \otimes (-)\op,\Spa) ,
\]
which is an equivalence by the defining property of the tensor product on $\Catx$. In particular, we could have equally well define $\Funbil(-)$ as the composite of the symmetric monoidal endofunctor $\C \mapsto \C \otimes \C$ of $\Catx$ and the lax symmetric monoidal functor $\C \mapsto \Funx(\C\op,\Spa) =\Ind(\C)$ from $\Catx$ to $\CAT$.
\end{remark}

\subsection{Poincaré categories associated to rigid symmetric monoidal structures}
\label{subsection:rigid-to-poinc}%

Recall that a symmetric monoidal $\infty$-category is said to be \defi{rigid} if every object $x \in \C$ is dualisable, that is, admits a dual $x^{\vee} \in \C$ equipped with a unit $1_{\C} \to x \otimes x^{\vee}$ and a counit $x^{\vee} \otimes x \to 1_{\C}$ satisfying the triangle inequalities.
If $\C$ is a rigid $\infty$-category then the association $x \mapsto x^{\vee}$ assembles to form a duality on $\C$ (that is, a $\Ct$-fixed structure on $\C$ with respect to the $\mop$-action of $\Ct$ on $\Cat$) which can be encoded via the perfect (space valued) pairing $(x,y) \mapsto \Map(x \otimes y,1_{\C})$. Furthermore, this duality is symmetric monoidal, that is, it is a $\Ct$-fixed structure also with respect to the $\mop$-action of $\Ct$ on $\CAlg(\Cat)$, and is natural in $\C \in \CAlg(\Cat)$, so that symmetric monoidal functors are canonically duality preserving. A proof of these statements can be found in~\cite{HLAS}*{Theorem 5.11}; more precisely, in loc.\ cit.\ the authors show that associating to each rigid $\infty$-category its duality determines a section of the forgetful functor $\Catr^{\hC} \to \Catr$, yielding a trivialization of the $\mop$-action on the full subcategory $\Catr \subseteq \CAlg(\Cat)$ spanned by the rigid $\infty$-categories. 
In fact, within the proof of~\cite{HLAS}*{Theorem 5.11}, a 
slightly more general construction is considered: given a tensor invertible object $L \in \C$, one may consider the twisted duality $\Dual_L(x) = x^{\vee}\otimes L$, which is associated to the perfect pairing $(x,y) \mapsto \map(x \otimes y,L)$. This construction can be organized into a commutative square of product preserving functors
\begin{equation}
\label{equation:naturality-L}%
\begin{tikzcd}[row sep = 5pt, column sep = 5pt]
 & (\C,L) \ar[rr,mapsto] \ar[d,phantom,"\rotatebox{270}{$\in$}"] && (\C,\Dual_L)\ar[d,phantom,"\rotatebox{270}{$\in$}"] \\
\Catr \times_{\Cat} \Cat_{*/\lax} & \Inv \ar[rr]\ar[dd]\ar[l,phantom,"\supseteq"] && \Cat^{\hC} \ar[dd] \\
& && \\
& \Catr \ar[rr] && \Cat \ ,
\end{tikzcd}
\end{equation}
where $\Cat_{*/\lax} \to \Cat$ is the cocartesian fibration classified by the identity functor $\Cat \to \Cat$ (also called the universal cocartesian fibration, see \cite{HTT}*{\S 3.3.2}), and $\Inv \subseteq \Catr \times_{\Cat} \Cat_{*/\lax}$ is the full subcategory spanned by those $(\C,L)$ such that $L$ is tensor invertible. 
Passing to commutative monoid objects one obtains a functor
\[
\CAlg(\Inv) \simeq \Catr \to \Catr \times_{\CAlg(\Cat)}\CAlg(\Cat)^{\hC} = \Catr^{\hC},
\]
yielding a section of the forgetful functor $\Catr^{\hC} \to \Catr$.

In the present subsection we adapt the above constructions to the setting of stable $\infty$-categories and Poincaré structures, along the way also extending the scope to allow for a $\Ct$-action on $L$. To begin, recall that the symmetric monoidal structure on $\Catx$ induces a "pointwise" symmetric monoidal structure on $\CAlg(\Catx)$ by means of internal mapping objects in $\infty$-operads (see~\cite[Example 3.2.4.4]{HA}), such that the forgetful functor 
\[
\CAlg(\Catx)^{\otimes} \to (\Catx)^{\otimes}
\]
is symmetric monoidal. 
We now construct a lax symmetric monoidal natural transformation
\[
\Funs(-)|_{\CAlg(\Catx)} \Rightarrow \Funx((-)\op,\Spa)^{\BC}|_{\CAlg(\Catx)}
\]
as follows. First, note that the monoidal product functor $m_{\C}\colon \C \times \C \to \C$ for $\C \in \CAlg(\Cat)$ is natural with respect to symmetric monoidal functors, and can hence be viewed as a natural transformation between two product preserving functors $\CAlg(\Cat) \to \Cat$. Restricting along the lax symmetric monoidal forgetful functor $\CAlg(\Catx)^{\otimes} \to \CAlg(\Cat)^{\times}$ we may view it as a lax symmetric monoidal natural transformation between two lax symmetric monoidal functors $\CAlg(\Catx) \to \Cat$. Post-composing with the lax symmetric monoidal functor $\Fun((-)\op,\Spa)$ we hence obtain a lax symmetric monoidal natural transformation
\[
m_!\colon \Fun((-)\op \times (-)\op,\Spa)|_{\CAlg(\Catx)} \to \Fun((-)\op,\Spa)|_{\CAlg(\Catx)} ,
\]
whose component at a given $\C \in \CAlg(\Catx)$ is given by left Kan extension along the monoidal product functor $m_{\C}\colon \C \times \C \to \C$. Now for $\C \in \CAlg(\Catx)$ the monoidal product $m_{\C}$ is bilinear and hence restriction along $m_{\C}$ sends exact functors to bilinear functors. It thus follows that left Kan extension along $m_{\C}$ sends bilinear equivalences to linear equivalences, and hence $m_!$ descends to a lax symmetric monoidal natural transformation
\[
\ovl{m}_!\colon \Funbil(-)|_{\CAlg(\Catx)} \Rightarrow \Funx((-)\op,\Spa)|_{\CAlg(\Catx)}.
\]
In addition, since $m_{\C}\colon \C \times \C \to \C$ is $\Ct$-equivariant with respect to the swap action on the source and the trivial action on $\C$ we have that $m_{\C}$ and $\ovl{m}_{\C}$ inherits the same structure,
and we may consequently consider the induced lax symmetric monoidal natural transformation  
\[
\ovl{m}^{\sym}_!\colon \Funs(-)|_{\CAlg(\Catx)} \Rightarrow \Funx((-)\op,\Spa)^{\BC}|_{\CAlg(\Catx)}
\]
on the level of $\Ct$-homotopy fixed points. Writing $\Zotimes_{\ex} \to \CAlg(\Catx)^{\otimes}$ for the symmetric monoidal unstraightening of the right hand side, the above lax symmetric monoidal natural transformation determines a symmetric monoidal functor
\[
\CAlg(\Catx)^{\otimes} \times_{(\Catx)^{\otimes}}(\Catsb)^{\otimes} \to \Zotimes_{\ex},
\]
which preserves cocartesian edges over $\CAlg(\Catx)$.
We note that this functor admits right adjoints fibrewise over $\CAlg(\Catx)$: indeed, on fibres over $\C^{\otimes}$ the functor in question $\Funs(\C) \to \Funx(\C\op,\Spa)^{\BC}$ admits a right adjoint given by restriction along $m^{\sym}_{\C}\colon \C\op \times \C\op \to \C\op$. By~\cite[Proposition 7.3.2.6]{HA} we conclude that the above functor admits a relative right adjoint
\[
\begin{tikzcd}
\Zotimes_{\ex} \ar[rr]\ar[dr] && \CAlg(\Catx)^{\otimes} \times_{(\Catx)^{\otimes}} (\Catsb)^{\otimes} \ar[dl] \\
&\CAlg(\Catx)^{\otimes} &
\end{tikzcd}
\]
which then inherits a lax symmetric monoidal structure. 

\begin{lemma}
\label{lemma:right-is-monoidal}%
The above right adjoint is symmetric monoidal (not only lax). 
\end{lemma}
\begin{proof}
Unwinding the definitions, this amounts to saying that for every pair of symmetric monoidal $\infty$-categories $\C,\D$, the square
\[
\begin{tikzcd}
\Funs(\C) \times \Funs(\D) \ar[r]\ar[d] & \Funs(\C \otimes \D) \ar[d] \\
\Funx(\C\op,\Spa)^{\BC} \times \Funx(\D\op)^{\BC} \ar[r] & \Funx(\C\op \otimes \D\op,\Spa)^{\BC}
\end{tikzcd}
\]
is vertically right adjointable. Now the right adjoints of the vertical functors exists also before taking homotopy fixed points, and are hence induced by the latter. Since forgetting homotopy fixed point structure is conservative it will suffice to show that the square
\[
\begin{tikzcd}
\Funbil(\C) \times \Funbil(\D) \ar[r]\ar[d] & \Funbil(\C \otimes \D) \ar[d] \\
\Funx(\C\op,\Spa) \times \Funx(\D\op) \ar[r] & \Funx(\C\op \otimes \D\op,\Spa)
\end{tikzcd}
\]
is vertically right adjointable. Identifying $\Funbil(-)$ with $\Funx((-)\op\otimes(-)\op,\Spa) = \Ind((-) \otimes (-))$ (see Remark~\ref{remark:two-approach}), the desired result now follows from the following claim: for every pair of exact functors $f\colon \C' \to \C$ and $g\colon \D' \to  \D$, the associated square
\[
\begin{tikzcd}
\Ind(\C') \times \Ind(\D') \ar[r]\ar[d] & \Ind(\C' \otimes \D') \ar[d] \\
\Ind(\C) \times \Ind(\D) \ar[r] & \Ind(\C \otimes \D)
\end{tikzcd}
\]
is vertically right adjointable. To prove the claim, note that by combining the facts that the horizontal functors and the vertical right adjoints preserve filtered colimits, that equivalences in $\Ind(\C' \otimes \D')$ are jointly detected by mapping spectra out of objects in $\C' \otimes \D'$, and that $\C' \otimes \D'$ is generated under finite colimits by the image of the bilinear functor $\C' \times \D' \to \C' \otimes \D'$, the claim in question amounts to showing that for every $(c,d) \in \C \times \D$ and $(c',d') \in \C'\times \D'$, the map
\[
\colim_{(x',f(x') \to c) \in \C'_{/c}, (y',g(y') \to d) \in \D'_{/d})}\map_{\C'\otimes \D'}(c'\otimes d',x'\otimes y') \to \map_{\C \otimes \D}(f(c')\otimes g(d'),c \otimes d)
\]
is an equivalence of spectra. Now by the pointwise formula for left Kan extension we can also write this map as 
\[
(f \times g)_!\map_{\C'\otimes \D'}(c'\otimes d',(-) \otimes (-)) \Rightarrow \map_{\C\otimes \D}(f(c') \otimes g(d'),(-)\otimes (-))
\]
where $(f \times g)_!$ denotes left Kan extension along $f \times g\colon \C'\times \D'\to \C \times \D$.
In light of the tensor formula for mapping spectra among pure tensors (see \S\ref{subsection:mult-hermitian-bilinear}) the desired result is now a direct consequence of the fact that the maps $f_!\map_{\C'}(c',-) \Rightarrow \map_{\C}(f(c'),-)$ and $g_!\map_{\D'}(d',-) \Rightarrow \map_{\D}(g(d'),-)$ are equivalences.
\end{proof}

With Lemma~\ref{lemma:right-is-monoidal} at hand, from the above relative right adjoint we now obtain a commutative square
\begin{equation}
\label{equation:functoriality}%
\begin{tikzcd}
\Zotimes_{\ex} \ar[r]\ar[d] & (\Catsb)^{\otimes} \ar[d] \\
\CAlg(\Catx)^{\otimes} \ar[r] & (\Catx)^{\otimes}
\end{tikzcd}
\end{equation}
of symmetric monoidal $\infty$-categories and symmetric monoidal functors between them. Concretely, on the level of underlying $\infty$-categories, the objects of the top left corners are pairs $(\C,L)$ where $\C$ is a stable $\infty$-category and $L \colon \C\op \to \Spa$ is a $\Ct$-equivariant functor (with respect to the trivial $\Ct$-action on both sides). 
The top horizontal map sends $(\C,L)$ to $(\C,\Bil_{L})$, where $\Bil_{L}$ is the symmetric bilinear functor obtained as the composite of $\Ct$-equivariant functors
\[
\C\op \times \C\op \xrightarrow{m_{\C}} \C\op \xrightarrow{L}\Spa .
\]
Equivalently, identifying $\Fun(\C\op,\Spa)$ with $\Ind(\C)$, we may think of $L \in \Ind(\C)^{\BC}$ as an Ind-object equipped with a $\Ct$-action, in which case we can write $\Bil_L$ via the formula 
\[
\Bil_L(x,y) = \map_{\Ind(\C)}(x \otimes y,L) ,
\]
where we have identified $x$ and $y$ with their images in $\Ind(\C)$ via the Yoneda embedding.

\begin{proposition}
\label{proposition:construction-gives-poincare}%
Suppose that $\C$ is a rigid stably symmetric monoidal $\infty$-category and $L \in \Ind(\C)^{\BC}$ a $\Ct$-equivariant Ind-object.
Then $(\C,\Bil_L)$ is non-degenerate if and only if $L$ belongs to $\C^{\BC} \subseteq \Ind(\C)^{\BC}$. It is furthermore perfect if and only if the underlying object of $L$ is tensor invertible in $\C$.
In addition, if $g\colon \C\to \C'$ is an exact functor, $L \in \C^{\BC}$ and $L'\in (\C')^{\BC}$ tensor invertible and $g(L)\to L'$ an equivalence, then the induced morphism $(\C,\QF^{\sym}_L)\to (\C',\QF^{\sym}_{L'})$ is duality preserving.
\end{proposition}

\begin{proof}
If $\Bil_L$ is represented by $\Dual\colon \C\op \to \C$ then $\Dual(1_{\C}) = L$, so that $L$ belongs to $\C$. On the other hand, if $L$ belongs to $\C$ then 
\[
\map_{\Ind(\C)}(x\otimes y,L) = \map_{\C}(x \otimes y,L) = \map_{\C}(x, y^{\vee} \otimes L)
\]
by the rigidity of $\C$, so that $\Bil_L$ is represented by 
$\Dual_L(y)=y^\vee \otimes L$.
Moreover, the double duality map $1_\C\to \Dual_L\Dual_L$ is
\[
y\to L\otimes L^\vee\otimes y
\]
induced by the adjoint of the $\Ct$-action $L\to L$, and so it is an equivalence if and only if $L$ is  tensor invertible.
 
Finally, if $L$ and $L'$ are tensor invertible objects of $\C$ and $\C'$ respectively, the natural transformation $g\Dual_L\to \Dual_{L'} g$ is given by
\[
g(L)\otimes fx^\vee\to L'\otimes fx^\vee
\]
and so it is an equivalence if and only if the map $g(L)\to L'$ is an equivalence.
\end{proof}

\begin{definition}
Let $\Inv^{\otimes} \subseteq \Zotimes_{\ex}$ denotes the suboperad spanned by those $(\C^{\otimes},L)$ such that $\C^{\otimes}$ is rigid and $L$ is (the Yoneda image of) a tensor invertible $\Ct$-equivariant object of $\C$, and those multi-operations $((\C_1,L_1),\ldots,(\C_n,L_n)) \to (\C,L)$ corresponding to a multi-exact functor $g\colon \C_1 \times \cdots \times \C_n \to \C$ and an equivalence $g(L_1,\ldots,L_n) \tosimeq L$. 
\end{definition}

By construction, an exact functor $\C_1 \otimes \cdots \otimes \C_n \to \C$ induced by a multi-operation $((\C_1,L_1),\ldots,(\C_n,L_n)) \to (\C,L)$ in $\Inv^{\otimes}$ sends the tensor invertible $\Ct$-equivariant object $L_1 \otimes \cdots \otimes L_n$ to $L$, and hence by Proposition~\ref{proposition:construction-gives-poincare} the associated morphism 
\[
(\C_1,\Bil_{L_1}) \otimes \cdots \otimes (\C_n,\Bil_{L_n}) \to (\C,\Bil_L)
\]
is duality preserving with perfect domain and target. 
As a consequence, the square~\eqref{equation:functoriality} restricts to a square of $\infty$-operads
\begin{equation}
\label{equation:functoriality-ps}%
\begin{tikzcd}[row sep = 5pt]
(\C,L) \ar[r,mapsto] \ar[d,phantom,"\rotatebox{270}{$\in$}"] & (\C,\Bil_L)\ar[d,phantom,"\rotatebox{270}{$\in$}"] \\
\Inv^{\otimes}\ar[dd]\ar[r] & (\Catps)^{\otimes} \ar[dd] \\
&\\
(\Catexr)^{\otimes} \ar[r]\ar[uur,dashed] & (\Catx)^{\otimes} \ .
\end{tikzcd}
\end{equation}
equipped with a dotted lift which we construct below. We now claim that all $\infty$-operads in this square are symmetric monoidal $\infty$-categories and all arrows are symmetric monoidal functors. Indeed, as we discussed in~\ref{subsection:mult-hermitian-bilinear}, the non-full suboperad $(\Catps)^{\otimes}$ of $(\Catsb)^{\otimes}$ is a symmetric monoidal $\infty$-category and the projection $(\Catps)^{\otimes} \to (\Catx)^{\otimes}$ is a symmetric monoidal functor.
Similarly, the full suboperad $(\Catexr)^{\otimes} \subseteq \CAlg(\Catx)^{\otimes}$ is a symmetric monoidal $\infty$-category and the forgetful functor $(\Catexr)^{\otimes} \to (\Catx)^{\otimes}$ is symmetric monoidal; indeed this follows from the fact that 
if $\C,\D$ are rigid then $\C \otimes \D$ is again rigid, since it receives a symmetric monoidal (bilinear) functor $\C \times \D \to \C \otimes \D$ whose image generates $\C \otimes \D$ under finite colimits, and all objects in $\C \times \D$ are dualisable.
In addition, an arrow in $\Zotimes_{\ex}$ whose domain is in $\Inv^{\otimes}$ lies entirely in $\Inv^{\otimes}$ if and only if it is cocartesian over $\CAlg(\Catx)$, so that 
\[
\Inv^{\otimes} \to \Catr^{\otimes}
\]
is a left fibration of $\infty$-operads. In particular, since $\Catr^{\otimes}$ is a symmetric monoidal $\infty$-category we conclude that $\Inv^{\otimes}$ is a symmetric monoidal $\infty$-category and the left vertical map is symmetric monoidal. Finally, since the right vertical functor is conservative by~\cite[Corollary 7.2.16]{9-authors-I} we have the top horizontal arrow must also be a symmetric monoidal functor.

The square~\eqref{equation:functoriality-ps} is the version of~\eqref{equation:naturality-L} relevant for our purposes. 
It admits a symmetric monoidal lift (indicated by a dotted arrow) induced by the symmetric monoidal section of the left vertical arrow sending $\C^{\otimes}$ to $(\C^{\otimes},\one_{\C})$. To see this, note that since the unit $\one_{\C}$ is always tensor invertible it will suffice to construct such a section for the pulled back left fibration 
\[
\CAlg(\Catx)^{\otimes} \times_{(\Catx)^{\otimes}} \Alg_{\EE_0}(\Catx)^{\otimes} \to \CAlg(\Catx)^{\otimes}
\]
where $\EE_0$ is the $\infty$-operad of pointed objects (that is, objects equipped with a map from the unit), and we consider the $\infty$-category of algebras as endowed with the pointwise symmetric monoidal structure (that is, as the corresponding internal mapping object in $\infty$-operads, see~\cite[Construction 3.2.4.1]{HA}).
The symmetric monoidal section $\C^{\otimes} \mapsto (\C^{\otimes},\one_{\C})$ can hence be constructed as the composite
\[
\CAlg(\Catx)^{\otimes} \simeq \Alg_{\EE_0}(\CAlg(\Catx))^{\otimes} \to \CAlg(\Catx)^{\otimes} \times_{(\Catx)^{\otimes}} \Alg_{\EE_0}(\Catx)^{\otimes}
\]
where the first equivalence is a chosen inverse for the projection 
$\Alg_{\EE_0}(\CAlg(\Catx)) \xrightarrow{\simeq} \CAlg(\Catx)$, which is well-defined 
since the unit is initial in $\CAlg(\Catx)^{\otimes}$. In particular, the dotted lift in the above square is given by the formula $\C^{\otimes} \mapsto (\C,\Dual_{\C})$, where $\Dual_{\C} = \Dual_{\one}$ is the canonical duality associated to $\C$ as a rigid $\infty$-category.

\begin{proposition}
\label{proposition:sm-sifted}%
In the square~\eqref{equation:functoriality-ps}, the underlying $\infty$-categories all admit sifted colimits and these are preserved by all the functors in the square (including the dotted lift). 
\end{proposition} 
\begin{proof}
First recall that $\Catx$ has all colimits~\cite[Proposition 6.1.1]{9-authors-I} and so $\Catps \simeq (\Catx)^{\hC}$ has all colimits as well, and these are preserved by the forgetful functor $\Catps \to \Catx$. Now since the tensor product on $\Catx$ preserves colimits in each variable (indeed, the operation $\C \otimes (-)$ admits a right adjoint $\Funx(\C,-)$ by its defining mapping property) we have that $\CAlg(\Catx)$ has all colimits as well and that the forgetful functor $\CAlg(\Catx) \to \Catx$ preserves sifted colimits. We now claim that the full subcategory $\Catexr \subseteq \CAlg(\Catx)$ is closed under sifted colimits. For this, note that since the forgetful functor $\CAlg(\Catx) \to \Catx$ is conservative, it also detects sifted colimits, that is, sifted colimits in $\CAlg(\Catx)$ are calculated in $\Catx$. This means that if $\chi\colon \I \to \CAlg(\Catx)$ is a sifted diagram with colimit $\C^{\otimes} = \colim_\I \chi$, then $\C$ is the colimit of the underlying $\Catx$-valued diagram, and so the collection of essential images $\im[\chi(i) \to \C]$ generates $\C$ as a stable $\infty$-category. Now if we assume that each $\chi(i)$ is rigid then each of the objects in $\im[\chi(i) \to \C]$ is dualisable, and hence any object in $\C$ is dualisable since in any stable symmetric monoidal $\infty$-category the collection of dualisable objects forms a stable subcategory. We hence conclude that $\Catexr$ admits sifted colimits and that these are preserved by the forgetful functor $\Catexr \to \Catx$.
Finally, the left vertical arrow $\Inv \to \Catexr$ is the left fibration classified by the functor 
\[
\Catexr \to \Sps
\]
sending a rigid stably symmetric monoidal $\infty$-category $\C^{\otimes}$ to the $\infty$-groupoid of tensor invertible objects in $\C$. Since any sifted category is weakly contractible and groupoids admit all colimits indexed by weakly contractible $\infty$-categories we conclude that $\infty$-groupoids admit sifted colimits, and hence the domain $\Inv$ of the left fibration $\Inv \to \Catexr$ admit sifted colimits and these are preserved by the projection to $\Catexr$. Since the forgetful functor $\Catps \to \Catx$ is conservative it now follows that the functor $\Inv \to \Catps$ preserves sifted colimits as well. A similar argument applies to the dotted lift.
\end{proof}

\subsection{Derived Poincaré categories of schemes}
\label{subsection:scheme-to-rigid}%

Let $\qSch$ denote the category of quasi-compact quasi-separated (qcqs for short) schemes.
As recalled in Appendix~\ref{subsection:perfect}, the operation which associates to each qcqs scheme $X$ its (symmetric monoidal) perfect derived $\infty$-category $\Dperf(X)^{\otimes}$ and to each map of schemes $f\colon X \to Y$ the corresponding (symmetric monoidal) pullback functor $f^*$ assembles to give a functor
\[
\qSch\op \to \CAlg(\Catx) \quad\quad X \mapsto \Dperf(X)^{\otimes}. 
\]
Endowing both sides with the cocartesian symmetric monoidal structure, which on the left hand side corresponds to cartesian products of schemes and on the right hand side to the pointwise tensor product of stably symmetric monoidal $\infty$-categories \cite[Proposition 3.2.4.10]{HA}, the functor $\Dperf(-)^{\otimes}$ carries an essentially unique lax symmetric monoidal structure, see~\cite[Proposition 2.4.3.8]{HA}. In addition, since each $\Dperf(X)$ is rigid as a symmetric monoidal $\infty$-category the functor above factors through the full subcategory $\Catexr \subseteq \CAlg(\Catx)$ spanned by the rigid $\infty$-categories (which, as explained in \S\ref{subsection:rigid-to-poinc}, is closed in $\CAlg(\Catx)$ under the pointwise tensor product).

\begin{construction}
Consider the fibre product 
\[
(\qSch\op)^{\amalg} \times_{(\Catexr)^{\otimes}}\Inv^{\otimes}
\]
realized in the $\infty$-category of $\infty$-operads. Since $\Inv^{\otimes} \to (\Catexr)^{\otimes}$ is a symmetric monoidal cocartesian fibration this fibre product is actually a symmetric monoidal $\infty$-category, and we define $\qSchPic^{\otimes}$ to be its opposite symmetric monoidal $\infty$-category. In particular, the objects of $\qSchPic$ are pairs $(X,L)$ where $X$ is a qcqs scheme and $L \in \Picspace(X)^{\BC}$ is an invertible perfect complex with $\Ct$-action, and the symmetric monoidal product is given by $(X,L) \otimes (X',L') = (X \times X',L \boxtimes L')$, where $L \boxtimes L'$ is the tensor product of the pullbacks of $L$ and $L'$ to $X \times X'$. 
\end{construction}

Combining the above with the square of symmetric monoidal $\infty$-categories~\eqref{equation:functoriality-ps}
we now obtain a diagram of symmetric monoidal $\infty$-categories
\begin{equation}
\label{equation:functoriality-schemes-1}%
\begin{tikzcd}[row sep = 5pt]
(X,L) \ar[d,phantom,"\rotatebox{270}{$\in$}"]\ar[r,mapsto] & (\Dperf(X),L) \ar[r,mapsto] \ar[d,phantom,"\rotatebox{270}{$\in$}"] & (\Dperf(X),\Dual_L)\ar[d,phantom,"\rotatebox{270}{$\in$}"] \\
(\qSchPic\op)^{\otimes} \ar[r]\ar[dd] & \Inv^{\otimes}\ar[dd]\ar[r] & (\Catps)^{\otimes} \ar[dd] \\
&&\\
(\qSch\op)^{\amalg} \ar[r] & (\Catexr)^{\otimes} \ar[r] & (\Catx)^{\otimes} \ .
\end{tikzcd}
\end{equation}
where the left square is cartesian and all the functors are symmetric monoidal except the two horizontal functors in the left square, which are only lax symmetric monoidal. 

Let us now fix a base qcqs scheme $S \in \qSch$. The projection
\[
\Ar(\qSchPic\op)^{\otimes} \to (\qSch\op)^{\amalg} \quad\quad [(X,L') \to (S,L)] \mapsto S
\]
is then a symmetric monoidal cocartesian fibration, where the arrow category is endowed with the pointwise symmetric monoidal structure. Since (any) $S$ is canonically an algebra object in $(\qSch\op)^{\amalg}$, the fibre over $S$ inherits a symmetric monoidal structure such that the fibre inclusion
\[
\Ar(\qSchPic\op)^{\otimes} \times_{(\qSch\op)^{\amalg}} \{S\} \to \Ar(\qSchPic\op)^{\otimes}
\]
is lax symmetric monoidal. At the same time, since $\qSchPic \to \qSch$ is a right fibration the forgetful  functor 
\[
\Ar(\qSchPic\op) \times_{\qSch\op} \{S\} \to \qSch_{S} \times \Picspace(S)^{\BC}
\]
is an equivalence, where $\qSch_S := \Ar(\qSch) \times_{\qSch} \{S\}$ is the category of qcqs $S$-schemes. Choosing an inverse, we hence obtain a composed lax symmetric monoidal functor
\[
\begin{tikzcd}[row sep=1ex]
(\qSch_{S}\op)^{\amalg} \times (\Picspace(S)^{\BC})^{\otimes} \ar[r] & \Ar(\qSchPic\op)^{\otimes} \ar[r] & (\qSchPic\op)^{\otimes} \\
(p\colon X \to S,L) \ar[r,mapsto] & \left[(X,p^*L) \to (S,L)\right] \ar[r,mapsto] & (X,p^*L) ,
\end{tikzcd}
\]
which we can compose with the formation of Poincaré derived $\infty$-categories as above to obtain a lax symmetric monoidal functor
\[
(\qSch_S\op)^{\amalg} \times (\Picspace(S)^{\BC})^{\otimes} \to (\Catps)^{\otimes} \quad\quad (p\colon X \to S,L) \mapsto (\Dperf(X),\Dual_{p^*L}).
\]
We note that since every object is canonically a module over the unit (which in the case of the left hand side, is given by $(S,\cO_S)$), the above functor canonically refines to a lax symmetric monoidal functor
\[
(\qSch_S\op)^{\amalg} \times (\Picspace(S)^{\BC})^{\otimes} \to \Mod_S(\Catps)^{\otimes} := \Mod_{(\Dperf(S),\Dual_S)}(\Catps)^{\otimes}
\]
where the symmetric monoidal product on the right hand side is given by relative tensor product over $(\Dperf(S),\Dual_S)$.
In particular, restricting to the symmetric monoidal subcategory $(\Sch_S\op)^{\amalg} \times \{\cO_S\}$ we obtain a lax symmetric monoidal functor
\[
\begin{tikzcd}[row sep = 5pt]
X  \ar[r,mapsto]\ar[d,phantom,"\rotatebox{270}{$\in$}"] & (\Dperf(X),\Dual_X)\ar[d,phantom,"\rotatebox{270}{$\in$}"] \\
(\qSch_S\op)^{\otimes} \ar[r] & \Mod_{S}(\Catps)^{\otimes} \ ,
\end{tikzcd}
\]
and similarly if we restrict to $(\Sch_S\op)^{\amalg} \times \{L\}$ for some $L \in \Picspace(S)^{\BC}$ we obtain a functor 
\[
\begin{tikzcd}[row sep = 5pt]
[p\colon X \to S]  \ar[r,mapsto]\ar[d,phantom,"\rotatebox{270}{$\in$}"] & (\Dperf(X),\Dual_{p^*L})\ar[d,phantom,"\rotatebox{270}{$\in$}"] \\
\qSch_S\op \ar[r] & \Mod_{S}(\Catps) 
\end{tikzcd}
\]
which is a module over the one above it with respect to Day convolution.

\begin{remark}
While the functor $[X \to S] \mapsto (\Dperf(X),\Dual_X) \in \Mod_S(\Catps)$ above is only lax symmetric monoidal, it does become strongly symmetric monoidal when restricting its domain to smooth $S$-schemes and localising its target by Karoubi equivalences. 
Since this stronger property will only be used in \S\ref{section:motivic-realization} and \S\ref{section:KQ}, we postpone the discussion of this point to \S\ref{subsection:more-derived}.
\end{remark}

\subsection{Genuine structures}%
\label{subsection:genuine}%

Recall that a t-structure on a stable $\infty$-category $\D$ is a pair of full subcategories $(\D_{\geq 0},\D_{\leq 0})$ of $\D$ satisfying the following properties 
\begin{enumerate}
\item
$\D_{\geq 0}$ is closed under suspensions and $\D_{\leq 0}$ under desuspensions.
\item
For every $x \in \D_{\geq 0}$ and $y \in \D_{\leq 0}$ we have $\Om\Map_{\D}(x,y) \simeq \ast$.
\item
For every object $y \in \D$ there exists an exact sequence $x \to y \to z$ such that $x \in \D_{\geq 0}$ and $\Sig y \in \D_{\leq 0}$. 
\end{enumerate}
Given a t-structure $(\D_{\geq 0},\D_{\leq 0})$ and $n \in \ZZ$ one writes $\D_{\geq n}$ and $\D_{\leq n}$ for the images of $\D_{\geq 0}$ and $\D_{\leq 0}$ under the $n$-fold suspension functor $\Sig^n$. 
The objects of $\D_{\geq n}$ are then called $n$-connective and those of $\D_{\leq n}$ are said to be $n$-coconnective. 
One can then show (see, e.g.,~\cite[Proposition 1.2.1.5]{HA}) that each of $\D_{\leq n}$ is a reflective subcategory of $\D$ with reflection functor $\tau_{\leq n}\colon \D \to \D_{\leq n}$ (called $n$-truncation), each $\D_{\geq n}$ is coreflective with coreflection functor $\tau_{\geq n}\colon \D \to \D_{\geq n}$ (called $n$-connective cover), and for every $x \in \D$ the unit and counit form an exact sequence
\[
\tau_{\geq n}(x) \to x \to \tau_{\leq (n-1)}(x) .
\]
In particular $\D_{\geq 0}$ and $\D_{\leq 0}$ determine each other via the identifications $\D_{\leq 0} = \ker(\tau_{\geq 1})$ and $\D_{\geq 0} = \ker(\tau_{\leq -1})$, which also means that $\D_{\geq 0}$ is closed under extensions and all colimits which exist in $\D$ and $\D_{\leq 0}$ is closed under extensions and all limits which exist in $\D$. We refer the reader to~\cite[\S 1.2.1]{HA} for a thorough discussion.

If $\D$ is presentable then we say that the t-structure is presentable if $\D_{\geq 0}$ is presentable.
The data of a presentable t-structure on $\D$ is then equivalent to the data of a colimit preserving fully-faithful functor $\D_{\geq 0} \to \D$ from a presentable $\infty$-category $\D_{\geq 0}$ whose image is closed under extensions, see~\cite[Proposition 1.4.4.11(1)]{HA}. This description can be used to define the $\infty$-category 
\[
\PrL_{\st,\tstruc} \subseteq \Ar(\PrL)
\]
of stable presentable $\infty$-categories equipped with presentable t-structures as the full $\infty$-category of $\Ar(\PrL)$ spanned by the fully-faithful arrows whose target is stable and whose image is closed under extensions. This is a reflective subcategory of $\Ar(\PrL)$ with reflection functor sending $\C \to \D$ to $\ovl{\im}(\C) \to \Spa(\D)$, where $\ovl{\im}(\C) \subseteq \Spa(\D)$ is the closure under colimits and extensions of the essential image of $\C \to \D \to \Spa(\D)$ (where we note that $\ovl{\im}(\C)$ is again presentable by~\cite[Proposition 1.4.4.11(2)]{HA}). In particular, $\PrL_{\st,\tstruc}$ is a localisation of $\Ar(\PrL)$. In addition, this localisation is multiplicative with respect to the pointwise tensor product on $\Ar(\PrL)$, so that $\PrL_{\st,\tstruc}$ inherits an induced symmetric monoidal structure. Concretely, the tensor product of $\D_{\geq 0} \subseteq \D$ and $\D'_{\geq 0} \subseteq \D'$ is given by $\ovl{\im}(\D_{\geq 0} \otimes \D'_{\geq 0}) \subseteq \D \otimes \D'$, where as above $\ovl{\im}(\D_{\geq 0} \otimes \D'_{\geq 0})$ is the closure under colimits and extensions of the essential image of $\D_{\geq 0} \otimes \D'_{\geq 0} \to \D \otimes \D'$.
We also point out that, since the target projection sends local equivalences to equivalences, the induced lax symmetric monoidal structure on the forgetful functor
\[
\PrL_{\st,\tstruc} \to \PrL_{\st}
\]
is in fact symmetric monoidal.

\begin{definition}
\label{definition:catext}%
Let $\C$ be a stable $\infty$-category.
By an \emph{orientation} on $\C$ we will mean a t-structure $\tstruc$ on $\Ind(\C)$, in which case we will call the pair $(\C,\tstruc)$ an \emph{oriented} $\infty$-category. We then write
$(\Catxt)^{\otimes}$ for the symmetric monoidal $\infty$-category of oriented $\infty$-categories, that is, $(\Catxt)^{\otimes}$ is given by the fibre product of symmetric monoidal functors
\[
\begin{tikzcd}
(\Catxt)^{\otimes} \ar[r]\ar[d] & (\PrL_{\st,\tstruc})^{\otimes} \ar[d] \\
(\Catx)^{\otimes} \ar[r, "\Ind"] & (\PrL_{\st})^{\otimes} \ .
\end{tikzcd}
\]
Given two oriented $\infty$-categories $(\C,\tstruc)$ and $(\C',\tstruc')$, we then write $\tstruc \boxtimes \tstruc'$ for the associated orientation on $\C \otimes \C'$, that is, the associated $\tstruc$-structure on $\Ind(\C \otimes \C') = \Ind(\C) \otimes \Ind(\C')$.

In a similar manner, we define the symmetric monoidal $\infty$-categories 
\[
(\Catpt)^{\otimes} := (\Catp)^{\otimes} \times_{(\Catp)^{\otimes}} (\Catxt)^{\otimes} \quad\text{and}\quad (\Catpst)^{\otimes} := (\Catps)^{\otimes} \times_{(\Catx)^{\otimes}} (\Catxt)^{\otimes} ,
\]
and use the term oriented Poincaré and oriented $\infty$-category with duality to designate their objects. Note that we do not enforce any compatibility constraints between the orientation and Poincaré/duality structure.
\end{definition}

Let $(\C,\QF, \tstruc)$ be an oriented Poincaré $\infty$-category. Then the first excisive approximation 
\[
\Psubone\QF=\colim_n \Om^n \QF \Sigma^n
\]
is an exact functor $\C\op\to\Spa$ and can thus be identified with an element of $\Ind\C$. 
We say that $\QF$ is \defi{$m$-connective} if $\Psubone\QF$ is $m$-connective with respect to the t-structure $\tstruc$. In this case we also say that $(\C,\QF,\tstruc)$ is an $m$-connective oriented Poincaré $\infty$-category. We then write $\Funq_{\ge m}(\C)\subseteq \Funq(\C)$ for the full subcategory spanned by the $m$-connective Poincaré structures and $\Catp_{\tstruc\ge m} \subseteq \Catpt$ for the full subcategory spanned by the $m$-connective Poincaré $\infty$-categories. For consistency of notation, let us also write $\Fun_{\geq -\infty}(\C) = \Funq(\C)$ and $\Funq_{\geq \infty}(\C) \subseteq \Funq(\C)$ for the full subcategory spanned by the quadratic functors whose linear part is trivial, and similarly $\Catp_{\tstruc \ge -\infty}$ and $\Catp_{\tstruc \ge \infty}$.

\begin{lemma}
\label{lemma:connected-cover-quadratic-functors}%
Let $(\C,\tstruc)$ be an oriented $\infty$-category.
For $m \in \ZZ \cup \{-\infty,\infty\}$ the inclusion of the subcategory $\Funq_{\ge m}(\C)\subseteq \Funq(\C)$ 
has a right adjoint $\QF \mapsto \QF^{\geq m}$, given by 
\[
\QF^{\ge m} = 
\begin{cases}
\QF\times_{\Psubone\QF}(\tau_{\ge m}\Psubone\QF) & \text{for $m \in \ZZ$,} \\
\QF & \text{for $m=-\infty$,} \\ 
\fib[\QF \to \Psubone\QF] & \text{for $m=\infty$.} 
\end{cases}
\]
\end{lemma}
\begin{proof}
For $m = -\infty$ the statement is vacuous and for $m=\infty$ it is contained in~\cite[Corollary 1.3.6]{9-authors-I}.
For $m \in \ZZ$, it suffices to show that the map $\QF^{\ge m}\to \QF$ induces an equivalence
\[
\Map(\QF',\QF^{\ge m})\to \Map(\QF',\QF)
\]
for every $m$-connective quadratic functor $\QF'$. But
\begin{align*}
\Map(\QF',\QF^{\ge m})&\simeq \Map(\QF',\QF)\times_{\Map(\QF',\Psubone\QF)} \Map(\QF',\tau_{\ge m}\Psubone\QF)\\
 & \simeq\Map(\QF',\QF)\times_{\Map(\Psubone\QF',\Psubone\QF)}\Map(\Psubone\QF',\tau_{\ge m}\Psubone\QF)
\end{align*}
and the map
\[
\Map(\Psubone\QF',\tau_{\ge m}\Psubone\QF)\to \Map(\Psubone\QF',\QF)
\]
is an equivalence since $\Psubone\QF'$ is $m$-connective. 
\end{proof}

\begin{proposition}
\label{proposition:genuine}%
For every $m$ the inclusion $\Catp_{\tstruc\ge m}\subseteq \Catpt$ has a right adjoint sending $(\C,\QF,\tstruc)$ to the triple $(\C,\QF^{\ge m},\tstruc)$, where $\C$ and $t$ are the same and $\QF^{\ge m}$ is the functor of Lemma~\ref{lemma:connected-cover-quadratic-functors}.
\end{proposition}

It follows directly from the description of the tensor product of Poincaré $\infty$-categories in~\cite[5.1.3. Proposition]{9-authors-I} that of $(\C,\QF,\tstruc)$ is an $m$-connective oriented Poincaré $\infty$-category and $(\C',\QF',\tstruc')$ is an $m'$-connective oriented Poincaré $\infty$-category then their tensor product $(\C \otimes \C',\QF \boxtimes \QF',\tstruc\boxtimes \tstruc')$ is $(m+m')$-connective. In particular, $\Catp_{\tstruc\ge 0}$ is closed under tensor products in $\Catp_{\tstruc}$ and hence inherits from it a symmetric monoidal structure under which the inclusion $\Catp_{\tstruc \geq 0} \to \Catp_{\tstruc}$ is symmetric monoidal and its right adjoint $\Cat_{\tstruc} \to \Cat_{\tstruc\geq 0}$ furnished by Proposition~\ref{proposition:genuine} is lax symmetric monoidal. Similarly, for every $m \in \ZZ \cup \{-\infty,\infty\}$ we have that $\Catp_{\tstruc \geq m}$ is a module over $\Catp_{\tstruc \geq 0}$ in a manner such that the inclusion $\Catp_{\tstruc\geq m} \to \Cat_{\tstruc}$ is $\Cat_{\tstruc\geq 0}$-linear and its right adjoint $\Catp_{\tstruc} \to \Catp_{\tstruc\geq m}$ is lax $\Catp_{\tstruc\geq 0}$-linear.

\begin{proof}[Proof of Proposition~\ref{proposition:genuine}]
For $m = -\infty$ the statement is vacuous and for $m=\infty$ it is contained in~\cite[Proposition 7.2.17]{9-authors-I}. We now consider the case $m \in \ZZ$.
Note that $\QF^{\ge m}$ has the same cross-effect as $\QF$, so it is a perfect quadratic functor on $\C$ with the same underlying duality. In particular $(\C,\QF^{\ge m})$ is a Poincaré $\infty$-category. Moreover the excisive approximation of $\QF^{\ge m}$ is $m$-connective by definition. So it suffices to show that if $(\C',\QF',t')$ is a Poincaré $\infty$-category such that $\QF'$ is $m$-connective, the map
\[
\Map_{\Catpt}((\C',\QF',t'),(\C,\QF^{\ge m},t))\to \Map_{\Catpt}((\C',\QF',t'),(\C,\QF,t))
\]
is an equivalence. Since both of these spaces lie above the space $\Map_{\Catxt}((\C',t'),(\C,t))$ of exact functors $\C' \to \C$ such that $f\op_!\colon \Ind(\C') \to \Ind(\C)$ is right t-exact, it suffices to show that the above map induces an equivalence on fibres over every such $f \colon \C'\to\C$. Now the map on fibres is the restriction of the map
\[
\Map_{\Funq(\C')}\left(\QF',\QF^{\ge m}\circ f\op\right)\to \Map_{\Funq(\C')}\left(\QF',\QF\circ f\op\right)
\]
to the subspaces of those maps that induce an equivalence $f\circ\Dual_{\QF'} \to \Dual_{\QF} \circ f\op$. Since the underlying duality of $\QF$ and of $\QF^{\ge m}$ is the same, it suffices to show that the above map is an equivalence. Finally we can rewrite it as the map
\[
\Map_{\Funq(\C)}\left(f\op_!\QF',\QF^{\ge m}\right)\to \Map_{\Funq(\C)}\left(f\op_!\QF',\QF\right)\,
\]
so by Lemma~\ref{lemma:connected-cover-quadratic-functors} it suffices to show that $f\op_!\QF'$ is $m$-connective. But $\Psubone(f\op_!\QF')=f\op_!(\Psubone\QF')$ and $f\op_!$ is right t-exact by assumption. So, since $\Psubone\QF'$ is $m$-connective, so is $f\op_!\QF'$.
\end{proof}

Specializing to derived categories of schemes, recall the square of lax symmetric monoidal functors
\begin{equation}
\label{equation:forming-derived}%
\begin{tikzcd}[row sep=7pt]
(X,L) \ar[r,mapsto]\ar[d,phantom,"\rotatebox{270}{$\in$}"] &(\Dperf(X),\Dual_{L}) \ar[d,phantom,"\rotatebox{270}{$\in$}"] \\
(\qSchPic\op)^{\otimes}\ar[dd]
 \ar[r] &(\Catps)^{\otimes}\ar[dd] \\
 &\\
(\qSch\op)^{\amalg} \ar[r] & (\Catx)^{\otimes} \ .
\end{tikzcd}
\end{equation}
obtained as the external rectangle in Diagram~\ref{equation:functoriality-schemes-1}.
As explained in \S\ref{section:appendix}, for a scheme $X$ the presentable $\infty$-category $\Derqc(X) = \Ind(\Dperf(X))$ of quasi-coherent complexes carries a canonical t-structure, sometimes called the standard t-structure, such that the associated $X \mapsto \Derqc(X)$ refines to a functor
\[
\qSch\op \to \CAlg(\PrL_{\st,\tstruc}) \quad\quad X \mapsto (\Derqc(X),\Derqc(X)_{\geq 0},\Derqc(X)_{\leq 0}).
\]
Now by the equivalence of~\cite[Theorem 2.4.3.18]{HA} this functor determines a lax symmetric monoidal functor
\[
(\qSch\op)^{\amalg} \to \PrL_{\st,\tstruc} ,
\]
and consequently a lax symmetric monoidal lift 
\[
(\qSchPic\op)^{\otimes} \to (\Catpst)^{\otimes}
\]
of the top horizontal arrow in~\eqref{equation:forming-derived}.
We may then further compose with the lax symmetric monoidal functor $\Catpst \to \Catpt$ induced by the right adjoint 
\[
\Catps \to \Catp \quad\quad (\C,\Dual) \mapsto (\C,\QF^{\sym}_{\Dual})
\]
of the forgetful functor (see Corollary~\ref{corollary:adj-catp-catps}), and subsequently with the $m$-truncation functor of Proposition~\ref{proposition:genuine} 
for various $m \in \ZZ \cup \{-\infty,\infty\}$ to obtain a collection of functors
\[
\qSchPic\op \to \Catpt \quad\quad (p\colon X \to S, L) \mapsto (\Dperf(X),\QF^{\ge m}_{L}),
\]
fitting together in a sequence 
\[
(\Dperf(X),\QF^{\qdr}_{L})=(\Dperf(X),\QF^{\ge\infty}_{L})\to \cdots \to (\Dperf(X),\QF^{\ge m}_{L})\to \cdots \to (\Dperf(X),\QF^{\ge-\infty}_{L})=(\Dperf(X),\QF^{\sym}_{L})\,.
\]
Here the functor at $m=0$ is lax symmetric monoidal and all the others are modules over it with respect to Day convolution, see the discussion following Proposition~\ref{proposition:genuine}. 
Similarly, as in \S\ref{subsection:rigid-to-poinc} we may restrict our domain along the lax symmetric monoidal functor $(\Sch_S\op)^{\amalg} \times (\Picspace(S)^{\BC})^{\otimes} \to (\qSchPic\op)^{\otimes}$ and obtain for every $m \in \ZZ \cup \{-\infty,\infty\}$ a functor
\begin{equation}
\label{equation:functoriality-schemes-Lt}%
\qSch_S\op \times \Picspace(S)^{\BC} \to \Catpt \quad\quad (p\colon X \to S, L) \mapsto (\Dperf(X),\QF^{\ge m}_{p^*L}),
\end{equation}
where again the one for $m=0$ is lax symmetric monoidal and all the others are modules over it with respect to Day convolution. 
Finally, since every object is canonically a module over the unit the above functors canonically refine to a sequence of functors 
\[
\qSch_S\op \times \Picspace(S)^{\BC} \to \Mod_{(\Dperf(S),\QF^{\geq 0}_S)}(\Catpt) \quad\quad (p\colon X \to S, L) \mapsto (\Dperf(X),\QF^{\ge m}_{p^*L})
\]
enjoying the same multiplicativity properties.

\begin{definition}
\label{definition:genuine}%
Given a scheme $X$ and $L \in \Picspace(X)^{\BC}$, we call the Poincaré structures $\QF^{\geq 0}_L$, $\QF^{\geq 1}_L$ and $\QF^{\geq 2}_L$ on $\Dperf(X)$ the \defi{genuine symmetric}, \defi{genuine even} and \defi{genuine quadratic} Poincaré structures, respectively, and use the notation $\QF^{\gs}_L$, $\QF^{\gev}_L$ and $\QF^{\gq}_L$ to refer to them.
\end{definition}

By~\cite{Hebestreit-Steimle}*{Theorem~A}, when $X=\spec R$ is affine the Grothendieck-Witt space of $(\Dperf(R),\QF^{r}_L)$ for $r=\gs,\gev,\gq$ coincides with the group completion of the 
symmetric monoidal groupoid of projective $R$-modules equipped with a perfect $L$-valued symmetric bilinear/even/quadratic form, respectively. In other words, the non-negative genuine symmetric/even/quadratic $\GW$-groups of $\spec(R)$ coincide with the classical symmetric/even/quadratic Grothendieck-Witt groups of $R$, as defined by Karoubi-Villamayor~\cite{karoubi-villamayor}. For the genuine symmetric case we will extend this comparison to all divisorial schemes in \S\ref{subsection:genuine-sym}.

\begin{notation}
\label{notation:functor-to-sheaf}%
For a functor $\F\colon \Catp \to \E$, a scheme $X$ and an invertible perfect complex with $\Ct$-action $L \in \Picspace(X)^{\BC}$ we write
\[
\F^{\geq m}(X,L) := \F(\Dperf(X),\QF^{\geq m}_L)
\]
for every $m \in \ZZ \cup \{-\infty,\infty\}$. For shorthand we will also write
\[
\F^{\geq m}(X) := \F^{\geq m}(X,\cO_X)
\]
and if $p\colon X \to S$ is a $S$-scheme and $L \in \Picspace(X)^{\BC}$ is an invertible perfect complex with $\Ct$-action on $S$ then we write 
\[
\F^{\geq m}_L(X) := \F^{\geq m}(X,p^*L) .
\]
More generally, for $n \in \ZZ$ and $m \in \ZZ \cup \{-\infty,\infty\}$ we define
\[
\F^{\ge m,[n]}_L(X) = \F\left(\Dperf(X),(\QF^{\geq m}_L)\qshift{n}\right).
\]
When $n=0$, we drop the shift in the notation, which becomes $\F^{\ge m}_L$. Similarly, for $m=\pm\infty$ we replace the superscripts $(-)^{\geq -\infty}$ and $(-)^{\geq \infty}$ with $(-)^{\sym}$ and $(-)^{\qdr}$, respectively, and for $L=\cO_S$ we use the subscript $S$ instead of $\cO_S$.
\end{notation}

Using the functoriality of~\eqref{equation:functoriality-schemes-Lt} we may view 
$\F^{\geq m}_L$ as a functor on $\qSch_S\op$. In particular, by the multiplicativity properties of~\eqref{equation:functoriality-schemes-Lt} we get that if $\E$ is symmetric monoidal and $\F\colon (\Catp)^{\otimes} \to \E^{\otimes}$ is lax symmetric monoidal structure then $\F^{\geq 0}_S\colon \qSch_S\op \to \E$ carries an induced lax symmetric monoidal structure with respect to the cocartesian monoidal structure $(\qSch_S\op)^{\amalg}$, and similarly $\F^{\geq m}_L\colon (\qSch_S)\op \to \E$ is a module over $\F^{\geq 0}_S$ with respect to Day convolution for every $m \in \ZZ \cup \{-\infty,\infty\}$ and $L \in \Picspace(S)^{\BC}$. Concretely, since Day convolution with respect to the cocartesian monoidal structure coincides with the pointwise product, this means that $\F^{\geq 0}_S$ lifts to a functor valued $\CAlg(\E)$ which acts pointwise on $\F^{\geq m}_L$ for every $m \in \ZZ \cup \{-\infty,\infty\}$ and $L \in \Picspace(S)^{\BC}$.

\section{Nisnevich descent}
\label{section:nisnevich}%

Our main goal in this section is to show that for a qcqs scheme $S$, a perfect invertible complex with $\Ct$-action $L \in \Picspace(S)^{\BC}$, and a Karoubi-localising functor $\F\colon \Catp \to \E$, the functor 
\[
\F^{\geq m}_L\colon (\qSch_S)\op \to \E \quad\quad (X,L) \mapsto \F(\Dperf(X),\QF^{\geq m}_X)
\]
of Construction~\ref{notation:functor-to-sheaf} satisfies Nisnevich descent, that is, it is a sheaf with respect to the Nisnevich topology for every $m \in \ZZ \cup \{\pm \infty\}$. To arrive to this, we first dedicate \S\ref{subsection:flat-linear} and~\ref{subsection:bounded-karoubi} to introduce and study the notions of flat $S$-linear $\infty$-categories (with duality) and bounded Karoubi sequences between them. The main reason for this setting is that it will eventually allow us in \S\ref{section:motivic-realization} to bootstrap Nisnevich descent into the construction of motivic realization functors with suitable multiplicative properties. In \S\ref{subsection:zariski-descent} we then prove Zariski descent for $\F^{\geq m}_L$, followed by Nisnevich descent in \S\ref{subsection:nisnevich-descent}. Finally, \S\ref{subsection:coniveau} and \S\ref{subsection:genuine-sym} contain applications of these descent results, the former to establishing a version of the coniveau filtration in the present setting, and the second for the study of genuine symmetric $\GW$-theory of schemes.

\subsection{Flat linear categories}
\label{subsection:flat-linear}%

As discussed in \S\ref{section:poinc-schemes} and in Appendix \S\ref{section:appendix}, for every scheme $X$, the stable $\infty$-category $\Dperf(X)$ comes equipped with a t-structure on its Ind-completion $\Derqc(X)$. To facilitate terminology, let us refer to the data of a t-structure on the Ind-completion of a given stable $\infty$-category $\A$ as an \defi{orientation} on $\A$. If $\A$ carries a symmetric monoidal structure (e.g., if $\A=\Dperf(X)$) then $\Ind(\A)$ acquires an induced symmetric monoidal structure, in which case we use the term \defi{multiplicative orientation} to mean a t-structure $(\Ind(\A)_{\geq 0}, \Ind(\A)_{\leq 0})$ on $\Ind(\A)$ which is compatible with the induced symmetric monoidal structure, that is, $\Ind(\A)_{\geq 0}$ contains $\one_{\A}$ and is closed under tensor product.

\begin{definition}
Let $\A$ be a stably symmetric monoidal $\infty$-category equipped with a multiplicative orientation.
Given $x \in \Ind(\A)$, we say that $x$ \defi{has Tor amplitude $\leq n$} for a given $n \in \ZZ$ if the operation $x \otimes (-)$ sends coconnective objects to $n$-coconnective objects. We say that $x$ has \defi{bounded Tor amplitude} if it has Tor amplitude $\leq n$ for some $n$. 
\end{definition}

\begin{remark}
\label{remark:tor-amp-tensor}%
If $x \in \Ind(\A)$ has Tor amplitude $\leq n$ and $\Ind(y) \in \A$ has Tor amplitude $\leq m$ then $x \otimes y$ has Tor amplitude $\leq n+m$.
\end{remark}

\begin{lemma}
\label{lemma:dual-amplitude}%
Let $\A$ be a rigid stably symmetric monoidal $\infty$-category equipped with a multiplicative orientation. Then for a given $n \in \ZZ$, an object $x \in \A$ has Tor amplitude $\leq n$ if and only if $\Dual x$ is $(-n)$-connective.
\end{lemma}
\begin{proof}
The dual $\Dual x$ has in particular the property that the functor $\Dual x \otimes (-)$ is left adjoint to $x \otimes (-)$. In addition, as in any t-structure, we have for every $i$ the equalities 
\[
\Ind(\A)_{\leq i} = \Ind(\A)_{\geq i+1}^{\perp} = \{z \in \A\; |\; \map(y,z) = 0, \;\forall y \in \Ind(\A)_{\geq i+1}\}
\]
and 
\[
\Ind(\A)_{\geq i+1} = {}^{\perp}\Ind(\A)_{\leq i} = \{x \in \A\; |\; \map(x,y) = 0, \;\forall y \in \Ind(\A)_{\leq i}\}.
\]
We conclude that $x \otimes (-)$ sends $\Ind(\A)_{\leq 0}$ to $\Ind(\A)_{\leq n}$ if and only if $\Dual x \otimes (-)$ sends $\Ind(\A)_{\geq n+1}$ to $\Ind(\A)_{\geq 1}$, hence if and only if $\Dual x \otimes (-)$ sends $\Ind(\A)_{\geq 0}$ to $\Ind(\A)_{\geq -n}$. Since the t-structure is multiplicative, this is equivalent to $\Dual x$ belonging to $\Ind(\A)_{\geq -n}$.
\end{proof}

Let $\A$ be a stably symmetric monoidal $\infty$-category. By an $\A$-linear $\infty$-category $\C$, we mean an $\A$-module object in $\Catx$. Such an $\A$-module structure determines a refinement of the mapping spectra in $\C$ to $\Ind(\A)$-valued mapping objects, which we denote by $\uline{\map}_{\C}(x,y) \in \Ind(\A)$. They are uniquely characterized by the property that
\[
\map_{\Ind(\A)}(j(a),\uline{\map}_{\C}(x,y)) = \map_{\C}(a \otimes x,y)
\]
for any $a \in \A$, where $j\colon \A \to \Ind(\A)$ is the Yoneda embedding. We refer to morphisms of $\A$-modules in $\Catx$ as $\A$-linear functors, and write $\Fun^{\A}(\C,\D)$ for the $\infty$-category of $\A$-linear functors from $\C$ to $\D$. Note that our terminology is such that an $\A$-linear functor is by definition exact.

\begin{definition}
\label{definition:flat}%
Let $\A$ be a stably symmetric monoidal $\infty$-category equipped with a multiplicative orientation.
We say that an $\A$-linear $\infty$-category $\E$ is \defi{flat} over $\A$ if the enriched mapping object $\uline{\map}(x,y) \in \Ind(\A)$ has bounded Tor amplitude for every $x,y \in \E$. 
\end{definition}

\begin{example}
\label{example:flat}%
Let $X \to S$ be a flat morphism of qcqs schemes. Then $\Dperf(X)$ is flat over $\A=\Dperf(S)$.
\end{example}
\begin{proof}
The assumption that $f$ is flat implies that $f^*\Derqc(S) \to \Derqc(X)$ is exact and hence its right adjoint $f_*\colon \Derqc(X) \to \Derqc(S)$ preserves objects with bounded Tor amplitude $\leq n$ by the projection formula. We may hence assume without loss of generality that $f$ is the identity. In this case the result follows directly from the characterization of Lemma~\ref{lemma:dual-amplitude}.
\end{proof}

Recall that the tensor product $(-) \otimes_\A (-)$ in $\Mod_\A(\Catx)$ is determined by the following universal property: there exists a natural map $i\colon \C \times \D \to \C \otimes_\A \D$ which is $\A$-linear in each variable separately and such that for every $\A$-linear $\infty$-category $\E$ restriction along $i$ determines an equivalence between $\A$-linear functors $\C \otimes_\A \D \to \E$ and functors $\C \times \D \to \E$ which are $\A$-linear in each variable separately. In particular, we have a canonical equivalence
\[
\Fun^{\A}(\C \otimes_\A \D,\E) = \Fun^{\A}(\C,\Fun^{\A}(\D,\E)),
\]
so that the functor $(-) \otimes_\A \D\colon \Mod_\A(\Catx) \to \Mod_\A(\Catx)$ admits a right adjoint given by $\E \mapsto \Fun^{\A}(\D,\E)$.
To obtain a more explicit description of the underlying stable $\infty$-category of $\C \otimes_\A \D$, we may ignore for the moment its $\A$-action and view the association $\C \mapsto \C \otimes_\A \D$ as a functor $\Mod_\A(\Catx) \to \Catx$. As such, it also admits a right adjoint, given by the formula $\E \mapsto \Funx(\D,\E)$, where $\Funx(\D,\E)$ is considered as an $\A$-module via the pre-composition action determined by its action on $\D$. We may hence describe $\C \otimes_\A \D$ by embedding it in its Ind-completion, which by the last adjunction can be described as
\begin{align}
\label{equation:formula-tensor}%
\Ind(\C \otimes_\A \D) &= \Funx((\C \otimes_\A \D)\op,\Spa) \\
\nonumber &= \Funx(\C\op \otimes_{\A\op} \D\op,\Spa) \\
\nonumber &= \Fun^{\A\op}(\C\op,\Funx(\D,\Spa)) \\
\nonumber &= \Fun^{\A\op}(\C\op,\Ind(\D)) \ ,
\end{align}
where the $\A\op$-action on $\Ind(\D) = \Funx(\D\op,\Spa)$ is the pre-composition action induced by the $\A\op$-action on $\D\op$. To identify the full subcategory of $\Fun^{\A\op}(\C\op,\Ind(\D))$ corresponding to $\C \otimes_\A \D$, we need to track the composite functor
\[
\C \times \D \to \C \otimes_\A \D \to \Ind(\C \otimes_\A \D) = \Fun^{\A\op}(\C\op,\Ind(\D)) .
\]
For this, one needs to embed $\Fun^{\A\op}(\C\op,\Ind(\D))$ as a reflective subcategory (and accessible localisation) of the $\infty$-category $\Fun^{\A\op-\oplax}(\C\op,\Ind(\D))$ of oplax $\A\op$-linear functors $\C\op \to \Ind(\D)$. Then, to every $(c,d) \in \C \times \D$ we may associate the oplax $\A\op$-linear functor $f_{c,d}\colon \C\op \to \Ind(\D)$ given by the formula 
\[
f_{c,d}(x) = \uline{\map}(x,c) \otimes d .
\]
The image of $(c,d)$ in $\Fun^{\A\op}(\C\op,\Ind(\D))$ is then given by applying to $f_{c,d}$ the reflection functor
\[
T\colon \Fun^{\A\op-\oplax}(\C\op,\Ind(\D)) \to \Fun^{\A\op}(\C\op,\Ind(\D)) ,
\]
and $\C \otimes_\A \D$ is the smallest stable subcategory of $\Fun^{\A\op}(\C,\Ind(\D))$ containing the objects $T(f_{c,d})$ for every $c,d \in \C \times \D$.
In general, this reflection functor can be hard to describe. However, under the assumption that $\A$ is rigid, we have that $f_{c,d}$ is already itself $\A\op$-linear (as opposed to just oplax) for every $c,d \in \C \times \D$. In fact:

\begin{lemma}
\label{lemma:rune-magic-lemma}%
Let $\I$ be a rigid symmetric monoidal $\infty$-category (not necessarily stable) and $\C,\D$ two $\I$-linear $\infty$-categories (that is, $\I$-module objects in $\Cat$). Then any functor $f\colon \C \to \D$ which is either lax or oplax $\I$-linear is already $\I$-linear.
\end{lemma}
\begin{proof}
Let $\B\I$ be the $(\infty,2)$-category with one object whose endomorphism $\infty$-category is $\I$. Then the $\I$-module objects $\C$ and $\D$ correspond to $(\infty,2)$-functors $F_{\C},F_{\D}\colon \B\I \to \Cat$. In this setup, an (op)lax $\I$-linear functor corresponds to an (op)lax natural transformation $F_{\C} \Rightarrow F_{\D}$. 
By~\cite[Corollary 4.8]{haugseng-lax} 
any lax (resp.\ oplax) natural transformation restricts to an honest natural transformation on the subcategory spanned by the right (resp.\ left) adjoint 1-morphism, that is, on the subcategory $\B\I_0 \subseteq \B\I$ where $\I_0 \subseteq \I$ is the full subcategory spanned by the right (resp.\ left) dualisable objects. But if $\I$ is rigid, any object is both left and right dualisable, so the desired result follows.
\end{proof}

In particular, when $\A$ is rigid, we may then simply view $f_{c,d}$ as objects of $\Fun^{\A\op}(\C,\Ind(\D))$, and $\C \otimes_\A \D$ is simply the smallest stable subcategory of $\Fun^{\A\op}(\C,\Ind(\D))$ containing the ``pure tensor'' objects $f_{c,d}$ for every $c,d \in \C \times \D$. 
We then have that
\begin{equation}
\label{equation:pure-tensors}%
\uline{\map}(f_{a,b},f_{c,d}) = \uline{\map}_{\D}(b, \uline{\map}_{\C}(a,c) \otimes d) = \uline{\map}_{\C}(a,c) \otimes \uline{\map}_{\D}(b,d) ,
\end{equation}
where the second equivalence is due to the assumption that $\A$ is rigid. Our discussion thus leads to the following conclusion:

\begin{corollary}
\label{corollary:flat-closed}%
Let $\A$ be a stably symmetric monoidal $\infty$-category equipped with a multiplicative orientation. If $\A$ is rigid, the collection of flat $\A$-linear $\infty$-categories 
is closed under the operation of taking tensor products over $\A$.
\end{corollary}
\begin{proof}
Since the collection of objects with bounded Tor amplitude in $\Ind(\A)$ is closed under finite colimits, desuspensions, and tensor products (see Remark~\ref{remark:tor-amp-tensor}), it follows from the formula~\eqref{equation:pure-tensors} that if $\C$ and $\D$ are flat over $\A$, the same holds for $\C \otimes_\A \D$.
\end{proof}

\subsection{Bounded Karoubi projections}
\label{subsection:bounded-karoubi}%

We now fix a stably symmetric monoidal $\infty$-category $\A$ equipped with a multiplicative orientation $(\Ind(\A)_{\geq 0},\Ind(\A)_{\leq 0})$ satisfying the following assumptions:
\begin{assumption}\
\label{assumption:rigid}%
\begin{enumerate}
\item
$\A$ is rigid. 
\item
$\one_\A$ is $m$-coconnective for some $m$.
\end{enumerate}
\end{assumption}
We write $\one_{\A} \in \A$ for the unit object.

\begin{remark}
\label{remark:tor-is-truncated}%
Since $\one_{\A}$ is assumed to be $m$-coconnective we have that 
if $x \in \Ind(\A)$ has Tor amplitude $\leq n$ then $x$ is $(n+m)$-coconnective.
\end{remark}

\begin{example}
If $S$ is any qcqs scheme then $\A = \Dperf(S)$ with the standard t-structure on $\Ind(\A) = \Derqc(S)$ satisfies the above assumptions. This is the main example we have in mind.
\end{example}

\begin{example}
If $R$ is an $\Einf$-ring spectrum which is both connective and $n$-coconnective for some $n \geq 0$ then the $\infty$-category $\A = \Mod(R)^\omega$ of compact $R$-module spectra, together with the standard t-structure on $\Ind(\A) = \Mod(R)$, satisfies the above assumptions.
\end{example}

Recall from \S\ref{subsection:rigid-to-poinc} that since $\A$ is rigid, it carries a canonical duality given by $\Dual_{\A}(x) = \uline{\map}_{\A}(x,\one_{\A})$, which is furthermore a symmetric monoidal duality, that is, $(\A,\Dual_{\A})$ is a commutative algebra object in $\Catps$. The assumption that the orientation on $A$ is multiplicative means that we can also consider it as an algebra object in $\Catpst$. We can consequently consider $(\A,\Dual_\A)$-modules in either $\Catps$ or $\Catpst$. 
In concrete terms, a module in $\Catps$ is an $\A$-linear $\infty$-category equipped with a duality $\Dual_\C\colon \C\op \to \C$ which is also $\A$-linear. The morphisms of such a given by duality preserving $\A$-linear functors
\[
(\C,\Dual_\C) \to (\D,\Dual_\D) .
\]
Similarly, $(\A,\Dual_A)$-modules in $\Catpst$ are such $\A$-linear $\infty$-categories with duality further equipped with an orientation which is compatible with the $A$-action (that is, the induced action of $\Ind(\A)_{\geq 0}$ on $\Ind(\C)$ preserves $\Ind(\C)_{\geq 0}$). 
Morphisms of such are then further required to induce a right t-exact functor on Ind completions.

\begin{definition}
We say that a duality preserving $\A$-linear functor $(\C,\Dual_\C) \to (\D,\Dual_\D)$ is an $\A$-linear Karoubi inclusion/projection if it is so on the level of underlying stable $\infty$-categories. 
Similarly, we will say that a sequence of $\A$-linear $\infty$-categories with duality is a Karoubi sequence if it is such on the level of underlying stable $\infty$-categories.
\end{definition}

Let $p\colon (\C,\Dual_{\C}) \to (\D,\Dual_{\D})$ be an $\A$-linear Karoubi projection 
and $g\colon \D \to \Pro(\C)$ the associated fully-faithful Pro-left adjoint (which is not necessarily duality preserving). We say that $p$ is \defi{bounded at $x \in \D$} if there exists an $n \in \ZZ$ and a cofiltered diagram $\{y_i\}_{i \in \I}$ in $\C$ with limit $g(x) \in \Pro(\C)$ such that $\uline{\map}(y_i,\Dual_{\C} y_i) \in \Ind(\A)$ has Tor amplitude $\leq n$ for every $i \in \I$.

\begin{example}
\label{example:split-is-flat}%
If $(\C,\Dual_\C)$ is a flat $\A$-linear $\infty$-category with duality, and $p\colon (\C,\Dual_\C) \to (\D,\Dual_\D)$ is a duality preserving $\A$-linear functor whose underlying functor is a split Verdier projection then $p$ is a bounded Karoubi projection in the sense of Definition~\ref{definition:bounded}.
\end{example}

\begin{definition}
\label{definition:bounded}%
We say that a square
\[
\begin{tikzcd}
(\C,\Dual_\C) \ar[r]\ar[d] & (\D,\Dual_\D) \ar[d] \\
(\C',\Dual_{\C'}) \ar[r] & (\D',\Dual_{\D'})
\end{tikzcd}
\]
of $\A$-linear $\infty$-categories with duality is a \defi{bounded Karoubi square} if it is cartesian and the vertical arrows are bounded $\A$-linear Karoubi projections. Such a bounded Karoubi square whose lower left corner is zero is also called a \defi{bounded Karoubi sequence}.
\end{definition}

Our main interest in the notion of bounded Karoubi projection is due to the following property:

\begin{proposition}
\label{proposition:flat-is-karoubi-sym}%
Let $p\colon (\C,\Dual) \to (\C',\Dual')$ be a bounded 
Karoubi projection and $m \in \ZZ \cup \{-\infty,\infty\}$. 
Then for $r \in \{\qdr,\sym\}$ the associated Poincaré functor
\[
p^{r}\colon (\C,\QF^{r}_{\Dual}) \to (\C',\QF^{r}_{\Dual'})
\]
is a Poincaré-Karoubi projection. Suppose in addition that $p$ refines to a morphism in $\Mod_S(\Catpst)$, that is, $(\C,\Dual)$ and $(\C',\Dual')$ are equipped with orientations which are compatible with the $\A$-action and the induced functor $\Ind(\C) \to \Ind(\C')$ is t-exact. Then for every $m \in \ZZ$ the Poincaré functor
\[
p^{\geq m}\colon (\C,\QF^{\geq m}_{\Dual}) \to (\C',\QF^{\geq m}_{\Dual'})
\]
is a Poincaré-Karoubi projection.
\end{proposition}

The proof of Proposition~\ref{proposition:flat-is-karoubi-sym} will require a fairly standard lemma concerning stable $\infty$-categories, which we include here since we could not locate a reference.

\begin{lemma}
\label{lemma:standard}%
Let $f\colon \C \to \D$ be an exact functor between stable $\infty$-categories. Suppose that both $\C$ and $\D$ admit t-structures such that the t-structure on $\D$ is right separated, and that there exists a $d \geq 0$ such that $f(\C_{\leq 0}) \subseteq \C_{\leq d}$ (e.g., $f$ is left t-exact). Let $m \in \ZZ$ be an integer, $K$ a finite type simplicial set (that is, $K$ has finitely many non-degenerate simplices in each degree) and $\vphi\colon K \to \C_{\leq m}$ a diagram which extends to a limit diagram $\ovl{\vphi}\colon K^{\triangleleft} \to \C$ in $\C$. Then $f(\ovl{\vphi})$ is a limit diagram in $\D$.  
\end{lemma}

We will eventually need the following special case of the above standard lemma:
\begin{corollary}
\label{corollary:standard}%
Let $\C$ be a stable $\infty$-category equipped with a right separated t-structure and suppose that $\C$ admits filtered colimits such that $\C_{\leq 0}$ is closed under filtered colimits. 
Then filtered colimits of uniformly bounded above diagrams commute with finite type limits.
\end{corollary}

\begin{proof}[Proof of Lemma~\ref{lemma:standard}]
Let $K_n$ be the $n$-skeleton of $K$, so that $K$ is a finite simplicial set. Since $\C$ is stable $K_n$-indexed limits exist in $\C$ and in $\D$. Furthermore, for any $r$ and any diagram $\psi\colon K \to \C_{\leq r}$ which admits a limit in $\C$, the fibre of
\[
\lim_{k \in K}\psi(k) \to \lim_{k \in K_n} \psi(k)
\]
is in $\C_{\leq r-n-1}$ by the Bousfield-Kan formula, and similarly for any diagram $\psi\colon K \to \D_{\leq r}$ admitting a limit in $\D$.
Now let $\vphi\colon K \to \C_{\leq m}$ be a diagram admitting a limit in $\C$. Consider the commutative square
\[
\begin{tikzcd}
f(\lim_{K} \vphi) \ar[r]\ar[d] & f(\lim_{K_n}\vphi) \ar[d, "\simeq"]\\
\lim_{K}(f\circ \vphi) \ar[r] & \lim_{K_n}(f \circ \vphi)
\end{tikzcd}
\]
where the right vertical map is an equivalence since $K_n$ is finite. 
Since $\vphi$ takes values in $\C_{\leq m}$ and $f$ raises truncatedness by at most $d$, the fibres of the horizontal maps lie in $\D_{\leq (m+d-n-1)}$. Since the right vertical map is an equivalence the fibre of the left vertical map is in $\D_{\leq (m+d-n-1)}$. Since the t-structure is right separated we conclude that the left vertical map is an equivalence.
\end{proof}

\begin{proof}[Proof of Proposition~\ref{proposition:flat-is-karoubi-sym}]
We prove all statements simultaneously by taking $m \in \ZZ \cup \{-\infty,\infty\}$, so that 
$\QF^{\geq \infty} = \QF^{\qdr}$ and $\QF^{\geq -\infty}=\QF^{\sym}$.
The Poincaré functor $p^{\geq m}$ determines a natural transformation
\[
\LKan_p(\QF^{\geq m}_{\Dual}) \Rightarrow \QF^{\geq m}_{\Dual'}
\]
from the left Kan extension of $\QF^{\geq m}_{\Dual}$ along $p\op\colon \C\op \to (\C')\op$ to $\QF^{\geq m}_{\Dual'}$, and since we already know that $p$ is a Karoubi projection on the level of underlying stable $\infty$-categories, we only need to show that this transformation is an equivalence. Since $\D$ is generated as a stable $\infty$-category by the objects $x \in \C'$ at which $p$ is bounded, it will suffice to show that the components of this natural transformation at such $x$ is an equivalence. Choose a cofiltered diagram $\{y_i\}_{i \in \I}$ in $\C$ with limit $g(x)$, where $g\colon \C'\to \Pro(\C)$ is the fully-faithful Pro-left adjoint of $p$. Then the left Kan extension can be computed via this presentation of $g(x)$, and we are reduced to showing that the induced map
\[
\colim_{i \in \I\op} \QF^{\geq m}_{\Dual}(y_i) \to \QF^{\geq m}_{\Dual'}(x)
\]
is an equivalence. If $m = \infty$ then 
\[
\QF^{\geq m}_{\Dual}(y_i) = \QF^{\qdr}_{\Dual}(y_i) = \map_{\C}(y_i,\Dual y_i)_{\hC},
\]
and since $(-)_{\hC}$ commutes with colimits the statement follows from the fact that the induced map 
\[
\colim_{i \in \I\op} \map_{\C}(y_i,\Dual y_i) = \colim_{i,j \in \I\op}\map_{\C}(y_i,\Dual y_j)  \to \map_{\C'}(x,\Dual x)
\]
is an equivalence by virtue of $p\colon \C \to \C'$ being a Karoubi projection.
On the other extremity, if $m=-\infty$ then  
\[
\QF^{\geq m}_{\Dual}(y_i) = \QF^{\sym}_{\Dual}(y_i) = \map_{\C}(y_i,\Dual y_i)^{\hC},
\]
and we can rewrite the above map is the limit-colimit interchange map
\[
\colim_{i \in \I} [\map_{\C}(y_i,\Dual y_i)^{\hC}] \to \Big[\colim_{i \in \I} \map_{\C}(y_i,\Dual y_i)\Big]^{\hC}.
\]
Now, by assumption, there exists an $n$ such that 
$\uline{\map}(y_i,\Dual y_i) \in \Ind(\A)$ has Tor amplitude $\leq n$ for every $i$, and is hence $(n+k)$-coconnective for every $i$, where $k$ is such that $\one_\A$ is $k$-coconnective, see Remark~\ref{remark:tor-is-truncated}. Since $\one_\A$ is also connective (by the assumption that the t-structure is multiplicative) we conclude that the family of spectra $\{\map_{\C}(y_i,\Dual y_i)\}$ is uniformly bounded above, and so the desired result 
follows from the fact that limits of finite type diagrams of spectra commute with uniformly bounded above colimits (see Corollary~\ref{corollary:standard}). 
Finally, given that the case $m=\pm\infty$ is established, to show the case $m \neq \pm \infty$ it will suffice to prove that the induced map
\[
\LKan_p \Lam^{\geq m}_{\Dual} \to \Lam^{\geq m}_{\Dual'}
\]
is an equivalence, where $\Lam^{\geq m}_{\Dual}$ and $\Lam^{\geq m}_{\Dual}$ are the linear parts of $\QF^{\geq m}_{\Dual}$ and $\QF^{\geq m}_{\Dual'}$, respectively. Indeed, by definition these are the $m$-connective covers of the linear parts of $\QF^{\sym}_{\Dual}$ and $\QF^{\sym}_{\Dual'}$, and $\LKan_p = \Ind(p)\colon \Ind(\C) \to \Ind(\C')$ is assumed t-exact when $m \neq \pm \infty$.
\end{proof}

Our principal examples of bounded Karoubi projections are the following:

\begin{proposition}
\label{proposition:open-is-bounded}%
Let $S$ be a qcqs scheme and $j\colon U \to X$ an open embedding of qcqs smooth $S$-schemes. Let $L\in\Picspace(X)^{\BC}$
be an invertible perfect complex with $\Ct$-action. Then, the duality preserving functor 
\[
j^*\colon (\Dperf(X),\Dual_L) \to (\Dperf(U),\Dual_{j^*L})
\]
is a bounded Karoubi projection (over $\A = \Dperf(S)$). 
\end{proposition}

\begin{corollary}
\label{corollary:open-is-karoubi}%
Let $j\colon U \to X$ an open embedding of qcqs schemes and let $L\in\Picspace(X)^{\BC}$ be an invertible perfect complex with $\Ct$-action. Then for every $m \in \ZZ \cup \{-\infty,\infty\}$ the Poincaré functor
\[
(j^*,\eta) \colon (\Dperf(X),\QF^{\geq m}_L) \to (\Dperf(U),\QF^{\geq m}_{j^*L})
\]
is a Poincaré-Karoubi projection. In particular, the sequence
\[
(\Dperf_Z(X),\QF_L^{\ge m})\to (\Dperf(X),\QF_L^{\ge m})\to (\Dperf(U),\QF^{\ge m}_{L|_U})\,.
\]
is a Poincaré-Karoubi sequence.
\end{corollary}

\begin{corollary}[Localisation sequence]
Let $S$ be a qcqs scheme and $\F\colon \Catp \to \E$ be a Karoubi-localising functor valued in some stable presentable $\infty$-category $\E$ (e.g., $\F=\KGW$ or $\F=\KL$).
Let $X$ be a qcqs scheme equipped with a invertible perfect complex with $\Ct$-action $L\in\Picspace(X)^{\BC}$. Then for every qc open subset $U$ with $Z=X\smallsetminus U$ and every $m \in \ZZ \cup \{-\infty,\infty\}$ we have a fibre sequence 
\[
\F(\Dperf_Z(X),\QF_L^{\ge m})\to \F(\Dperf(X),\QF_L^{\ge m})\to \F(\Dperf(U),\QF^{\ge m}_{L|_U})\,.
\]
\end{corollary}

The proof of Proposition~\ref{proposition:open-is-bounded} uses the following lemma:

\begin{lemma}
\label{lemma:non-positive-tor-amplitude}%
Let $\A$ be a stably symmetric monoidal $\infty$-category equipped with a multiplicative orientation $(\Ind(\A)_{\geq 0},\Ind(\A)_{\leq 0})$. Suppose that every object in $\Ind(\A)_{\geq 0}$ is a filtered colimit of objects in $\A \cap \Ind(\A)_{\geq 0}$. Then every object $x \in \Ind(\A)$ which has Tor amplitude $\leq n$ is a filtered colimits of objects with Tor amplitude $\leq n$ which furthermore belong to $\A \subseteq \Ind(\A)$. 
\end{lemma}
\begin{proof}
Let us fix $x \in \Ind(\A)$ with Tor amplitude $\leq n$. 
Then $x$ is the colimit of the canonical filtered diagram $\A_{/x}\to\Ind(\A)$. 
Let $\I \subseteq \A_{/x}$ be the full subcategory spanned by those $z \in x$ where $z$ has Tor amplitude $\leq n$. We claim that the inclusion $\I \subseteq \A_{/x}$ is cofinal. 
By Quillen's theorem A what we need to verify is that if $f\colon z\to x$ is a map with $z \in \A$, then the $\infty$-category $\A_{z//x} \times_{\A} \I$ of factorizations of $z \to w \to x$ such that $w $ is in $\A$ and has Tor amplitude $\leq n$, is weakly contractible.
 
First let us prove that $\A_{z//x} \times_{\A} \I$ is non empty. 
Since $x$ has Tor amplitude $\leq n$ the object $x\otimes \tau_{\le-n-1}(\Dual z)$ is $(-1)$-coconnective, and since the unit $\one_\A$ is connective (by the assumption that the orientation is multiplicative) we get that the dual map 
\[
\tilde f\colon \one_\A\to x\otimes \Dual z
\]
lifts to $x\otimes t_{\ge -n}(\Dual z)$. Now, by our assumption, we have that $t_{\ge -n}(\Dual z)$ is a filtered colimit of $(-n)$-connective objects which belong to $\A$, and so we can find a $(-n)$-connective object $y \in \A$ such that $\tilde f$ factors as
\[
\one_\A \to x \otimes y \to x \otimes \Dual z .
\]
It then follows that $f\colon z \to x$ factors as
\[
z \to \Dual y \to z ,
\]
where $\Dual y$ has Tor amplitude $\leq n$ by Lemma~\ref{lemma:dual-amplitude}, and so $\A_{z//x} \times_{\A} \I$ is non-empty.
\end{proof}

\begin{proof}[Proof of Proposition~\ref{proposition:open-is-bounded}]
We first note that the underlying exact functor $j^*\colon \Dperf(X) \to \Dperf(U)$ is indeed a Karoubi projection by~\ref{theorem:perfect-generators} \ref{item:karoubi-sequence-intersection-closed}.
Let $g\colon \Dperf(U) \to \Pro\Dperf(X)$ be its Pro-left adjoint. We need to show that for every $R \in \Dperf(U)$ the object $g(R) \in \Pro\Dperf(X)$ can be written as the limit of a cofiltered diagram $\{P_\alp\}_{\alp \in \I}$ such that 
\[
\uline{\map}(P_{\alp},\Dual_L P_{\alp}) = \uline{\map}(P_{\alp},\Dual_X P_{\alp} \otimes L) = p_*(\Dual_X P_{\alp} \otimes \Dual_X P_{\alp} \otimes L) \in \Dperf(S)
\]
has bounded Tor amplitude uniformly in $\alp$, where $p\colon X \to S$ is the structure map of $X$. Now the duality $\Dual_X\colon \Dperf(X) \tosimeq \Dperf(X)\op$ induces an equivalence 
\[
\wtl{\Dual}_X\colon \Der^{\qc}(X) = \Ind\Dperf(X) \tosimeq \Ind\Dperf(X)\op = \Pro(\Dperf(X))\op
\]
which intertwines the Pro-left and Ind-right adjoints of $j^*$. More precisely, we have
\[
g(R) = \wtl{\Dual}_X j_*(\Dual_U(R)),
\]
where $j_*\colon \Dperf(X) \to \Der^{\qc}(U)$ is the push-forward functor, which is the Ind-right adjoint of $j^*$. Replacing $R$ with $\Dual_UR$, it will hence suffice to show that for every $R \in \Dperf(U)$, the quasi-coherent complex $j_*U \in \Der^{\qc}$ can be written as a filtered colimit of perfect complexes $\{Q_{\bet}\}$ such that $p_*(Q_{\bet} \otimes Q_{\bet} \otimes L)$ has bounded Tor amplitude uniformly in $\bet$. 

Since $p\colon X \to S$ is smooth it is in particular flat, and hence the pullback functor $p^*\colon \Der^{\qc}(S) \to \Der^{\qc}(X)$ preserves coconnective objects. By the projection formula, it follows that for any given $n$ the push-forward functor $p_*\colon \Der^{\qc}(X) \to \Derqc(S)$ preserves the property of having bounded Tor amplitude $\leq n$. It will hence suffice to find $\{Q_{\bet}\}$ such that $\{Q_{\bet} \otimes Q_{\bet} \otimes L\}$ has uniformly bounded Tor amplitude in $\Derqc(X)$. Now since $L$ is a perfect complex it has Tor amplitude $\leq n'$ for some $n'$, and so it will suffice to show that $Q_{\bet}$ can be chosen to have a uniformally bounded Tor amplitude. Indeed, since $j\colon U \to X$ is flat we have that $j^*$ preserves coconnective objects and hence by the projection formula $j_*$ preserves Tor amplitudes. In particular, since $R$ is perfect it has Tor amplitude $\leq n$ for some $n$ by Lemma~\ref{lemma:dual-amplitude}, which means that $j_*R$ has Tor amplitude $\leq n$, and so by Lemma~\ref{lemma:non-positive-tor-amplitude} it can be written as a filtered colimit of objects with Tor amplitude $\leq n$, as desired.
\end{proof}

Finally, let us prove that bounded Karoubi sequences are stable under tensoring with flat $\infty$-categories with duality. This property is what makes the notion of bounded Karoubi sequence useful to set up a multiplicative framework for motivic realization in \S\ref{subsection:nisnevich-invariants}.

\begin{proposition}
\label{proposition:flat-multi}%
Let $\A$ be stably symmetric monoidal $\infty$-category equipped with a multiplicative orientation and satisfying Assumption~\ref{assumption:rigid}.
Let $(\C,\Dual)$ be an $\A$-linear $\infty$-category with duality. Then the operation of tensoring with $(\C,\Dual)$ 
\[
(\C,\Dual) \otimes_A (-)\colon \Mod_{(\A,\Dual_\A)}(\Catps) \to \Mod_{(\A,\Dual_\A)}(\Catps)
\]
preserves Karoubi sequences of $\A$-linear $\infty$-categories with duality. Furthermore, if $(\C,\Dual)$ is flat then $(\C,\Dual) \otimes_A (-)$ preserves bounded Karoubi sequences. 
\end{proposition}
\begin{proof}
For the first claim, note that since the operation $\C \otimes_\A (-)$ preserves colimits, it suffices to show that it preserves Karoubi inclusions, that is, fully-faithful embeddings. For this, it will suffice to show that if $\D \to \D'$ is fully-faithful then the induced functor
\[
\Ind(\C \otimes_\A \D) \to \Ind(\C \otimes_\A \D')
\]
is fully-faithful.
The formula~\eqref{equation:formula-tensor} gives us canonical equivalences
\begin{align*} 
\Ind(\C \otimes_\A \D) &= \Funx((\C \otimes_A \D)\op,\Spa) \\
&= \Fun^{\A\op}(\C\op,\Ind(\D)) \ .
\end{align*}
To avoid confusion, let us point out that the functoriality of $\Fun^{\A\op}(\C\op,\Ind(\D))$ in $\D$ is slightly delicate: given an $\A$-linear functor $f\colon \D \to \D'$, one has an associated $\A\op$-linear functor $f^*\colon \Ind(\D') \to \Ind(\D)$ induced by restriction along $f$, and hence an induced $\A\op$-linear functor 
\[
\Fun^{\A\op}(\C\op,\Ind(\D')) \to \Fun^{\A\op}(\C\op,\Ind(\D)),
\]
whose left adjoint is the $\A$-linear functor encoding the covariant dependence in $\D$. This left adjoint is generally \emph{not} induced by post-composition with $f_!\colon \Ind(\D) \to \Ind(\D')$, since the latter is generally not $\A\op$-linear (only oplax $\A\op$-linear). However, if the oplax $\A\op$-linear structure on $f_!$ happens to be strict, then post composition with $f_!$ is well-defined and hence yields a left adjoint to post-composition with $f^*\colon \Ind(\D') \to \Ind(\D)$. This phenomenon happens for example when $\A$ is assumed to be rigid, since then by adjunction the $\A\op$-action on $\Ind(\D)$ is given by pre-composing the $\A$-action on it with the duality $\Dual_\A\colon \A\op \to \A$, and hence any $\A$-linear functor $\Ind(\D) \to \Ind(\D')$ is also $\A\op$-linear. In particular, if $f\colon \D \to \D'$ is fully-faithful then $f_!\colon \Ind(\D) \to \Ind(\D')$ is fully-faithful and so the post-composition functor
\[
f_! \circ\colon \Fun^{\A\op}(\C\op,\Ind(\D)) \to \Fun^{\A\op}(\C\op,\Ind(\D'))
\]
is fully-faithful, so that 
\[
\Ind(\C \otimes_A \D) \to \Ind(\C \otimes_\A \D')
\]
is fully-faithful, as desired.

To prove the second claim, let us assume that $(\C,\Dual)$ is flat and let $p\colon (\D,\Dual_\D) \to (\E,\Dual_\E)$ be a bounded Karoubi projection of $\A$-linear $\infty$-categories with duality. By the first part of the proposition we have that the induced functor
\[
\C \otimes_\A p \colon (\C,\Dual_\C) \otimes_{\A} (\D,\Dual_\D) \to (\C,\Dual_\C) \otimes_{\A}(\E,\Dual_\E)
\]
is also a Karoubi projection. We now show that $\C \otimes_\A p$ is bounded. Since $\C \otimes_\A \E$ is generated as a stable $\infty$-category by the pure tensors $z \otimes x \in \C \otimes_\A \E$, and by assumption $\E$ is generated as a stable $\infty$-category by the collection of $x \in \E$ at which $p$ is bounded, it will suffice to show that if $p$ is bounded at $x$ then $\C \otimes_\A p$ is bounded at each $z \otimes x$ for every $z \in \E$. Let $\{y_i\}_{i \in \I}$ be a cofiltered in $\D$ family with limit $g(x)$, where $g\colon \E \to \Pro(\D)$ is a Pro-left adjoint of $p$. Then $\{z \otimes y_i\}_{i \in \I}$ is a cofiltered family in $\C \otimes \D$ with limit $z \otimes g(x)$ and by~\eqref{equation:pure-tensors} we get that
\[
\{\uline{\map}_{\C}(z \otimes y_i, \Dual_{\C}(z) \otimes \Dual_{\C}(y_i))\}_{i \in \I}= \{\uline{\map}_{\C}(z,\Dual_{\C}(z)) \otimes \uline{\map}_{\C}(y_i,\Dual(y_i))\}_{i \in \I} ,
\]
and so $\C \otimes_\A p$ is bounded at $z \otimes x$ by Remark~\ref{remark:tor-amp-tensor}.
\end{proof}

\subsection{Zariski descent}
\label{subsection:zariski-descent}%

In this section, we assemble the results of the previous two sections into various forms of Zariski descent.

\begin{proposition}
\label{proposition:zariski-descent-poinc}%
Let $X$ be a qcqs scheme and $j_0\colon U_0 \hrar X \hookleftarrow U_1\cocolon j_1$ a pair of qcqs open subschemes and write $j_{01}\colon U_0 \cap U_1 \hrar X$ for the inclusion of the intersection.
Then for any invertible perfect complex with $\Ct$-action $L \in \Picspace(X)^{\BC}$ on $X$ and 
for every $-\infty \leq m \leq \infty$ the corresponding square of Poincaré $\infty$-categories
\begin{equation}
\label{equation:zariski-square}%
\begin{tikzcd}
(\Dperf(X),\QF^{\ge m}_L)\ar[r] \ar[d] & \ar[d] (\Dperf(U_1),\QF^{\ge m}_{j_1^*L})\\
(\Dperf(U_0),\QF^{\ge m}_{j_0^*L})\ar[r] & (\Dperf(U_0 \cap U_1),\QF^{\ge m}_{j_{01}^*L})
\end{tikzcd}
\end{equation}
is a Poincaré-Karoubi square.
\end{proposition}

\begin{corollary}[Excision]
\label{corollary:poinc-excision}%
Let $X$ be a qcqs scheme, $U_0 \subseteq X$ a qcqs open subscheme with closed complement $Z = X \setminus U_0$ and $U_1$ an open subscheme which contains $Z$. Then, for every invertible perfect complex with $\Ct$-action $L$ and every $m \in \{-\infty,\infty\} \cup \ZZ$ the functor
\[
(\Dperf_Z(X),\QF^{\geq m}_L) \to (\Dperf_Z(U_1), \QF^{\geq m}_{L|_{U_1}})
\]
is an equivalence of Poincaré $\infty$-categories. 
\end{corollary}
\begin{proof}
This is the induced Poincaré functor on vertical fibres in the square~\eqref{equation:zariski-square}.
\end{proof}

\begin{proof}[Proof of Proposition~\ref{proposition:zariski-descent-poinc}]
By Corollary~\ref{corollary:open-is-karoubi} all the Poincaré functors in the square~\eqref{equation:zariski-square} are Poincaré-Karoubi projections. To finish the proof we now show that the square is cartesian in $\Catp$. For this, note that the underlying square of stable $\infty$-categories is cartesian by Remark~\ref{remark:inherited-perfect}, and hence to prove that the square is cartesian in $\Catp$ it will suffice to show that the square of Poincaré structures
\[
\begin{tikzcd}
\QF^{\ge m}_{L}(-) \ar[r]\ar[d] & \QF^{\ge m}_{j_1^*L}(j_1^*(-))\ar[d] \\
\QF^{\ge m}_{j_0^*L}(j_0^*(-))\ar[r] & \QF^{\ge m}_{j_{01}^*L}(j^*_{01}(-))
\end{tikzcd}
\]
is exact. Now since the square~\eqref{equation:zariski-square} is cartesian on the level of underlying stable $\infty$-categories it is also cartesian on the level of underlying stable $\infty$-categories with duality, and hence the above square of Poincaré structures is cartesian for $m=-\infty,\infty$. For a general $m$ it will hence suffice to show that this square becomes cartesian after passing to linear parts. Now, by definition, the linear part of $\QF^{\sym}_L$ is represented by the object $\tau_{\geq m}\LT_L \in \Derqc(X)$. Since the Poincaré functors in the square~\eqref{equation:zariski-square} are Karoubi projections we have that for $\eps\in \{0,1,01\}$ the natural transformation 
\[
\QF^{\ge m}_{L}(-) \Rightarrow \QF^{\ge m}_{j_{\eps}^*L}(j_{\eps}^*(-))
\]
exhibits $\QF^{\ge m}_{j_{\eps}^*L}(-)$ as the left Kan extension of $\QF^{\ge m}_L$ along $j_{\eps}^*$. To finish the proof it will hence suffice to show that the square
\[
\begin{tikzcd}
\tau_{\geq m}\E_L \ar[r]\ar[d] & (j_1)_*(j_1)^*(\tau_{\geq m}\E_{L})\ar[d] \\
(j_0)_*(j_0)^*(\tau_{\geq m}\E_{L}) \ar[r] & (j_{01})_*(j_{01})^*(\tau_{\geq m}\E_{L})
\end{tikzcd}
\]
is exact in $\Derqc(X)$. Indeed, this follows from the fact that the square
\[
\begin{tikzcd}
\Derqc(X)\ar[r] \ar[d] & \ar[d] \Derqc(U_1) \\
\Derqc(U_0) \ar[r] & \Derqc(U_0 \cap U_1)
\end{tikzcd}
\]
is cartesian, see Corollary~\ref{corollary:descent-qc}.	
\end{proof}

\begin{corollary}
\label{proposition:poincare-descent}%
Let $X$ be a scheme equipped with a finite open covering $X = \cup_{i=1}^n U_i$ by opens. For every non-empty subsets $S \subseteq \{1,\ldots,n\}$ let us denote by $U_S = \cap_{i \in S} U_i$ and $j_S\colon U_S \hrar X$ the associated inclusion.
Then for any invertible perfect complex with $\Ct$-action $L \in \Picspace(X)^{\BC}$ on $X$ and 
for every $-\infty \leq m \leq \infty$ the Poincaré functor 
\begin{equation}
\label{equation:descent-poincare}%
(\Dperf(X),\QF^{\geq m}_L) \to \lim_{\emptyset \neq S \subseteq \{1,\ldots,n\}} (\Dperf(U_S),\QF^{\geq m}_{j_S^*L}) 
\end{equation}
is an equivalence of Poincaré $\infty$-categories.
\end{corollary}

\begin{corollary}
\label{corollary:zariski-descent-pn}%
Let $S$ be a qcqs scheme and $L \in \Picspace(S)^{\BC}$ a invertible perfect complex with $\Ct$-action. Then for every $m \in \ZZ \cup \{-\infty,+\infty\}$, the functor 
\[
\qSch_{/S} \to \Catp \quad\quad [p\colon X \to S] \mapsto (\Dperf(X),\QF^{\geq m}_{p^*L})
\]
is a $\Catp$-valued sheaf with respect to the Zariski topology.
\end{corollary}

\begin{corollary}[Zariski descent]
\label{corollary:zariski-descent}%
Let $S$ be a qcqs scheme equipped with a invertible perfect complex with $\Ct$-action $L\in\Picspace(X)^{\BC}$. Let $\F\colon \Catp \to \A$ be a Karoubi-localising functor valued in some presentable $\infty$-category $\A$ (e.g., $\A=\Spa$ and $\F$ is $\KGW$ or $\KL$, or, say, $\A=\Sps$ and $\F=\GWspace(-^{\natural})$). 
Then for every $m \in \ZZ \cup \{-\infty,+\infty\}$, the functor
\[
\F^{\geq m}_L\colon \qSch_{/S} \to \A \qquad \F^{\geq m}_L(p\colon X\to S) = \F(\Dperf(X),\QF^{\ge m}_{p^*L})
\]
is a Zariski sheaf. 
\end{corollary}
\begin{proof}
Combine Proposition~\ref{proposition:zariski-descent-poinc} and Corollary~\ref{corollary:zariski-sheaf}.
\end{proof}

\begin{corollary}
\label{corollary:open-is-karoubi-2}%
Let $X$ be qcqs scheme and $Z,W \subseteq X$ two closed subschemes with quasi-compact complements $V,U$, respectively. Let $L \in \Picspace(X)^{\BC}$ be an invertible perfect complex with $\Ct$-action. Then for every $m \in \ZZ \cup \{-\infty,+\infty\}$, the Poincaré functor
\[
(\Dperf_Z(X),\QF^{\geq m}_L) \to (\Dperf_{U \cap Z}(U),\QF^{\geq m}_{L|_U})
\]
is a Poincaré-Karoubi projection. In particular, the sequence
\[
(\Dperf_{Z \cap W}(X),\QF_L^{\ge m})\to (\Dperf_Z(X),\QF_L^{\ge m})\to (\Dperf_{U \cap Z}(U),\QF^{\ge m}_{L|_U})\,.
\]
is a Poincaré-Karoubi sequence.
\end{corollary}
\begin{proof}
By excision (Corollary~\ref{corollary:poinc-excision}) the Poincaré functor in question fits in a cartesian square of Poincaré $\infty$-categories
\[
\begin{tikzcd}[column sep = 5pt]
(\Dperf_Z(X),\QF^{\geq m}_L) \ar[rr]\ar[d] && (\Dperf_{U \cap Z}(U \cup V),\QF^{\geq m}_{L|_{U \cup V}}) \ar[r,phantom, "\simeq"] \ar[d] & (\Dperf_{U \cap Z}(U),\QF^{\geq m}_{L|_{U}})  \\
(\Dperf(X),\QF^{\geq m}_L) \ar[rr] && (\Dperf(U \cup V),\QF^{\geq m}_{L|_{U \cup V}})
\end{tikzcd}
\]
By~\cite[Proposition 8.1.9]{9-authors-II} Poincaré-Verdier projections are preserved under base change. Since the bottom horizontal arrow is a Poincaré-Verdier projection by Corollary~\ref{corollary:open-is-karoubi-2} we conclude that the top horizontal arrow is a Poincaré-Verdier projection as well.
\end{proof}

\subsection{Nisnevich descent}
\label{subsection:nisnevich-descent}%

In this section, we upgrade the Zariski descent result of the previous section to Nisnevich descent.

\begin{proposition}
\label{proposition:affine-exicision}%
Let $f\colon A \to B$ be a flat homomorphism of commutative rings and $S \subseteq A$ a set of elements such that for every $s \in S$, the map $A/sA \to B/sB$ is an isomorphism of $A$-modules. Let $M$ be an invertible $A$-module with $A$-linear involution and $N = B \otimes_A M$ the induced invertible $B$-module (with $B$-linear involution). 
Then for every $m \in \ZZ \cup \{-\infty,+\infty\}$, the square
\[
\begin{tikzcd}
(\Dperf(A),\QF^{\geq m}_M) \ar[r]\ar[d] & (\Dperf(B),\QF^{\geq m}_M) \ar[d] \\
(\Dperf(A[S^{-1}]),\QF^{\geq m}_{M[S^{-1}]}) \ar[r] & (\Dperf(B[f(S)^{-1}]),\QF^{\geq m}_{N[f(S)^{-1}]}) 
\end{tikzcd}
\]
is a Poincaré-Karoubi square.
\end{proposition}
\begin{proof}
We apply~\cite[Proposition 4.4.20]{9-authors-II}, and thus need to verify that its assumptions (i)-(v) hold in the present case. For (i)-(iv) this is clear: assumption (i) is by construction, (ii) holds in the commutative case whenever the involution on $M$ is $A$-linear (see~\cite[Example 1.4.4]{9-authors-II}), (iii) is assumed directly and (iv) (the Ore condition) is automatic in the commutative case. To prove that (v) holds, we need to show that the boundary maps $\hat{\mathrm{H}}^n(\Ct,N[f(S)^{-1}]) \to \hat{\mathrm{H}}^n(\Ct,M)$ associated to the short exact sequence of abelian groups
\[
0 \to M \to N \oplus M[S^{-1}] \to N[f(S)^{-1}] \to 0
\]
vanish. Equivalently, we need to show that the map 
\[
0 \to \hat{\mathrm{H}}^n(\Ct,M) \to \hat{\mathrm{H}}^n(\Ct,N) \oplus \hat{\mathrm{H}}^n(\Ct,M[S^{-1}])
\]
is injective for every $n$. Now, for odd $n$, the Tate cohomology group $\hat{\mathrm{H}}^n(\Ct,M)$ is given by the kernel of the norm map $M_{\Ct} \to M^{\Ct}$, while for even $n$ by the cokernel of this map, or, equivalently, by the kernel of the norm map $M(-1)_{\Ct} \to M(-1)^{\Ct}$, where $M(-1)$ denotes $M$ with its $\Ct$-action sign twisted. It will hence suffice to show that the maps
\[
M_{\Ct} \to  N_{\Ct} \oplus M[S^{-1}]_{\Ct}
\]
and
\[
M(-1)_{\Ct} \to  N(-1)_{\Ct} \oplus M(-1)[S^{-1}]_{\Ct}
\]
are injective. Indeed, since the $\Ct$-action on $M$ is $A$-linear we can identify these maps with the result of applying the functors $M_{\Ct} \otimes_A (-)$ and $M(-1)_{\Ct} \otimes_A (-)$ to the injective $A$-module map
\[
A \to B \oplus A[S^{-1}],
\]
and so the desired result follows from the fact that the cokernel of the last map is $B[p(S)^{-1}]$, which is flat as an $A$-module by assumption.
\end{proof}

\begin{corollary}[Nisnevich descent]
\label{corollary:nisnevich-descent}%
Let $S$ be a qcqs scheme equipped with a invertible perfect complex with $\Ct$-action 
$L\in\Picspace(X)^{\BC}$.
Let $\F\colon \Catp \to \A$ be a Karoubi-localising functor valued in a presentable $\infty$-category $\A$ (e.g., $\A=\Spa$ and $\F$ is $\KGW$ or $\KL$, or, say, $\A=\Sps$ and $\F=\GWspace(-^{\natural})$).
Then for every $m \in \{-\infty,+\infty\} \cup \ZZ$, 
the functor
\[
\F_L^{\geq m}\colon \qSch_{/S} \to \A \quad\quad \F_L^{\geq m}(p\colon X\to S) =  \F(\Dperf(X),\QF^{\ge m}_{p^*L})
\]
is a Nisnevich sheaf. 
\end{corollary}
\begin{proof}
We first reduce to the case where $L$ is a shift of a line bundle. For this, recall from Remark~\ref{remark:shift-of-line-bundle} that there exists a disjoint union decomposition $S = \coprod_n S_n$, and for each $n$ a line bundle $L_n$ on $S_n$, such that $L|_{S_n} = L_n[n]$. Now, since $S$ is quasi-compact, the $S_n$'s must be empty for all but finitely many $n$. In particular, we may write $S = S_{n_1} \amalg \cdots \amalg S_{n_r}$ for some $r \geq 1$. We now note that the category of qcqs schemes is extensive, so that, in particular, we have an equivalence of categories
\begin{equation}
\label{equation:extensive-equiv}%
\prod_{i=1}^{r}\qSch_{/S_{n_i}} \tosimeq \qSch_{/S} \quad\quad \{[X_i \to S_{n_i}]\}_{i=1}^{r} \mapsto \Big[\coprod_{i=1}^{r} X_i \to S\Big] .
\end{equation}
Let $\F^{\geq m}_i$ be the restriction of $\F_L^{\geq m}$ along the functor $(\qSch_{/S_{n_i}})\op \to (\qSch_{/S})\op$ induced by post-composition with $S_{n_i} \hrar S$. Since $\F_L^{\geq m}$ is a Zariski sheaf by Corollary~\ref{corollary:zariski-descent} it sends finite disjoint unions of schemes to products in $\E$. We may hence identify $\F_L^{\geq m}$ with the composite
\[
\prod_{i=1}^{r}\qSch_{/S_{n_i}} \xrightarrow{\prod \F^{\geq m}_i} \prod_{i=1}^{r} \E \xrightarrow{\times} \E .
\]
In addition, by Lemma~\ref{lemma:zariski-nisnevich-coverage}, the equivalence of categories~\eqref{equation:extensive-equiv} identifies Nisnevich epis on the right hand side with tuples of Nisnevich epis on the right hand side. It will hence suffice to show that each of the presheafs $\F^{\geq m}_i \colon (\qSch_{/S_{n_i}})\op \to \E$ is a Nisnevich sheaf. Replacing $S$ with $S_{n_i}$ we may consequently assume without loss of generality that $L$ is a shift of a line bundle. 

Now since we already know that $\F_L^{\geq m}$ is a Zariski sheaf it will suffice by Proposition~\ref{proposition:reduce-affine} to show the restriction $\F^{\geq m}_{\aff} := (\F^{\geq m}_L)|_{\Aff_{/S}}$ is a Nisnevich sheaf. We note that since all coverings in the Nisnevich topology consist of étale maps we have that $\F^{\geq m}_{\aff}$ is a Nisnevich sheaf if and only if for very affine scheme $\spec(R) \to S$ the restriction of $\F^{\geq m}_{\aff}$ to the category of affine étale $R$-schemes is a Nisnevich sheaf. By~\cite[Theorem B.5.0.3]{SAG} the latter property is equivalent to $\F^{\geq m}_{\aff}$ vanishing on the empty scheme and satisfying Nisnevich excision for affine schemes in the sense of~\cite[Definition B.5.0.1]{SAG}, that is, for every map $f\colon A \to B$ of étale $R$-algebras, and every $s \in A$ such that the map $A/sA \to B/f(s)B$ is an isomorphism of $A$-modules, the square
\[
\begin{tikzcd}
\F^{\geq m}_{\aff}(\spec(A)) \ar[r]\ar[d] & \F^{\geq m}_{\aff}(\spec(B)) \ar[d] \\
\F^{\geq m}_{\aff}(\spec(A[s^{-1}]) \ar[r] & \F^{\geq m}_{\aff}(\spec(B[f(s)^{-1}]) 
\end{tikzcd}
\]
is cartesian.
To finish the proof, we now point out that $\F^{\geq m}_{\aff}$ indeed vanishes on the empty scheme since it is a Zariski sheaf, and satisfies Nisnevich excision for affine schemes 
since $\F$ is Karoubi-localising and for every invertible $A$-module with $A$-linear involution $M$ and any $n \in \ZZ$ and the square
\[
\begin{tikzcd}
\big(\Dperf(A),(\QF^{\geq m}_M\big)\qshift{n}) \ar[r]\ar[d] & (\Dperf(B),\big(\QF^{\geq m}_M)\qshift{n}\big) \ar[d] \\
\big(\Dperf(A[s^{-1}]),(\QF^{\geq m}_{M[s^{-1}]})\qshift{n}\big) \ar[r] & \big(\Dperf(B[f(s)^{-1}]),(\QF^{\geq m}_{N[f(s)^{-1}]})\qshift{n}\big) 
\end{tikzcd}
\]
is a Poincaré-Karoubi square by Proposition~\ref{proposition:affine-exicision}, where $N=M \otimes_A B$. 
\end{proof}

\subsection{The coniveau filtration}
\label{subsection:coniveau}%

Our goal in this subsection is to prove the following Poincaré analogue of the coniveau filtration along the lines of \cite{Balmer-filtrations}. Let $Z\subseteq X$ a closed subscheme with quasi-compact complement. For $0\le c\le \infty$ we denote by $Z^{(c)}$ the set of points of $Z$ of codimension $c$, that is, the set of points for which $\cO_{Z,x}$ is $\geq c$. 
We write $\Dperf_Z(X)^{\ge c} \subseteq \Dperf_Z(X)$ for the full subcategory of $\Dperf_Z(X)$ spanned by those perfect complexes such that any point in the support of $P$ belongs to $Z^{(c)}$. 
In particular $\Dperf_Z(X)^{\ge 0} = \Dperf_Z(X)$, and if $Z$ has Krull dimension $e$ then $\Dperf_Z(X)^{\geq e+1} = 0$. Since the duality preserves the support, the subcategory $\Dperf_Z(X)^{\ge c}$ inherits from $\Dperf_Z(X)$ its Poincaré structure.

\begin{proposition}
\label{proposition:coniveau-sequence}%
Let $X$ be a Noetherian scheme and $Z\subseteq X$ a closed subset with quasi-compact complement. Then the restriction functors $(\Dperf_Z(X)^{\ge c},\QF^{\ge m}_L) \to (\Dperf_x(\cO_{X,x}),\QF^{\ge m}_L)
$ for $x \in Z^{(c)}$ fit in a Poincaré-Verdier sequence 
\[
(\Dperf_Z(X)^{\ge c+1},\QF^{\ge m}_L)\to (\Dperf_Z(X)^{\ge c},\QF^{\ge m}_L)\to \bigoplus_{x\in Z^{(c)}} (\Dperf_x(\cO_{X,x}),\QF^{\ge m}_L) \ .
\]
\end{proposition}

The main ingredient in constructing the Poincaré-Verdier sequence of Proposition~\ref{proposition:coniveau-sequence} is the Poincaré-Verdier sequence of Corollary~\ref{corollary:open-is-karoubi-2}. The proof will require however a couple of additional lemmas.

\begin{lemma}
\label{lemma:complexes-with-disjoint-support}%
Let $X$ be a quasi-compact and quasi-separated scheme and let $Z_1,Z_2\subseteq X$ two disjoint closed subsets such that $X\smallsetminus Z_i$ is quasi-compact. Then the canonical map of Poincaré $\infty$-categories
\[
(\Dperf_{Z_1}(X),\QF^{\ge m}_L)\oplus (\Dperf_{Z_2}(X),\QF^{\ge m}_L)\to (\Dperf_{Z_1\cup Z_2}(X),\QF^{\ge m}_L)\qquad (P_1,P_2)\mapsto P_1\oplus P_2
\]
is an equivalence.
\end{lemma}
\begin{proof}
For every $P_i\in \Dperf_{Z_i}(X)$ both $P_1\otimes P_2$ and $P_1\otimes P_2^\vee$ are supported on $Z_1\cap Z_2=\varnothing$ and so are 0. In particular, the map in the statement is a fully faithful Poincaré inclusion. It suffices to show it is essentially surjective, that is that every $P\in \Dperf_{Z_1\cup Z_2}(X)$ can be written as $P\simeq P_1\oplus P_2$ with $P_i$ supported on $Z_i$.
    
Write $U_i=X\smallsetminus Z_i$ for the open complements and consider the $\cO_X$-algebras $\cO_{U_i}$. Since $U_1\cap U_2=X\smallsetminus (Z_1\cup Z_2)$ we have that $P\otimes_{\cO_X}\cO_{U_1}\otimes_{\cO_X}\cO_{U_2}=0$. We can write $P_1=P\otimes_{\cO_X}\cO_{U_2}$ and $P_2=P\otimes_{\cO_X}\cO_{U_1}$, so that $P_i$ is supported on $Z_i$. Consider now the map
\[
P\to P_1\oplus P_2\,,
\]
obtained by tensoring $\cO_X\to \cO_{U_1}\oplus \cO_{U_2}$ with $P$. This is an equivalence after restricting to $U_i$, and so it is an equivalence on $U_1\cup U_2=X$. Finally $P_1$ and $P_2$ are perfect because summand of a perfect complex.
\end{proof}

\begin{lemma}
\label{lemma:stalk-of-D_Z}%
Let $X$ be a quasi-compact quasi-separated scheme and $Z\subseteq X$ a closed subset with quasi-compact complement. Let $x \in Z$ be a point.
Then there is a natural equivalence of Poincaré $\infty$-categories
\[
\colim_{V\ni x} (\Dperf_{Z\cap V}(V),\QF^{\ge m}_L)\simeq (\Dperf_{x}(\cO_{X,x}),\QF^{\ge m}_L)\,,
\]
where the colimit runs over all open neighborhoods of $x$ which contain $X \setminus Z$.
\end{lemma}
\begin{proof}
Write $V_x = \spec(\cO_{X,x})$.
Given an open neighborhood $V$ of $x$, write $U = V \setminus (V \cap Z)$ and consider the commutative diagram
\[
\begin{tikzcd}
(\Dperf_{V \cap Z}(V),\QF^{\geq m}_{L|_V}) \ar[d]\ar[r] & (\Dperf(V),\QF^{\geq m}_{L|_V})\ar[r] \ar[d] & (\Dperf(U),\QF^{\geq m}_{L|_U}) \ar[d] \\
(\Dperf_x(V_x),\QF^{\geq m}_{L|_{V_x}}) \ar[r] & (\Dperf(V_x),\QF^{\geq m}_{L|_{V_x}}) \ar[r] & (\Dperf(V_x \setminus x), \QF^{\geq m}_{L|_{V_x \setminus x}})
\end{tikzcd}
\]
whose rows are Poincaré-Verdier sequences by Corollary~\ref{corollary:open-is-karoubi}. Since $x$ is in $Z$ we may identify $V_x \setminus x$ with $V_x \cap U$. We claim that the middle and right vertical arrows become equivalences when one takes the colimit over all open neighborhoods $V$ of $x$. Equivalently, since the family of neighborhoods is a filtered poset, this is the same as saying that the middle and right vertical arrows exhibit their targets as the colimit of their domains when $V$ ranges over all neighborhoods of $x$. Indeed, this holds on the level of underlying stable $\infty$-categories by Lemma~\ref{lemma:inverse-system}. At the same time, for every $V$ the middle and right most vertical maps are Poincaré-Karoubi projections by Corollary~\ref{corollary:limit-open-karoubi}, and so their colimit is again a Poincaré-Karoubi projection. But any Poincaré-Karoubi projection whose underlying exact functor is an equivalence is itself an equivalence of Poincaré $\infty$-categories.

Now since $\Catp$ is compactly generated filtered colimits commute with finite limits, and hence we conclude that the family of left most vertical Poincaré functors also exhibit their target as the colimit of their domains when $V$ ranges over all neighborhoods of $x$. It is left to explain why one may as well consider only neighborhoods which contain $X \setminus Z$. For this, note that if $V \subseteq V'$ is a pair of neighborhoods of $x$ such that $V \cap Z = V' \cap Z$ then the induced functor
\[
(\Dperf_{V \cap Z}(V),\QF^{\geq m}_{L|_V}) \to (\Dperf_{V' \cap Z}(V'),\QF^{\geq m}_{L|_{V'}})
\]
is an equivalence of Poincaré $\infty$-categories by excision (Corollary~\ref{corollary:poinc-excision}). It follows that the diagram $V \mapsto (\Dperf_{V \cap Z}(V),\QF^{\geq m}_{L|_V})$ factors through the poset map $V \mapsto V \cup (X \setminus Z)$, considered as a map from the poset of open neighborhoods of $x$ (with reverse inclusion) to the poset of open neighborhoods of $x$ which contain $X \setminus Z$. This last map is cofinal and hence the desired result follows.
\end{proof}

\begin{proof}[Proof of Proposition~\ref{proposition:coniveau-sequence}]
For ease of readability we will suppress the Poincaré structures from the proof.
    
For every $c\ge0$, write $\A_Z(c)$ for the poset consisting of those closed subsets $W\subseteq Z$ such that all points of $W$ have codimension $\geq c$ in $Z$, where the partial order is given by inclusion. Let $\cB_Z(c) \subseteq \A_Z(c+1) \times \A_Z(c)$ be the full subposet spanned by those pairs $(W',W)$ such that $W'\subseteq W$. Since the collection of subsets $\A_Z(c)$ is non-empty (it always contains the empty set) and closed under disjoint union we see that $\A_Z(c)$ is filtered. Similarly, $\cB_Z(c)$ is filtered. For every $(W',W) \subseteq \cB_Z(c)$ we have by Corollary~\ref{corollary:open-is-karoubi-2} a Poincaré-Karoubi sequence
\[
(\Dperf_{W'}(X),\QF^{\geq m}_{L}) \to (\Dperf_W(X),\QF^{\geq m}_L) \to (\Dperf_{W\smallsetminus W'}(X\smallsetminus W'),\QF^{\geq m}_L) \,.
\]
Now since $\Catp$ is compactly generated filtered colimits there compute with finite limits, and hence the collection of Poincaré-Verdier sequences is closed under filtered colimits. Taking the colimit of the above Poincaré-Verdier sequences for $(W',W) \in \cB_Z(c)$ and using that both projections $\A_Z(c+1) \leftarrow \cB_Z(c) \rightarrow \A_Z(c)$ are cofinal (their corresponding comma categories are non-empty and hence automatically filtered) we thus obtain a Poincaré-Verdier sequence
\[
\colim_{W'\in \A_Z(c+1)}(\Dperf_{W'}(X),\QF^{\geq m}_{L}) \to \colim_{W\in \A_Z(c)}(\Dperf_W(X),\QF^{\geq m}_L) \to \colim_{(W',W)\in \cB_Z(c)}(\Dperf_{W\smallsetminus W'}(X\smallsetminus W'),\QF^{\geq m}_L) \,.
\]
Now for $W \in \A_Z(c)$ we have that $(\Dperf_{W}(X),\QF^{\geq m}_{L})$ is by construction a full Poincaré subcategory of $\Dperf(X)$, and so their colimit in $\Catp$ coincides with their union, considered as a full subcategory of $\Dperf(X)$. This union is nothing but $\Dperf(X)^{\geq c}$: indeed, by definition a perfect complex $P \in \Dperf_Z(X)$ belongs to $\Dperf_Z(X)^{\geq 0}$ if and only if its support, which is a closed subset of $Z$, contains only points of codimension $\geq c$, that is, if and only if $\mathrm{supp}(P) \in \A_Z(c)$. We hence conclude that the Poincaré-Verdier inclusion $(\Dperf_Z(X)^{\geq c+1},\QF^{\geq m}_L) \subseteq (\Dperf_Z(X)^{\geq c},\QF^{\geq m}_L)$ fits in a Poincaré-Verdier sequence of the form
\[
(\Dperf_{Z}(X)^{\geq c+1},\QF^{\geq m}_{L}) \to (\Dperf_Z(X)^{\geq c},\QF^{\geq m}_L) \to \colim_{(W',W)\in \cB_Z(c)}(\Dperf_{W\smallsetminus W'}(X\smallsetminus W'),\QF^{\geq m}_L) \,.
\]
To conclude the proof it will now suffice to show that the last term in the above sequence is naturally equivalent to $\bigoplus_{x\in Z^{(c)}} (\Dperf_x(\cO_{X,x}),\QF^{\geq m}_L)$. Indeed, the projection $\cB_Z(c) \to \A_Z(c)$ is a cocartesian fibration and hence we may calculate the colimit in question as
\[
\colim_{(W',W)\in \cB_Z(c)}(\Dperf_{W\smallsetminus W'}(X\smallsetminus W'),\QF^{\geq m}_L) = \colim_{W \in \A_Z(c)} \colim_{\stackrel{W' \in \A_Z(c+1)}{W'\subseteq W}}(\Dperf_{W\smallsetminus W'}(X\smallsetminus W'),\QF^{\geq m}_L).
\]
Let us fix a $W\in \A_Z(c)$ and assume first that it is irreducible of codimension $c$ in $Z$. In particular, it contains a single point $w$ of $Z^{(c)}$.
But then $\A_W(c+1)$ is exactly the poset of proper closed subsets of $W$. That is, it is the opposite of the poset of open neighborhoods of the generic point $w$ of $W$ containing $X\smallsetminus W$. Therefore, by Lemma~\ref{lemma:stalk-of-D_Z} we have that
\[
\colim_{\stackrel[W'\subseteq W]{W' \in \A_Z(c+1)}{}}(\Dperf_{W\smallsetminus W'}(X\smallsetminus W'),\QF^{\geq m}_L) \simeq (\Dperf_w(\cO_{X,w}),\QF^{\geq m}_L) 
\]
Now for a general $W\in\A_W(c)$, as $X$ is noetherian, we can write $W=W_1\cup\cdots \cup W_n\cup \tilde W$, where the $W_i$ are the irreducible components of codimension $c$ and $\tilde W\in\A_W(c+1)$ is the union of all the irreducible components of codimension $\ge c+1$. Set 
\[
W_0=\tilde W \cup \bigcup_{i\neq j} W_i\cap W_j\in \A_W(c+1)\,,
\]
so that for every $W'$ containing $W_0$ we have $W\smallsetminus W'= (W_1\smallsetminus W')\amalg\cdots \amalg (W_n\smallsetminus W')$, where the unions are disjoint. Therefore by Lemma~\ref{lemma:complexes-with-disjoint-support} we have
\[
\Dperf_{W\smallsetminus W'}(X\smallsetminus W')\simeq\bigoplus_{i=1}^n\Dperf_{W_i\smallsetminus W'}(X\smallsetminus W')\,.
\]
Taking the colimit in $W'$ we obtain, by the previous case
\[
\colim_{\stackrel[W'\subseteq W]{W' \in \A_Z(c+1)}{}} (\Dperf_{W\smallsetminus W'}(X\smallsetminus W'), \QF^{\geq m}_L) \simeq \bigoplus_{x\in Z^{(c)}\cap W} (\Dperf_x(\cO_{X,x}),\QF^{\geq m}_L)\,.
\]
Finally taking the colimit for $W\in \A_Z(c)$ we get the desired result. 
\end{proof}

\subsection{Genuine symmetric GW-theory of schemes}
\label{subsection:genuine-sym}%

For a scheme $X$ equipped with a line bundle $L$, its classical $\GW$-groups were defined and studied in Hornbostel's thesis~\cite{hornbostel-thesis} under the assumption that 2 is invertible, and later in the works of Schlichting~\cite{schlichting-exact, schlichting-mv} without this assumption, by considering the collection of vector bundles on $X$ as an exact category with duality $\Dual_L = \underline{\Hom}_X(-,L)$, to which one can associate a symmetric $\GW$-space using the hermitian $\Q$-construction. More general flavours of forms were later considered in~\cite{schlichting-higher} (see in particular~\cite[Theorem 1.1]{schlichting-higher} for a proof that this approach indeed generalizes classical Grothendieck-Witt groups of rings).

Using vector bundles gives a well-behaved theory when the scheme $X$ is divisorial, a condition which insures that $X$ has a sufficiently supply of such bundles. Our goal in this subsection is to show that for a divisorial scheme $X$, its classical symmetric $\GW$-space in the above sense coincides with the $\GW$-space of the Poincaré $\infty$-category $(\Dperf(X),\QF^{\gs}_L)$, where $\QF^{\gs}_L = \QF^{\geq 0}_L$ is the genuine symmetric Poincaré structure, see~\ref{definition:genuine}. The idea is to use the Zariski descent result of Corollary~\ref{corollary:zariski-descent-pn} in order to reduce the comparison to the case where $X$ is affine, which in turn follows from the main result of~\cite{Hebestreit-Steimle}. 

Despite its relation with classical symmetric $\GW$-theory, our main focus in the present paper is not genuine symmetric $\GW$-theory, but rather the homotopy symmetric one, that is, the one associated to the Poincaré structure $\QF^{\sym}_L$. These are not completely unrelated to each other: for regular Noetherian $X$ of finite Krull dimension $d$, the genuine symmetric and homotopy symmetric $\GW$-groups coincide in degrees $\geq d-1$, see Proposition~\ref{proposition:gen-sym-compare} below.

Let us now fix a qcqs scheme $X$ and a line bundle $L$ on $X$.
Recall that by the Bott-Genauer sequence (see~\cite[\S 4.3]{9-authors-II}), the genuine symmetric $\GW$-space $\GWspace^{\gs}(X,L) := \GWspace(\Dperf(X),\QF^{\gs}_L)$ fits into a fibre sequence of $\Einf$-groups
\[
\GWspace^{\gs}(X,L) \to |\core \Q_\bullet(\Dperf(X))| \to |\Poinc\Q_\bullet(\Dperf(X),\QF^{\gs}_L)|.
\]
To compare this with the classical symmetric $\GW$-space of $X$ defined using vector bundles note that by definition, the cofibre of the natural map $\QF^{\gs}_L \Rightarrow \QF^{\sym}_L$ is an exact functor on $\Dperf(X)$ which is represented by the $(-1)$-truncated quasi-coherent complex $\tau_{\leq -1}\LT_L \in \Derqc(X) = \Ind(\Dperf(X))$. For any vector bundle $V$ on $X$ the cofibre of $\QF^{\gs}_L(V) \to \QF^{\sym}(V)$ is hence $(-1)$-truncated, and so the map
\[
\Om^{\infty}\QF^{\gs}_L(V) \to \Om^{\infty}\QF^{\sym}_L(V) = \Map(V \otimes V,L)^{\hC}
\]
is an equivalence of spaces. In particular, $\Om^{\infty}\QF^{\gs}_L(V)$ is a discrete space whose $\pi_0$ is the abelian group $\Hom(V \otimes V,L)^{\Ct} = \Hom_X(V,\Dual_L(V))^{\Ct}$.
We may consequently consider the commutative square
\[
\begin{tikzcd}
{|\core\Q_\bullet(\Vect(X))|} \ar[r]\ar[d] & {|\Poinc\Q_\bullet(\Vect(X),\Dual_L)|} \ar[d] \\
{|\core\Q_\bullet(\Dperf(X))|} \ar[r] & {|\Poinc\Q_\bullet(\Dperf(X),\QF^{\gs}_L)|} \ ,
\end{tikzcd}
\]
where on the top left corner we have the $\Q$-construction of the exact category $\Vect(X)$ of vector bundles on $X$, and the top right corner the hermitian $\Q$-construction of the exact category $\Vect$ equipped with the duality $\Dual_L = \underline{\Hom}_X(-,L)$, see~\cite{schlichting-exact}. Taking horizontal fibres yields a map of $\Einf$-groups
\[
\vphi_{X,L}\colon \GWspace^{\cl}(X,L)\to \GWspace^{\gs}(X,L) ,
\]
from the classical symmetric $\GW$-space $\GWspace^{\cl, \sym}(X,L) := \GWspace^{\sym}(\Vect(X),\Dual_L)$ of $X$ to the associated genuine symmetric $\GW$-space. 

\begin{proposition}
\label{proposition:genuine-is-classical}%
If $X$ is divisorial then the map $\vphi_{X,L}$ is an equivalence.
\end{proposition}
\begin{proof}
Let us consider the associations $U \mapsto \GWspace^{\cl}(U,L|_U)$ and $U \mapsto \GWspace^{\gs}(X,L|_{U})$ as presheaves of spaces on the site of qcqs open subschemes of $X$. By~\cite[Theorem 16]{schlichting-mv} the former is a sheaf for the Zariski topology, and by Corollary~\ref{corollary:zariski-descent} the latter is one as well. The collection of maps $\vphi_{U,L|_{U}}$ determines a map between these two sheaves, and the claim we wish to prove is that this map is an equivalence on global sections. But by the main result of~\cite{Hebestreit-Steimle} the map $\vphi_{U,L|_{U}}$ is an equivalence whenever $U$ is affine. It is hence an equivalence on stalks, and so an equivalence of sheaves, and in particular an equivalence on global sections.
\end{proof}

Proposition~\ref{proposition:genuine-is-classical} tells us that, when $X$ is a divisorial qcqs scheme, its classical symmetric Grothendieck-Witt groups $\GW^{\cl,\sym}_n(X,L)$ with coefficient in a line bundle $L$ are recovered in the present setting by the genuine symmetric Grothendieck-Witt groups $\GW^{\gs}_n(X,L)$ for $n \geq 0$. If $X$ is not divisorial we don't expect this to be always true. At the same time, for non-divisorial schemes we don't expect the $\GW$-theory of vector bundles to be a well-behaved invariant, and in particular do not expect it to satisfy Zariski descent. We may hence consider genuine symmetric $\GW$-theory as the natural extension of classical symmetric $\GW$-theory from rings to schemes, in the sense that it agrees with the latter on affine schemes, while also satisfying Zariski, and even Nisnevich (Corollary~\ref{corollary:nisnevich-descent}), descent. In fact, by~\cite[Proposition 3.7.4.5]{SAG} it is the unique such extension.

In the present paper we are mostly interested in homotopy symmetric $\GW$-theory, which, as we prove in \S\ref{subsection:Aone-invariance} has the additional advantage of being $\Aone$-invariant on regular Noetherian schemes of finite Krull dimension, and that constitutes our principal input in order to construct the hermitian $\K$-theory spectrum in \S\ref{section:KQ}. In this context let us mention that the genuine symmetric $\GW$-spectrum does not satisfy such $\Aone$-invariance. It is hence natural to try and compare genuine and homotopy symmetric $\GW$-theories. The proposition below uses the Zariski descent and the results of~\cite{9-authors-III} in order to obtain such a comparison in sufficiently high degrees for regular Noetherian schemes of finite Krull dimension.

In what follows, recall that a space $\Z$ is said to be $(-1)$-truncated if it is either empty or contractable, and $n$-truncated for $n \geq 0$ if the homotopy groups of each of its connected components have vanishing homotopy groups vanish in degrees $> n$. For $n < -1$ we take being $n$-truncated as meaning contractible.

\begin{proposition}
\label{proposition:gen-sym-compare}%
Let $X$ be a regular Noetherian scheme of finite Krull dimension $d$ and $L$ a line bundle with $\Ct$-action on $X$. Then the homotopy fibres of the map
\[
\GWspace^{\gs}(X,L) \to \GWspace^{\sym}(X,L)
\]
are all $(d-3)$-truncated. In particular, if $d=0,1$ then this map is an equivalence.
\end{proposition}
\begin{proof}
As in the proof of Proposition~\ref{proposition:genuine-is-classical}, let us consider the associations $U \mapsto \GWspace^{\gs}(U,L|_{U})$ and $U \mapsto \GWspace^{\sym}(U,L|_{U})$ as presheaves of spaces on the site of quasi-compact open subsets on $X$, which by Corollary~\ref{corollary:zariski-descent} are actually Zariski sheaves.
A point $s \in \GWspace^{\sym}(X,L)$ is hence by definition a global section of the sheaf $\GWspace^{\sym}(-,L|_{-})$, which we can equally consider as a map $s\colon \ast \to \GWspace^{\sym}(-,L|_{-})$ from the terminal sheaf $\ast$. 
Let
\[
\cK_s := \GWspace^{\gs}(-,L|_{-})\times_{\GWspace^{\sym}(-,L|_{-})} \ast
\]
be the associated fibre product sheaf, which we can equally view as the fibre of the map $\GWspace^{\gs}(-,L|_{-})\to \GWspace^{\sym}(-,L|_{-})$ over $s$.
The desired claim can be reformulated as saying that for every $s$ the spectrum of global sections of $\cK_s$ is $(d-3)$-truncated. Since the collection of $(d-3)$-truncated spectra is closed under limits, it will suffice to show that the stalks of $\cK_s$ are $(d-3)$-truncated. This now follows from the fact that $\cK_s(U)$ is $(d-3)$-truncated for any affine $U \subseteq X$ by~\cite{9-authors-III}*{Corollary~1.3.10}.
\end{proof}

\section{Dévissage}
\label{section:devissage}%

Our goal in this section is to establish a dévissage result for symmetric $\GW$- and $\L$-theory of schemes, see Theorem~\ref{theorem:global-devissage} below. Technical details aside, this theorem says that if $X$ is a regular Noetherian scheme of finite Krull dimension, $L \in \Picspace(X)^{\BC}$, and $i\colon Z \subseteq X$ is a regular embedding, then the $Z$-supported symmetric $\GW$-theory of $X$ with coefficients in $L$ identifies with symmetric $\GW$-theory of $Z$ with coefficients in a suitable invertible complex $i^!L$ (and similarly for symmetric $\L$-theory). To properly formulate this result we first investigate in \S\ref{subsection:pushforward} the conditions under which the push-forward functor refines to a Poincaré functor. It turns out that this holds in particular for regular embeddings, which allows us to formulate Theorem~\ref{theorem:global-devissage} below. After reducing the claim to the level of $\L$-theory, our strategy consists of first addressing the case where $X$ is the spectrum of a regular local ring in \S\ref{subsection:local-devissage}, and then extending to the general case in \S\ref{subsection:global-devissage} using the coniveau filtration of \S\ref{subsection:coniveau}.

\subsection{Push-forward as a Poincaré functor}
\label{subsection:pushforward}%

Recall that a map $f\colon X \to Y$ of qcqs schemes is said to be \defi{quasi-perfect} if $f_*\colon \Derqc(X) \to \Derqc(Y)$ preserves perfect complexes (equivalently, if its right adjoint $f^!$ preserves colimits, see Definition~\ref{definition:quasi-perfect} in Appendix~\ref{section:appendix} and the surrounding discussion).
In particular, when $f$ is quasi-perfect we may view $f_*$ as an exact functor $\Dperf(X) \to \Dperf(Y)$.

\begin{definition}
\label{definition:poincare}%
Let $f\colon X \to Y$ be a map of qcqs schemes. It is called \defi{quasi-Gorenstein} if $f^!\cO_Y$ is an invertible perfect complex on $X$. We say that $f$ is
\defi{qGqp} if it is quasi-Gorenstein and quasi-perfect. Equivalently (by Lemma~\ref{lemma:eta-equivalence}), $f$ is qGqp if it is quasi-perfect and $f^!$ preserves tensor invertible objects. 
\end{definition}

\begin{remark}
\label{remark:invertible-closed-compositoin}%
By its second characterization we see that the property of being a qGqp map is closed under composition.
\end{remark}

\begin{construction}
\label{construction:pushforward-hermitian}%
Let $X,Y$ be qcqs schemes, let $N,M \in \Picspace(X)^{\BC}$ be invertible complexes with $\Ct$-action on $X$ and let $N \in \Picspace(Y)^{\BC}$ be an invertible complex with $\Ct$-action on $Y$. Let $f\colon X \to Y$ be a qGqp map and $\tau\colon N \otimes M^{\otimes 2} \tosimeq f^!L$ a natural equivalence.
We construct a natural transformation $\eta_{f,\tau}\colon \QF^{\sym}_{N} \Rightarrow \QF^{\sym}_L(f_*(-\otimes M))$ as the composite 
\begin{align}
\nonumber\eta_{f,\tau}\colon \QF^{\sym}_{N}(P)&=\map_X(P\otimes P,N)^\hC \\
\label{align:arrow-pf-1}%
 & \tosimeq \map_X(P \otimes P \otimes M^{\otimes 2}, N \otimes M^{\otimes 2})^\hC\\
\label{align:arrow-pf-2}%
 & \tosimeq \map_X((P \otimes M) \otimes (P \otimes M),f^!L)^\hC\\
\nonumber & = \map_Y(f_*(P \otimes M\otimes P \otimes M), L)^\hC\\
\label{align:arrow-pf-3}%
 & \to \map_Y(f_*(P \otimes M)\otimes f_*(P \otimes M),L)^{\hC} \\
\nonumber & =\QF^{\sym}_L(f_*(P \otimes M)) \ .
\end{align}
where~\eqref{align:arrow-pf-1} is an equivalence since $M$ is invertible, the arrow~\eqref{align:arrow-pf-2} is induced by post-composing with $\tau$ (and rearranging the terms in the domain) and is hence an equivalence, and~\eqref{align:arrow-pf-3} is induced from the lax symmetric monoidal structure of $f_*$ (which in turn is induced by the symmetric monoidal structure on its left adjoint $f^*$, see~\cite[Corollary 7.3.2.7]{HA}).
We then consider the pair $(f_*(-\otimes M),\eta_{f,\tau})$ as a hermitian functor
\[
(f_*(-\otimes M),\eta_{f,\tau})\colon (\Dperf(X),\QF^{\sym}_{N})\to (\Dperf(Y),\QF^{\sym}_L)\ .
\]
When $f$ and $\tau$ are implied, we denote $\eta_{f,\tau}$ simply by $\eta$.
\end{construction}

Our next goal is to show that the hermitian functor of Construction~\ref{construction:pushforward-hermitian} is always Poincaré. For this, we first show that the construction $(f_*(-\otimes M),\eta_{f,\tau})$ is compatible with composition in the following sense. Suppose given a composable pair
\[
X \xrightarrow{f} X \xrightarrow{f'} Z
\]
and invertible perfect complexes with $\Ct$-action $N,M \in \Picspace(X)^{\BC}$, $L,M' \in \Picspace(X')^{\BC}$ and $P \in \Picspace(Z)$ together with natural equivalences 
\[
\tau\colon N \otimes M^{\otimes 2} \tosimeq f^!L \quad\text{and}\quad \tau'\colon L \otimes (M')^{\otimes 2} \tosimeq (f')^!P .
\]
Applying Construction~\ref{construction:pushforward-hermitian} to $(f,\tau)$ and $(f',\tau')$, we then obtain a composable pair of hermitian functors
\[
(\Dperf(X),\QF^{\sym}_{N})\xrightarrow{(f_*(- \otimes M),\eta_{f,\tau})} (\Dperf(Y),\QF^{\sym}_L)
\xrightarrow{(f'_*(-\otimes M'),\eta_{f',\tau'})} (\Dperf(Z),\QF^{\sym}_P)
\ .
\]
On the other hand, we can also apply Construction~\ref{construction:pushforward-hermitian} to the map $f'' := f'\circ f$ and the data of $N, M'' := M\otimes f^*M',P$ and the composite
\[
\tau''\colon N \otimes M^{\otimes 2} \otimes f^*(M')^{\otimes 2} \xrightarrow{\tau \otimes f^*(M')^{\otimes 2}} f^!L \otimes f^*(M')^{\otimes 2} \tosimeq f^!(L \otimes (M')^{\otimes 2}) \xrightarrow{f^!\tau'} f^!(f')^!P ,
\]
yielding a hermitian functor
\[
(\Dperf(X),\QF^{\sym}_{N})\xrightarrow{(f''_*(-\otimes M''),\eta_{f'',\tau''})} (\Dperf(Z),\QF^{\sym}_P)
\ .
\]

\begin{lemma}
\label{lemma:composite-pf}%
The three hermitian functors just constructed fit into a commutative triangle
\[
\begin{tikzcd}
& (\Dperf(Y),\QF^{\sym}_L) \ar[dr, "{(f'_*(-\otimes M'),\eta_{f',\tau'})}"] & \\
(\Dperf(X),\QF^{\sym}_{N}) \ar[ur, "{(f_*(- \otimes M),\eta_{f,\tau})}"]\ar[rr, "{(f''_*(-\otimes M''),\eta_{f'',\tau''})}"'] && (\Dperf(Z),\QF^{\sym}_P) \ .
\end{tikzcd}
\]
\end{lemma}
\begin{proof}
We first note that the claim directly follows from the relevant definitions if either $M'= \cO_{Y}$ or $f=\id$. It hence suffices to prove the claim in the case where $M=\cO_X$ and $f'=\id$. We may then also assume without loss of generality that $N=f^!L$, $P = L \otimes (M')^{\otimes 2}$ and both $\tau$ and $\tau'$ are the respective identities. Now the projection formula gives an equivalence
\[
f_*(-) \otimes M' \tosimeq f_*(- \otimes f^*M')
\]
which traces a commuting homotopy in the underlying triangle of exact functors. To lift this to a commuting homotopy on the level of hermitian functors we then need to construct, for $L,M \in \Picspace(Y)^{\BC}$ and $f\colon X \to Y$ a qGqp map, a commuting homotopy in the diagram
\[
\begin{tikzcd}
\QF^{\sym}_{f^!L}(-) \ar[r]\ar[d] &\QF^{\sym}_{f^!L \otimes (f^*M')^{\otimes 2}}(-\otimes f^*M') \ar[r] & \QF^{\sym}_{f^!(L \otimes (M')^{\otimes 2})}(-\otimes f^*M') \ar[d]\\
\QF^{\sym}_L(f_*(-))  \ar[r] & \QF^{\sym}_{L \otimes (M')^{\otimes 2}}(f_*(-) \otimes M') & \ar[l] \QF^{\sym}_{L \otimes (M')^{\otimes 2}}(f_*(-\otimes f^*M'))
\end{tikzcd}
\]
where the top right horizontal map is induced by the equivalence $f^!L \otimes (f^*M')^{\otimes 2} \tosimeq f^!(L \otimes (M')^{\otimes 2})$ of Lemma~\ref{lemma:eta-equivalence}, the bottom right horizontal map is induced by the equivalence $f_*(-) \otimes M' \tosimeq f_*(-\otimes f^*M')$ of the projection formula, and the rest of the maps are instances of the natural transformation described in Construction~\ref{construction:pushforward-hermitian}. 
Unwinding the definitions further, it will suffice to construct, naturally in $F \in \Dperf(X)$, a commuting homotopy in the external octagon of the diagram of spectra with $\Ct$-action
\begin{equation}
\label{equation:big-diagram}%
\begin{tikzcd}[column sep=-5pt]
& \map_X(F\otimes F \otimes f^*(M')^{\otimes 2},f^!L \otimes f^*(M')^{\otimes 2}) \ar[dr] &\\
\map_X(F\otimes F,f^!L)\ar[ur]\ar[d] &  & \map_X(F \otimes F \otimes f^*(M')^{\otimes 2}, f^!(L \otimes (M')^{\otimes 2}))\ar[d] \\
\map_Y(f_*(F \otimes F),L)\ar[dd]\ar[dr] && \map_X(f_*(F \otimes F \otimes f^*(M')^{\otimes 2}),L \otimes (M')^{\otimes 2}) \ar[dl]\ar[dd]\\
& \map_Y(f_*(F \otimes F) \otimes (M')^{\otimes 2},L \otimes (M')^{\otimes 2})\ar[dd] & \\
\map_Y(f_*F \otimes f_*F,L)\ar[dr] && \map_Y(f_*(F \otimes f^*M') \otimes f_*(F\otimes f^*M'),L\otimes (M')^{\otimes 2})\ar[dl] \\
& \map_Y(f_*F \otimes f_*F\otimes (M')^{\otimes 2},L\otimes (M')^{\otimes 2}) & 
\end{tikzcd}
\end{equation}
To construct this homotopy, it will suffice to fill in the two bottom squares and the top hexagon. 
Now the bottom left square is easily filled: its horizontal arrows are induced on mapping spectra by the functor $(-) \otimes (M')^{\otimes 2}$, its left vertical map is induced by pre-composition with $\sig\colon f_*F \otimes f_*F \to f_*(F \otimes F)$ and its right vertical map is induced by pre-composition with the map $\sig \otimes \id\colon f_*F \otimes f_*F \otimes (M')^{\otimes} \to f_*(F \otimes F) \otimes (M')^{\otimes 2}$. The right bottom square, in turn, can be filled by observing that the counit map $f^*f_* \Rightarrow \id$ is a lax symmetric monoidal natural transformation, and hence the same holds for the projection formula transformation $f_*(-) \otimes (-) \Rightarrow f_*(-\otimes f^*(-))$ (considered as a transformation between functors in two arguments). 

We now proceed to construct a commuting homotopy in the top hexagon of~\eqref{equation:big-diagram}. Consider, for $G \in \Dperf(X)$ and $L,Q \in \Dperf(Y)$, the natural commutative diagram
\[
\begin{tikzcd}[column sep=10pt]
 & \map_Y(f_*G,f_*f^!L)\ar[dr] & \\
\map_X(G,f^!L) \ar[ur]\ar[d] && \map_Y(f_*G \otimes Q,f_*(f^!L) \otimes Q) \ar[d, "\simeq"] \\
\map_X(G \otimes f^*Q,f^!L \otimes f^*Q) \ar[dd, "\simeq"]\ar[dr] && \map_Y(f_*G \otimes Q,f_*(f^!L \otimes f^*Q)) \ar[dd, "\simeq"]\\
& \map_Y(f_*(G \otimes f^*Q),f_*(f^!L \otimes f^*Q)) \ar[ur,"\simeq"]\ar[d, "\simeq"] & \\
\map_X(G \otimes f^*Q, f^!(L \otimes Q))  \ar[dr]\ar[r] & 
\map_X(f_*(G \otimes f^*Q),f_*f^!(L \otimes Q)) \ar[r, "\simeq"]\ar[d]&
\map_X(f_*G \otimes Q,f_*f^!(L \otimes Q)) \ar[d] \\
&\map_X(f_*(G \otimes f^*Q),L \otimes Q) \ar[r,"\simeq"]&
\map_X(f_*G \otimes Q,L \otimes Q) \ .
\end{tikzcd}
\]
Here, the top commuting hexagon is induced on mapping spectra by the natural transformation
\[
f_*(-) \otimes Q \Rightarrow f_*(-\otimes f^*Q)
\]
underlying the projection formula, the three vertical maps in the middle row are induced by post-composition with $f_*(f^!L \otimes f^*Q) \tosimeq f_*f^!(L \otimes Q)$, the two vertical maps in the bottom row are induced by post-composition with $f_*f^!(L \otimes Q) \to L \otimes Q$, and the top right vertical map is induced by post-composition with $f_*(f^!L) \otimes Q \to f_*(f^!L \otimes f^*Q)$. In addition, the three horizontal maps on the right side which are marked as equivalences are induced by pre-composition with $f_*G \otimes Q \tosimeq f_*(G \otimes f^*Q)$, and all the other maps not marked as equivalences are induced on mapping spaces either by $f_*$, by $(-) \otimes Q$, or by $(-) \otimes f^!Q$. In particular, the right vertical total composite is induced by post-composition with the total composite
\[
f_*(f^!L) \otimes Q \to f_*(f^!L \otimes f^*Q) \to f_*f^!(L \otimes Q) \to L \otimes Q .
\]
We claim that this composite is homotopic to the map induced by the counit $f_*f^!L \to L$ after tensoring with $Q$. Indeed, the functor $f_*$ is $\Dperf(Y)$-linear by the projection formula, and this structure induces a lax $\Dperf(Y)$-linear structure on the right adjoint $f^!$, which is an strong $\Dperf(Y)$-structure in the case at hand by Lemma~\ref{lemma:eta-equivalence}. In this situation the counit of $f_* \dashv f^!$ is automatically a $\Dperf(Y)$-linear transformation, and so the claim follows. We hence obtain a commutative diagram of the following form, natural in $G,Q$ and $L$:
\[
\begin{tikzcd}
& \map_X(G \otimes f^*Q,f^!L \otimes f^*Q)\ar[dr] & \\
\map_X(G,f^!L)\ar[ur]\ar[d] && \map_X(G \otimes f^*Q, f^!(L \otimes Q))\ar[d] \\
\map_Y(f_*G,L)\ar[dr] && \map_X(f_*(G \otimes f^*Q),L \otimes Q)\ar[dl] \\
& \map_Y(f_*G \otimes Q,L \otimes Q) &
\end{tikzcd}
\]
Substituting $G = F \otimes F$ and $Q = M' \otimes M'$ and reflecting we hence obtain the top hexagon in~\eqref{equation:big-diagram}, as desired.
\end{proof}

\begin{lemma}
\label{lemma:devissage-functor}%
Let $f\colon X\to Y$ be a qGqp map (see Definition~\ref{definition:poincare}). 
Then the hermitian functor 
\[
(f_*(-\otimes M),\eta_{f,\tau})\colon (\Dperf(X),\QF^{\sym}_{N})\to (\Dperf(Y),\QF^{\sym}_L)\ .
\]
of Construction~\ref{construction:pushforward-hermitian} is Poincaré for any choice of $M,N,L$ and $\tau\colon N \otimes M^{\otimes 2} \tosimeq f^!L$.
\end{lemma}
\begin{proof}
By Lemma~\ref{lemma:composite-pf} it will suffice to prove separately the case where $M=\cO_X$ and the case where $f=\id$. Now when $f=\id$, or, more generally, when $f$ is an equivalence, we have that the underlying exact functor $f_*(- \otimes M)$ is an equivalence and also that the only possibly non-invertible natural transformation~\eqref{align:arrow-pf-3} in the definition of $\eta_{f,\tau}$ is an equivalence, so that $\eta_{f,\tau}$ is a natural equivalence. The hermitian functor $(f_*(- \otimes M),\eta_{f,\tau})$ is then an equivalence, and in particular a Poincaré functor. 

We now treat the case where $M=\id$. Here, we may as well suppose that $N=f^!L$ and $\tau$ is the identity. We need to show that the induced map $f_*\Dual_{f^!L}(M) \to \Dual_{L}(f_*M)$ is an equivalence for every $M \in \Dperf(X)$. Mapping a test object $N \in \Dperf(Y)$ into this map and using the adjunction $f^* \dashv f_*$, it will suffice to verify that the induced map
\[
\Bil_{f^!L}(f^*N,M) = \map(N,f_*\Dual_{f^!L}(M)) \to \map(N,\Dual_{L}(f_*M)) = \Bil_L(N,f_*M)
\]
is an equivalence. Unwinding the definitions (see~\cite[Proof of Lemma~1.2.4]{9-authors-I}), this last map is 
is simply the composite 
\[
\Bil_{f^!L}(f^*N,M) \xrightarrow{(\eta_{f,\id})_*} \Bil_{L}(f_*f^*N,f_*M) \to \Bil_L(N,f_*M) .
\]
Finally, by the very definition of $\eta_{f,\id}$, this composite is given by the composite
\[
\map(f^*N \otimes M, f^!L) \simeq \map(f_*(f^*N \otimes M),L) \to \map(f_*f^*N \otimes f_*M,L)  \to \map(N \otimes f_*M, L) .
\]
It will hence suffice to show that for every $M \in \Dperf(X)$ and $N \in \Dperf(Y)$, the composite
\[
N \otimes f_*M \to f_*f^*N \otimes f_*M \to f_*(f^*N \otimes M)
\]
is an equivalence. Indeed, this is exactly what the projection formula says.
\end{proof}

For later purposes, let us also record the following observation:
\begin{lemma}
\label{lemma:pushforward-projection}%
Let $f\colon X\to Y$ be a qGqp map and suppose in addition that $f_*\colon \Dperf(X) \to \Dperf(Y)$ is a split Verdier projection. Then for every choice of $M,N,L$ and $\tau\colon N \otimes M^{\otimes 2} \tosimeq f^!L$ as in Construction~\ref{construction:pushforward-hermitian}, the Poincaré functor 
\[
(f_*(-\otimes M),\eta_{f,\tau})\colon (\Dperf(X),\QF^{\sym}_{N})\to (\Dperf(Y),\QF^{\sym}_L)\ .
\]
established in Lemma~\ref{lemma:devissage-functor} is a split Poincaré-Verdier projection.
\end{lemma}
\begin{proof}
We note that a left adjoint to $f_*(- \otimes M)$ is given by $f^*(-) \otimes M^{-1}$, and the assumption that $f_*$ is a split Verdier projection implies that $f^*$, and hence also $f^*(-) \otimes M^{-1}$, is fully-faithful. By~\cite[Corollary 1.2.3]{9-authors-II} it will now suffice to check that the natural transformation 
\[
\eta_{f,\tau}\colon \QF^{\sym}_N(-) \Rightarrow \QF^{\sym}_L(f_*(-\otimes M))
\]
is an equivalence when evaluated on perfect complexes of the form $P = f^*(Q) \otimes M^{-1}$ for $Q \in \Dperf(Y)$. By definition, $\eta_{f,\tau}$ is given by a composite of natural transformations, all invertible except one, namely, the map~\eqref{align:arrow-pf-3} in Construction~\ref{construction:pushforward-hermitian}. Plugging in $P = f^*(Q) \otimes M^{-1}$ in~\eqref{align:arrow-pf-3} what we hence need to show is that the map
\[
f_*(f^*(Q) \otimes f^*(Q)) \to f_*(f^*(Q)) \otimes f_*(f^*(Q))
\]
determines by the lax symmetric monoidal structure of $f_*$, is an equivalence. For this, note that the adjunction $f^* \dashv f_*$ is a symmetric monoidal adjunction (that is, $f^*$ is symmetric monoidal and $f_*$ carries an induced lax symmetric monoidal structure), so that, in particular, the unit and counit are lax symmetric monoidal natural transformations. Since we assume in addition that $f_*$ is a split Verdier projection we have that the unit $\id \Rightarrow f_*\circ f^*$ is an equivalence. The composite $f_* \circ f^*$ is hence equivalent to the identity as a lax symmetric monoidal functor, and is in particular symmetric monoidal. It then follows that lax symmetric monoidal structure maps of $f_*$ are equivalences on objects in the image of $f^*$, as desired.
\end{proof}

\begin{proposition}
\label{proposition:proper-lci-poincare}%
Any proper local complete intersection map $f\colon X \to Y$ is qGqp.
\end{proposition}

\begin{corollary}
\label{corollary:lci-pushforward}%
For any proper local complete intersection map $f\colon X \to Y$ and any line bundle $L \in Y$ the hermitian functor
\[
(f_*,\eta)\colon (\Dperf(X),\QF^{\sym}_{f^!L})\to (\Dperf(Y),\QF^{\sym}_L)
\]
is Poincaré.
\end{corollary}

For the proof of Proposition~\ref{proposition:proper-lci-poincare} we first verify that
being a qGqp map is a local property on the codomain:

\begin{lemma}
\label{lemma:poincare-local}%
Let $f\colon X \to Y$ be a map of qcqs schemes, $Y = \cup_i V_i$ an open covering of $Y$. Then $f$ is qGqp if and only if its base change $f_i\colon U_i = X \times_Y V_i \to V_i$ to $V_i$ is qGqp for every $i$. In addition, when these equivalent conditions hold we have that $(f^!\cO_Y|)_{U_i} \simeq f_i^!(\cO_{V_i})$ for every $i$.
\end{lemma}
\begin{proof}
By Corollary~\ref{corollary:quasi-perfect-local} we have that $f$ is quasi-perfect if and only if each $f_i$ is quasi-perfect, and that when these equivalent conditions hold we have that $(f^!\cO_Y|)_{U_i} \simeq f_i^!(\cO_{V_i})$. It will hence suffice to show that $M \in \Derqc(X)$ is invertible it and only if $M|_{U_i}$ is invertible in $\Derqc(U_i)$ for every $i$. The only if direction is clear since the pullback functor $\Derqc(X) \to \Derqc(U_i)$ is monoidal. On the other hand, if each $M|_{U_i}$ is invertible then each $M|_{U_i}$ is in particular dualisable, hence compact, hence perfect. By Remark~\ref{remark:perfect-local} we have that $M$ is perfect, hence dualisable with some dual $\Dual M$. Since each $M|_{U_i}$ is invertible the coevaluation map $M \otimes \Dual M \to \cO_X$ maps to an equivalence in $\Derqc(U_i)$ for each $i$ and is hence an equivalence by Zariski descent (Proposition~\ref{proposition:descent}). We conclude that $M$ is invertible.
\end{proof}

\begin{proof}[Proof of Proposition~\ref{proposition:proper-lci-poincare}]
By Lemma~\ref{lemma:poincare-local}, it suffices to prove the claim when $f$ is a (proper) global complete intersection, that is, when $f$ is the composite of a regular closed embedding and a proper smooth map. By Remark~\ref{remark:invertible-closed-compositoin}, it suffices to show this when $f$ is either a regular closed embedding or a proper smooth map.

If $f$ is a regular embedding, then in particular it is affine, and by replacing $Y$ by sufficiently small affine open subschemes, we may assume that $f$ is of the form $\spec(A/I) \to \spec(A)$, where $I \subseteq A$ is an ideal generated by a regular sequence $a_1,\ldots,a_r \in A$. The functor $f_*\colon \Der(A/I) \to \Der(A)$ is then the forgetful (or inflation) functor obtained by precomposing an $A/I$-action with $A \to A/I$, and its right adjoint $f^!\colon \Der(A) \to \Der(A/I)$ sends an $A$-module $M$ to the $A/I$-module $\mathrm{RHom}_A(A/I,M)$. The regularity assumption then implies that the Koszul complex 
\[
K_\bullet = \bigotimes_{i=1}^{r} [A \xrightarrow{a_i} A]
\]
is a finite free resolution of the $A$-module $A/I$, where $A \xrightarrow{a_i} A$ is considered as a complex sitting in degrees $1,0$. In particular, $A/I$ is a perfect $A$-complex, and so $f$ is quasi-perfect. In addition, for every $A/I$-module $M$ we have that $f^!M = \mathrm{RHom}_A(A/I,M)$ can be modelled by the complex $\bigotimes_{i=1}^{r} [M \xrightarrow{a_i} M]$, where $[M \xrightarrow{a_i} M]$ is considered as a complex sitting in degree $0,-1$. Taking $M=A$ we obtain that $f^!A \simeq A/I[-r]$ is invertible.

We now consider the case where $f\colon X \to Y$ is smooth and proper. In particular, since $f$ is proper it is by definition separated, and so the diagonal map 
\[
\Del\colon X \to X \times_Y X
\]
is a closed embedding. In addition, since $f$ is smooth, we have by~\cite[Tag 067U]{stacks-project} that the closed embedding $\Del$ is a regular embedding. By the first part of the theorem, we deduce that $\Del$ is a qGqp morphism, so that $\Del^!(\cO_{X \times_Y X})$ is an invertible perfect complex on $X$. Now, consider the diagram
\[
\begin{tikzcd}
X \ar[dr, "\Del"] & & \\
& X \times_Y X \ar[r, "p_2"]\ar[d, "p_1"] & X \ar[d, "f"] \\
& X \ar[r, "f"] & Y
\end{tikzcd}
\]
Since $f$ is smooth, it is in particular flat. Combining Lemma~\ref{lemma:flat-quasi-perfect} and Lemma~\ref{lemma:eta-equivalence}, we then have that for $M \in \Derqc(Y)$ we have
\begin{align*}
f^*M &= \Del^!p_1^!f^*M \\
&= \Del^!p_2^*f^!M \\
&= \Del^!(\cO_{X \times_Y X}) \otimes \Del^*p_2^*f^!M\\
&= \Del^!(\cO_{X \times_Y X}) \otimes f^!M 
\end{align*}
and so the functors $f^!$ and $f^*$ differ by an invertible factor $\Del^!(\cO_{X \times_Y X})$. We conclude that $f^!$ preserves colimits and hence that $f$ is quasi-perfect. In addition, for $M=\cO_X$ we get that $\cO_X = f^*\cO_Y = \Del^!(\cO_{X \times_Y X}) \otimes f^!\cO_Y$ and so $f^!\cO_Y$ is an invertible object. We conclude that $f$ is invertible, as desired.
\end{proof}

\begin{remark}
\label{remark:proper-lci}%
In the proof of Proposition~\ref{proposition:proper-lci-poincare}, we show in particular that if $f\colon X \to Y$ is either a closed regular embedding or a proper smooth map, then $f$ is qGqp. Elaborating slightly on the arguments presented above, one can show more precisely that if $X$ is a closed regular embedding with rank $r$ normal bundle $\mathcal{N}$, then $f^!\cO_Y \simeq \det \mathcal{N}[-r]$ (see also \cite{stacks-project}*{\href{https://stacks.math.columbia.edu/tag/0BR0}{Tag 0BR0}}), and if $f$ is a proper smooth map with rank $r$ relative cotangent bundle $\Om_{X/Y}$, then $f^!\cO_Y \simeq (\Del^! \cO_{X \times_Y X})^{-1} \simeq \det \Om_{X/Y}[r]$, where $\Del\colon X \to X \times_Y X$ is the diagonal map.
\end{remark}

\subsection{Global dévissage}
\label{subsection:devissage}%

Let us now consider a closed embedding $i \colon Z \to X$ of finite dimensional regular Noetherian schemes with open complement $j\colon U \hrar X$. In particular, $i$ is automatically a regular embedding and $U$ is regular Noetherian. Given a line bundle with $\Ct$-action $L$ on $X$ we have by Proposition~\ref{proposition:proper-lci-poincare} an associated Poincaré functor
\[
(i_*,\eta)\colon \big(\Dperf(Z),\QF^{\sym}_{f^!L}\big)\to \big(\Dperf(X),\QF^{\sym}_L\big) .
\]
The functor $i_*\colon \Dperf(Z) \to \Dperf(X)$ takes values in the duality invariant full subcategory $\Dperf_Z(X) \subseteq \Dperf(X)$ spanned by the perfect complexes supported on $Z$. 
We consequently obtain that the Poincaré functor $(i_*,\eta)$ above induces a Poincaré functor
\[
\big(\Dperf(Z),\QF^{\sym}_{i^!L}\big)\to \big(\Dperf_Z(X),{\QF^{\sym}_L}|_Z\big) \ ,
\]
where we have denoted by ${\QF^{\sym}_L}|_Z$ the restriction of $\QF^{\sym}_L$ to $\Dperf_Z(X)$.

\begin{theorem}[Dévissage]
\label{theorem:global-devissage}%
Let $i\colon Z\hrar X$ be a closed embedding of finite dimensional regular Noetherian schemes (automatically a regular embedding, see \cite{stacks-project}*{\href{https://stacks.math.columbia.edu/tag/0E9J}{Tag 0E9J}}).
Let $L$ be a line bundle with $\Ct$-action on $X$ and $n \in \ZZ$ an integer. 
Then the induced maps
\[
\GW\big(\Dperf(Z),\QF^{\sym}_{i^!L[n]}\big)\to \GW\big(\Dperf_Z(X),\QF^{\sym}_{L[n]}|_Z\big) \quad\text{and}\quad \L\big(\Dperf(Z),\QF^{\sym}_{i^!L[n]}\big)\to \L\big(\Dperf_Z(X),\QF^{\sym}_{L[n]}|_Z\big) 
\]
are equivalences. 
\end{theorem}

Recall that on any regular Noetherian scheme the collection of perfect complexes coincides with the coherent ones, and so the stable subcategory of perfect complexes inherits from its embedding in quasi-coherent complexes a t-structure whose heart is the abelian category of coherent sheaves. In addition, any perfect complexes has only finitely many non-trivial homotopy sheaves, so that this t-structure is bounded. 
Now since $j\colon U \to X$ is flat the functor $j^*\colon \Dperf(X) \to \Dperf(U)$ is t-exact, and hence the bounded t-structure of $\Dperf(X)$ restricts to a bounded t-structure on $\Dperf_Z(X)$, with $\Dperf_{Z}(X)^{\heartsuit} \subseteq \Coh(X)$ the full subcategory spanned by the coherent sheaves supported on $Z$. 
Since closed embeddings are in particular affine, the functor $i_*\colon \Dperf(Z) \to \Dperf_Z(X)$ is t-exact. 
We consequently obtain that for every $m \in \{-\infty\} \cup \ZZ$ the Poincaré functor $(i_*,\eta)$ above induces a Poincaré functor
\[
(i_*,\eta^{\geq m})\colon \big(\Dperf(Z),\QF^{\geq m}_{i^!L}\big)\to \big(\Dperf_Z(X),{\QF^{\geq m}_L}|_Z\big) \ .
\]

\begin{theorem}[Genuine dévissage]
\label{theorem:genuine-devissage}%
Let $X$ be a regular Noetherian scheme of finite Krull dimension $d$, and $i\colon Z\hrar X$ a closed embedding with $Z$ regular. Fix a line bundle with $\Ct$-action $L \in \Dperf(X)$ and an $m \in \{-\infty,\infty\} \cup \ZZ$. 
Then the map
\[
\L_n\big(\Dperf(Z),\QF^{\geq m}_{i^!L}\big)\to \L_n\big(\Dperf_Z(X),\QF^{\geq m}_L|_Z\big)
\]
is an isomorphism for and $n \geq 2m-1+d$ and injective for $n = 2m-2+d$.
\end{theorem}

\begin{remark}
\label{remark:K-theory-devissage}%
In the situation of Theorem~\ref{theorem:global-devissage}, one may also consider the corresponding map
\[
\K\big(\Dperf(Z)\big)\to \K\big(\Dperf_Z(X)\big) 
\]
on the level of $\K$-theory. Applying the theorem of the heart for $\K$-theory (see~\cite{barwick-heart}) to the t-structures described above this map can be identified with the map
\[
\K(\Coh(Z))\to \K(\Coh_Z(X)) \ ,
\]
which is an equivalence by Quillen's dévissage theorem for abelian categories, see~\cite[Theorem 4]{quillen}.
\end{remark}

Theorem~\ref{theorem:genuine-devissage} implies the $\L$-theory part of Theorem~\ref{theorem:global-devissage} by taking $m=-\infty$ (where we note that shifting the line bundle translates to suspending $\L$-theory by bordism invariance). At the same time, by Remark~\ref{remark:K-theory-devissage} and the fundamental fibre sequence (see~\cite[Corollary~11.4.14]{9-authors-II}) we have that the $\L$-theory equivalence and $\GW$-theory equivalence in Theorem~\ref{theorem:global-devissage} imply each other. We hence conclude that Theorem~\ref{theorem:genuine-devissage} implies Theorem~\ref{theorem:global-devissage}. The following two subsections are dedicated to the proof of Theorem~\ref{theorem:genuine-devissage}.

\begin{remark}
\label{remark:karoubi-GW}%
In the situation of Theorem~\ref{theorem:global-devissage}, the natural transformations $\GW \Rightarrow \KGW$ and $\L \Rightarrow \KL$ are equivalences when evaluated on both  $(\Dperf(Z),\QF^{\sym}_{i^!L})$ and $(\Dperf_Z(X),\QF^\sym_L|_{Z})$.  
To see this, note that using the exact squares \eqref{equation:cofinality} (Karoubi cofinality), it suffices to show that the natural transformation $\K \Rightarrow \KK$ is an equivalence when evaluated on both $\Dperf(Z)$ and $\Dperf_{Z}(X)$. 
This, in turn, follows from the fact that both these (idempotent complete) stable $\infty$-categories carry bounded t-structures with Noetherian hearts, and such stable $\infty$-categories have vanishing negative $\K$-groups, see~\cite{theorem-of-the-heart}*{Theorem~1.2}. We conclude that the dévissage equivalence expressed in Theorem~\ref{theorem:global-devissage} equally holds if one replaces all appearances of $\GW$ by $\KGW$ and all appearances of $\L$ by $\KL$.
\end{remark}

\subsection{Local dévissage}
\label{subsection:local-devissage}%

Our goal in this section is to prove Theorems~\ref{theorem:genuine-devissage} 
in the case where $X = \spec(R)$ for $R$ a regular Noetherian local ring with maximal ideal $\fm$ residue field $k := R/\fm$, and $Z = \spec(k)$ is the associated unique closed point. We denote by $d$ the Krull dimension of $R$.

Recall that a finitely generated $R$-module $M$ is said to have \defi{finite length} if there is a finite filtration $0=M_0 \subseteq \cdots \subseteq M_n=M$ such that $M_i/M_{i-1}$ is annihilated by $\fm$. A perfect $R$-complex $M$ is then supported on $\spec(k)$ if and only if its homologies have finite length. In this case, we will also say that the $R$-complex $M$ itself has finite length. 
We then let $\Dperf_{\fm}(R) \subseteq \Dperf(R)$ denote the full stable subcategory spanned by the perfect $R$-complexes of finite length. 
As discussed in \S\ref{subsection:devissage}, 
the t-structure on $\Dperf(R)$ then restricts to a t-structure on $\Dperf_{\fm}(R)$ with heart $\Dperf_{\fm}(R)^{\heartsuit} = \Coh_{\fm}(R) \subseteq \Coh(R)$ the abelian category of finite length $R$-modules.

Let us fix an invertible perfect $R$-complex $L$ equipped with an $R$-linear involution. Let $\QF^{\sym}_L$ be the associated symmetric Poincaré structure and write $\QF^\sym_L|_{\fm} = {\QF^{\sym}_R}|_{\Dperf_{\fm}(R)}$ for its restriction to $\Dperf_{\fm}(R)$. As in \S\ref{subsection:devissage} we then consider the associated Poincaré functor
\begin{equation}
\label{equation:inf}%
(i_*,\eta)\colon (\Dperf(k),\QF^{\sym}_{i^!L}) \to (\Dperf_{\fm}(R),\QF^{\sym}_L|_{\fm}) .
\end{equation}

\begin{remark}
Since $R$ is local the invertible perfect complex $L$ is equivalent to a shift $R[r]$ for some $r$.
This isomorphism is however not unique, and it will be convenient to avoid choosing a particular one. In a similar manner, since $R$ is regular of Krull dimension $d$ we have that $i^!L$ is equivalent to $(L \otimes_R k)[-d] \simeq k[r-d]$ (see Lemma~\ref{lemma:purity-2} below), though not canonically. 
\end{remark}

\begin{theorem}[Dévissage for local rings]
\label{theorem:devissage}%
Let $r$ be the unique integer for which $\pi_rL \neq 0$. Then for $-\infty \leq m < \infty$ the map 
\[
\L_n(\Dperf(k),\QF^{\geq m}_{i^!L}) \tosimeq \L_n(\Dperf_{\fm}(R),\QF^{\geq m}_L|_{\fm})
\]
induced by~\eqref{equation:inf} is an isomorphism for $n \geq 2m-1+d-r$ and injective for $n = 2m-2+d-r$.
\end{theorem}

\begin{remark}
\label{remark:assume-r-0}%
In the situation of Theorem~\ref{theorem:devissage} we have equivalences
\[
\QF^{\geq m-1}_{i^!L[-1]} \simeq (\QF^{\geq m}_{i^!L})\qshift{-1} \quad\text{and}\quad \QF^{\geq m-1}_{L[-1]} \simeq  (\QF^{\geq m}_{L})\qshift{-1} .
\]
Replacing $L$ with $L[d-r]$, $m$ with $m+d-r$ and $n$ with $n+d-r$ we see that to prove Theorem~\ref{theorem:devissage} it will suffice to prove the case $r=d$.
\end{remark}

The remainder of this subsection is devoted to the proof of Theorem~\ref{theorem:devissage}.

\begin{lemma}
\label{lemma:purity}%
The $(d-r)$-shifted counit map $i_*i^!L[d-r] \to L[d-r]$ factors as
\[
i_*i^!L[d-r] \to L_{\fm} \to L[d-r]
\]
where $L_{\fm} \in \Mod(R) = \Der(R)^{\heartsuit}$ is an injective $R$-module which satisfies the following property:
for every $M \in \Dperf_{\fm}(R)$ the map
\[
\map_R(M,L_{\fm}) \to \map_R(M,L[d-r]) = \Sig^d\Dual_L(M)
\]
is an equivalence in $\Der(R)$. 
\end{lemma}
\begin{proof}
Let $k \subseteq I_k \in \Mod(R)$ be an injective envelop of the $R$-module $k$. By~\cite{gille}*{Lemma 3.3~(4)} there exists a map of $R$-module $\phi\colon I_k[-d] \to R$ such that for every $M \in \Dperf_{\fm}(R)$ the induced map
\[
\map(M,I_k[-d]) \to \map(M,R)
\]
is an equivalence. Since the underlying $R$-complex of $L$ is equivalent to $R[r]$ the same holds for the map $L_{\fm} \to L[d-r]$, where $L_{\fm} = L[-r] \otimes_R I_k$. Since $i_*i^!L[d-r]$ belongs to $\Dperf_{\fm}(R)$ we can factor in particular the counit map via $L_{\fm}$. We now finish by noting that Hom into the injective $R$-module is t-exact. 
\end{proof}

\begin{lemma}
\label{lemma:purity-2}%
The underlying $k$-complex of $i^*L$ is concentrated in degree $r-d$ and $\pi_{r-d}i^*L$ is a 1-dimensional $k$-vector space. In particular, there exists a (non-canonical) equivalence $i^*L \simeq k[r-d]$.
\end{lemma}
\begin{proof}
Since $i\colon \spec(k) \to \spec(R)$ is a closed embedding and $R$ is regular we have that $i^!L$ is a dualising complex on $\spec(k)$, and hence an invertible perfect complex. In particular, $i^!L$ can have at most one non-trivial homotopy group, and the non-trivial homotopy group must be an 1-dimensional $k$-vector space. It will hence suffice to show that 
\[
\pi_{n}i^!L = \pi_{n}\map_R(k,L) = 0
\]
for $n \neq r-d$. Indeed, by Lemma~\ref{lemma:purity} we have 
\[
\pi_{n}\map_R(k,L) = \pi_{n}\map_R(k,L_{\fm}[r-d]) = \Ext^{-n+r-d}_R(k,L_{\fm})
\]
where $L_{\fm} \in \Der(R)^{\heartsuit} = \Mod(R)$ is an injective $R$-module, so that the last $\Ext$-group vanishes whenever $n \neq r -d$.
\end{proof}

We summarize a key conclusion of the previous two lemmas:
\begin{corollary}
\label{corollary:purity}%
Suppose that $r=d$. Then, with respect to the standard t-structures, we have that
\begin{enumerate}
\item
The duality of $(\Dperf_{\fm}(R),\QF^{\geq m}_{L}|_{\fm})$ sends connective objects to coconnective objects and vice-versa. The induced duality on $\Dperf_{\fm}(R)^{\heartsuit} = \Coh_{\fm}(R)$ is given by $M \mapsto \Hom_R(M,L_{\fm})$.
\item
The duality of $(\Dperf(k),\QF^{\geq m}_{i^!L})$ sends connective objects to coconnective objects and vice-versa. The induced duality on $\Dperf(k)^{\heartsuit} = \Coh(k)$ is given by $M \mapsto \Hom_k(M,\pi_{-d}i^!L)$ (in particular, $\pi_{-d}i^!L$ is a $k$-vector space of dimension 1).
\end{enumerate}
\end{corollary}
\begin{proof}
The first statement follows from Lemma~\ref{lemma:purity} and the fact that $L_{\fm}$ is injective, and the second statement from Lemma~\ref{lemma:purity-2}.
\end{proof}

\begin{proposition}
\label{proposition:key}%
Let $\C$ be a stable $\infty$-category with equipped with a duality functor $\Dual\colon \C \to \C^{{\op}}$ and $\QF^{\sym}_{\Dual}\colon \C\op \to \Spa$ the associated symmetric Poincaré structure on $\C$. Suppose that $\C$ admits a $t$-structure such that $\Dual$ sends connective objects to coconnective objects and vice-versa and let $\Dual^{\heartsuit}\colon \C^{\heartsuit} \to \C^{\heartsuit}$ be the induced duality on $\C^{\heartsuit}$.
We then write $-\Dual^{\heartsuit}$ for the sign twist of $\Dual$. For $-\infty \leq m < \infty$ we may consider the truncated Poincaré structure $\QF^{\geq m}_{\Dual}$ on $\C$, defined using the induced t-structure on $\Ind(\C)$. Then the following holds:
\begin{enumerate}
\item
\label{item:sym}%
The groups $\L_n(\C,\QF^{\sym}_{\Dual})$ vanish for odd $n$, and for $n=2k$ the duality preserving map
\[
(\C^{\heartsuit},(-1)^k\Dual^{\heartsuit}) \to (\C,\Om^{2k}\Dual) \quad\quad x \mapsto x[-k]
\]
induces an isomorphism 
\[
\W(\C^{\heartsuit},\Dual^{\heartsuit}) \xrightarrow{\cong} \L_{2k}(\C,\QF^{\sym}_{\Dual}) .
\]
In addition, every stably metabolic Poincaré object in $(\C^{\heartsuit},\Dual^{\heartsuit})$ is metabolic. 
\item
\label{item:2m-3}%
For $n \geq 2m-3$ odd the group $\L_n(\C,\QF^{\geq m}_{\Dual})$ vanishes.
\item
\label{item:2m-2}%
For $n$ even the map
\[
\L_n(\C,\QF^{\geq m}_{\Dual}) \rightarrow \L_n(\C,\QF^{\sym}_{\Dual}) 
\]
is an isomorphism for $n \geq 2m$ and injective for $n=2m-2$.  
\end{enumerate}
\end{proposition}
\begin{proof}
For~\ref{item:sym} apply~\cite{9-authors-III}*{Proposition~1.3.1}, with $d=0,r=\infty,b=0$ and $a=-1$ to deduce that $\L_n(\C,\QF^{\sym}_{\Dual})$ vanishes for odd $n$ and with $d=0,r=2-m,b=0$ and $a=-1$ to deduce that $\L_n(\C,\QF^{\geq m}_{\Dual})$ vanishes for odd $n \geq 2r+1 = 2m-3$. This covers all statement involving odd symmetric or genuine $\L$-groups. For even $n$-groups, we first reduce for simplicity to the case $n=0$ as follows.
For a given $k \in \ZZ$, if $\Dual$ interchanges $\C_{\geq 0}$ and $\C_{\leq 0}$ then $\Om^{2k}\Dual$ interchanges $\C_{\geq -k}$ and $\C_{\leq -k}$, and so the assumptions of the theorem also hold when $\Dual$ is replaced with the shifted duality $\Om^{2k}\Dual$ and $(\C_{\geq 0},\C_{\leq 0})$ with the shifted $t$-structure $(\C_{\geq -k},\C_{\leq -k})$. We that the t-structure on $\Ind(\C)$ induced by the shifted $t$-structure on $\C$ is just the one obtained by shifting the $t$-structure on $\Ind(\C)$ induced by $(\C_{\geq 0},\C_{\leq 0})$. We then have that 
\[
\QF^{\sym}_{\Om^{2k}\Dual} = \Om^{2k}\QF^{\sym}_{\Dual} \quad\text{and}\quad
\QF^{\geq m-k}_{\Om^{2k}\Dual} = \Om^{2k}\QF^{\geq m}_{\Dual} ,
\]
where the truncation in $\QF^{\geq m-k}_{\Om^{2k}\Dual}$ is calculated with respect to the t-structure $(\C_{\geq -k},\C_{\leq k})$ and the truncation in $\Om^{2k}\QF^{\geq m}_{\Dual}$ is calculated with respect to the t-structure $(\C_{\geq 0},\C_{\leq 0})$. 
Replacing the duality and t-structures by their shifts we may thus reduce to proving the statements involving even $\L$-groups to the case $n=0$. In particular, to prove~\ref{item:sym} it will suffice to show that the map $\W(\C,\Dual) \to \L_0(\C,\QF^{\sym}_{\Dual})$ is an isomorphism and that every stably metabolic object in $\W(\C,\Dual)$ is metabolic, and to prove~\ref{item:2m-2} it will suffice to show that the map
\[
\L_n(\C,\QF^{\geq m}_{\Dual}) \to \L_n(\C,\QF^{\sym}_{\Dual})
\]
is an isomorphism for $m \leq 0$ and injective for $m = 1$.

For~\ref{item:sym}, we first apply~\cite{9-authors-III}*{Proposition~1.3.1}, with $d=0,r=\infty$ and $n=a=b=0$ to deduce that the map
\[
\L^{0,0}_0(\C,\QF^{\sym}_{\Dual}) \to \L_0(\C,\QF^{\sym}_{\Dual})
\]
is an isomorphism. By definition, $\L^{0,0}_0(\C,\QF^{\sym}_{\Dual})$ is the quotient of the monoid of Poincaré objects $(x,q)$ such that $x$ is connective modulo the submonoid of those which admits a Lagrangian $\cob \to x$ such that $\cob$ and $\cof[\cob \to x] \simeq \Dual \cob$ are connective. Since $\Dual$ maps connective objects to coconnective objects and vice verse we see that such Poincaré objects and Lagrangians are in fact contained in $\C^{\heartsuit}$, and Lagrangian inclusions $\cob \to x$ are monomorphisms in the abelian category $\C^{\heartsuit}$. We may consequently identify $\L^{0,0}_0(\C,\QF^{\sym}_{\Dual})$ with $\W(\C,\Dual)$, establishing the even $\L$-groups part of~\ref{item:sym}. To obtain final part of~\ref{item:sym} note that every stably metabolic object $(x,q)$ represents zero in $\W(\C,\QF^{\sym}_{\Dual})$, and hence $(x[0],q)$ represents zero in $\L_0(\C,\QF^{\sym}_{\Dual})$. It follows that $(x[0],q)$ is metabolic, and so by~\cite{9-authors-III}*{Proposition~1.3.1~(ii)}, with $d=0,r=\infty$ and $n=a=b=0$ it admits a Lagrangian $\cob \to x$ such that $\cob$ and $\cof[\cob \to x] \simeq \Dual \cob$ are connective, so that $\cob$ belongs to $\C^{\heartsuit}$ and $\cob \to x$ is a monomorphism in $\C^{\heartsuit}$. 
This exactly means that $(x,q)$ is metabolic in $(\C^{\heartsuit},\Dual^{\heartsuit})$.

Finally, let us prove~\ref{item:2m-2} in the case of $n=0$.
For this, we apply~\cite{9-authors-III}*{Proposition~1.3.1}, with $d=0,r=2-m$ and $n=a=b=0$ to deduce that the map
\[
\L^{0,0}_0(\C,\QF^{\geq m}_{\Dual}) \to \L_0(\C,\QF^{\geq m}_{\Dual})
\]
is an isomorphism whenever $m \leq 2$. As above, $\L^{0,0}_0(\C,\QF^{\geq m}_{\Dual})$ is the quotient of the monoid of Poincaré objects $(x,q)$ such that $x$ belongs to $\C^{\heartsuit}$ modulo the submonoid of those which admits a Lagrangian $\cob \to x$ such that $\cob \in \C^{\heartsuit}$ and $\cob \to x$ is a monomorphism in $\C^{\heartsuit}$. 
Now for $x \in \C^{\heartsuit}$ the fibre of
\[
\QF^{\leq m}_{\Dual}(x) \to \QF^{\sym}_{\Dual}(x)
\]
is $(m-2)$-truncated (essentially by definition). Suppose first that $m \leq 0$, so that this fibre is $(-2)$-truncated. 
Then for $x \in \C^{\heartsuit}$ we have that the map 
\[
\Om^{\infty}\QF^{\leq m}(x) \to \Om^{\infty}\QF^{\sym}_{\Dual}(X) 
\]
is an equivalence.
Similarly, for a map $\cob \to x$ in $\C^{\heartsuit}$ we have that the total fibre of the square
\[
\begin{tikzcd}
\QF^{\leq m}(x) \ar[r]\ar[d] & \QF^{\sym}(x) \ar[d] \\
\QF^{\leq m}(\cob) \ar[r] & \QF^{\sym}(\cob) 
\end{tikzcd}
\]
is $(-2)$-truncated, so that the map 
\[
\Om^{\infty}\fib[\QF^{\leq m}(x) \to \QF^{\leq m}(\cob)] \to \Om^{\infty}\fib[\QF^{\sym}(x) \to \QF^{\sym}(\cob)]
\]
Combining this with the above we therefore conclude that when $m \leq 0$ the top horizontal and both vertical maps in the square
\[
\begin{tikzcd}
\L_0^{0,0}(\C,\QF^{\geq m}_{\Dual}) \ar[r]\ar[d, "\cong"] & \L_0^{0,0}(\C,\QF^{\sym}_{\Dual}) \ar[d, "\cong"] \\
\L_0(\C,\QF^{\geq m}_{\Dual}) \ar[r] & \L_0(\C,\QF^{\sym}_{\Dual})  
\end{tikzcd}
\]
are isomorphisms, and hence the bottom horizontal map is an isomorphism as well. Finally, suppose that $m=1$. 
Then for every $x \in \C^{\heartsuit}$ the map $\QF^{\geq 0}_{\Dual}(x) \to \QF^{\sym}_{\Dual}(x)$ is $(-1)$-truncated. Now as we saw above the map $\L^{0,0}_0(\C,\QF^{\geq m}_{\Dual}) \to \L_0(\C,\QF^{\geq m}_{\Dual})$ is still an isomorphism for $m=1$, and so every class in $\L_0(\C,\QF^{\geq m}_{\Dual})$ can be represented by a Poincaré object $(x,q)$ such that $x \in \C^{\heartsuit}$. For such a Poincaré objects, if its image in $\L_0(\C,\QF^{\sym}_{\Dual})$ is zero then by~\cite{9-authors-III}*{Proposition~1.3.1~(ii)}, with $d=0,r=\infty$ and $n=a=b=0$ the image of $(x,q)$ in $(\C,\QF^{\sym}_{\Dual})$ admits a Lagrangian $\cob \to x$ with $\cob \in \C^{\heartsuit}$. 
Now if we consider the associated commutative rectangle with exact rows
\[
\begin{tikzcd}
\QF^{\geq m}_{\Met}(\cob \to x) \ar[r]\ar[d] & \QF^{\leq m}(x) \ar[r]\ar[d] & \QF^{\leq m}(\cob) \ar[d] \\
\QF^{\sym}_{\Met}(\cob \to x) \ar[r] & \QF^{\sym}(x) \ar[r] & \QF^{\sym}(\cob) \ ,
\end{tikzcd}
\]
then the fibre of the right most vertical map is $(-1)$-truncated and hence the total fibre of the left square is $(-2)$-truncated. We hence conclude that the Lagrangian $\cob \to x$ lifts to a Lagrangian with respect to $\QF^{\geq m}_{\Dual}$, so that $[x,q] = 0$ in $\L_0(\C,\QF^{\geq 0}$. We conclude that the map
\[
\L_0(\C,\QF^{\geq m}_{\Dual}) \to \L_0(\C,\QF^{\sym}_{\Dual})
\]
is injective when $m = 1$, as desired.
\end{proof}

\begin{proof}[Proof of Theorem~\ref{theorem:devissage}]
By Remark~\ref{remark:assume-r-0} we may, and will, assume that $r=d$, so that the underlying $R$-complex of $L$ is equivalent to $R[d]$. By Corollary~\ref{corollary:purity} we can apply Proposition~\ref{proposition:key} to both $(\Dperf(k),\QF^{\geq m}_{i^!L})$ and $(\Dperf_{\fm}(R),\QF^{\geq m}_L|_{\fm})$. 
To finish the proof it will now suffice to show that the map
\begin{equation}
\label{equation:witt}%
\inflation_*\colon \W^{\pm}(\Coh(k)) \to \W^{\pm}(\Coh_{\fm}(R)) 
\end{equation}
of classical (symmetric and anti-symmetric) Witt groups induced by the inflation functor $\inflation \colon \Coh(k) \to \Coh_{\fm}(R)$, is an isomorphism. Here, $\Coh(k)$ is endowed with the duality $V \mapsto \Hom_k(V,\pi_{-d}i^!L)$ and $\Coh_{\fm}(R)$ is endowed with the duality $M \mapsto \Hom_R(M,L_{\fm})$, see Corollary~\ref{corollary:purity}.
We note that by Proposition~\ref{proposition:key}\;\ref{item:sym} any (anti-)symmetric Poincaré object in $\Coh(k)$ which maps to $0$ in $\W^{\pm}(\Coh(k))$ contains a Lagrangian in $\Coh(k)$, and the same holds for Poincaré (anti-)symmetric objects in $\Coh_{\fm}(R)$. But if $V$ is a finite dimensional $k$-vector space then $\inflation\colon \Coh(k) \to \Coh_{\fm}(R)$ induces a bijection between $k$-vector subspaces of $V$ and $R$-submodules of $\inflation(V)$. It then follows that~\eqref{equation:witt} is injective. We also note that a finite length $R$-module is in the image of $\inflation$ if and only if it is annihilated by $\fm \subseteq R$. To finish the proof it will hence suffice to show that if $M$ is a finite length $R$-module equipped with a non-degenerate (anti-)symmetric bilinear form then $M$ contains a sub-Lagrangian $N \subseteq M$ such that $N^{\perp}/N$ is annihilated by $\fm$. Indeed, this is a special case of~\cite{Quebbemann-Scharlau-Schulte}*{Theorem 6.10} for $\mathcal{M} = \Coh_{\fm}(R)$ and $\mathcal{M}_0 = \Coh(k)$. 
\end{proof}

\begin{corollary}
\label{corollary:relative-devissage}%
Let $g\colon \spec(R') \to \spec(R)$ be a closed embedding associated to a surjective ring homomorphism $R \to R'$ between regular Noetherian local rings of Krull dimensions $d$ and $d'$ and with maximal ideals $\fm$ and $\fm'$, respectively. Let $r \in \ZZ$ be an integer and $L$ an invertible perfect $R$-complex with $\Ct$-action whose underlying $R$-complex is equivalent to an $r$-shift of a line bundle. Then the map
\[
\L_n(\Dperf_{\fm}(R),\QF^{\geq m}_L|_{\fm}) \to \L_n(\Dperf_{\fm'}(R'),\QF^{\geq m}_{g_!L}|_{\fm})
\]
induced by push-forward along $g$ is an isomorphism when $n \geq 2m-1+d-r$ and injective for $n=2m-2+d-r$.
\end{corollary}
\begin{proof}
Write $k = R/\fm = R'/\fm'$ for the common residue field of $R$ and $R'$.
Since $\spec(R)$ is regular $L$ is a dualising complex on $\spec(R)$ and since $\spec(R') \to \spec(R)$ is a closed embedding we have that $g^!L$ is a dualising complex on $\spec(R')$, see \cite[\href{https://stacks.math.columbia.edu/tag/0BZI}{Tag 0BZI}]{stacks-project}. But since $\spec(R')$ is regular $g^!L$ must be an invertible perfect complex, hence of the form $R'[r']$ for some $r'$. Let $i'\colon \spec(k) \to \spec(R')$ be the inclusion of the unique closed point of $\spec(R')$, so that $i = r \circ i'\colon \spec(k) \to \spec(R)$ is the inclusion of the unique closed point of $\spec(R)$. By Lemma~\ref{lemma:purity-2} we then have  
\[
k[r-d] \simeq i^!L = (i')^!g^!L \simeq (i')^!R'[r'] \simeq k[r'-d']
\]
and so $r-d = r'-d'$. Replacing $L$ with $L[d-r]$ and using Remark~\ref{remark:assume-r-0} we may hence assume without loss of generality that $r=d$ and $r'=d'$. 

Consider the commutative diagram
\[
\begin{tikzcd}
\L_n(\Dperf(k),\QF^{\geq m}_{i^!L}) \ar[r]\ar[d] & \L_n(\Dperf_{\fm'}(R'),\QF^{\geq m}_{r^!L}|_{\fm'}) \ar[r]\ar[d] & \L_n(\Dperf_{\fm}(R),\QF^{\geq m}_{L}|_{\fm}) \ar[d] \\
\L_n(\Dperf(k),\QF^{\sym}_{i^!L}) \ar[r] & \L_n(\Dperf_{\fm'}(R'),\QF^{\sym}_{r^!L}|_{\fm'}) \ar[r] & \L_n(\Dperf_{\fm}(R),\QF^{\sym}_{L}|_{\fm}) \ .
\end{tikzcd}
\]
By Theorem~\ref{theorem:devissage} and the 2-out-of-3 property the top horizontal arrows are isomorphisms for $n \geq 2m-1$ (recall that we have set $r=d$ and $r'=d'$) and the bottom horizontal arrows are isomorphisms for every $n$. In addition, for $n=2m-2$ we have by Proposition~\ref{proposition:key} that the vertical maps are injective and hence the top horizontal maps are injective as well.
\end{proof}

\subsection{End of the proof of global dévissage}
\label{subsection:global-devissage}%

We are finally ready to conclude the proof of the dévissage theorem. As explained in \S\ref{subsection:devissage}, given dévissage for $\K$-theory (Remark~\ref{remark:K-theory-devissage}) and the fundamental fibre sequence~\cite{9-authors-II}*{Corollary~11.4.14}, Theorem~\ref{theorem:global-devissage} is a consequence of Theorem~\ref{theorem:genuine-devissage}.

\begin{proof}[Proof of Theorem~\ref{theorem:genuine-devissage}]
We prove by descending induction on $c$ that the map
\[
\L(\Dperf(Z)^{\ge c},\QF^{\sym}_{i^!L}) \to \L(\Dperf_Z(X)^{\ge c},\QF^{\sym}_L) 
\]
has a $(2m-2-d)$-truncated cofibre. For the base of the induction we note that for $c > \dim(Z)$ the domain and codomain of this map are zero. Now suppose that $c \geq 0$ is such that the claim holds for $c+1$. Consider the commutative diagram
\begin{equation}
\label{equation:filtration}%
\begin{tikzcd}
\L(\Dperf(Z)^{\ge c+1},\QF^{\geq m}_{i^!L}) \ar[r]\ar[d] & \L(\Dperf(Z)^{\ge c},\QF^{\geq m}_{i^!L})\ar[r]\ar[d] & \bigoplus_{x\in Z^{(c)}} \L(\Dperf_x(\cO_{Z,x}),\QF^{\geq m}_{i^!L}|_x)\ar[d]\\
\L(\Dperf_Z(X)^{\ge c+1},\QF^{\geq m}_L)\ar[r] & (\Dperf_Z(X)^{\ge c},\QF^{\geq m}_L)\ar[r] & \bigoplus_{x\in Z^{(c)}} \L(\Dperf_x(\cO_{X,i(x)}),\QF^{\geq m}_L|_x)
\end{tikzcd}
\end{equation}
whose rows are fibre sequences by Lemma~\ref{proposition:coniveau-sequence}. Then the cofibre of the left most vertical map is $(2m-2-d)$-truncated by the induction hypothesis and the cofibre of the right most vertical map is $(2m-2-d)$-truncated by Corollary~\ref{corollary:relative-devissage} (applied with $r=0$). We hence conclude that the cofibre of the middle vertical map has $(2m-2-d)$-truncated cofibre, as desired.
\end{proof}

\section{The projective bundle formula and $\Aone$-invariance}
\label{section:projective-bundle}%

In this section we prove a projective bundle formula for symmetric Karoubi-Grothendieck-Witt theory (and more generally for invariants of schemes arising from Karoubi-localising invariants of Poincaré $\infty$-categories as in Notation~\ref{notation:functor-to-sheaf}), generalizing in particular results obtained in \cite{Walter}, \cite{nenashev} and \cite{Rohrbach}, which were proven under the assumption that 2 was invertible in the base scheme. This is achieved in \S\ref{subsection:pbf}, see Theorem~\ref{theorem:projective-bundle-formula}. In \S\ref{subsection:genuine-projective} we investigate the analogue result for the genuine symmetric Poincaré structure, but only in the case of constant $\Pone$-bundles. Finally, in \S\ref{subsection:Aone-invariance} we combine the projective bundle formula with dévissage to deduce $\Aone$-invariance of symmetric $\GW$-theory over regular Noetherian schemes of finite Krull dimension.

\subsection{The projective bundle formula}
\label{subsection:pbf}%

Let $X$ be a qcqs scheme, $L$ an invertible perfect complex over $X$ equipped with $C_2$-action and $V$ a vector bundle over $X$ of rank $r+1$. Write $p\colon\PP_X V=\Proj(\Sym^\bullet V)\to X$ for the corresponding projective bundle over $X$, with tautological line bundle $\cO(1)$. For $n \in \ZZ$ we will write $\cO(n) := \cO(1)^{\otimes n}$ and $p^*L(n) := p^*L \otimes \cO(n)$. Our goal is to describe the symmetric $\GW$-spectra $\GW^{\sym}(\PP_X V,p^*L(n))$
in terms of certain symmetric $\GW$-spectra of $X$. In fact, we provide a structural result for the Poincaré $\infty$-category $(\Dperf(\PP_X V),\QF^\sym_{p^*L(n)})$ that allows us to give a projective bundle formula for every additive invariant.

\begin{remark}
\label{remark:square-line-bundles}%
Suppose $M$ is any line bundle on a scheme $Y$. Then the canonical $C_2$-action on $M\otimes_Y M$ is trivial. Indeed, since $M$ is a line bundle we have that $M \otimes_Y M$ is a line bundle as well, and in particular belongs to the heart of $\Derqc(Y)$. We may hence view it as an object in the ordinary category $\Mod^\qc(Y)$ of quasi-coherent sheaves on $Y$, so that the triviality of the $\Ct$-action amounts to an equality between the action of the generator and the identity map inside the abelian group $\Hom_Y(M \otimes_Y M,M \otimes_Y M)$. By definition of mappings between two Zariski sheaves this equality can be checked Zariski-locally. We may hence assume that $M=\cO_Y$. But then the multiplication map $\cO_Y\otimes\cO_Y\to \cO_Y$ is a $\Ct$-equivariant isomorphism for the trivial action on the target. In particular, this shows that
\[
(-)\otimes \cO(t)\colon (\Dperf(\PP_XV),\QF^s_{p^*L(n)})\to (\Dperf(\PP_XV),\QF^s_{p^*L(n-2t)})
\]
is an equivalence of Poincaré $\infty$-categories.  
\end{remark}

By~\cite{Khan-blow-up}*{Theorem~3.3}(i) we have that the functor 
\[
p^*(-)\otimes \cO(i)\colon \Dperf(X)\to \Dperf(\PP_XV)
\]
is fully faithful for every $i$. Let us write $\cA(i)$ for its image, and for $i \leq j$ write $\cA(i,j)$ for the stable subcategory of $\Dperf(\PP_XV)$ generated by $\cA(i),\cA(i+1),...,\cA(j)$. Then by~\cite{Khan-blow-up}*{Theorem~3.3~(ii)} the full subcategories $\cA(i+r), \cA(i+r-1), \ldots,\cA(i)$ form a semi-orthogonal decomposition of $\Dperf(\PP_XV)$ for every $i$. In particular, if $i,j \in \ZZ$ are such that $-r \leq j-i \leq -1$ then $\map_{\PP_X V}(M,N) = 0$ for every $M \in \cA(i)$ and $N \in \cA(j)$. In addition, the left orthogonal complement of $\cA(i-r,i-1)$ is given by $\cA(i)$ and its right orthogonal complement is given by $\cA(i-r-1)$. Finally, note that a right adjoint to the full faithful functor $p^*(-)\otimes \cO(i)$ is given by $p_*((-) \otimes \cO(-i))$, and the latter also admits a right adjoint given by $p^!((-) \otimes \cO(i))$.
We hence obtain a split Verdier sequence
\[
\cA(i-r,i-1) \to \Dperf(\PP_X V) \xrightarrow{p_*((-) \otimes \cO(-i)} \Dperf(X).
\]

\begin{lemma}
\label{lemma:duality-shift}%
For every $i,n\in \ZZ$ we have 
\[
\Dual_{p^*L(n)}(\cA(i))= \cA(n-i)
\]
as full subcategories of $\Dperf(\PP_XV)$.
\end{lemma}
\begin{proof}
Since $\Dual_{p^*L(n)}^2\simeq \id$, it suffices to prove one inclusion, and in doing so we may restrict attention to the generators of $\cA(i)$. For $P \in \Dperf(X)$ we then have
\begin{align*}
\Dual_{p^*L(n)}(p^*(P) \otimes \cO(i)) &= \Dual_{\PP_XV}(p^*(P) \otimes \cO(i)) \otimes p^*L(n) \\
&= \Dual_{\PP_X V}(p^*(P)) \otimes \Dual_{\PP_X V}(\cO(i)) \otimes p^*L(n) \\
&= p^*\Dual_X(P) \otimes \cO(-i) \otimes p^*L(n) \in \cA(n-i) 
\end{align*}
where we have used that $\Dual_{\PP_X V}$ is a symmetric monoidal duality (it coincides with the canonical duality of $\PP_X V$ as a rigid symmetric monoidal $\infty$-category, see~\S\ref{subsection:rigid-to-poinc}), and that $p^*$ is duality preserving for $\Dual_X$ and $\Dual_{\PP_X V}$.
\end{proof}

Write $r'=r$ if $n-r$ is even and $r'=r-1$ if $n-r$ is odd. In particular $n-r'$ is always even. 
Then $\cA\big(\frac{n-r'}{2},\frac{n+r'}{2}\big)$ is stable under the duality, hence a Poincaré subcategory. Moreover, if $n-r$ is even then 
\[
\cA\big(\frac{n-r'}{2},\frac{n+r'}{2}\big)=\Dperf(\PP_XV)
\]
by~\cite{Khan-blow-up}*{Theorem~3.3~(ii)}. When $n-r$ is odd $\cA\big(\frac{n-r'}{2},\frac{n+r'}{2}\big)\neq \Dperf(\PP_XV)$ but one can nonetheless describe explicitly its Poincaré-Verdier quotient. For this, let $\mathcal{V} = p^*V$ be the pullback of $V$ to $\PP_X V$. By the construction of $\PP_X V$ as $\Proj(\Sym^{\bullet} V)$ the vector bundle $\mathcal{V}$ comes with a tautological map $\mathcal{V} \to \cO(1)$. The $\cO(-1)$-twisted map $\mathcal{V}(-1)\to \cO_{\PP_X V}$ then fits in an exact complex 
\[
0\to\Lambda^{r+1}\mathcal{V}(-r-1)\to \cdots \to \Lambda^1\mathcal{V}(-1)\to\cO_{\PP_X V}\to 0\,.
\]
of vector bundles on $\PP_X V$.
In particular, $\det\mathcal{V}(-r-1)[r] = \Lambda^{r+1}\mathcal{V}(-r-1)[r]$ is quasi-isomorphic to the complex
\[
\Lambda^{r}\mathcal{V}(-r)\to \cdots \to \Lambda^1\mathcal{V}(-1)\to\cO_{\PP_X V} \,.
\]
concentrated in degrees $[0,r]$. 
The inclusion of the bottom term gives a map 
\begin{equation}
\label{equation:important-map}%
\cO_{\PP_X V}\to\det\mathcal{V}(-r-1)[r].
\end{equation}
Since $p_*\Lambda^i\mathcal{V}(-i) = 0$ for $i=1,\ldots,r$, this map induces an equivalence
\[
p_*\cO_{\PP_X V} \tosimeq p_*\det\mathcal{V}(-r-1)[r] = p_*\cO(-r-1) \otimes \det V[r],
\]
and hence an equivalence 
\[
p^!p_*\cO_{\PP_X V} \tosimeq p^!p_*\det\mathcal{V}(-r-1)[r] \simeq \det\mathcal{V}(-r-1)[r],
\]
where the latter equivalence is because $\det\mathcal{V}(-r-1)[r]$ belongs to $\cA(-r-1)$, and is hence right orthogonal to $\cA(-r,-1) = \ker(p_*)$.
Now, since $p^*$ is fully-faithful the unit map $\cO_X \to p_*p^*\cO_X = p_*\cO_{\PP_X V}$ is an equivalence, and so we obtain equivalences
\begin{equation}
\label{equation:tau}%
\tau\colon p^!\cO_X \simeq \det\mathcal{V}(-r-1)[r] .
\end{equation}
Let us now suppose that $n-r$ is odd.
Applying Construction~\ref{construction:pushforward-hermitian} with $M = \cO(-\frac{n+r+1}{2})$ using the composed equivalence
\begin{align*}
p^!(L\otimes \det V^\vee[-r]) \otimes M^{\otimes 2} \simeq& p^*L\otimes \det \mathcal{V}^\vee[-r] \otimes p^!\cO_X \otimes M^{\otimes 2} \\
\stackrel{\tau}{\simeq}& p^*L\otimes \det \mathcal{V}^\vee[-r] \otimes \det\mathcal{V}(-r-1)[r] \otimes M^{\otimes 2}\\
\simeq& p^*L(n)
\end{align*}
now refines the split Verdier projection $p_*((-) \otimes \cO(-\frac{n+r+1}{2})$ to a Poincaré functor
\[
\big(\Dperf(\PP_XV),\QF^\sym_{p^*L(n)}\big)\xrightarrow{p_*(-\otimes\cO(-\frac{n+r+1}{2}))} \big(\Dperf(X),\QF^{\sym}_{L \otimes \det V^\vee[-r]}\big)
\]
with kernel $\cA\big(\frac{n-r'}{2},\frac{n+r'}{2}\big)$.

\begin{proposition}
\label{proposition:proj-bundle-I}%
Suppose $n-r$ is odd. Then the sequence
\[
\big(\cA\big(\frac{n-r'}{2},\frac{n+r'}{2}\big),\QF^\sym_{p^*L(n)}\big)\to \big(\Dperf(\PP_XV),\QF^\sym_{p^*L(n)}\big)\xrightarrow{p_*(-\otimes\cO(-\frac{n+r+1}{2}))} \big(\Dperf(X),\QF^{\sym}_{L \otimes \det V^\vee[-r]}\big)
\]
is a split Poincaré-Verdier sequence (where we have written $\QF^\sym_{p^*L(n)}$ also for its restriction to $\cA\big(\frac{n-r'}{2},\frac{n+r'}{2}\big)$). In particular, when $n=0$ and $r=1$ we have a split Poincaré-Verdier sequence
\[
\big(\Dperf(X),\QF^\sym_L\big)\xrightarrow{p^*}\big(\Dperf(\PP_XV),\QF^\sym_{p^*L}\big) \xrightarrow{p_*(-\otimes\cO(-1))} \big(\Dperf(X),\QF^\sym_{L \otimes \det V^\vee[-1]}\big).
\]
\end{proposition}
\begin{proof}
Apply Lemma~\ref{lemma:pushforward-projection}.
\end{proof}

Now, let $\Lag\subseteq \cA\big(\frac{n-r'}{2},\frac{n+r'}{2}\big)$ be the stable subcategory generated by $\cA(\frac{n+1}{2})\cup\cdots\cup\cA(\lceil\frac{n+r'}{2}\rceil)$. 
By~\cite{Khan-blow-up}*{Theorem~3.3~(ii)} $\Lag$ has a semi-orthogonal decomposition of the form $(\cA(\frac{n+r'}{2}),\dots,\cA(\lceil\frac{n+1}{2}\rceil))$. In particular, for every additive invariant $\F$ there's an equivalence
\[
\F^\hyp(\Lag)\simeq \F^\hyp(\Dperf(X))^{\oplus \lfloor\frac{r'+1}{2}\rfloor}
\]
induced by the functors $\{p^*(-) \otimes \cO(i)\}_{i=\frac{n+r'}{2},\dots,\lceil\frac{n+1}{2}\rceil}$ (see \cite{Khan-blow-up}*{Lemma~2.8}).

\begin{lemma}
\label{lemma:proj-bundle-II}%
If $n$ is odd, then the stable subcategory $\Lag = \cA\big(\frac{n+1}{2},\frac{n+r'}{2}\big) \subseteq \cA\big(\frac{n-r'}{2},\frac{n+r'}{2}\big)$ is a Lagrangian with respect to the Poincaré structure restricted from $\QF^{\sym}_{p^*L}$. If $n$ is even, then $\Lag$ is isotropic with respect to the Poincaré structure and its homology inclusion can be identified with the Poincaré functor
\[
p^*(-) \otimes \cO\big(\frac{n}{2}\big)\colon (\Dperf(X),\QF^{\sym}_{L}) \to (\Dperf(\PP_X V),\QF^{\sym}_{p^*L(n)}) .
\]
\end{lemma}
\begin{proof}
Since $(\cA(\frac{n+r'}{2}),\dots,\cA(\frac{n-r'}{2}))$ is a semi-orthogonal decomposition of $\cA\big(\frac{n-r'}{2},\frac{n+r'}{2}\big)$, the inclusion $\Lag\subseteq \cA\big(\frac{n-r'}{2},\frac{n+r'}{2}\big)$ has a right adjoint $r\colon \cA\big(\frac{n-r'}{2},\frac{n+r'}{2}\big) \to \Lag$ whose kernel can be identified using Lemma~\ref{lemma:duality-shift} with either $\Dual_{p^*L(n)}\Lag$ if $n$ is odd or with the stable subcategory generated by $\Dual_{p^*L(n)}\Lag$ and $\cA(\frac{n}{2})$ if $n$ is even. Hence by \cite{9-authors-II}*{Remark~3.2.3} and Lemma~\ref{lemma:duality-shift} the orthogonal complement $\Lag^\perp$ of $\Lag$ is either $\Lag$ if $n$ is odd or the subcategory generated by $\Lag$ and $\cA(\frac{n}{2})$ if $n$ is even. In particular, $\Lag\subseteq \Lag^\perp$, and since the Poincaré structure is a symmetric one this suffices to prove that $\QF^\sym_{p^*L(n)}|_\Lag=0$. Thus $\Lag$ is an isotropic subcategory of $\cA\big(\frac{n-r'}{2},\frac{n+r'}{2}\big)$ which is furthermore a Lagrangian when $n$ is odd. Finally, when $n$ is even $\Lag^{\perp}$ admits a semi-orthogonal decomposition of the form $(\Lag,\cA(\frac{n}{2}))$ and so the duality invariant subcategory $\Lag^{\perp} \cap \Dual_{p^*L(n)}\Lag^{\perp} = \Lag^{\perp} \cap \ker(r)$ coincides with $\cA(\frac{n}{2})$. We may then identify the associated homology inclusion with the resulting Poincaré functor
\[
\big(\cA\big(\frac{n}{2}\big),\QF^{\sym}_{p^*L(n)}|_{\cA(\frac{n}{2})}\big) \hrar \big(\Dperf(\PP_X V),\QF^{\sym}_{p^*L(n)}\big) .
\]
Now since the functor $p^*(-) \otimes \cO(\frac{n}{2})\colon \Dperf(X) \to \Dperf(\PP_X V)$ is fully-faithful with image $\cA(\frac{n}{2})$ the last Poincaré functor is equivalent to the Poincaré functor
\[
(\Dperf(X),\QF) \hrar \big(\Dperf(\PP_X V),\QF^{\sym}_{p^*L(n)}\big)
\]
where $\QF$ is obtained by restricting $\QF^{\sym}_{p^*L(n)}$ along $p^*(-) \otimes \cO(\frac{n}{2})$. We then compute
\begin{align*}
\QF(P)&=\map_{\PP_XV}(p^*P\otimes p^*P \otimes \cO(n),p^*L(n))^\hC\\
&\simeq \map_{\PP_XV}(p^*(P\otimes P) ,p^*L)^\hC \\
&\simeq\map_X(P\otimes P,L)^\hC \\
&=\QF^{\sym}_L(P) \ ,
\end{align*}
as desired.
\end{proof}

Putting the two lemmas together we obtain the projective bundle formula.

\begin{theorem}[The projective bundle formula]
\label{theorem:projective-bundle-formula}%
Let $X$ be a qcqs scheme, let $L$ be an invertible perfect complex over $X$ with $\Ct$-action and let $V$ be a vector bundle over $X$ of rank $r+1$. Let $\F\colon \Catp\to \E$ be an additive functor valued in some stable $\infty$-category $\E$. Then the following holds
\begin{enumerate}
\item
If $n$ and $r$ are both odd then the exact functors $p^*(-) \otimes \cO(i)\colon\Dperf(X)\to \Dperf(\PP_XV)$ for $i=\frac{n+r}{2},\dots,\frac{n+1}{2}$ induce an equivalence
\[
\F^\hyp(\Dperf(X))^{\oplus (r+1)/2} \tosimeq \F\big(\Dperf(\PP_X V),\QF^{\sym}_{p^*L(n)}) .
\]
\item
If $n$ and $r$ are both even then the exact functors $p^*(-) \otimes \cO(i)\colon\Dperf(X)\to \Dperf(\PP_XV)$ for $i=\frac{n+r}{2},\dots,\frac{n+2}{2}$ together with the Poincaré functor $p^*(-) \otimes \cO(\frac{n}{2})\colon(\Dperf(X),\QF^{\sym}_{L}) \to (\Dperf(\PP_XV),\QF^{\sym}_{p^*L(n)})$ induce an equivalence
\[
\F^\hyp(\Dperf(X))^{\oplus r/2} \oplus \F(\Dperf(X),\QF^{\sym}_L) \tosimeq \F\big(\Dperf(\PP_X V),\QF^{\sym}_{p^*L(n)}) .
\]
\item
\label{equation:pbf-II}%
If $n$ is odd and $r$ is even then we have a fibre sequence 
\begin{equation*}
\F^\hyp(\Dperf(X))^{\oplus r/2}\to \F\big(\Dperf(\PP_XV),\QF^\sym_{p^*L(n)}\big)\xrightarrow{p_*((-) \otimes \cO(-\frac{n+r+1}{2}))} \F\big(\Dperf(X),\QF^\sym_{L \otimes \det V^\vee[-r]}\big) \ ,
\end{equation*}
where the first arrow is induced by the functors $p^*(-) \otimes \cO(i)\colon\Dperf(X)\to \Dperf(\PP_XV)$ for $i=\frac{n+r-1}{2},\dots,\frac{n+1}{2}$.
\item
\label{equation:pbf-I}%
If $n$ is even and $r$ is odd then we have a fibre sequence
\begin{equation*}
\F^\hyp(\Dperf(X))^{\oplus (r-1)/2} \oplus \F(\Dperf(X),\QF^\sym_L)\to \F(\Dperf(\PP_XV),\QF^\sym_{p^*L(n)})\xrightarrow{p_*((-) \otimes \cO(-\frac{n+r+1}{2}))} \F(\Dperf(X),\QF^\sym_{L \otimes \det V^\vee[-r]}\big) \ ,
\end{equation*}
where the first arrow is induced by the functors $p^*(-) \otimes \cO(i)\colon\Dperf(X)\to \Dperf(\PP_XV)$ for $i=\frac{n+r-1}{2},\dots,\frac{n+2}{2}$ together with the Poincaré functor
$p^*(-) \otimes \cO(\frac{n}{2})\colon(\Dperf(X),\QF^{\sym}_{L}) \to (\Dperf(\PP_XV),\QF^{\sym}_{p^*L(n)})$.
\end{enumerate}
\end{theorem}
\begin{proof}
Note that when $n-r$ is even we have $\cA\big(\frac{n-r'}{2},\frac{n+r'}{2}\big)=\Dperf(\PP_XV)$ by~\cite{Khan-blow-up}*{Theorem~3.3~(ii)}, and when $n-r$ is odd this full subcategory participates in a split Poincaré-Verdier sequence as recorded in Proposition~\ref{proposition:proj-bundle-I}.
All the statements now follow by applying \cite{9-authors-II}*{Theorem~3.2.10} and \cite{Khan-blow-up}*{Lemma~2.8} to the isotropic subcategory inclusion $\Lag\subseteq\cA\big(\frac{n-r'}{2},\frac{n+r'}{2}\big)$ of Lemma~\ref{lemma:proj-bundle-II}.
\end{proof}

Our goal is now to study the splitting of the fibre sequences \ref{equation:pbf-II} and \ref{equation:pbf-I}. To do so, we need to compute a special case of the boundary operator. Suppose $r=\rk V-1$ is odd. Then the non-degenerate symmetric bilinear form
\[
\wedge\colon\Lambda^{\frac{r+1}{2}}V [\frac{r+1}{2}] \otimes\Lambda^{\frac{r+1}{2}}V [\frac{r+1}{2}] \to \det V [r+1]\,.
\]
determines a class $e(V) \in \L^{\sym}_{0}(X,\det V[r+1]) = \L^{\sym}_{-r-1}(X,\det V)$, called the \defi{Euler class} of $V$. Note that $e(V)=0$ whenever $V$ has a quotient bundle of odd rank \cite{Walter}*{Proposition~8.1}, in particular whenever $V$ is trivial.

\begin{lemma}
\label{lemma:pbf-boundary}%
Suppose that $r$ is odd and $n$ is even. Then the boundary map
\[
\partial\colon\L^{\sym}_0(X)\to\L^{\sym}_{-1}(X,\det V[r]) = \L^{\sym}_{-r-1}(X,\det V)
\]
of the fibre sequence of Theorem~\ref{theorem:projective-bundle-formula}~\ref{equation:pbf-I}, with $\F=\L$ and $L=\det V[r]$, sends 1 to $e(V)$.
\end{lemma}
\begin{proof}
We first note that since $\L^{\hyp} = 0$ the domain and target of this boundary map are indeed as indicated. Secondly, the parity conditions on $n$ and $r$ imply that $n+r+1$ is even. Hence, by Remark~\ref{remark:square-line-bundles}, we may assume without loss of generality that $n=-r-1$; this assumption simplifies the notation in the following argument, although mathematically it makes no essential difference. In particular, our Poincaré-Verdier projection takes the form
\[
p_*\colon \big(\Dperf(\PP_X V),\QF^{\sym}_{p^!\cO_X}\big) = \big(\Dperf(\PP_X V),\QF^{\sym}_{\det\mathcal{V}(-r-1)[r]}\big) \to (\Dperf(X),\cO_X) .
\]
We now carry out the explicit procedure to produce the boundary map as described in \cite{9-authors-II}*{Proposition~4.4.8}. For this, we need to find a hermitian object in $\big(\Dperf(\PP_X V),\QF^{\sym}_{\det\mathcal{V}(-r-1)[r]}\big)$ lifting the unit Poincaré object $\cO_X$ in $(\Dperf(X),\QF^{\sym}_X)$. Since $p_*$ is a Poincaré-Verdier projection there is a natural candidate for such a lift, obtained by applying to $\cO_X$ the left adjoint $p^*$, in which case the hermitian structure map
\[
\QF^{\sym}_{\det\mathcal{V}(-r-1)[r]}(\cO_{\PP_X V}) = \QF^{\sym}_{\det\mathcal{V}(-r-1)[r]}(p^*\cO_X)\to \QF^{\sym}_X(\cO_X)
\]
is an equivalence. Let $q \in \Om^{\infty}\QF^{\sym}_{\det\mathcal{V}(-r-1)[r]}(\cO_{\PP_X V})$ be the hermitian form corresponding to $\id \in \Om^{\infty}\QF^{\sym}_X(\cO_X)$ under the above equivalence. 
To calculate the boundary of $q$, we first note that the induced map of spaces
\[
\Map_{\PP_X V}\big(\cO_{\PP_X V}, \Dual_{\det\mathcal{V}(-r-1)[r]}\cO_{\PP_X V}\big) \to \Map_X(\cO_X, \Dual_X\cO_X)
\]
is already an equivalence before taking $\Ct$-homotopy fixed points, and that the second, and hence also the first, are discrete spaces. In particular, $q$ is determined up to homotopy by the associated map
\[
q_{\sharp}\colon p^*\cO_X = \cO_{\PP_X V} \to \Dual_{\det\mathcal{V}(-r-1)[r]}\cO_{\PP_X V} = \det\mathcal{V}(-r-1)[r],
\]
which in turn is determined up to homotopy by the condition that it maps to the identity map $\cO_X \to \cO_X$
by $p_*$, under the identifications 
\[
\cO_X \xrightarrow{\simeq} p_*p^*\cO_X \quad\text{and}\quad p_*\det\mathcal{V}(-r-1)[r] \stackrel{\tau}{\simeq} p_*p^!\cO_X \xrightarrow{\simeq} \cO_X
\]
determined by the unit of $p^* \dashv p_*$, the counit of $p_* \dashv p^!$, and the map $\tau$ 
of~\eqref{equation:tau} 
used to define to Poincaré structure on $p_*$. By the construction of $\tau$ the map $q_{\sharp}$ must then be homotopy equivalent to the map
\[
\cO_{\PP_X V}\to \det \mathcal{V}(-r-1)[r]
\]
of~\eqref{equation:important-map}, determined by the inclusion of the bottom term in the complex
\[
\Lambda^{r}\mathcal{V}(-r)\to \cdots \to \Lambda^1\mathcal{V}(-1)\to\cO_{\PP_X V}\,.
\]
Unwinding the definitions, we then see that the boundary of the hermitian object $(\cO_{\PP_X V},q)$ is the object $K \in \cA\big(\frac{n-r'}{2},\frac{n+r'}{2}\big) \subseteq \Dperf(\PP_X V)$ represented by the complex
\[
\Lambda^r\mathcal{V}(-r)\to \cdots \to \Lambda^1\mathcal{V}(-1) 
\]
concentrated in degrees $[1,r]$, with the Poincaré form $\beta \in \QF^{\sym}_{\det\mathcal{V}[r+1](-r-1)}(K)$ given by the collection of wedge pairings
\[
\Lambda^i \mathcal{V}(-i) \otimes \Lambda^{r+1-i}\mathcal{V}(-r-1+i) \to \Lambda^{r+1}\mathcal{V}(-r-1).
\]
Note that the subcomplex of $K$ given by
\[
\Lambda^{\frac{r-1}{2}}\mathcal{V}\big(-\frac{r-1}{2}\big)\to \cdots \to \Lambda^1\mathcal{V}(-1) 
\]
lives in $\Lag$ and so it forms canonically a surgery datum (since $\Lag$ is isotropic).
The result of the surgery is exactly $\Lambda^{\frac{r+1}{2}}\mathcal{V}[\frac{r+1}{2}]\big(\frac{-r-1}{2}\big) = \Lambda^{\frac{r+1}{2}}\mathcal{V}[\frac{r+1}{2}]\big(\frac{n}{2}\big)$ with bilinear form given by the wedge pairing
\[
\Lambda^{\frac{r+1}{2}}\mathcal{V}[\frac{r+1}{2}]\big(\frac{n}{2}\big)\otimes \Lambda^{\frac{r+1}{2}}\mathcal{V}[\frac{r+1}{2}]\big(\frac{n}{2}\big)\to \Lambda^{r+1}\mathcal{V}[r+1](n).
\]
that is, it's the pullback of $e(V)$ along $p^*(-) \otimes \cO\big(\frac{n}{2}\big)$, as required.
\end{proof}

\begin{corollary}
Suppose $n+r$ is odd. If either $r$ is even or $r$ is odd and $e(V)=0$ then for every $L \in \Picspace(X)^{\BC}$ the Poincaré functor
\[
(\Dperf(\PP_XV), \QF^{\sym}_{p^*L(n)}) \xrightarrow{p_*((-) \otimes \cO(-\frac{n+r+1}{2}))} (\Dperf(X),\QF^{\sym}_{L \otimes \det V^{\vee}[-r]})
\]
admits a Poincaré section. In particular, in these cases the fibre sequence in 
Items~\ref{equation:pbf-II} and~\ref{equation:pbf-I} of Theorem~\ref{theorem:projective-bundle-formula} split.
\end{corollary}
\begin{proof}
Assume that $n+r$ is odd.
We claim that if either $r$ is even or $r$ is odd and $e(V)=0$ then $1 \in \L^{\sym}_0(X)$ belongs to the image of the map
\[
\L^{\sym}_0(\PP_XV, \det\mathcal{V}(n)[r]) \xrightarrow{p_*((-) \otimes \cO(\frac{n+r+1}{2}))} \L^{\sym}_0(X) \ .
\]
Indeed when $r$ is even this map is an equivalence by Theorem~\ref{theorem:projective-bundle-formula}\ref{equation:pbf-II} and when $r$ is odd the result follows from Lemma~\ref{lemma:pbf-boundary}. This means that there exists a Poincaré object $(x,q)$ in $(\Dperf(\PP_X V),\QF^{\sym}_{\det\mathcal{V}(n)[r]})$ whose image under $p_*((-) \otimes \cO(-\frac{n+r+1}{2}))$ is cobordant to $(\cO_x,\id)$ (where $\id \in \Om^{\infty}\QF^{\sym}_X(\cO_X)$ is the identity form on $\cO_X$). Now since $p_*((-) \otimes \cO(-\frac{n+r+1}{2}))$ is a split Poincaré-Verdier projection~\cite{9-authors-II}*{Theorem 2.5.3} tells that the induced functor on cobordism categories is bicartesian fibration, and hence we can lift any cobordism between $p_*(x \otimes \cO(-\frac{n+r+1}{2}))$ (with its induced form) and $(\cO_X,\id)$ to a cobordism between $(x,q)$ and some Poincaré object $(x',q')$ in $\big(\Dperf(\PP_X V),\QF^{\sym}_{\det\mathcal{V}(n)[r]}\big)$ lying above $(\cO_X,\id)$. In particular, $p_*(x' \otimes \cO(-\frac{n+r+1}{2})) \simeq \cO_X$.

We are now going to use the Poincaré object $(x',q')$ in order to construct the desired Poincaré section. For this, note that the symmetric monoidal structure maps of the top horizontal functor in~\eqref{equation:functoriality-schemes-1} determine in particular a Poincaré functor
\[
(\Dperf(\PP_X V),\QF^{\sym}_{\det\mathcal{V}(n)[r]}) \otimes_{(\Dperf(X),\QF^{\sym}_X)} \big(\Dperf(X),\QF^{\sym}_{L \otimes \det V^{\vee}[-r]}\big) \to (\Dperf(\PP_X V),\QF^{\sym}_{p^*L(n)})
\]
which is an equivalence on underlying stable $\infty$-categories with duality. In particular, the Poincaré object $(x',q')$ in $(\Dperf(\PP_X V),\QF^{\sym}_{\det\mathcal{V}(n)[r]})$ determines a $(\Dperf(X),\QF^{\sym}_X)$-linear Poincaré functor
\[
F_{(x',q')}\colon \big(\Dperf(X),\QF^{\sym}_{L \otimes \det V^{\vee}[-r]}\big) \to \big(\Dperf(\PP_X V),\QF^{\sym}_{p^*L(n)}\big) ,
\]
whose underlying exact functor is $F_{x'}(-) = p^*(-) \otimes x'$. Consider the composite Poincaré functor
\[
G_{(x',q')}\colon \big(\Dperf(X),\QF^{\sym}_{L \otimes \det V^{\vee}[-r]}\big) \xrightarrow{F_{(x',q')}} \big(\Dperf(\PP_X V),\QF^{\sym}_{p^*L(n)}\big) \xrightarrow{p_*((-) \otimes \cO(-\frac{n+r+1}{2})} \big(\Dperf(X),\QF^{\sym}_{L \otimes \det V^{\vee}[-r]}\big).
\]
On the level of underlying exact functors this composite is given by 
\[
z \mapsto p_*\big(p^*(z) \otimes x' \otimes \cO\big(-\frac{n+r+1}{2}\big)\big) = z \otimes p_*\big(x'\otimes \cO\big(-\frac{n+r+1}{2}\big)\big) = z,
\]
where the first equivalence is by the projection formula and the second by the construction of $x'$.  We conclude that $G_{(x',q')}$ is an equivalence on underlying stable $\infty$-categories, and hence on the level of underlying stable $\infty$-categories with duality. Since $\QF^{\sym}_{L \otimes \det V^{\vee}[-r]}$ is the symmetric Poincaré structure associated to the underlying duality we conclude that $G_{(x',q')}$ is an equivalence of Poincaré $\infty$-categories. In fact, one can show that $G_{(x',q')}$ is homotopic to the identity Poincaré functor on $\big(\Dperf(X),\QF^{\sym}_{L \otimes \det V^{\vee}[-r]}\big)$, so that $F_{(x',q')}$ gives the desired section, though we will not have to prove this: even if $G_{(x',q')}$ is merely some equivalence, we get a section by taking $F_{(x',q')} \circ G_{(x',q')}^{-1}$.
\end{proof}

\subsection{Genuine refinement for the projective line}
\label{subsection:genuine-projective}%

In~\S\ref{subsection:pbf} we have proven a general projective bundle formula for the symmetric Poincaré structure. In this section, we want to address the analogous question for the genuine Poincaré structures $\QF^{\ge m}$. For simplicity, and since it is the only case we use, we only consider the case $V=\cO^{\oplus 1}$ and $m=0$, so that $\PP_X V=X\times\Pone$ and $r=1$. 

\begin{construction}
\label{construction:projective-line}%
Consider the sequence of $S$-linear $\infty$-categories with duality 
\[
(\Dperf(X),\Dual_L)\xrightarrow{q^*} (\Dperf(\Pone_X),\Dual_{q^*L}) \xrightarrow{q_*(- \otimes \cO(-1))} (\Dperf(X),\Dual_{L[-1]})
\]
underlying the split Poincaré-Verdier sequence of Proposition~\ref{proposition:proj-bundle-I}. We refine it to a sequence in $\Mod_S(\Catpst)$ as follows.
For the left and middle $S$-linear $\infty$-categories, we use the standard $\tstruc$-structure on their Ind-completions $\Derqc(\Pone_S)$ and $\Derqc(S)$, with respect to which the pullback functor $q^*$ is right $\tstruc$-exact. As for the right arrow, we note that while the induced functor
\[
q_*(- \otimes \cO(-1)) \colon \Derqc(\Pone_X) \to \Derqc(X)
\]
on the level of Ind-completions is not right $\tstruc$-exact with respect to the corresponding standard t-structures, it does become so if we endow the target with the $(-1)$-shift of the standard $\tstruc$-structure (that is, the $\tstruc$-structure in which the connective objects are those which are $(-1)$-connective in the standard $\tstruc$-structure). 
Using the shifted $\tstruc$-structure on right most term we may then view the above sequence as a sequence in $\Mod_S(\Catpst)$. By Example~\ref{example:split-is-flat}, the resulting sequence is a bounded Karoubi sequence. Applying the composed functor $\Catpst \to \Catpt \to \Catp_{t\geq m}$ (where the latter is the truncation functor of Proposition~\ref{proposition:genuine}) for a given $m \in \ZZ$ we then obtain a sequence of Poincaré $\infty$-categories
\[
(\Dperf(X),\QF^{\geq m}_L)\xrightarrow{p^*} (\Dperf(\Pone \times X),\QF^{\geq m}_{p^*L})\xrightarrow{p_*(-\otimes\cO(-1))} (\Dperf(X),\QF^{\geq m-1}_{L[-1]}).
\]
To avoid confusion, let us point out that on the right most term, we have $(-)^{\geq m-1}$ instead of $(-)^{\geq m}$ because the truncation is taken with respect to the shifted $\tstruc$-structure on $\Dperf(X)$.
\end{construction}

Our goal in this subsection is to prove the following refinement of this result:
\begin{proposition}
\label{proposition:beilinson}%
For every $m \in \ZZ \cup \{-\infty,\infty\}$ the sequence
\[
(\Dperf(X),\QF^{\geq m}_L)\xrightarrow{p^*} (\Dperf(\Pone \times X),\QF^{\geq m}_{p^*L})\xrightarrow{p_*(-\otimes\cO(-1))} (\Dperf(X),\QF^{\geq m-1}_{L[-1]}).
\]
of Construction~\ref{construction:projective-line} is a split Poincaré-Verdier sequence.
\end{proposition}

Before going into the proof of Proposition~\ref{proposition:beilinson}, we note that the underlying split Verdier sequence induces for every $M\in\Der^{\qc}(X\times\Pone)$ a fibre sequence of the form
\begin{equation}
\label{equation:beilinson}%
p^*p_*M\to M\to p^*p_*(M(-1)) \otimes \Sig\cO(-1) \ ,
\end{equation}
also known as the Beilinson sequence. The proof of Proposition~\ref{proposition:beilinson} will then be based on the following lemma:

\begin{lemma}
\label{lemma:truncation}%
\phantom{}
\begin{enumerate}
\item
\label{item:beilinson-sequence-p}%
There are natural equivalences $p_*\LT_{p^*L} \simeq \LT_{L}$ and $p_*\LT_{p^*L}(-1) \simeq \Sig^{-1}\LT_{L}$. 
In particular, the Beilinson sequence~\eqref{equation:beilinson} for $M=\LT_{p^*L}$ has the form
\[
p^*\LT_{L} \to \LT_{p^*L} \to p^*\LT_{L} \otimes\cO(-1).
\]
\item
\label{item:sequence-zariski-splits}%
The above exact sequence splits Zariski locally on  $\Pone \times X$. 
\item
\label{item:connective-cover-exact}%
For every $m \in \ZZ$ the associated sequence of $m$-connective covers
\[
\tau_{\geq m}(p^*\LT_{L}) \to \tau_{\geq m}(\LT_{p^*L}) \to \tau_{\geq m}(p^*\LT_{L} \otimes\cO(-1))
\]
remains exact in $\Der^{\qc}(\Pone \times X)$.
\end{enumerate}
\end{lemma}
\begin{proof}
We first prove \ref{item:beilinson-sequence-p}.
On the level of the underlying split Verdier sequence, the Verdier inclusion $p^*$ has a right adjoint $p_*$ and a left adjoint $p_*((-) \otimes \Sig\cO(-2))$, and the Verdier projection $p_*((-) \otimes \cO(-1))$ has a left adjoint $p^*(-) \otimes \cO(1)$ and a right adjoint $p^*(-) \otimes \Sig \cO(-1)$. For $M \in \Der^{\qc}(X)$ we consequently have
\begin{align*} 
\map_X(M,p_*\LT_{p^*L}) &= \map_{\Pone \times X}(p^*M,\LT_{p^*L}) \\
&= \map_{\Pone \times X}(p^*M \otimes p^*M, p^*L)^{\tC} \\
&= \map_{X}(M \otimes M, L)^{\tC} \\
&= \map_X(M,\LT_L) \ ,
\end{align*}
and
\begin{align*} 
\map_X(M,p_*\LT_{p^*L}(-1)) &= \map_{\Pone \times X}(p^*M \otimes \cO(1),\LT_{p^*L}) \\
&= \map_{\Pone \times X}(p^*M \otimes p^*M, p^*L \otimes \cO(-2))^{\tC} \\
&= \map_{X}(M \otimes M, L \otimes p_*\cO(-2))^{\tC} \\
&= \map_X(M,\Sig^{-1}\LT_{L}) \ .
\end{align*}
This establishes \ref{item:beilinson-sequence-p}.

The statement of \ref{item:sequence-zariski-splits}, being local in $X$, we may as well assume that $X = \spec(R)$ is affine. We will prove that 
the sequence in question splits when restricted to $\Aone_R = \Aone \times R$ for any embedding $j\colon \Aone_R\hrar \Pone_R$ over $\spec(R)$. Since $\Pone_R$ can be covered by (two) such copies of $\Aone_R$, this implies the desired result. Now, for an open embedding $j\colon \Aone_R \hrar \Pone_R$, it follows from Corollary~\ref{corollary:open-is-karoubi} (applied to the $\geq -\infty = \sym$ decoration) and the fact that left Kan extensions commute with taking linear parts that $\LT_{j^*p^*L} = j^*\LT_{p^*L}$. Wring $f = pj \colon \Aone_R \to \spec(R)$, we are reduced to showing that the comparison map
\[
f^*\LT_L \to \LT_{f^*T}
\]
admits a retraction in $\Der^{\qc}(\Aone_R)$. Making this map explicit, we note that $L$ is an invertible $R$-module with an $R$-linear $\Ct$-action and $\LT_L$ is the object of $\Der(R)$ corresponding to the $\GEM(R)$-module spectrum $\GEM(L)^{\tC}$, with $\GEM(R)$ acting via the Tate Frobenius $\GEM(R) \to \GEM(R)^{\tC}$, the $\Ct$ action on $\GEM(R)$ being the trivial one. On the level of $\GEM(R[t])$-module spectra, the above comparison map can then be written as the $\GEM(R[t])$-linear map
\begin{equation}
\label{equation:comparison}%
\GEM(L)^{\tC} \otimes_{\GEM(R)} \GEM(R[t])  \to \GEM(L \otimes_R R[t])^{\tC} .
\end{equation}
induced by the $\GEM(R)$-linear map $\GEM(L)^{\tC} \to  \GEM(L \otimes_R R[t])^{\tC}$.
We want to construct a retraction of~\eqref{equation:comparison}. For this we first reduce to the case where $L=R$ with constant $\Ct$-action. 
Consider the commutative square of $\Einf$-ring spectra
\[
\begin{tikzcd}
\GEM(R) \ar[r]\ar[d] & \GEM(R[t]) \ar[d] \\
\GEM(R)^{\tC} \ar[r] & \GEM(R[t])^{\tC}
\end{tikzcd}
\]
where the vertical maps are the corresponding Tate Frobenius maps. 
This square determines a map of $\Einf$-ring spectra
\begin{equation}
\label{equation:comparison-2}%
\GEM(R)^{\tC} \otimes_{\GEM(R)} \GEM(R[t])  \to \GEM(R[t])^{\tC} .
\end{equation}
At the same time, the map~\eqref{equation:comparison} factors as a composite
\[
\GEM(L)^{\tC} \otimes_{\GEM(R)} \GEM(R[t]) \to (\GEM L)^{\tC} \otimes_{\GEM(R)^{\tC}} \GEM(R[t])^{\tC}  \to \GEM(L \otimes_R R[t])^{\tC} ,
\]
where the first map is induced by the various Tate Frobenius maps and the second map by the lax monoidal structure of $(-)^{\tC}$. We then observe that the second map is an equivalence: indeed, since the Tate construction commutes with infinite direct sums of uniformly truncated objects we have that $\GEM(R[t])^{\tC}$ is an infinite direct sum of copies of $\GEM(R)^{\tC}$ indexed by the monomials $t^i$, and $\GEM(L \otimes_R R[t])^{\tC}$ is an infinite direct sum of copies of $\GEM(L)^{\tC}$ indexed in a corresponding manner. 
We then see that the map~\eqref{equation:comparison} is obtained from~\eqref{equation:comparison-2} by tensoring with $(\GEM L)^{\tC}$ over $\GEM(R)^{\tC}$. 
It will hence suffice to show that the map~\eqref{equation:comparison-2} admits a $[\GEM(R)^{\tC} \otimes_{\GEM(R)} \GEM(R[t])]$-linear retraction.

Explicitly, $(\GEM R)^{\tC} \otimes_{\GEM R} \GEM(R[t])$ is an infinite direct sum of copies of $\GEM(R)^{\tC}$, indexed by the monomials $t^i$, where the multiplication is determines by the $\Einf$-ring structure of $\GEM(R)^{\tC}$ and the rule $t^it^j = t^{i+j}$. As above, since $(-)^{\tC}$ commutes with uniformly truncated infinite direct sums, the underlying spectrum of $\GEM(R[t])^{\tC}$ has the exact same form, and the exact same multiplication law. The map~\eqref{equation:comparison-2} hence corresponds to the endomorphism
\begin{equation}
\label{equation:comparison-3}%
\bigoplus_{i \geq 0} \GEM(R)^{\tC}\<t^i\> \to \bigoplus_{i \geq 0} \GEM(R)^{\tC}\<t^i\> 
\end{equation}
induced by the Tate Frobenius of $\GEM(R[t])$. We claim that this map identifies the $\GEM(R)^{\tC}\<t^i\>$ factor on the left with the $\GEM(R)^{\tC}\<t^{2i}\>$ factor on the right. Indeed, since this map is $\GEM(R)^{\tC}$ linear it will suffice to check that it sends the component of $\GEM(R)^{\tC}\<t^i\>$ corresponding to the unit of $\GEM(R)^{\tC}$ to the component of $\GEM(R)^{\tC}\<t^{2i}\>$ corresponding to the same unit. Now since the $\Ct$-action on $R$ is trivial we have that $\pi_0\GEM(R[t])^{\tC} = \hat{H}^0(R[t]) = R[t]/2 = (R/2)[t]$, and the Tate Frobenius is the map 
\[
R[t] \to (R/2)[t] \quad\quad x \mapsto [x^2]
\]
on the level of $\pi_0$. In particular, it sends $t^i \in R[t]$ to $t^{2i} \in (R/2)[t]$, as needed. We conclude that, as an $\bigoplus_{i \geq 0} \GEM(R)^{\tC}\<t^i\>$-module, the right hand side of~\eqref{equation:comparison-3} splits as a direct sum of two copies of $\bigoplus_{i \geq 0} \GEM(R)^{\tC}\<t^i\>$ (one spanned by the even powers of $t$ and one by the odd powers), such that the map~\eqref{equation:comparison-2} is the inclusion of the even summand. It follows that this map admits an $\bigoplus_{i \geq 0} \GEM(R)^{\tC}\<t^i\>$-linear retraction, as desired.

Finally, to prove \ref{item:connective-cover-exact}, we note that for a sequence of quasi-coherent sheaves $M \to N \to P$ the comparison map $\eta_m\colon \cof[\tau_{\geq m}M \to \tau_{\geq m}N] \to \tau_{\geq m}P$ is an isomorphism on all homotopy sheaves $\pi_i$ for $i \neq m$ by the 5-lemma, while the induced map on $\pi_m$ is injective. 
But if the sequence splits Zariski locally, then the map $N \to P$ is surjective on homotopy sheaves and hence the map $\eta_m$ is an equivalence. 
\end{proof}

\begin{proof}[Proof of Proposition~\ref{proposition:beilinson}]
To check that the Poincaré functor
\[
(p^*,\eta^{\geq m})\colon (\Der^{\qc}(X),\QF^{\geq m}_L) \to (\Der^{\qc}(\Pone \times X),\QF^{\geq m}_{p^*L}) \ ,
\]
is a Poincaré-Verdier inclusion one needs to check that the Poincaré structure map
\[
\eta^{\geq m}\colon \QF^{\geq m}_{L} \Rightarrow \QF^{\geq m}_{p^* L}(p^*(-))
\]
is an equivalence. This map is already known to be an equivalence on bilinear parts (since these are the same as those of the corresponding symmetric Poincaré structures), and hence it will suffice to verify that $\eta^{\geq m}$ is an equivalence on linear parts. Unwinding the definitions and using that $p^*$ is t-exact this amounts to the composed map
\[
\tau_{\geq m}(\LT_L) \tosimeq p_*p^*\tau_{\geq m}(\LT_{L}) \tosimeq p_*\tau_{\geq m}(p^*\LT_{L}) \to p_*\tau_{\geq m}(\LT_{p^*L})
\]
being an equivalence in $\Der^{\qc}(X)$, and hence by Lemma~\ref{lemma:truncation} to the vanishing of $p_*\tau_{\geq m}(p^*\LT_{L} \otimes\cO(-1))$. To see that this last term indeed vanishes, note that the natural map $\tau_{\geq m}(p^*\LT_{L}) \otimes \cO(-1) \to \tau_{\geq m}(p^*\LT_L \otimes \cO(-1))$ is an equivalence since it is an equivalence Zariski locally, and hence the push-forward of the latter quasi-coherent sheaf to $X$ vanishes.

Now on the projection side, since $p_*((-) \otimes \cO(-1))$ is a split Verdier projection, it automatically refines to a split Poincaré-Verdier projection
\[
(\Dperf(\Pone \times X),\QF^{\geq m}_{p^*L}) \to (\Dperf(X),\QF^{\geq m}_{p^*L},\QFD)
\]
where $\QFD$ is obtained by precomposing $\QF^{\geq m}_{p^*L}$ with the left adjoint of $p_*((-) \otimes \cO(-1))$, which 
is given by $p^*(M) \otimes \cO(1)$. The natural transformation $\QF^{\geq m}_{p^*L} \Rightarrow \QF^{\sym}_{p^*L}$ then induces a natural transformation 
\[
\tau\colon \QFD \Rightarrow\QF^{\sym}_{L[-1]} \ ,
\]
which is an equivalence on bilinear parts. It will hence suffice to check that $\tau$ is an $(m-1)$-connective cover on linear parts (with respect to the standard t-structure; equivalently, an $m$-connective cover with respect to the shifted t-structure). Now, by Lemma~\ref{lemma:truncation}, we get that the map 
\[
\tau_{\geq m}(\LT_{p^*L}) \to \tau_{\geq m}(p^*\LT_{L} \otimes\cO(-1)) = \tau_{\geq m}(p^*\LT_L) \otimes \cO(-1)
\]
induces an equivalence on mapping spectra with domain of the form $p^*M \otimes \cO(1)$ for $M \in \Dperf(X)$ (indeed, the fibre of this map is the constant sheaf $\tau_{\geq m}p^*\LT_L = p^*\tau_{\geq m}\LT_L$). We hence get that
\begin{align*}
\Lam_{\QFD}(M) &= \map(p^*M\otimes\cO(1),\tau_{\geq m}\LT_{p^*L}) \\
&= \map(p^*M \otimes \cO(1),\tau_{\geq m}(p^*\LT_L) \otimes \cO(-1)) \\
&= \map(M,p_*(\tau_{\geq m}p^*\LT_{L} \otimes \cO(-2))),\\
&= \map(M,p_*(p^*\tau_{\geq m}\LT_{L} \otimes \cO(-2))),\\
&= \map(M,\Sig^{-1}\tau_{\geq m}\LT_{L}),
\end{align*}
and the map $\Sig^{-1}\tau_{\geq m}\LT_{L} \to \Sig^{-1} \LT_{L}$ is indeed an $(m-1)$-connective cover, as desired.
\end{proof}

\subsection{Homotopy invariance of symmetric Grothendieck-Witt}
\label{subsection:Aone-invariance}%

As an application of the projective bundle formula (Theorem~\ref{theorem:projective-bundle-formula}) and dévissage (Theorem~\ref{theorem:global-devissage}), we show here that the symmetric Grothendieck-Witt spectrum is $\Aone$-invariant over a regular Noetherian base of finite Krull dimension. Here, it does not matter if we consider Grothendieck-Witt or Karoubi-Grothendieck-Witt theory: when $X$ is regular the canonical map $\GW^{\sym}(X) \to \KGW^{\sym}(X)$ is an equivalence. Indeed, by Karoubi cofinality \eqref{equation:cofinality}, this follows from the map $\K(X) \to \KK(X)$ being an equivalence, which is a well-known consequence of $X$ being regular, see~\cite{thomason-trobaugh}*{Proposition~6.8}, or in more modern language of \cite{theorem-of-the-heart}*{Theorem~1.2}.

\begin{theorem}[$\Aone$-invariance]
\label{theorem:A1-invariance}%
Let $X$ be a regular Noetherian scheme of finite Krull dimension and $L$ a line bundle on $X$. Then the pullback map
\[
\begin{tikzcd}
\GW^{\sym}(X,L)\ar[d,equal]\ar[r] & \GW^{\sym}(X\times\Aone,q^*L) \ar[d,equal] \\
\KGW^{\sym}(X)\ar[r] & \KGW^{\sym}(X\times\Aone,q^*L) 
\end{tikzcd}
\]
is an equivalence, where $q\colon X \times \Aone \to X$ is the projections. 
\end{theorem}
\begin{proof}    
Write $p\colon X \times \Pone \to X$ for the associated projection. 
Consider the commutative diagram of Poincaré $\infty$-categories
\[
\begin{tikzcd}
 & (\Dperf(X \times \Aone),\QF^{\sym}_{q^*L}) & & \\
(\Dperf(X),\QF^{\sym}_L)\ar[r,hook,"p^*"]\ar[ur, "q^*"] & (\Dperf(X\times \Pone),\QF^{\sym}_{p^*L})\ar[rr,"p_*(-\otimes\cO(-1))"]\ar[u] && (\Dperf(X),\QF^{\sym}_{L[-1]})\\
 & (\Dperf_{X\times\{\infty\}}(X\times\Pone),\QF^{\sym}_{p^*L})\ar[u,hook] \ar[urr,"\vphi"'] &&  
\end{tikzcd}
\]
where the horizontal sequence is the fibre sequence of Theorem~\ref{theorem:projective-bundle-formula}\ref{equation:pbf-I} for $V=\cO^{\oplus 2}$, and the vertical sequence is the localisation sequence of Corollary~\ref{corollary:open-is-karoubi} associated to the closed subscheme inclusion $X\times\{\infty\}\subseteq X\times\Pone$. Since both the vertical and horizontal sequences are Karoubi-Poincaré sequence, we see that $q^*$ induces an equivalence on $\KGW$ if and only if $\vphi$ induces an equivalence on $\KGW$; indeed, both statements are equivalent to the map
\[
\KGW(\Dperf(X\times \Pone),\QF^{\sym}_{p^*L}) \to \KGW(\Dperf(X\times \Aone),\QF^{\sym}_{q^*L}) \oplus   \KGW(\Dperf(X),\QF^{\sym}_{L[-1]})
\]
induced by the right horizontal and top vertical arrows being an equivalence. Now to show that $\vphi$ induces an equivalence on $\KGW$ it will suffice to show that the composite Poincaré functor 
\[
(\Dperf(X),\QF^{\sym}_{i^!p^*L}) \xrightarrow{i_*} (\Dperf_{X\times\{\infty\}}(X\times\Pone),\QF^{\sym}_{p^*L}) \xrightarrow{\vphi} 
(\Dperf(X),\Sigma^{-1}\QF^{\sym}_L)
\]
is an equivalence, where the functor on the left is the dévissage Poincaré functor constructed in \S\ref{subsection:global-devissage}, and which is an equivalence by Theorem~\ref{theorem:global-devissage}. Indeed, since $\vphi$ is a composite 
\[
\big(\Dperf_{X\times\{\infty\}}(X\times\Pone),\QF^{\sym}_{p^*L}\big) \subseteq \big(\Dperf(X\times \Pone),\QF^{\sym}_{p^*L}\big) \xrightarrow{p_*((-) \otimes \cO(-1))} (\Dperf(X),\QF^{\sym}_{L})
\]
and hence it will suffice to show that the composite
\[
(\Dperf(X),\QF^{\sym}_{i^!p^*L}) \xrightarrow{i_*} \big(\Dperf(X\times \Pone),\QF^{\sym}_{p^*L}\big) \xrightarrow{p_*((-) \otimes \cO(-1))} (\Dperf(X),\QF^{\sym}_{L})
\]
is an equivalence. Indeed, this now follows from Lemma~\ref{lemma:composite-pf}.
\end{proof} 

Let us finally build on a classical argument to show that on regular Noetherian schemes, outside of the symmetric case, neither $\L$-theory nor $\GW$-theory is homotopy invariant.

\begin{example}
\label{example:non-homotopy-invariance}%
For any $m \in \ZZ \cup \{+\infty\}$, the pull-back map $\L(\FF_2,\QF^{\geq m}) \to \L(\FF_2[t],\QF^{\geq m})$ is not an equivalence. 
\end{example}
\begin{proof}
Let us start by the case $m=+\infty$, i.e.\ the quadratic case.
The quadratic $\L$-groups of rings coincide in non-negative degrees with Ranicki's periodic $\L$-groups, by \cite{9-authors-I}*{Example~2.3.17}. The $\L_0$-group is therefore the classical Witt group of non-degenerate quadratic forms over $R$.

For any scheme $X$, consider the sections $s_0,s_1: X \to \Aone \times X$ at $0$ and $1$ of the projection $p:\Aone \times X \to X$. We have $s_0^*p^*=\id=s_1^*p^*$ on $\L$-groups, so if $p^*$ is an isomorphism, then $s_0^*=s_1^*$, both being inverse to $p^*$. We now show this cannot happen over $X=\spec(\FF_2)$. The class of the non-singular quadratic form $q$ of rank $2$ given by $q(x,y)=x^2+xy+ty^2$ is an element of $\L_0(\FF_2[t],\QF^\qdr)$, and it pulls back via $s_i$ to the class of $q_i(x,y)=x^2+xy+ixy$ in $\L_0(\FF_2,\QF^\qdr)$, whose Arf invariant is $i \in \FF_2$. Thus $s_0^*\neq s_1^*$ on $\L_0(-,\QF^\qdr)$.
Actually, by \cite{9-authors-II}*{4.5.6}, the quadratic $\L$-groups are 4-periodic by a natural isomorphism commuting to these pull-backs, so $s_0^* \neq s_1^*$ on $\L_{4n}(-,\QF^\qdr)$ for any $n\in \ZZ$.

When $m \in \ZZ$, for $n<<0$, we have $\L_n(R,\QF^{\geq m})\simeq \L_n(R,\QF^\qdr)$ via the natural map, by \cite{9-authors-I}*{Corollary~1.2.12}. The quadratic case thus implies that $s_0^*\neq s_1^*$ on $\L_n(-,\QF^{\geq m})$ for some $n<<0$.
\end{proof}

In fact, we suspect that forcing homotopy invariance on the quadratic $\GW$ or any of the truncated versions might turn it into the symmetric $\GW$.

\section{Motivic realization of localising invariants}
\label{section:motivic-realization}%

Our goal in this section is to construct, for a qcqs scheme $S$, a lax symmetric monoidal functor
\[
\M^{\sym}_{S}\colon \Fun^{\kloc}(\Catp,\Spa) \to \SH(S), 
\]
which takes as input a Karoubi-localising functor $\F\colon \Catp \to \Spa$ on Poincaré $\infty$-categories and returns a motivic spectrum over $S$. 
Taking $\F$ to be the Karoubi-Grothendieck-Witt functor then yields a hermitian K-theory motivic spectrum. The study of this and related motivic spectra, such as the one associated to Karoubi L-theory, will be taken on in \S\ref{section:KQ} below. The present section is, in turn, essentially dedicated to the construction of $\M^{\sym}_S$, which is somewhat involved. In fact, it is convenient to construct a slightly more flexible version of $\M^{\sym}_S$, which takes as input a Karoubi-localising functor $\F\colon \Catp \to \Spa$ and an $S$-linear stable $\infty$-category with duality $(\C,\Dual)$. We call the resulting motivic spectrum $\M^{\sym}_S(\F;(\C,\Dual))$ the \defi{motivic realization of $\F$ with coefficients in $(\C,\Dual)$}. 
Finally, in \S\ref{subsection:unique-delooping}, we prove a technical lemma which says that, in a certain precise sense, the motivic spectrum $\M^{\sym}_S(\F)$ only depends on the underlying $\Grp_{\Einf}$-valued localising functor $\Om^{\infty}\F$. This lemma is later used in \S\ref{subsection:pullback-invariance} in order to prove the pullback invariance of hermitian K-theory.

\subsection{Recollections on stabilizations and motivic spectra}
\label{subsection:recollection-stabilization}%

Let $S$ be a qcqs scheme.
Recall that we denote by $\Sm_S$ the category of smooth $S$-schemes $p\colon X \to S$. In the present section, we always consider $\Sm_S$ as endowed with the Nisnevich topology.

There are several equivalent ways to define the $\infty$-category of motivic spectra over $S$. To facilitate the manoeuvre between these, let us fix a coefficient $\infty$-category $\E$ that is presentable. In practice, $\E$ will either be the $\infty$-category of spaces, pointed spaces, $\Einf$-groups, or spectra.
Recall that a Nisnevich sheaf $\F\colon \Sm_S\op \to \E$ is said to be \defi{$\Aone$-invariant} if the pullback map $\F(S) \to \F(\Aone_S)$ is an equivalence in $\E$. We then write $\Hinv(S,\E) \subseteq \Sh^{\Nis}(S,\E)$ for the full subcategory spanned by the $\Aone$-invariant Nisnevich sheaves. When $\E = \Sps$, these are known as \defi{motivic spaces}.

If $\E$ is pointed then it is canonically tensored and cotensored over the presentable $\infty$-category $\Sps_*$ of pointed spaces, in which case $\Psh(S,\E)$ becomes tensored and cotensored over $\Psh(S,\Sps_*)$ by means of Day convolution (see~\cite[\S 2.2.6]{HA}) with respect to the fibre product monoidal structure on $\Sm_S\op$ (equivalently, the pointwise product on presheaves). 
In particular, if $X \to S$ is a smooth $S$-scheme equipped with an $S$-point $s\colon S \to X$, then the pointed $S$-scheme $(X,s)$ determines via representability a presheaf $[X,s] \in \Psh(S,\Sps_*)$ of pointed spaces, and for $\F \in \Psh(S,\Sps_*)$, the cotensor operation is given by
\[
\F^{[X,s]}(T) = \fib[\F(X \times_S T) \xrightarrow{s^*} \F(T)].
\]
When $X = \Pone_S$ and $s$ is the section at $\infty$, we also write
$\Om_{\Pone}(\F) := \F^{[\Pone,\infty]}$.
By~\cite[\S 3]{9-authors-IV}
the full subcategories $\Hinv(S,\E) \subseteq \Sh^{\Nis}(S,\E) \subseteq \Psh(S,\E)$ are both $\Psh(S,\Sps_*)$-linear as accessible localisations of $\Psh(S,\E)$. In particular, they both inherit the structure of being tensored and cotensored over $\Psh(S,\Sps_*)$, with the cotensor operation given by the same formula.

Let us now fix a pointed presentable $\infty$-category $\E$.
To construct the $\infty$-category of motivic spectra, one needs to form the \emph{stabilization} of $\Hinv(S,\E)$ with respect to the action of $[\Pone,\infty]$, that is, the initial $\Psh(S,\Sps_*)$-linear presentable $\infty$-category under $\Hinv(S,\E)$ on which $[\Pone,\infty]$ acts invertibly. Such stabilizations can generally be realized via the \emph{symmetric spectra} construction, see~\cite[Proposition 1.3.14]{stabilization}. More precisely, let $\EE_0$ be the $\infty$-operad of pointed objects. If $\C$ is any symmetric monoidal $\infty$-category and $c \in \C$ is an object, then tensoring with $c$ determines a $\C$-linear morphism $(-)\otimes c \colon \C \to \C$, through which we may view $\C$ as a $\EE_0$-algebra object in $\Mod_\C(\PrL)$ (whose underlying object is the unit of $\Mod_{\C}(\PrL)$ but whose pointed structure is not the identity). This $\EE_0$-algebra object then determines an essentially unique symmetric monoidal functor
\[
\chi_c\colon \Env(\EE_0) = \Inj^{\otimes} \to \Mod_\C(\PrL)^{\otimes}
\]
from the symmetric monoidal envelope of $\EE_0$ (see~\cite[Construction 2.2.4.1]{HA}), which can be readily identified with the symmetric monoidal category $\Inj$ of finite sets and injective maps as morphisms, with the symmetric monoidal structure given by disjoint union. 
Concretely, $\chi_c$ sends a finite set $I \in \Inj$ to 
$\C^{\otimes_{\C}^{I}} = \C$, and every inclusion $I \subseteq J$ to the functor $(-) \otimes \bigotimes_{i \in J \setminus I} c$. One then sets
\[
\Spa_{c}^{\Sig}(\C) := \colim_{\Inj} \chi_c \in \Mod_\C(\PrL) ,
\]
where the colimit is computed in $\Mod_{\C}(\PrL)$.
The underlying lax symmetric monoidal structure on $\chi_c$ then refines this colimit to an $\Einf$-algebra in $\Mod_{\C}(\PrL)$ , so that $\Spa_{c}^{\Sig}(\C)$ is canonically a $\C$-linearly symmetric monoidal $\infty$-category. 
We note that colimits in $\Mod_\C(\PrL)$ are computed in $\PrL$, and we may identify $\PrL$ with $(\PrR)\op$ by passing to right adjoints. The underlying presentable $\infty$-category of $\Spa_{c}^{\Sig}(\C)$ can then be identified with the limit over $\Inj\op$ of a diagram obtained from $\chi_c$ by replacing all arrows with their right adjoints. 
Explicitly, the objects of $\Spa_{c}^{\Sig}(\C)$ are then given by the data of an object $X_I \in \C$ for every finite set $I$, together with, for every inclusion of finite sets $I \subseteq J$, an equivalence $X_I \xrightarrow{\simeq} X_J^{\bigotimes_{J \setminus I} c}$, where $(-)^{\bigotimes_{J \setminus I} c}$ is the associated cotensor, that is, the right adjoint of $(-) \otimes \bigotimes_{i \in J \setminus I} c$, together with all the associated coherence homotopies. By~\cite[Proposition 1.3.14]{stabilization} the $\C$-linear symmetric monoidal functor 
\[
\Sig^{\infty}_c\colon \C \to \Spa_{c}^{\Sig}(\C)
\]
induced by the symmetric monoidal functor $\{\emptyset\} \to \Inj$ exhibits $\Spa_{c}^{\Sig}(\C)$ as the (symmetric monoidal) stabilization of $\C$ by $c$, that is, $c$ maps to an invertible object in $\Spa_{c}^{\Sig}(\C)$ and $\Spa_{c}^{\Sig}(\C)$ is the initial such recipient of a symmetric monoidal functor from $\C$. In addition, $\Spa_{c}^{\Sig}(\C)$ is also initial as a $\C$-linear $\infty$-category under $\C$ on which $c$ acts invertibly. More generally, if $\D$ is any $\C$-linear $\infty$-category then
\[
\Spa^{\Sig}_c(\D) := \colim_{I \in \Inj}[\chi_c(I) \otimes_{\C} \D] = \Spa^{\Sig}_c(\C) \otimes_{\C} \D \in \Mod_{\C}(\PrL)
\]
is the initial $\C$-linear $\infty$-category under $\D$ on which $c$ acts invertibly. In fact, the object $\Spa^{\Sig}_c(\C) \in \Mod_{\C}(\PrL)$ is an idempotent algebra object and tensoring with it is a smashing localisation on $\Mod_{\C}(\PrL)$ whose image consists of those $\C$-linear $\infty$-categories on which $c$ acts invertibly. When $\D$ is $\C$-linearly symmetric monoidal (that is, an $\Einf$-algebra object in $\Mod_{\C}(\PrL)$) then $\Spa^{\Sig}_c(\D)$ inherits a symmetric monoidal structure and identifies with the symmetric monoidal localisation of $\D$ by the image of $c$ via the unit functor $\C \to \D$.

The objects of $\Spa_c^{\Sig}(\D)$ are called \emph{symmetric $c$-spectra} in $\D$. To work with symmetric $c$-spectra, it is generally convenient to realize $\Spa_c^{\Sig}(\D)$ as an accessible localisation of the $\infty$-category $\PSpa_c^{\Sig}(\D)$ of symmetric $c$-prespectra (called lax $c$-spectra in \cite{stabilization}), which is obtained by replacing the colimit in the definition 
with the associated lax colimit (so that the underlying $\infty$-category is given by the associate oplax limit of right adjoints). More concretely, the objects of $\PSpa_{c}^{\Sig}(\D)$ are given by the data of an object $X_I \in \D$ for every finite set $I$, and for every inclusion of finite sets $i\colon I \hrar J$ a map $s_{i}\colon X_I \to X_J^{\bigotimes_{J \setminus i(I)} c}$, together with all the associated coherence homotopies. The $c$-spectra are then included as those $c$-prespectra for which all the structure maps $s_i$ are equivalences. The inclusion
\[
\Spa^{\Sig}_c(\D) \subseteq \PSpa^{\Sig}_c(\D)
\]
admits a left adjoint, called the spectrification functor. In full generality, the spectrification functor is unfortunately not very explicit. 
Nonetheless, we can describe its value on $c$-spectra of a certain special type via the sequential stabilization formula familiar from $S^1$-spectra. For this, note that the diagram $\chi_c$ above is shift-invariant. 
More precisely, since $\chi_c$ is symmetric monoidal we have $\chi_c((-) \amalg I) \simeq \chi_c(-) \otimes \chi_c(I) \simeq \chi_c(-)$ since $\chi_c(I) = \C$ is the unit of $\Mod_{\C}$.
This equivalence then determines an equivalence on the level of lax colimits in $\Mod_{\C}(\PrL)$, and so we obtain an associated $I$-shift functor 
\[
\sh_I\colon \PSpa^{\Sig}_c(\D) \to \PSpa^{\Sig}_c(\D),
\]
given informally by the formula $\sh_I(X)_J = X_{I \amalg J}$. For a $c$-prespectrum $X$ there are natural maps 
\[
X \to \sh_I X^{c^{\otimes I}},
\]
where the cotensor by $c^{\otimes I}$ is applied levelwise (using the symmetric monoidal structure on $\C$ to commute the $c$-cotensor operation with itself), and $X$ is a $c$-spectrum if and only if this map is an equivalence for every $I$ (equivalently, for $I=\{\ast\}$). In fact, when restricted to $\Spa^{\Sig}_c(\D)$, the functors $\sh_I$ and $(-)^{c^{\otimes I}}$ are inverse equivalences, so that the action of tensoring with $c ^{\otimes I}$ on $\Spa^{\Sig}_c(\D)$ is given by $\sh_I$.
Now for $n \geq 0$ we write $\sh_n$ for $\sh_{\{1,...,n\}}$, so that for a $c$-prespectrum $X$ we have natural maps $X \to \sh_nX^{c^{\otimes n}}$ fitting into a sequence
\[
X \to \sh_1X^c \to \sh_2X^{c \otimes c} \to ... \to \sh_n X^{c^{\otimes n}} \to ...,
\]
and we write $X \to \sh_{\infty} X^{c^{\otimes \infty}}$ for the map into the associated sequential colimit. We call $\sh_{\infty} X^{c^{\otimes \infty}}$ the \emph{sequential stabilization} of $X$.
Following Schwede~\cite{schwede}, we say that a symmetric $c$-prespectrum $X$ is \emph{semi-stable} if $\sh_{\infty} X^{c^{\otimes \infty}}$ is a $c$-spectrum. For example, for any $c$-spectrum the map $X \to \sh_{\infty} X^{c^{\otimes \infty}}$ is an equivalence and hence in particular $X$ is semi-stable as a $c$-prespectrum.
We may consequently consider the intermediate full subcategory 
\[
\Spa^{\Sig}_c(\C) \subseteq \PSpa^{\semi}_{c}(\C) \subseteq \PSpa^{\Sig}_c(\C)
\]
spanned by the semi-stable $c$-prespectra. 

\begin{proposition}
\label{proposition:sequential-stabilization}%
Suppose that the cotensor functor $(-)^{c}\colon \D \to \D$ commutes with sequential colimits. Then the functor$X \mapsto \sh_{\infty} X^{c^{\otimes \infty}}$ is left adjoint to the inclusion $\Spa^{\Sig}_c(\C) \subseteq \PSpa^{\semi}_{c}(\C)$.
\end{proposition}

Proposition~\ref{proposition:sequential-stabilization} is a special case of a more general statement, which has nothing to do specifically with spectra or pre-spectra. To formulate this, note that the association $I \mapsto \sh_I(-)^{\otimes_I c}$ determines an action of $\Inj$ on $\PSpa^{\Sig}_c(\C)$ (if we view the underlying $\infty$-category of $\PSpa^{\Sig}_c(\C)$ as a certain $\Inj$-indexed lax limit, the action of $\Inj$ is induced by the action of $\Inj$ on itself as an object in the lax slice $\Cat^{\lax}_{/\Inj}$). Evaluating at a given object $X \in \PSpa^{\Sig}_c(\C)$ then yields a functor $X_{\sharp}\colon \Inj \to \PSpa^{\Sig}_c(\C)$, and $X$ is a $c$-spectrum precisely when $X_{\sharp}$ inverts all morphisms in $\Inj$. The sequential stabilization procedure, in turn, only depends on this action of $\Inj$ by construction. We can then formulate (and prove) a generalization of Proposition~\ref{proposition:sequential-stabilization} as follows:

\begin{proposition}
\label{proposition:sequential-stabilization-general}%
Let $\E$ be an $\infty$-category with sequential colimits equipped with a sequential colimit preserving action of $\Inj$, that is, a monoidal functor
\[
\beta\colon \Inj^{\otimes} \times_{\Com^{\otimes}} \Ass^{\otimes} \to \Fun^\seq(\E,\E)^{\circ},
\]
where the target is the $\infty$-category of sequential colimit preserving endo-functors equipped with the composition monoidal structure. Say that $x \in \E$ is $\beta$-stable if $\beta_I(x) \to \beta_J(x)$ is an equivalence for every $I \hookrightarrow J$ in $\Inj$, and that $x$ is $\beta$-semistable if 
\[
F(x) := \colim_{n \geq 0} \beta_{\{1,...,n\}}(x)
\]
is $\beta$-stable. Then for every $\beta$-semistable $x$ we have that $x \to F(x)$ exhibits $F$ as the initial $\beta$-stable object under $x$. In particular, the inclusion $\E^{\stable} \subseteq \E^{\semi}$ of $\beta$-stable objects inside $\beta$-semistable objects admits a left adjoint given by $F$.
\end{proposition}
\begin{proof}
Let $\wtl{\Inj}$ be the symmetric monoidal category of countable sets and injective maps between them, so that $\Inj$ sits inside $\wtl{\Inj}$ as a full symmetric monoidal subcategory. We note that for every infinite $I \in \wtl{\Inj}$ the comma category $\Inj_{/I}$ is just the poset of finite subsets of $I$, and always contains $(\NN,>)$ as a cofinal subposet. Since $\E$ admits sequential colimits we can operadically left Kan extend the action $\beta$ to a lax monoidal functor 
\[
\wtl{\beta}\colon \wtl{\Inj}^{\otimes}\times_{\Com^{\otimes}}\Ass^{\otimes} \to \Fun^\seq(\E,\E)^{\circ}.
\]
In addition, since for every $I,J \in \wtl{\Inj}$ the map $\Inj_{/I} \times \Inj_{/J} \to \Inj_{/I \amalg J}$ sending $(I_0 \hookrightarrow I, J_0 \hookrightarrow J)$ to $I_0 \amalg J_0 \hookrightarrow I \amalg J$ is an equivalence of posets the lax monoidal functor $\wtl{\beta}$ is actually monoidal, that is, $\wtl{\beta}$ encodes an action of $\wtl{\Inj}$ on $\E$ via sequential colimits preserving functors, extending $\beta$. Since every object in $\Inj$ is compact in $\wtl{\Inj}$ we also see that $\wtl{\beta}$ preserves sequential colimits in $\wtl{\Inj}$. Finally, note that if $x$ is $\beta$-stable then it is also $\wtl{\beta}$-stable, that is, for every map $I \hookrightarrow J$ in $\wtl{\Inj}$ the map $\wtl{\beta}_I(x) \to \wtl{\beta}_J(x)$ is an equivalence; indeed, this map is the image under $\wtl{\beta}_I$ of the map $x = \wtl{\beta}_{\emptyset}(x) \to \wtl{\beta}_{J\setminus I}(x)$, and the latter is a sequential transfinite composite of inclusions of finite sets. 

By construction, the action of the set $\NN$ is given by $F$, that is, $\wtl{\beta}_{\NN}(x) = F(x)$. We now claim that the endomorphism $\wtl{\beta}_{\NN}\colon \E^{\semi} \to \E^{\semi}$ satisfies the conditions of~\cite[Proposition 5.2.7.4]{HTT} with respect to the natural transformation $\id = \wtl{\beta}_{\emptyset} \to \wtl{\beta}_{\NN}$. Unwinding the definitions, this amounts to saying that the two component inclusions $i_1,i_2\colon \NN \hrar \NN \amalg \NN$ induce both equivalences
\[
\wtl{\beta}_{\NN}(x) \to \wtl{\beta}_{\NN \amalg \NN}(x) = \wtl{\beta}_{\NN}(\wtl{\beta}_{\NN}(x))
\]
for every $x \in \E^{\semi}$. But this is simply because both $i_1,i_2$ are inclusions of infinite countable sets,
and the condition that $x$ is $\beta$-semistable implies (and in fact equivalent to saying) that $\wtl{\beta}_{I}(x) \to \wtl{\beta}_J(x)$ is an equivalence whenever $I \hookrightarrow J$ is an inclusion of infinite countable sets.
Invoking~\cite[Proposition 5.2.7.4]{HTT} we hence conclude that $\wtl{\beta}_{\NN}$ determines a left adjoint to the inclusion of $\im(\wtl{\beta}) \subseteq \E^{\semi}$ of its image. 
To finish the proof we now observe that the essential image of $\wtl{\beta}_{\NN}$ is given by the full subcategory $\E^{\stable} \subseteq \E^{\semi}$ spanned by the $\beta$-stable object; indeed, this essential image is contained there by the definition of being $\beta$-semistable, and every $\beta$-stable $x$ is in the image since $\wtl{\beta}_{\NN}(x) \simeq x$ for such $x$.
\end{proof}

\begin{definition}
\label{definition:motivic}%
Let $\E$ be a pointed presentable $\infty$-category. We define $\Spa_{\Pone}^{\Sig}(S,\E) := \Spa_{\Pone}^{\Sig}(\Sh^{\Nis}(S,\E))$ and $\SH(S,\E) = \Spa_{\Pone}^{\Sig}(\Hinv(S,\E))$ to be the stabilizations of $\Sh^{\Nis}(S,\E)$ and $\Hinv(S,\E)$ respectively, with respect to the action of $(\Pone,\infty) \in \Psh(S,\Sps_*)$. 
The fully-faithful inclusion $\Hinv(S,\E) \subseteq \Sh^{\Nis}(S,\E)$ then induces a fully-faithful inclusion 
\[
\SH(S,\E) \subseteq \Spa_{\Pone}^{\Sig}(S,\E).
\]
We refer to the objects of $\SH(S,\E)$ as \defi{$\E$-valued motivic spectra}, and to those of $\Spa_{\Pone}^{\Sig}(S,\E)$ as \defi{$\E$-valued pre-motivic spectra}. If $\E=\Spa$ then we drop the prefix ``$\E$-valued'' and simply write $\SH(S)$ and $\Spa_{\Pone}^{\Sig}(S)$. Given a finite set $I$, we then write $\Sig^I_{\Pone} := \sh_I$ for the $I$-fold shift functor on either $\Spa_{\Pone}^{\Sig}(S,\E)$ or $\SH(S,\E)$, and $\Om^I_{\Pone}$ for its inverse equivalence, given by cotensoring with $\wedge_{i \in I}[\Pone,\infty]$ pointwise. We refer to them as the ($I$-fold) $\Pone$-suspension and $\Pone$-loops functors. For $I=\{1,...,n\}$ we also write $\Sig^n_{\Pone}$ and $\Om^n_{\Pone}$. Finally, we write $\sh_{\infty}\Om^{\infty}_{\Pone} = \colim_n\sh_n\Om^{n}_{\Pone}$ for the sequential stabilization functor as above. 
\end{definition}

Closely related to (and simpler than) symmetric spectra is the construction of \emph{sequential spectra}, which can be defined in a similar fashion by replacing $\Inj$ with the poset $(\NN,\geq)$, that is,
\[
\Spa_{c}^{\NN}(\D) := \colim[\D \xrightarrow{(-) \otimes c} \D \xrightarrow{(-) \otimes c} ... ]
\]
where the colimit is calculated in $\Mod_{\C}(\PrL)$. As above, the underlying presentable $\infty$-category can be calculated as a suitable limit in $\PrR$, so that objects of $\Spa_{c}^{\NN}(\D)$ are given by sequences $X_0,X_1,...$ together with, for every $n=0,1....$ an equivalence $X_n \xrightarrow{\simeq} X_{n+1}^c$. The inclusion $\NN \to \Inj$ sending $n \in \N$ to the finite set $\{1,...,n\}$ then determines an adjunction
\begin{equation}
\label{equation:sequential-to-symmetric}%
\Spa_{c}^{\NN}(\D) \adj \Spa_{c}^{\Sig}(\D) 
\end{equation}
relating sequential and symmetric $c$-spectra, such that the left adjoint is $\C$-linear. While simpler, the construction of sequential spectra has less favourable formal properties: it does not carry any obvious symmetric monoidal structure, and on its underlying $\C$-linear $\infty$-category the object $c$ does not in general act invertibly.
Nonetheless, when $c$ is a \emph{symmetric} object, that is, the cyclic permutation on $c \otimes c \otimes c$ is homotopic to the identity, then the above adjunction is an equivalence for every $\C$-module $\D$ by~\cite[Proposition 1.6.3]{stabilization}, so that the simpler construction $\Spa_{c}^{\NN}(\D)$ also enjoys all the favourable properties of $\Spa_{c}^{\Sig}(\D)$. We note that if $\D$ itself is $\C$-linearly symmetric monoidal then it is enough that the image of $c$ in $\D$ be symmetric for the above adjunction to be an equivalence.

By~\cite[Theorem 4.3 and Lemma 4.4]{voevodsky} the image of $[\Pone,\infty]$ in $\Hinv(S,\Sps_*)$ (and hence in $\Hinv(S,\E)$ for any pointed $\E$) is symmetric. In those cases the adjunction~\eqref{equation:sequential-to-symmetric} is an equivalence and so we have
\[
\SH(S,\E) = \Spa_{\Pone}^{\Sig}(\Hinv(S,\E)) = \Spa_{\Pone}^{\NN}(\Hinv(S,\E))
\]
on the level of underlying $\infty$-categories.
We warn the reader that the image of $[\Pone,\infty]$ is not generally symmetric in $\Sh^{\Nis}(S,\E)$, so that a-priori the $\Pone$-stabilization of $\Sh^{\Nis}(S,\E)$ cannot be realized using sequential spectra. We thank Marc Hoyois for bringing this subtle point to our attention. Though the spectrification functor over Nisnevich sheaves plays a role in some of the arguments below, the relevant $\Pone$-prespectra are always semi-stable.

\begin{proposition}
\label{proposition:motivic-spectra-is-stable}%
Let $\E$ be a pointed presentable $\infty$-category. Then the $\infty$-category $\SH(S,\E)$ is stable, and the functor
\[
\SH(S,\Spa(\E)) \to \SH(S,\E)
\]
induced by $\Om^{\infty}\colon \Spa(\E) \to \E$ is an equivalence. In particular, the functors 
\[
\SH(S,\Spa) \to \SH(S,\Grp_{\Einf}) \to \SH(S,\Sps_*)
\]
are both equivalences.
\end{proposition}

We will consider any of the equivalent $\infty$-categories 
\[
\SH(S,\Spa) = \SH(S,\Grp_{\Einf}) = \SH(S,\Sps_*)
\]
as a model for the $\infty$-category of \defi{motivic spectra} over $S$.

\begin{proof}[Proof of Proposition~\ref{proposition:motivic-spectra-is-stable}]
Since $\Hinv(S,\E)$ is closed in $\Psh(S,\E)$ under limits we have that limits in $\Hinv(S,\E)$ are computed levelwise. This applies in particular to taking loop objects, and so $\SH(S,\Spa(\E)) = \Spa(\SH(S,\E))$ over $\SH(S,\E)$. All claims hence follow once we show that $\SH(S,\E)$ is stable. This, in turn, follows from the fact that in $\Hinv(S,\Sps_*)$ one has the equivalence $\Sone \wedge \Gm \simeq \Pone$. As a consequence, universally inverting the action of $\Pone$ inverts both the action of $\Sone$ -- thus stabilizing the category in the standard sense -- and $\Gm$. 
\end{proof}

The advantage of working with coefficients in $\Spa$ (or more generally in any stable $\E$) is that the localisation functor
\[
\Loc_{\Aone}\colon \Sh^{\Nis}(S,\Spa) \to \Hinv(S,\Spa)
\]
is given by the explicit formula 
\[
(\LAone\X)_n(T) = |\X_n(T \times_S \Del_S^{\bullet})|,
\]
where $\Del_S^n$ is the algebraic $n$-simplex over $S$, that is, the closed subscheme of $\Aa^n \times S$ determined by the equation $\sum_i x_i=1$. In fact, this formula always gives the reflection to the full subcategory spanned by $\Aone$-invariant presheaves, but for a general $\E$, it does not preserve the Nisnevich sheaf property. However, when $\E$ is stable the sheaf property is preserved under colimits, and hence $\Loc_{\Aone}$ sends Nisnevich sheaves to Nisnevich sheaves, which implies that its universal property persists to the sheaf context. In fact, when $\E$ is stable this formula also applies for $\Pone$-spectrum objects:

\begin{lemma}
\label{lemma:LAoneOmPone}%
Suppose that $\E$ is stable. If $F\colon \Sm_S\op \to \E$ is a Nisnevich sheaf, then the interchange map $\LAone \OmPone F \to \OmPone \LAone F$ is an equivalence. In particular, the induced functor
\[
\LAone \colon \Spa_{\Pone}^{\Sig}(S,\E) \to \SH(S,\E)
\]
on stabilizations (which we denote by the same name) is computed $\Pone$-levelwise. 
\end{lemma}
\begin{proof}
The cartesian products with $\Aa^n$ and $\Pone$ commute, and taking fibres commutes with colimits because in a stable $\infty$-category finite limits commute with arbitrary colimits.
\end{proof}

\subsection{More on derived Poincaré categories of schemes}
\label{subsection:more-derived}%

Let us now revisit the functor
\[
\qSch\op \to \CAlg(\Catx) \quad\quad X \mapsto \Dperf(X)^{\otimes}
\]
discussed in Appendix~\ref{subsection:perfect}, and from which the lax symmetric monoidal functor $X \mapsto (\Dperf(X),\Dual_X) \in \Catps$ is obtained in \S\ref{subsection:scheme-to-rigid}. Our goal in the present subsection is to show that when restricted to smooth $S$-scheme, the above functor can be made symmetric monoidal (when the target is suitably formulated), and that this is compatible with base change along maps $T \to S$ of qcqs schemes (see Points~\ref{item:forming-derived} and~\ref{item:base-change} below).

We begin with the observation that the functor $X \mapsto \Dperf(X)^{\otimes}$ above almost preserves pushouts, which on the left correspond to fibre products of schemes and on the right hand side correspond to relative tensor products of stably symmetric monoidal $\infty$-categories. More precisely, two issues arise when considering pushout preservation:
\begin{itemize}
\item
The functor $\Dperf(-)$ takes values in idempotent complete stable $\infty$-categories, but those are generally not closed under relative tensor products. 
\item
Even up to idempotent completion the functor $\Dperf(-)$ does not preserve pushouts, due to the difference between tensor products and derived tensor products. For example, if $R$ and $S$ are two $A$-algebras then the idempotent completion of $\Dperf(R) \otimes_{\Dperf(A)} \Dperf(S)$ is not $\Dperf(R \otimes_A S)$ but rather $\Dperf(R \otimes^{\h}_A S)$, where $\otimes^{\h}_A$ denotes derived tensor product over $A$. 
\end{itemize}
Let us begin by addressing the first issue. Recall that a map of stable $\infty$-categories $\C \to \D$ is called a Karoubi equivalence if it induces an equivalence on idempotent completions. 
We then write $\Catxi$ for the localisation of $\Catx$ by the collection of Karoubi equivalences. This localisation is left Bousfield (and in fact also right Bousfield, but we will not need this here). In particular, the localisation functor $\Catx \to \Catxi$ 
admits fully-faithful right adjoint whose essential image consists of the idempotent complete stable $\infty$-categories. The unit of the resulting adjunction on a given $\C$ is the inclusion $\C \to \C^{\natural}$ of $\C$ in its idempotent completion. The key point of us is that the resulting localisation adjunction
\[
\Catx \adj \Catxi \cocolon (-)^{\natural}
\]
is multiplicative, that is, the collection of Karoubi equivalences are closed under pointwise tensor product in the arrow category, so that $\Catxi$ inherits from $\Catx$ a symmetric monoidal structure such that the left and right adjoints carry compatibly a symmetric monoidal and a lax symmetric monoidal structure, respectively. 
Concretely, if we identify the localisation $\Catxi$ with the full subcategory of $\Catx$ spanned by the idempotent complete stable $\infty$-categories then the localised tensor product is computed by first taking the tensor product in $\Catx$ and then taking idempotent completion. 
This adjunction then induces an adjunction
\[
\CAlg(\Catx) \adj \CAlg(\Catxi) 
\]
on the level of commutative algebra objects, where both functors are compatible with forgetting the commutative algebra structure. 
In particular, the right adjoint is still fully-faithful, with essential image those stably symmetric monoidal $\infty$-categories whose underlying stable $\infty$-category is idempotent complete. Pushouts in $\CAlg(\Catxi)$ are then given by pushouts in $\CAlg(\Catx)$, that is, relative tensor products, followed by idempotent completion. To account for the first point above we may then replace $\CAlg(\Catx)$ with $\CAlg(\Catxi)$.

As for the second point, one possibility is to replace schemes with derived schemes, so that all products and fibre products are by definition derived, thus resolving the remaining issue, see~\cite{integral}*{Proposition 4.6}. Alternatively, pushouts with one leg flat are still preserved, since in this case the derived and underived tensor products coincide. Since in what follows we will mostly be interested in smooth schemes over a fixed base and smooth maps are flat, this latter solution is sufficient for our purposes.

Taking now into account both points above we may view $\Dperf(-)^{\otimes}$ as a functor 
\[
\qSch\op \to \CAlg(\Catxi) \quad\quad X \mapsto \Dperf(X)^{\otimes} ,
\]
which then preserves pushouts with one leg flat by~\cite{integral}*{Proposition 4.6}.
In addition, since each $\Dperf(X)^{\otimes}$ is rigid as a symmetric monoidal $\infty$-category the functor above factors through the full subcategory $\Catexri \subseteq \CAlg(\Catxi)$ spanned by the idempotent complete rigid $\infty$-categories. We note that $\Catexri$ is again a reflective subcategory of $\Catexr$,
and that $\Catexri$ is closed under pushouts in $\CAlg(\Catxi)$: indeed, if $\C,\D$ are two rigid stably symmetric monoidal $\infty$-categories under $\A$, then the pushout in $\CAlg(\Catxi)$ is given by idempotent completion of $\C \otimes_{\A} \D$, which is again rigid since it receives a canonical symmetric monoidal functor $\C \times \D \to \C \otimes_{\cA} \D$ whose image generates $\C \otimes_{\cA} \D$ under finite colimits and retracts. Restricting the target we hence obtain a functor
\begin{equation}
\label{equation:integral}%
\qSch\op \to \Catexri \quad\quad X \mapsto \Dperf(X) 
\end{equation}
which as above preserves pushouts with one leg flat.

In what follows we would like to consider the above functor on smooth schemes over a fixed base, but also keep track of what happens when the base changes. For this, we first apply the functor on the level of arrow categories to obtain a commutative square
\[
\begin{tikzcd}
\Ar(\qSch)\op \ar[d,"t\op"']\ar[r] & \Ar(\Catexri) \ar[d,"s"] \\
\qSch\op \ar[r] & \Catexri
\end{tikzcd}
\]
where the vertical arrows are the corresponding source projections (where we use the notation $t\op$ to indicate that we think of it as induced by the target projection of $\qSch$ under taking opposites), which by the existence of pushouts on both sides are cocartesian fibrations, with cocartesian transition functors given by cobase change. By the above we have that the top horizontal functor preserves cocartesian edges with flat domain, and all cocartesian arrows lying over flat maps in $\qSch\op$.
Let $\Ar^\sm(\qSch) \subseteq \Ar(\qSch)$ be the full subcategory spanned by the smooth maps $X \to S$. Since smooth maps are preserved under base change (which corresponds to cobase change in $\qSch\op$) the target projection $t\colon \Ar^\sm(\qSch) \to \qSch$ is still a cartesian fibration, and in the resulting square
\[
\begin{tikzcd}
\Ar^\sm(\qSch)\op \ar[d,"t\op"']\ar[r] & \Ar(\Catexri) \ar[d,"s"] \\
\qSch\op \ar[r] & \Catexri
\end{tikzcd}
\]
the top horizontal map now preserves cocartesian edges. Since all corners in the above square admit coproducts we may endow all corners in this square with the corresponding cocartesian symmetric monoidal structure, in which case the square uniquely refines to a lax symmetric monoidal one (see~\cite[Proposition 2.4.3.8]{HA}). In addition, the vertical source projection are coproduct preserving, and the collection of cocartesian arrows (corresponding in this case to cocartesian squares) is closed under pointwise coproduct. We thus conclude that after passing to the corresponding cocartesian symmetric monoidal $\infty$-categories the vertical maps become symmetric monoidal cocartesian fibrations and the top horizontal arrow preserves cocartesian edges.

We now combine this construction with the symmetric monoidal functor $(\Catexr)^{\otimes} \to (\Catps)^{\otimes}$ sending a rigid $\infty$-category $\C^{\otimes}$ to the underlying stable $\infty$-category with duality $(\C,\Dual_{\C})$, where $\Dual_{\C}(-) = \Dual_{\one_{\C}}$ is the internal mapping object to $\one_{\C}$ (concretely, this symmetric monoidal functor is the dashed lift in the square~\eqref{equation:functoriality-ps}). For this, let us write 
\[
\Catpsi = \Catps \times_{\Catx} \Catxi \subseteq \Catps
\]
for the full subcategory of $\Catps$ spanned by the perfect symmetric bilinear $\infty$-categories (equivalently, of stable $\infty$-categories with duality) whose underlying stable $\infty$-category is idempotent complete. Then $\Catpsi$ is again a multiplicative localisation of $\Catp$, and since the functor $\Catexr \to \Catps$ preserves Karoubi equivalences (being compatible with the forgetful functor to $\Catx$ on both sides) it induces a symmetric monoidal functor
\[
(\Catexri)^{\otimes} \to (\Catpsi)^{\otimes}
\]
on the level of localisations.
Let now $\MCom_*$ denote the $\infty$-operad encoding pairs $(A,M)$ where $A$ is a commutative algebra object and $M$ is a pointed $A$-module, that is, an $A$-module equipped with an $A$-module map $A \to M$. 
We note that $\MCom_*$ is a unital $\infty$-operad whose $\infty$-category of colors is simply $\Del^1$. Now the cocartesian structure on $\Catexri$ coincides with the one we considered in \S\ref{subsection:rigid-to-poinc}, which was restricted from the pointwise structure on $\CAlg(\Catx)$: indeed, on the latter the pointwise and cocartesian structure coincide by~\cite[Proposition 3.2.4.10]{HA}. Using~\cite[Proposition 2.4.3.9]{HA} we may now form a commutative rectangle of symmetric monoidal $\infty$-categories
\begin{equation}
\label{equation:functoriality-smooth}%
\begin{tikzcd}
\Ar^\sm(\qSch\op)^{\amalg} \ar[d,"t\op"']\ar[r] & \Ar(\Catexri)^{\otimes} \ar[d,"s"]\ar[r,equal] & \Alg_{\MCom_*}(\Catexri)^{\otimes}\ar[r]\ar[d] & \Alg_{\MCom_*}(\Catpsi)^{\otimes} \ar[d] \\
(\qSch\op)^{\amalg} \ar[r] & (\Catexri)^{\otimes} \ar[r,equal] & \CAlg(\Catexri)^{\otimes}\ar[r] & \CAlg(\Catpsi)^{\otimes}
\end{tikzcd}
\end{equation}
where all functors are symmetric monoidal except the two horizontal functors on the left most column which are only lax symmetric monoidal. In addition, all vertical maps are symmetric monoidal cocartesian fibrations and all horizontal arrows in the top row preserve cocartesian arrows. Indeed, for the left most square we showed this above, and for the right most square this follows from the fact that $\Catexri$ and $\Catpsi$ admit sifted colimits which are preserved by the functor $\Catexri \to \Catpsi$ (indeed, this is true for $\Catexr$, $\Catps$ and $\Catexr \to \Catps$ by Proposition~\ref{proposition:sm-sifted} and these properties pass to left Bousfield localisations); the cocartesian transition functors in this case are given by forming relative tensor products followed by idempotent completion. 
Fixing a qcqs scheme $S \in \qSch$ we may now pass to the induced functors on the level of fibres. Post-composing with the (symmetric monoidal) forgetful functor from pointed modules to modules we now obtain a sequence of symmetric monoidal functors
\begin{equation}
\label{equation:functoriality-schemes-3}%
\Sm_S\op \to (\Catexri)_{\Dperf(S)^{\otimes}/} \to \Mod_{(\Dperf(S),\Dual_S)}(\Catpsi) =: \Mod_S(\Catpsi) .
\end{equation}
Here, the second functor is symmetric monoidal because the horizontal arrows in the right most square of~\eqref{equation:functoriality-smooth} are symmetric monoidal (and cocartesian edge preserving), while the first functor, which is a-priori only lax symmetric monoidal, is actually symmetric monoidal because the functor $\qSch \to \Catexri$ preserves pushouts with one leg smooth. We summarize this discussion by recording the following two points:
\begin{enumerate}
\item
\label{item:forming-derived}%
For $S$ a qcqs scheme the functor $[X \to S] \mapsto (\Dperf(X),\Dual_X)$ is symmetric monoidal when considered from smooth $S$-schemes to $(\Dperf(S),\Dual_S)$-modules in $\Catpsi$.
\item
\label{item:base-change}%
If $T \to S$ is map of qcqs schemes then the (cocartesian arrow preserving) top horizontal composite in~\eqref{equation:functoriality-smooth} determines a commutative square
\[
\begin{tikzcd}[column sep=40pt]
(\Sm_S\op)^{\otimes} \ar[d]\ar[rrr, "{T \times_S (-)}"] &&& (\Sm_T\op)\op \ar[d]\\
\Mod_S(\Catpsi)^{\otimes} \ar[rrr, "{(\Dperf(T),\Dual_T) \otimes_{(\Dperf(S),\Dual_S)} (-)}"]  &&& \Mod_T(\Catpsi)^{\otimes},
\end{tikzcd}
\]
expressing the compatibility of the functor in the previous item with base change along $T \to S$.
\end{enumerate}

\subsection{Bounded localising invariants and Nisnevich sheaves}
\label{subsection:nisnevich-invariants}%

Let $S$ be a qcqs scheme.

\begin{construction}
We denote by 
\[
\Mod_S^{\flat}(\Catpsi) \subseteq \Mod_S(\Catpsi)
\]
the full subcategory spanned by those $S$-linear $\infty$-categories with duality which are flat over $\Dperf(S)$ in the sense of Definition~\ref{definition:flat}. 
Our previous results establish the following:
\begin{itemize}
\item
By Example~\ref{example:flat}, this full subcategory contains the object $\Dperf(X)$ for every flat qcqs $p\colon X \to S$, and so in particular for any smooth qcqs $p$.
\item
By Corollary~\ref{corollary:flat-closed}, this full subcategory is closed under tensor products in $\Mod_S(\Catpsi)$, and hence inherits from $\Mod_S(\Catpsi)$ its symmetric monoidal structure. 
\end{itemize}
Combining these facts we find that the symmetric monoidal functor 
\[
(\Sm_S\op)^{\otimes} \to \Mod_S(\Catpsi)^{\otimes} \quad\quad [X \to S] \mapsto (\Dperf(X),\Dual_X)
\]
of~\eqref{equation:functoriality-schemes-3} factors uniquely through the symmetric monoidal full subcategory
$\Mod_S^\flat(\Catpsi)^{\otimes}$. 
Now since the $\infty$-category $\Sm_S$ of smooth $S$-schemes is essentially small, we may fix a large enough regular cardinal $\kappa$ such that the image of this functor lies in the full subcategory 
\[
\Mod_S^{\flat,\kappa}(\Catpsi) := \Mod_S^\flat(\Catpsi) \cap \Mod_S^\kappa(\Catpsi)
\]
spanned by the flat $\kappa$-compact $S$-linear $\infty$-categories with duality. Since $\Mod_S(\Catpsi)$ is presentably symmetric monoidal, we can ensure that $\Mod_S^{\flat,\kappa}(\Catpsi)$ is closed under tensor products in $\Mod_S^\flat(\Catpsi)$ by taking $\kappa$ sufficiently large. In particular, $\Mod_S^{\flat,\kappa}(\Catpsi)$ inherits a symmetric monoidal structure and the above symmetric monoidal functor further factors through a symmetric monoidal functor
\begin{equation}
\label{equation:functoriality-kappa}%
(\Sm_S\op)^{\otimes}  \to \Mod_S^{\flat,\kappa}(\Catpsi)^{\otimes} \quad\quad [X \to S] \mapsto (\Dperf(X),\Dual_X).
\end{equation}
\end{construction}

Let us now fix a presentable $\infty$-category $\E$. 

\begin{definition}
We say that a functor $\F\colon \Mod_{S}^{\flat,\kappa}(\Catpsi) \to \E$ is \defi{bounded localising} if it is reduced and sends bounded Karoubi squares (see Definition~\ref{definition:bounded}) to fibre squares. We denote by
\[
\Fun^{\bloc}(\Mod_S^{\flat,\kappa}(\Catpsi), \E) \subseteq \Fun(\Mod_S^{\flat,\kappa}(\Catpsi), \E)
\]
the full subcategory spanned by the bounded localising functors.
\end{definition}

\begin{remark}
If $\E$ is a stable then a reduced functor $\F\colon \Mod_{S}^{\flat,\kappa}(\Catpsi) \to \E$ is bounded localising if and only if it sends bounded Karoubi sequences to exact sequences.
\end{remark}

\begin{proposition}
\label{proposition:nisnevich-descent-variant}%
Let $S$ be a qcqs scheme equipped with a invertible perfect complex with $\Ct$-action $L\in\Picspace(X)^{\BC}$. 
Let $\F\colon \Mod_{S}^{\flat,\kappa}(\Catpsi) \to \E$
be a bounded localising functor.
Then the presheaf
\[
\F_S\left(\qSch_{/S}\right)\op \to \A,
\qquad
[p\colon X\to S] \mapsto \F(\Dperf(X),\Dual_{p^*L})
\]
is a Nisnevich sheaf. 
\end{proposition}
\begin{proof}
By Morel and Voevodsky's characterization of Nisnevich sheaves, what we need to check is that $\F_S$ vanishes on the empty $S$-scheme and sends pullback squares of qcqs $S$-schemes
\[
\begin{tikzcd}
W \ar[r]\ar[d] & V \ar[d,"f"] \\
U \ar[r,"i"] & X \ ,
\end{tikzcd}
\]
such that $f$ is étale, $i$ is an open embedding, and the induced map $V \setminus W \to X \setminus U$ is an isomorphism, to pullback squares in $\E$. Indeed, the first claim holds since $\F$ is assumed reduced, and the second because the induced square
\[
\begin{tikzcd}
(\Dperf(X),\Dual_{L|_X}) \ar[r]\ar[d] & (\Dperf(V),\Dual_{L|_{V}}) \ar[d] \\
(\Dperf(U),\Dual_{L|_U}) \ar[r] & (\Dperf(W),\Dual_{L|_{W}})
\end{tikzcd}
\]
is a bounded Karoubi square: its vertical legs are bounded Karoubi projections by Proposition~\ref{proposition:open-is-bounded} and it is cartesian since it is cartesian on the level of underlying stable $\infty$-categories, see Corollary~\ref{corollary:nisnevich-karoubi}.
\end{proof}

Now suppose that our presentable $\infty$-category $\E$ is presentably symmetric monoidal. Then the functor $\infty$-categories 
\[
\Fun(\Mod_S^{\flat,\kappa}(\Catpsi),\E) \quad\text{and}\quad \Psh(S,\E) = \Fun(\Sm_S\op,\E)
\]
inherit symmetric monoidal structures in the form of Day convolution, (see~\cite[\S 2.2.6]{HA}) and the symmetric monoidal functor~\eqref{equation:functoriality-kappa} induces by means of left Kan extension a symmetric monoidal functor
\begin{equation}
\label{equation:functoriality-psh}%
\Psh(S,\E)^{\otimes} \to \Fun(\Mod_S^{\flat,\kappa}(\Catpsi),\E)^{\otimes}. 
\end{equation}
Note that the symmetric monoidal structure on $\Sm_S\op$ is the cocartesian one, and so Day convolution on $\Psh(S,\E)$ amounts to pointwise products of presheaves.

Now, on the sheaf side, we have that $\Psh(S,\E)$ is presentable and its full subcategory $\Sh^{\Nis}(S,\E) \subseteq \Psh(S,\E)$ spanned by the Nisnevich sheaves is an accessible localisation of it. Furthermore, this localisation is multiplicative with respect to Day convolution: 
indeed, the monoidal product on $\Sm_S$ is the fibre product over $S$, and coverings families are closed under base change, see, e.g., the criterion in~\cite{9-authors-I}*{Lemma~5.3.4}. In particular, $\Sh^{\Nis}(S,\E)$ inherits a symmetric monoidal structure such that the sheafification functor
\[
(-)^{\Nis}\colon \Psh(S,\E) \to \Sh^{\Nis}(S,\E)
\]
refines to a symmetric monoidal functor.
Explicitly, the symmetric monoidal product on $\Sh^{\Nis}(S,\E)$ is given by taking pointwise products followed by sheafification.

Next, on the functor side, the full subcategory $\Fun^{\bloc}(\Mod_S^{\flat,\kappa}(\Catpsi), \E)$ is an accessible localisation of $\Fun(\Mod_S^{\flat,\kappa}(\Catpsi), \E)$ with associated localisation functor
\[
\Lbloc\colon \Fun(\Mod_S^{\flat,\kappa}(\Catpsi), \E) \to \Fun^{\bloc}(\Mod_S^{\flat,\kappa}(\Catpsi), \E).
\]
By Proposition~\ref{proposition:nisnevich-descent-variant}, in the left-Kan-extension/restriction adjunction
\[
\Psh(\Sm_S,\E) \adj \Fun(\Mod_S^{\flat,\kappa}(\Catpsi), \E),
\]
the right adjoint sends bounded localising functors to Nisnevich sheaves, and so the left adjoint sends Nisnevich local equivalences to $\Lbloc$-equivalences. 
By the universal property of localisations, we now obtain a commutative square of left adjoint functors
\[
\begin{tikzcd}
\Psh(\Sm_S,\E) \ar[r]\ar[d, "(-)^{\Nis}"] & \Fun(\Mod_S^{\flat,\kappa}(\Catpsi), \E) \ar[d,"\Lbloc"] \\
\Sh^{\Nis}(\Sm_S,\E) \ar[r, "\T_S"] & \Fun^{\bloc}(\Mod_S^{\flat,\kappa}(\Catpsi), \E) \ .
\end{tikzcd}
\]
where $\T_S$ is given by left Kan extension followed by $\Lbloc$.
As explained above, the top horizontal arrow in this square refines to a symmetric monoidal left adjoint, and since the localisation $(-)^{\Nis}$ is multiplicative the same holds for the left vertical arrow. 
We claim that this is the case for the right vertical arrow as well: 
\begin{proposition}
\label{proposition:lnis-is-mult}%
The accessible localisation
\[
\Lbloc\colon \Fun(\Mod_S^{\flat,\kappa}(\Catpsi), \E) \to \Fun^{\bloc}(\Mod_S^{\flat,\kappa}(\Catpsi), \E)
\]
is multiplicative with respect to Day convolution. In particular, $\Fun^{\bloc}(\Mod_{S}^{\flat,\kappa}(\Catpsi), \E)$ inherits a symmetric monoidal structure such that $\Lbloc$ refines to a symmetric monoidal functor.
\end{proposition}

By the universal property of multiplicative localisations the above square refines to a commutative square of symmetric monoidal left adjoint functors
\[
\begin{tikzcd}
\Psh(S,\E)^{\otimes} \ar[r]\ar[d, "(-)^{\Nis}"] & \Fun(\Mod_S^{\flat,\kappa}(\Catpsi), \E)^{\otimes} \ar[d,"\Lbloc"] \\
\Sh^{\Nis}(S,\E)^{\otimes} \ar[r, "\T_S^{\otimes}"] & \Fun^{\bloc}(\Mod_S^{\flat,\kappa}(\Catpsi), \E)^{\otimes}.
\end{tikzcd}
\]

\begin{proof}[Proof of Proposition~\ref{proposition:lnis-is-mult}]
By the criterion in~\cite{9-authors-I}*{Proposition~5.3.4} it will suffice to check that bounded Karoubi squares in $\Mod_S^{\flat,\kappa}(\Catpsi)$ are closed under tensoring with a fixed object. Indeed, by~\cite[\S 3]{9-authors-IV}
the collection of Poincaré-Karoubi squares in $\Catp$ is closed under tensoring with a fixed Poincaré $\infty$-category, and so any bounded Karoubi square at least remains cartesian after tensoring with a fixed $S$-linear $\infty$-category $(\C,\Dual)$. In addition, by Proposition~\ref{proposition:flat-multi} tensoring with $(\C,\Dual)$ also preserves bounded Karoubi projections when $(\C,\Dual)$ is flat, and so the desired result follows.
\end{proof}

\subsection{Motivic and pre-motivic realization}
\label{subsection:bott-periodicity}%

In this subsection we construct the motivic realization functor, see Constructions~\ref{construction:motivic} and~\ref{construction:R-from-kloc} below. As above, we fix a presentably symmetric monoidal $\infty$-category $\E$, with unit $u\colon \Sps \to \E$, and we consider the symmetric monoidal left adjoint functor
\[
\T_S^{\otimes} \colon \Sh^{\Nis}(S,\E)^{\otimes} \to \Fun^{\bloc}(\Mod_S^{\flat,\kappa}(\Catpsi), \E)^{\otimes}
\]
constructed at the end of the previous subsection.

\begin{definition}
Given an object $X \to S$ of $\Sm_S$, let us denote by $\E[X]$ the $\E$-valued sheaf obtained by sheafifying the presheaf 
\[
Y \mapsto u\Map_{\Sm_S}(Y,X).
\]
If $x\colon S \to X$ is an $S$-point of $X$ and $\E$ is pointed then we denote by $\E[X,x]$ the cofibre of the induced map $\E[S] \xrightarrow{x} \E[X]$ in $\Sh^{\Nis}(S,\E)$.
\end{definition}

In what follows, we will make use of the localised $\E$-valued co-Yoneda functor
\begin{equation}
\label{equation:jbloc}%
\jbloc\colon \Mod_S^{\flat,\kappa}(\Catpsi)\op \to \Fun^{\bloc}(\Mod_S^{\flat,\kappa}(\Catpsi), \E)
\end{equation}
sending $(\C,\Dual) \in \Mod_S^{\flat,\kappa}(\Catpsi)$ to $\Lbloc u\Map((\C,\Dual),-)$, where $\Loc_b$ is the multiplicative localisation functor of Proposition~\ref{proposition:lnis-is-mult}. Since $\Lbloc$, $u$ and the space-valued co-Yoneda embedding itself are all symmetric monoidal we have that $\jbloc$ is a symmetric monoidal functor.

\begin{proposition}
\label{proposition:projective-line}%
Suppose that $\E$ is additive, so in particular pointed. Then the symmetric monoidal functor $\T_S^{\otimes}$ sends the object $\E[\Pone_S,\infty]$ to $\jbloc (\Dperf(S),\Sig^{-1}\Dual_S)$. In particular, it sends $\E[\Pone_S,\infty]$ to an \emph{invertible} object in $\Fun^{\bloc}(\Mod_S^{\flat,\kappa}(\Catpsi), \E)$ (with respect to the localised Day convolution structure of Proposition~\ref{proposition:lnis-is-mult}). 
\end{proposition}
\begin{proof}
Since left Kan extension sends representable functors to representable functors we have that $\T_S\E[\Pone_S,\infty]$ is given by the cofibre of the map
\[
\jbloc(\Dperf(S),\Dual_S) \xrightarrow{\jbloc(\infty^*)} \jbloc (\Dperf(\Pone_S),\Dual_{\Pone_S})
\]
induced by the restriction functor $\infty^*\colon (\Dperf(\Pone_S),\Dual_{\Pone_S}) \to (\Dperf(S),\Dual_S)$ along the base point $\infty\colon S \to \Pone_S$. This functor admits a retraction which is induced by the terminal map $q\colon \Pone_S \to S$ in $\Sm_S$, and since $\Fun^{\bloc}(\Mod_S^{\flat,\kappa}(\Catpsi), \E)$ inherits from $\E$ the property of being additive, it follows from the splitting lemma (see, e.g,~\cite{9-authors-II}*{Lemma~1.5.12}) that the above cofibre is canonically equivalent to the fibre of the map
\[
\jbloc(\Dperf(\Pone_S),\Dual_{\Pone_S}) \xrightarrow{\jbloc(q^*)} \jbloc(\Dperf(S),\Dual_S) 
\]
induced by the associated restriction functor $q^*\colon (\Dperf(S),\Dual_S) \to (\Dperf(\Pone_S),\Dual_{\Pone_S})$. But now the bounded Karoubi sequence
\[
(\Dperf(S),\Dual_S)\xrightarrow{q^*} (\Dperf(\Pone_S),\Dual_{\Pone_S}) \xrightarrow{q_*(- \otimes \cO(-1))} (\Dperf(S),\Sig^{-1}\Dual_S)
\]
of Construction~\ref{construction:projective-line} exhibits an equivalence between the fibre of $\jbloc(q^*)$ and the object $\jbloc(\Dperf(S),\Sig^{\-1}\Dual_S)$, which is invertible since $(\Dperf(S),\Sig^{-1}\Dual_S)$ is invertible in $\Mod_S^{\flat,\kappa}(\Catpsi)$ and $\jbloc$ is symmetric monoidal.
\end{proof}

By Proposition~\ref{proposition:projective-line} and the universal property of stabilizations, if $\E$ is an additive presentably symmetric monoidal $\infty$-category, then the symmetric monoidal left adjoint functor $\T_S^{\otimes}$ from above factors essentially uniquely via a symmetric monoidal left adjoint functor
\[
\begin{tikzcd}
\Sh^{\Nis}(S,\E)^{\otimes} \ar[r,"\T_S^{\otimes}"]\ar[d,"\Sig^{\infty}_{\Pone}"'] & \Fun^{\bloc}(\Mod_S^{\flat,\kappa}(\Catpsi), \E)^{\otimes}\\
\Spa_{\Pone}^{\Sig}(S,\E)^{\otimes} \ar[ur,dashed,"\wtl{\T}_S^{\otimes}"'] &
\end{tikzcd}
\]
whose right adjoint
\[
\R_S\colon \Fun^{\bloc}(\Mod_S^{\flat,\kappa}(\Catpsi), \E) \to \Spa_{\Pone}^{\Sig}(S,\E)
\]
then inherits a lax symmetric monoidal structure.

\begin{construction}[Motivic realization for bounded localising functors]
\label{construction:motivic}%
Given a bounded localising functor $\F\colon \Mod_S^{\flat,\kappa}(\Catpsi) \to \E$ and a flat $S$-linear $\infty$-category with duality $(\C,\Dual) \in \Mod_S^{\flat,\kappa}(\Catpsi)$, we will denote by 
\[
\R_S(\F;(\C,\Dual)) \in \Spa_{\Pone}^{\Sig}(S,\E)
\]
the image under the right adjoint $\R_S$ above of the bounded localising functor $\F((-) \otimes (\C,\Dual))$ obtained by pre-composing $\F$ with $(-) \otimes (\C,\Dual)$ (and operation which preserves bounded localising functors by Proposition~\ref{proposition:flat-multi}). We refer to $\R_S(\F;(\C,\Dual))$ as the \defi{pre-motivic realization} of $\F$ with coefficients in $(\C,\Dual)$, and to
\[
\M_S(\F;(\C,\Dual)) := \Loc_{\Aone}\R_S(\F;(\C,\Dual)) \in \SH(S,\E)
\]
as the \defi{motivic realization} of $\F$ with coefficients in $(\C,\Dual)$.
\end{construction}

The association $(\F,(\C,\Dual)) \mapsto \R_S(\F;(\C,\Dual))$ assembles to form a functor
\[
\R_S(-;-)\colon \Fun^{\bloc}(\Mod_S^{\flat,\kappa}(\Catpsi), \E) \times \Mod_S^{\flat,\kappa}(\Catpsi) \to \Spa_{\Pone}^{\Sig}(S,\E) .
\]
This functor carries a canonical lax monoidal functor: indeed, it is the composite of the lax symmetric monoidal functor $\R_S(-)$ and the functor $(\F,(\C,\Dual)) \mapsto \F((-) \otimes (\C,\Dual))$, which is nothing but the internal mapping object with respect to Day convolution from $\jbloc(\C,\Dual)$ to $\F$, and is hence lax symmetric monoidal in its entries as a pair (see~\cite[Proposition 3.4.9]{HHLN}). Since the $\Aone$-localisation functor $\Loc_{\Aone}\colon \Spa_{\Pone}^{\Sig}(S,\E) \to \SH(S,\E)$ is also symmetric monoidal we similarly have that
\[
\M_S(-;-)\colon \Fun^{\bloc}(\Mod_S^{\flat,\kappa}(\Catpsi), \E) \times \Mod_S^{\flat,\kappa}(\Catpsi) \to \SH(S,\E)
\]
is lax symmetric monoidal.
In particular, if 
\[
\F\colon \Mod_S^{\flat,\kappa}(\Catpsi) \to \E
\]
is a commutative algebra object with respect to Day convolution, that is, $\F$ is a lax symmetric monoidal functor, then $\M_S(\F)$ is a commutative algebra object in $\SH(S,\E)$, 
and the same holds for $\M_S(\F;(\C,\Dual))$ for every commutative algebra object in 
$\Mod_S^{\flat,\kappa}(\Catpsi)$. Similarly, if $\F$ is lax symmetric monoidal and $(\C,\Dual)$ is any object of $\Mod_S^{\flat,\kappa}(\Catpsi)$ then $\M_S(\F;(\C,\Dual))$ inherits the structure of a module over $\M_S(\F)$.

\begin{remark}
\label{remark:realization-exact}%
By construction, for any bounded localising functor $\Mod_S^{\flat,\kappa}(\Catpsi) \to \E$, the association $(\C,\Dual) \mapsto \R_S(\F;(\C,\Dual))$ takes bounded Karoubi squares $\Mod_S^{\flat,\kappa}(\Catpsi) $ to fibre squares.  
\end{remark}

\begin{notation}
\label{notation:L-in-pic}%
Given a line bindle $L \in \Picspace(S)^{\BC}$ the $\infty$-category with duality $(\Dperf(S),\Dual_L)$ is a module over $(\Dperf(S),\Dual_S)$, and we will use the shorthand notation
\[
\R_L(\F) := \R_S(\F;(\Dperf(S),\Dual_L)) \quad\text{and}\quad \M_L(\F) := \M_S(\F;(\Dperf(S),\Dual_L))
\]
to denote the associated pre-motivic and motivic realizations. When $\F$ is lax symmetric monoidal, these are modules over $\R_S(\F)$ and $\M_S(\F)$, respectively. 
\end{notation}

\begin{construction}
\label{construction:realization-cotensor}%
The functor $\wtl{\T}_S^{\otimes}$ being symmetric monoidal, we have for every $\E$-valued pre-motivic spectrum $E \in \Spa^{\Sig}_{\Pone}(S,\E)$ a commutative square
\[
\begin{tikzcd}
\Fun^{\bloc}(\Mod_S^{\flat,\kappa}(\Catpsi), \E) \ar[d,"(-)\otimes \wtl{\T}_S(E)"']& \ar[l, "\wtl{\T}_S"']  \Spa_{\Pone}^{\Sig}(S,\E) \ar[d, "{(-) \otimes E}"] \\
\Fun^{\bloc}(\Mod_S^{\flat,\kappa}(\Catpsi), \E) & \ar[l, "\wtl{\T}_S"']  \Spa_{\Pone}^{\Sig}(S,\E) \ . \end{tikzcd}
\]
Passing to right adjoints, we obtain a commutative square
\[
\begin{tikzcd}
\Fun^{\bloc}(\Mod_S^{\flat,\kappa}(\Catpsi), \E) \ar[r, "\R_S"] & \Spa_{\Pone}^{\Sig}(S,\E) \\
\Fun^{\bloc}(\Mod_S^{\flat,\kappa}(\Catpsi), \E) \ar[u,"(-)^{\wtl{\T}_S(E)}"]\ar[r, "\R_S"] & \Spa_{\Pone}^{\Sig}(S,\E)  \ar[u, "{(-)^E}"'] \ ,
\end{tikzcd}
\]
where the vertical arrows denote the operation of internal mapping objects from $E$ and $\wtl{\T}_S(E)$, respectively. In particular, for $\F \in \Fun^{\bloc}(\Mod_S^{\flat,\kappa}(\Catpsi), \E)$ we obtain a natural equivalence
\[
\R_S(\F)^{E} \simeq \R_S(\F^{\wtl{\T}_S(E)}) .
\]
In addition, the symmetry of the monoidal structure on $\Fun^{\bloc}(\Mod_S^{\flat,\kappa}(\Catpsi), \E)$ means that cotensors commute with cotensors, and hence for every $(\C,\Dual) \in \Mod_S^{\flat,\kappa}(\Catpsi)$ we obtain a natural equivalence
\[
\R_S(\F;(\C,\Dual))^{E} \simeq \R_S(\F^{\wtl{\T}_S(E)};(\C,\Dual)) .
\]
In particular, taking $(\C,\Dual) = (\Dperf(S),\Dual_L)$ for some $L \in \Picspace(L)^{\BC}$ we obtain a natural equivalence
\[
\R_L(\F)^{E} \simeq \R_L(\F^{\wtl{\T}_S(E)}),
\]
see Notation~\ref{notation:L-in-pic}.
\end{construction}

\begin{notation}
Given a finite set $I$ we will denote by $\Sig^I$ the operation of $I$-fold suspension (that is, tensoring with the $I$-fold smash product $\wedge_{i \in I} S^1$ of $S^1$ with itself), and by $\Om^I$ its right adjoint, namely, taking $I$-fold loops. If $\C$ is a stable $\infty$-category equipped with a duality $\Dual$ then we also write $\Dual^{[I]} := \Sig^I\Dual$ and $\Dual^{[-I]} := \Om^I\Dual$ for the $I$-fold suspension/loops of $\Dual$, and if $\QF$ is a Poincaré structure on $\C$ then we similarly write $\QF^{[I]} := \Sig^I\QF$ and $\QF^{[-I]} := \Om^I\QF$. In addition, if $L$ is a line bundle on a scheme $X$ then we also use the notation $L[I] := \Sig^IL$ and $L[-I] := \Om^I L$ to denote the $I$-fold suspension/loops, so that we have, for example $(X,\Dual_L\qshift{I}) = (X,\Dual_{L[I]})$ and $(\Dperf(X),(\QF^{\sym}_L)\qshift{I}) = (\Dperf(X),\QF^{\sym}_{L[I]})$.

Given a functor $\F$ defined on either $\infty$-categories with duality or Poincaré $\infty$-categories we denote by $\F\qshift{I}(-,-) := \F(-,-\qshift{I})$
the functor obtained by pre-composing with the $\pm I$-shift functor. We note that this operation preserves bounded localising and Karoubi localising functors.

While these operations only depend on the sizes $n=|I|$, this notation allows us to keep track of the functorial dependence of this construction in $I$ with respect to set isomorphisms. When this information is not important we also use the notation $(-)\qshift{n}$ to denote $(-)\qshift{\mathrm{sign}(n)\{1,...,n\}}$, as we have been using in the context of Poincaré $\infty$-categories. 
\end{notation}

\begin{example}[Bott periodicity]
\label{example:bott}%
Applying Construction~\ref{construction:realization-cotensor} in the case of $E = \Sig^{\infty}_{\Pone}\E[\Pone_S,\infty]$ and using Proposition~\ref{proposition:projective-line} we obtain a natural equivalence
\[
\Om_{\Pone}\R_L(\F) = \R_L(\F((-) \otimes (\Dperf(S),\Dual_{S}\qshift{-1})) =  \R_L(\F\qshift{-1}) = \R_{L[-1]}.
\]
Since the operations $\Om_{\Pone}(-)$ and $(-)^{[-1]}$ are invertible this equivalence determines an equivalence
\[
\Sig_{\Pone}\R_L(\F) = \R_L(\F\qshift{1}) = \R_{L[1]}(\F).
\]
More generally, for every finite set $I$ we have an equivalence
\[
\Sig_{\Pone}^I\R_L(\F) = \R_L(\F^{[I]}) = \R_{L[I]}(\F),
\]
Applying the $\Aone$-localisation functor $\Loc_{\Aone}\colon\Spa^{\Sig}_{\Pone}(\Sh^{\Nis}(S,\E)) \to \SH(S,\E)$ we then obtain an equivalence
\[
\Sig_{\Pone}^I\M_L(\F) = \M_L(\F^{[I]}) = \M_{L[I]}(\F).
\]
\end{example}

\begin{remark}
\label{remark:underlying}%
For every $S$-smooth scheme $p\colon X \to S$ we have $(\Dperf(X),\Dual_X) \otimes (\Dperf(S),\Dual_L) = (\Dperf(X),\Dual_{p^*L})$, and so the underlying Nisnevich sheaf $\Om^{\infty}_{\Pone}\R_L(\F)$ is given by $[p\colon X \to S] \mapsto \F(\Dperf(X),\Dual_{p^*L})$. More generally, By Example~\ref{example:bott}
the pre-motivic spectrum $\R_L(\F) \in \Spa^{\Sig}_{\Pone}(S,\E)$ can concretely be described by the formula
\[
\R_L(\F)_I(p\colon X \to S) = \F(\Dperf(X),\Dual_{p^*L[I]}) .
\]
Applying the $\Aone$-localisation functor and using Lemma~\ref{lemma:LAoneOmPone} we conclude that the motivic spectrum $\M_L(\F)$ is given by
\[
\M_L(\F)_I(p\colon X \to S) = (\Loc_{\Aone}\R^{\sym}_L(\F)_I)(p\colon X \to S) = |\F(\Dperf(X\times_S \Del_S^{\bullet}),\Dual_{p_\bullet^*L[I]})| ,
\]
where $p_n\colon X\times_S \Del_S^n \to S$ is the structure map of $X \times_S \Del_S^n$.
\end{remark}

\begin{lemma}
\label{lemma:colimits}%
The functor $\R_S$ preserves all limits. If in addition $\E$ is stable, then $\R_S$ (and hence also $\M_S$) preserves all colimits.
\end{lemma}
\begin{proof}
The first claim is because $\R_S$ is a right adjoint. To see the second claim, note that if $\E$ is stable, then the full subcategory of Nisnevich sheaves is closed under colimits inside all presheaves and the functor $\Om_{\Pone}$ commutes with colimits. It then follows that colimits in symmetric $\Pone$-spectrum objects in Nisnevich sheaves are computed on the level of symmetric $\Pone$-prespectrum objects in presheaves.
At the same time, when $\E$ is stable the full subcategory of bounded localising functors is closed under colimits inside all functors, and so colimits of bounded localising functors are computed levelwise. 
The desired result now follows from the fact that, by Remark~\ref{remark:underlying}, the value at $X \in \Sm_S$ of the $I$-components of the $\Pone$-spectrum object $\R_S(\F)$ is given by 
\[
\R_S(\F)_I(X) = \F(\Dperf(X),\Dual_X\qshift{I}),
\]
which, as a functor of $\F$, is colimit preserving by the above. 
\end{proof}

Finally, we now construct motivic realization functors for Karoubi localising invariants of Poincaré $\infty$-categories:

\begin{construction}[Motivic realization for Karoubi localising functors]
\label{construction:R-from-kloc}%
By Proposition~\ref{proposition:flat-is-karoubi-sym}, restriction along the lax symmetric monoidal composite functor 
\begin{equation}
\label{equation:lax-composite}%
\Mod_S^{\flat,\kappa}(\Catpsi) \to \Catpsi \to \Catps \to \Catp \quad\quad (\C,\Dual) \mapsto (\C,\QF^{\sym}_\Dual) 
\end{equation}
sends Karoubi-localising functors to bounded localising functors. By means of pre-composition we hence obtain a lax symmetric monoidal functor
\begin{equation}
\label{equation:RS-from-kloc}%
\R^{\sym}_S\colon \Fun^{\kloc}(\Catp,\E) \to  \Fun^{\bloc}(\Mod_S^{\flat,\kappa}(\Catpsi), \E)  \xrightarrow{\R_S} \Spa_{\Pone}^{\Sig}(S,\E).
\end{equation}
For a Karoubi-localising functor $\F\colon \Catp \to \E$ and an $S$-linear stable $\infty$-category with duality $(\C,\Dual)$ we then write
\[
\R^{\sym}_S(\F;(\C,\Dual)) = \R_S(\F|_{\Mod_S^{\flat,\kappa}(\Catpsi)};(\C,\Dual)) \in \Spa_{\Pone}^{\Sig}(S,\E)
\]
and
\[
\M^{\sym}_S(\F;(\C,\Dual)) := \Loc_{\Aone}\R^{\sym}_S(\F;(\C,\Dual)) \in \SH(S,\E)
\]
for the corresponding pre-motivic and motivic realizations with coefficients. 

As in Notation~\ref{notation:L-in-pic}, for an invertible perfect complex with $\Ct$-action $L \in \Picspace(S)^{\BC}$, we write
\[
\R^{\sym}_L(\F) = \R^{\sym}_S(\F;(\Dperf(S),\Dual_L)) \quad\text{and}\quad \M^{\sym}_L(\F) = \M^{\sym}_S(\F;(\Dperf(S),\Dual_L)),
\]
so that $\R^{\sym}_L$ is a module over $\R^{\sym}_S$ with respect to Day convolution, and similarly for $\M^{\sym}_L$ is a module over $\M^{\sym}_S(\F)$. 
In particular, $\R^{\sym}_L(\F)$ is a module over $\R^{\sym}_S(\F)$ and $\M^{\sym}_L(\F)$ is a module over $\M^{\sym}_S(\F)$ for any lax symmetric monoidal Karoubi-localising functor $\F$. 
\end{construction}

\begin{remark}
\label{remark:underlying-sym}%
By Remark~\ref{remark:underlying} the pre-motivic spectrum $\R^{\sym}_L(\F) \in \Spa^{\Sig}_{\Pone}(S,\E)$ is given explicitly by the formula
\[
\R^{\sym}_L(\F)_I(p\colon X \to S) = \F(\Dperf(X),\QF^{\sym}_{p^*L[I]}) ,
\]
while the motivic spectrum $\M^{\sym}_L(\F)$ is given by
\[
\M^{\sym}_L(\F)_I(p\colon X \to S) =  |\F(\Dperf(X\times_S \Del_S^{\bullet}),\QF^{\sym}_{p^*_\bullet L[I]})|.
\]
where $p_n\colon X\times_S \Del_S^n \to S$ is the structure map of $X \times_S \Del_S^n$.
\end{remark}

\begin{remark}
To avoid confusion, we note that since the composite~\eqref{equation:lax-composite} is only lax symmetric monoidal, the functor $\R^{\sym}_S$ is not part of a symmetric monoidal adjunction; more precisely, its left adjoint is only oplax symmetric monoidal, but not symmetric monoidal.
\end{remark}

\subsection{The free delooping lemma}
\label{subsection:unique-delooping}%

Let us now consider the case where the target $\E$ is either the $\infty$-category $\Spa$ of spectra or the $\infty$-category $\Grp_{\Einf}$ of $\Einf$-groups. These two $\infty$-categories are related via an adjunction
\[
\iota\colon \Grp_{\Einf} \adj \Spa \cocolon \Om^{\infty} ,
\]
where $\Om^{\infty}$ is the lift of the usual loop infinity functor from spaces to $\Einf$-groups and its left adjoint $\iota$ is the embedding of $\Grp_{\Einf}$ as the full subcategory of connective spectra. In fact, this adjunction also exhibits $\Spa$ as the stabilization of $\Grp_{\Einf}$, and so we could have called $\iota$ also $\Sig^{\infty}$, but we will stick with $\iota$ in order to avoid confusion with $\Sig^{\infty}$ of pointed spaces (which is not fully-faithful).
We then have an induced adjunction
\[
\iota_* \colon \Spa_{\Pone}^{\Sig}(S,\Grp_{\Einf}) \adj  \Spa_{\Pone}^{\Sig}(S,\Spa) \cocolon \Om^{\infty}_*,
\]
on the level of symmetric $\Pone$-spectrum objects (in Nisnevich sheaves), where $\Om^{\infty}_*$ is given by object-wise post-composing with $\Om^{\infty}$ and $\iota_*$ is given by post-composition with $\iota$ followed by Nisnevich sheafification and the spectrification functor
\[
\PSpa_{\Pone}^{\Sig}(S,\Spa) \to \Spa_{\Pone}^{\Sig}(S,\Spa)
\]
from symmetric $\Pone$-prespectra to $\Pone$-spectra.

\begin{lemma}[Free delooping]
\label{lemma:unique-delooping}%
Let $\F\colon \Mod_S^{\flat,\kappa}(\Catpsi) \to \Spa$ be a bounded localising functor and $L \in \Pic(S)^{\B C}$ an invertible perfect complex with $\Ct$-action.
Then the counit map
\[
\iota_*\Om^{\infty}_*\R_L(\F) \to \R_L(\F)
\]
is an equivalence. 
\end{lemma}

\begin{remark}
Since $\Om^{\infty}_*\R_L(\F) \simeq \R_L(\Om^{\infty}\F)$, Lemma~\ref{lemma:unique-delooping} says in particular that $\R_L(\F)$ only depends on the underlying $\Grp_{\Einf}$-valued functor $\Om^{\infty}\F$. This also means that if $\F$ is a Karoubi-localising functor on $\Catp$ then the motivic spectrum $\M^{\sym}_L(\F)$ only depends on $\Om^{\infty}\F$.
This consequence is a-priori not surprising, to the extent that Karoubi-localising functors are uniquely determined by their connective covers, see~\cite[\S 1]{9-authors-IV}.
Lemma~\ref{lemma:unique-delooping} is however a bit more precise, giving the exact manner in which the eventual dependence on $\Om^{\infty}\F$ is expressed. 
\end{remark}

\begin{proof}
Replacing $\F$ with $\F((-) \otimes_{(\Dperf(S),\Dual_S)}(\Dperf(S),\Dual_L))$ we may assume that $L = \cO_S$.
Since the composite $\iota\circ \Om^{\infty}$ is naturally equivalent to the connective cover functor the desired claim amounts to checking that the map of symmetric $\Pone$-prespectra objects in presheaves
\[
\tau_{\geq 0}\R_S(\F) \to \R_S(\F)
\]
induces an equivalences on mapping spaces to any symmetric $\Pone$-spectrum object in sheaves. Here, the connective cover functor $\tau_{\geq 0}$ is applied to the $\Pone$-spectrum object $\R_S(\F)$ pointwise and without sheafification, so that $\tau_{\geq 0}\R_S(\F)$ is taken as a $\Pone$-prespectrum object in presheaves. 

In what follows, let us consider $\Aa^1, \Gm$ and $\Pone\smallsetminus\{0\}$ as pointed schemes over $\spec(\ZZ)$ with a compatible base point given by $1 \in \Gm$ (and its images in $\Aa^1$ and $\Pone\smallsetminus\{0\}$). We can then view them as representable sheaves on $\Sm_{\ZZ}$ taking values in pointed spaces, and we write
\[
B = [\Aa^1,1]\amalg_{[\Gm,1]}[\Pone\smallsetminus\{0\},1] \in \Psh(\spec(\ZZ),\Sps_*)
\]
for the associated pushout in presheaves of pointed spaces.
Since $\Grp_{\Einf}$ is tensored and cotensored over pointed spaces we can form the associated loop functor
\[
\Om_{B}\colon \Psh(S,\Grp_{\Einf}) \to \Psh(S,\Grp_{\Einf}).
\]
The obvious map $B \to [\Pone,1]$ of pointed presheaves then induces a natural transformation
$\Om_{\Pone} \Rightarrow \Om_B$,
yielding a natural map
\[
\X \to \sh_1\Om_{\Pone}(\X) \to \sh_1\Om_B(\X) .
\]
We then write 
\[
L_B\X = \colim_n \sh_n\Om_B^n(\X)
\]
for the colimit of the corresponding sequence.
We note that since the map $B \to [\Pone,\infty]$ is a Nisnevich local equivalence the maps $\Y \to \sh_1\Om_{\Pone}\Y \to \sh_1\Om_{B}\Y$ are both equivalences for $\Y$ a $\Pone$-spectrum object in Nisnevich sheaves, in which case the map $\Y \to L_B\Y$ is an equivalence.

We now claim that if $\Y$ is such that the map $\Y \to L_B\Y$ is an equivalence and $\X \to \X'$ is a map of $\Pone$-prespectrum objects in presheaves such that the induced map $L_B\X \to L_B\X'$ is an equivalence then the restriction map
\[
\map(\X',\Y) \to \map(\X,\Y)
\]
is an equivalence. To see this, consider the commutative diagram
\[
\begin{tikzcd}
\map(\X',\Y) \ar[r]\ar[d] & \map(\X,\Y) \ar[d]\\
\map(L_B\X',L_B\Y) \ar[r]\ar[d] & \map(L_B\X,L_B\Y)\ar[d] \\
\map(\X',L_B\Y) \ar[r] & \map(\X,L_B\Y) \ ,
\end{tikzcd}
\]
where the maps from the top to the middle row are induced by the functoriality of $L_B$ and the maps from the middle to the bottom row are induced by pre-composition with $\X'\to L_B\X'$ and $\X \to L_B\X$, respectively. The naturality of $(-) \Rightarrow L_B(-)$ then determines a homotopy between the vertical composites and the maps induced by post-composition with $\Y \to L_B\Y$, and so these vertical composites are equivalences by our assumption. This entire diagram is hence equivalent to a retract diagram in the arrow category, and since our assumptions imply that the middle horizontal map is an equivalence the same holds for the top and bottom horizontal maps.

With this at hand, we now claim that in the case of $\X = \R_S(\F)$, the induced map
\[
\eta\colon L_B\tau_{\geq 0}\X \to L_B\X
\]
is an equivalence.
By the above, this implies that the map $\tau_{\geq 0}\X \to \X$
induces an equivalence on mapping spectra into any $\Pone$-spectrum objects in Nisnevich sheaves, so that the desired result follows.
Now to prove that $\eta$ is an equivalence, we show that 
for every $n \geq 0$ the map
\[
\eta\colon \sh_n\Om^n_B\tau_{\geq 0}\R_S(\F) \to \sh_n\Om^n_B\R_S(\F) \simeq \R_S(\F)
\]
is pointwise an equivalence on homotopy presheaves in degrees $\geq -n$. This holds for $n=0$ by definition. Let us fix a finite set $I$. 
Arguing by induction, it remains to show that if $\X' \to \R_S(\F)_I$ is a presheaf map which exhibits $\X'$ as the $(-n+1)$-connective cover of $\R_S(\F)_I$ (in presheaves) for some $n \geq 1$,
then the induced map $\Om_B\X' \to \Om_B\R_S(\F)_I$ exhibits $\Om_B\X'$ as the $(-n)$-connective cover of $\Om_B\R_S(\F)_I$. 
For this, consider the commutative diagram of presheaves
\[
\begin{tikzcd}
\Om_B\X'\ar[r]\ar[d] & \Om_{\Aa^1}\X' \oplus \Om_{\Pone\smallsetminus\{0\}}\X' \ar[r]\ar[d] & \Om_{\Gm}\X' \ar[d] \\
\Om_B\R_S(\F)_I\ar[r] & \Om_{\Aa^1}\R_S(\F)_I \oplus \Om_{\Pone\smallsetminus\{0\}}\R_S(\F)_I \ar[r] & \Om_{\Gm}\R_S(\F)_I 
\end{tikzcd}
 \]
whose rows are exact by construction. For every pointed scheme $U \in \Sm_{/\spec(\ZZ)}$ and presheaf $\Z$, the presheaf $\Om_U\Z$ is a summand of $\Z((-) \times U)$, and hence the middle and right most vertical maps are isomorphisms on homotopy presheaves in degrees $\geq -n+1$ by the induction hypothesis. By the long exact sequence in homotopy presheaves and the five lemma, it will suffice to show that the square
\[
\begin{tikzcd}
\pi_{-n}(\Om_{\Aa^1}\X') \oplus \pi_{-n}(\Om_{\Pone\smallsetminus\{0\}}\X') \ar[r]\ar[d] & \pi_{-n}(\Om_{\Gm}\X') \ar[d] \\
\pi_{-n}(\Om_{\Aa^1}\R_S(\F)_I) \oplus \pi_{-n}(\Om_{\Pone\smallsetminus\{0\}}\R_S(\F)_I) \ar[r] & \pi_{-n}(\Om_{\Gm}\R_S(\F)_I) 
\end{tikzcd}
\]
induces an isomorphism on horizontal kernels. In the upper row, both terms are zero by the induction hypothesis (and the fact that $\Om_U\Z$ is a summand of $\Z((-) \times U)$ as above), so this amounts to showing that the bottom horizontal arrow is injective. Since $\R_S(\F)_I$ is a Nisnevich sheaf, this is equivalent by the long exact sequence in homotopy presheaves to the claim that the map
\[
\pi_{-n}(\Om_{\Pone}\R_S(\F)_I) \to \pi_{-n}(\Om_{\Aa^1}\R_S(\F)_I) \oplus \pi_{-n}(\Om_{\Pone\smallsetminus\{0\}}\R_S(\F)_I)
\]
is the zero map. In fact, we claim that each of the maps of presheaves
\[
\Om_{\Pone}\R_S(\F)_I \to \Om_{\Aa^1}\R_S(\F)_I \quad\text{and}\quad \Om_{\Pone}\R_S(\F)_I \to \Om_{\Pone\smallsetminus\{0\}}\R_S(\F)_I
\]
is a null-homotopic. Since there is an automorphism of $\Pone$ which fixes 1 and switches between 0 and $\infty$ it will suffice to prove the claim for the first map. 
Fix a smooth $S$-scheme $p\colon X \to S$ and write $q\colon \Pone_X \to X$ and $r\colon \Aone_X \to X$ for the corresponding projections. To simplify notation, let us write $L' := p^*L[I]$, so that 
\[
\R_L(\F)_I(X) = \F(\Dperf(X),\Dual_{L'})
\]
by Remark~\ref{remark:underlying}. Since $\Om_{\Aa^1}\R_S(\F)_I(X)$ is a direct summand of $\R_S(\F)_I(\Aone_X) = \F(\Dperf(\Aone_X),\Dual_{r^*L'})$ it will suffice to show that the map
\[
\Om_{\Pone}\R_S(\F)_I(X) = \fib[\F(\Dperf(\Pone_X),\Dual_{q^*L'}) \xrightarrow{i_1^*} \F(\Dperf(X),\Dual_{L'})] \to \F(\Dperf(\Aone_X),\Dual_{r^*L'})
\]
is null-homotopic, where $i_1^*$ is pullback along the inclusion $i_1\colon X \to \Pone_X$ induced by the base point $1 \in \Pone$.
Now the inclusion 
$i_{\infty} \colon X \hrar \Pone_X$
induced by the point $\infty \in \Pone$ is a regular codimension 1 embeddings (even if $X$ and $\Pone_X$ are not necessarily regular), and hence by Lemma~\ref{lemma:devissage-functor} we have the associated duality preserving push-forward functors
\[
(i_{\infty})_*\colon (\Dperf(X),\Dual_{L'[-1]}) = (\Dperf(X),\Dual_{i_{\infty}^!q^*L'}) \to (\Dperf(\Pone_X),\Dual_{q^*L'}).
\]
We now claim that the zero-composite sequence
\begin{equation}
\label{equation:push-pull}%
(\Dperf(X),\Dual_{L'[-1]}) \xrightarrow{(i_\infty)_*} (\Pone_X, \Dual_{q^*L'})  \xrightarrow{i_1^*} (X,\Dual_{L'})
\end{equation}
is sent by $\F$ to an exact sequence, or equivalently, that the vertical arrow in the induced diagram 
\[
\begin{tikzcd}
\F(\Dperf(X),\Dual_{L'[-1]})\ar[d]\ar[dr,dashed] & \\
\fib[\F(\Dperf(\Pone_X),\Dual_{q^*L'}) \to  \F(\Dperf(X),\Dual_{q^*L'[-1]})] \ar[r] & \F(\Dperf(\Aone_X),\Dual_{r^*L'}) 
\end{tikzcd}
\]
is an equivalence. Note that this will finish the proof, since the composite of $(i_\infty)_*$ with pullback along $\Aa^1 \subseteq \Pone$ is the zero functor, and so the dotted composite above is null homotopic. Now to show that~\eqref{equation:push-pull} is indeed sent by $\F$ to an exact sequence let us place it horizontally in the commutative diagram of $S$-linear $\infty$-categories with duality
\[
\begin{tikzcd}
& (\Dperf(X),\Dual_{L'})\ar[d,"q^*"]\ar[dr,dashed] & \\
(\Dperf(X),\Dual_{L'[-1]}) \ar[r,"(i_\infty)_*"]\ar[dr,dashed] & (\Pone_X, \Dual_{q^*L'}) \ar[r,"i_1^*"]\ar[d,"q_*(- \otimes \cO(-1))"] & (X,\Dual_{L'})  \\
& (\Dperf(X),\Dual_{L'[-1]}) & 
\end{tikzcd}
\]
where the vertical sequence is the one underlying the bounded Karoubi sequence of Construction~\ref{construction:projective-line}, which is sent to an exact sequence by $\F$. Now both dashed composites are homotopic to the respective identities, the top right one since $q \circ i_1 = \id$, and the bottom left one by Lemma~\ref{lemma:composite-pf} since $q \circ i_{\infty} = \id$ and $i_{\infty}^*\cO(-1) \simeq \cO_X$. We hence conclude that $\F$ also sends the horizontal sequence to an exact sequence, as desired.
\end{proof}

\section{Hermitian K-theory as a motivic spectrum}
\label{section:KQ}%

In this section we build on the results proven so far to construct the hermitian K-theory motivic spectrum $\KQ_S$ over a qcqs scheme $S$, without assuming that 2 is invertible in $S$ (as is done in previous constructions, such as \cite{hornbostel-motivic} or \cite{hermitian-bank-robbery}). Given an invertible perfect complex with $\Ct$-action $L \in \Picspace(S)^{\BC}$, we also construct a motivic spectrum $\KQ_L$ of hermitian $\K$-theory with coefficients in $L$, as well as corresponding versions $\KW_S$ and $\KW_L$ of motivic Witt spectra. We spend the majority of the section in order to establish the key fundamental properties of these motivic spectra. In particular, we prove:

\begin{enumerate}
\item
$\KQ_S$ and $\KW_S$ are motivic $\Einf$-ring spectra, and the natural map $\KQ_S \to \KW_S$ is an $\Einf$-map whose kernel is $\KGL_{\hC}$, where $\KGL$ is the motivic $\K$-theory spectrum. The motivic spectra $\KQ_L$ and $\KW_L$ are modules over $\KQ_S$ and $\KW_S$.
\item
The forgetful map $\KQ_S \to \KGL_S$ is also one of $\Einf$-ring spectra. It fits in the Wood fibre sequence
\[
\Om_{\Pone}\KQ_S \to \KGL_S \to \KQ_S .
\]
\item
The hermitian $\K$-theory spectrum is absolute: for any map $T \to S$ of qcqs schemes the natural map $f^*\KQ_S \to \KQ_T$ is an equivalence, see Proposition~\ref{proposition:KQ-stable-under-basechange}. 
\item
The Thom isomorphism, namely
\[
\Om^{V}\KQ_L \simeq \KQ_{L \otimes \det V[-r]}
\]
for any rank $r$ vector bundle $V \to S$, see Proposition~\ref{proposition:thom-iso}.
\item
If $S$ is Noetherian of finite Krull dimension then the motivic ring spectrum $\KQ_S$ is \emph{absolutely pure}, see Theorem~\ref{theorem:purity}.
\end{enumerate}

\subsection{Motivic hermitian K-theory spectra}
\label{subsection:construction-KQ}%

Fix a qcqs scheme $S$. Recall the lax symmetric monoidal functor
\[
\M^{\sym}_S\colon \Fun^{\kloc}(\Catp,\Spa) \to \SH(S,\Spa))
\]
and its line bundle generalisation $\M^{\sym}_L$ from Construction~\ref{construction:R-from-kloc}, and the fact that $\KK$, $\KGW$ and $\KL$ are $\Einf$-algebra objects in $\Fun^{\kloc}(\Catp,\Spa)$ with respect to Day convolution, that is, each carries a canonical lax symmetric monoidal structure, see~\cite[\S 3]{9-authors-IV}, and both the universal natural transformations $\KGW \to \KK$ and $\KGW \to \KL$ are lax symmetric monoidal.

\begin{definition}
We define the motivic $\Einf$-ring spectra $\KGL_S$, $\KQ_S$ and $\KW_S, \in \SH(S)$ by the formulas
\[
\KGL_S = \M^{\sym}_S(\KK),
\qquad
\KQ_S = \M^{\sym}_S(\KGW)
\quad\text{and}\quad 
\KW_S = \M^{\sym}_S(\KL) .
\]
Given an invertible perfect complex with $\Ct$-action $L \in \Picspace(S)^{\BC}$, we also define
the motivic spectra $\KGL_L,\KQ_L, \KW_L\in \SH(S)$ by the formulas
\[
\KGL_L = \M^{\sym}_L(\KK),
\qquad
\KQ_L = \M^{\sym}_L(\KGW)
\quad\text{and}\quad 
\KW_L = \M^{\sym}_L(\KL) .
\]
By construction of $\M^{\sym}_L$, these are naturally modules over the $\Einf$-ring spectra $\KGL_S$, $\KQ_S$ and $\KW_S$, respectively.
\end{definition}

\begin{remark}
\label{remark:action-on-k}%
Since the functor $(\C,\QF) \mapsto \KK(\C)$ does not depend on $\QF$ we have that $\KGL_L$ does not depend on $L$. The reason why we did choose to use this notation is that, as a functor on $\Catp$, $\KK$ carries a natural $\Ct$-action (where for a given Poincaré $\infty$-category $(\C,\QF)$ the $\Ct$-action on $\KK(\C)$ is induced by the duality). As a result, each of the motivic spectra $\KGL_L = \M^{\sym}_L(\KK)$ acquires an induced $\Ct$-action. This $\Ct$-action does depend on $L$, and so the notation $\KGL_L$ allows one to record which of these implicit $\Ct$-actions is being considered. 
\end{remark}

\begin{remark}
When $L = \cO_S$, we have $\KGL_L = \KGL_S$, $\KQ_L=\KQ_S$ and $\KW_L = \KW_S$.
\end{remark}

Since $\KGL_L$ does not depend on $L$ (see Remark~\ref{remark:action-on-k}) 
we have by Example~\ref{example:bott} an equivalence
\[
\beta\colon \KGL_L \xrightarrow{\simeq} \KGL_{L[-1]} = \Om_{\Pone}\KGL_L ,
\]
which we call the Bott map. Similarly, by Example~\ref{example:bott} and Ranicki periodicity \cite[Corollary 3.5.16]{9-authors-I} the motivic spectra $\KQ_L$ and $\KW_L$ carry periodicity equivalences of the form
\[
\wtl{\beta}\colon \KQ_L \xrightarrow{\simeq} \KQ_{L[-4]} = \Om_{\Pone}^4\KQ_L \quad\text{and}\quad \ovl{\bet}\colon \KW_L \xrightarrow{\simeq} \KW_{L[-4]} = \Om_{\Pone}^4\KW_L .
\]
for every $L \in \Picspace(S)^{\B C}$. To avoid confusion, let us note that the map $\beta$ is not compatible with the $\Ct$-action on $\KGL_L$ determined by $L$ (see Remark~\ref{remark:action-on-k}), but its 4-fold iteration $\KGL_L \xrightarrow{\beta^4} \Om^4_{\Pone}\KGL_L$ is $\Ct$-equivariant, again by Ranicki periodicity (and the fact that double suspension induces on K-theory the same map as the identity).

\begin{remark}
\label{remark:regular-base}%
If $S$ is a regular Noetherian scheme of finite Krull dimension then by Theorem~\ref{theorem:A1-invariance} the pre-motivic spectrum $\R^{\sym}_L(\GW)$ is $\Aone$-invariant, and hence already a motivic spectrum. In this case $\R^{\sym}_L(\KGW)$ and $\M^{\sym}_L(\KGW)$ coincide, so that we may identify $\KQ_L$ with $\R^{\sym}_L(\KGW)$. 
\end{remark}

\begin{remark}
\label{remark:either-way}%
Let $\KGW^{\sym}\colon \Catpsi \to \Grp_{\Einf}$ be the functor given by 
\[
\KGW^{\sym}(\C,\Dual) = \KGW(\C,\QF^{\sym}_{\Dual}),
\]
and write $\GWspace^{\sym} := \Om^{\infty}\KGW^{\sym}$. Then 
$\KGW^{\sym}$ is a bounded localising functor on $\Catpsi$ and so by Lemma~\ref{lemma:unique-delooping}, the canonical map
\[
\iota_*\R_L(\GWspace^{\sym}) \to \R_L(\KGW^{\sym})
\]
is an equivalence in $\Spa_{\Pone}^{\Sig}(\Sh^{\Nis}(S,\Spa))$. We consequently conclude that 
\[
\KQ_L \simeq \Loc_{\Aone}\iota_*\R_L(\GWspace^{\sym}),
\]
that is, it can equivalently be described as the image of $\R_L(\GWspace^{\sym})$ under the reflection functor
\[
\Spa^{\Sig}_{\Pone}(\Sh^{\Nis}(S,\Grp_{\Einf})) \to \Spa_{\Pone}^{\Sig}(\Hinv(S,\Spa)) = \SH(S,\Spa).
\]
A similar statement holds for $\KGL_L$ and $\KW_L$ (though we note that in the case of Karoubi L-theory, the underlying infinite loop space $\Om^{\infty}\KL$ is quite different from the usual L-theory space $\mathcal{L}$, even when evaluated on idempotent complete Poincaré $\infty$-categories, see~\cite[\S 1]{9-authors-IV}). 
\end{remark}

\begin{notation}
Given a scheme $X$ and $L \in \Picspace(X)^{\B C}$, define
\[
\HGW^{\sym}(X,L) := |\KGW^{\sym}(X \times \Del^{\bullet},p_\bullet^*L)|,
\quad
\HL^{\sym}(X,L) := |\HL^{\sym}(X \times \Del^{\bullet},p_\bullet^*L)|,
\]
and
\[
\HK(X) :=  |\KK(X \times \Del^{\bullet})| ,
\]
where $\Del^n$ is the algebraic $n$-simplex and $p_n\colon X \times \Del^n \to X$ is the projection. When $L$ is the structure sheaf, we also abbreviate $\HGW^{\sym}(X,\cO_X)$ as $\HGW^{\sym}(X)$ and $\HL^{\sym}(X,\cO_X)$ as $\HL^{\sym}(X)$. In particular, $\HK$ is homotopy $\K$-theory, and similarly $\HGW^{\sym} = \Loc_{\Aone}\KGW^{\sym}$ and $\HL^{\sym} = \Loc_{\Aone}\KL^{\sym}$ are the $\Aone$-invariant replacements of the Nisnevich sheaves $\KGW^{\sym}$ and $\KL^{\sym}$, respectively. We refer to them as \defi{homotopy $\KGW$-theory} and \defi{homotopy $\KL$-theory}, respectively.
\end{notation}

From Remark~\ref{remark:underlying-sym} we immediately get:
\begin{theorem}
\label{theorem:values-of-KQ}%
Let $p\colon X \to S$ be a smooth $S$-scheme and $L \in \Picspace(S)^{\BC}$. Then for a finite set $I$ there are natural equivalences
\begin{alignat*}{2}
& \map_{\SH(S)}(\Sigma^\infty_{\Pone} X_+,\Sigma^{I}_{\Pone}\KGL_S) &\, & \simeq \HK(X), \\
& \map_{\SH(S)}(\Sigma^\infty_{\Pone} X_+,\Sigma^{I}_{\Pone}\KQ_L) &\, & \simeq  \HGW^{\sym}(X,p^*L[I]),  \quad \text{and} \\
& \map_{\SH(S)}(\Sigma^\infty_{\Pone} X_+,\Sigma^{I}_{\Pone}\KW_L) &\, & \simeq  \HL^{\sym}(X,p^*L[I]).
\end{alignat*}
In particular, $\KGL_S$ coincides with the usual algebraic $\K$-theory motivic spectrum 
in a manner that identifies $\bet$ with the usual motivic Bott map. 
\end{theorem}

Combining Theorem~\ref{theorem:values-of-KQ} with Theorem~\ref{theorem:A1-invariance}, we furthermore obtain:

\begin{corollary}
\label{corollary:values-of-KQ-regular}%
In the situation of Theorem~\ref{theorem:values-of-KQ}, if $X$ is regular Noetherian of finite Krull dimension then the equivalences of that corollary become
\begin{alignat*}{2}
& \map_{\SH(S)}(\Sigma^\infty_{\Pone} X_+,\Sigma^I_{\Pone}\KGL_S) &\, & \simeq \K(X), \\
& \map_{\SH(S)}(\Sigma^\infty_{\Pone} X_+,\Sigma^I_{\Pone}\KQ_L) &\, & \simeq \GW^{\sym}(X,p^*L[I]), \quad \text{and} \\
& \map_{\SH(S)}(\Sigma^\infty_{\Pone} X_+,\Sigma^I_{\Pone}\KW_L) &\, & \simeq \L^{\sym}(X,p^*L[I]). 
\end{alignat*}
\end{corollary}

Moving on, let us exploit a bit more the properties of the motivic realization functor $\M^{\sym}_S$. In particular, applying it to the lax symmetric monoidal transformation $\KGW \to \KL$ we obtain a map of $\Einf$-motivic spectra
\[
w_S\colon \KQ_S \to \KW_S.
\]
Similarly, applying $\M^{\sym}_L$ for $L \in \Picspace(S)^{\B C}$ we obtain a $\KQ_S$-module map
\[
w_L\colon \KQ_L \to \KW_L.
\]
Concerning the relation between hermitian and algebraic K-theory, let us write
\[
\KGL_L \xrightarrow{h_{L}} \KQ_L \xrightarrow{f_L} \KGL_L
\]
for the maps of motivic spectra obtained by applying $\M^{\sym}_L$ to the natural transformations
\[
\KK \stackrel{\hyp}{\to} \KGW \stackrel{\fgt}{\to} \KK,
\]
and set $h_S := h_{\cO_S}$ and $f_S := f_{\cO_S}$.
Since $\fgt$ is a lax symmetric monoidal transformation we have that $f_S$ is a map of motivic $\Einf$-ring spectra. Similarly, $f_L$ and $h_L$ are $\KQ_S$-module maps. In addition, since the functor $\KK$ carries a $\Ct$-action such that the natural transformations $\hyp$ and $\fgt$ are $\Ct$-equivariant (with respect to the trivial action on $\KGW$), we have that the same structure is inherited at the level of motivic spectra, that is, $\KGL_L$ carries a $\Ct$-structure such that $f_L$ and $h_L$ are $\Ct$-equivariant (see Remark~\ref{remark:action-on-k}). In particular, $h_L$ and $f_L$ induce maps
\[
\left(\KGL_L\right)_{\hC} \to \KQ_L \to \left(\KGL_L\right)^{\hC}.
\]
 
\begin{corollary}[The Tate sequence]
The map $w_L$ and the map induced by $h_L$ fit into an exact sequence
\[
\left(\KGL_L\right)_{\hC} \to \KQ_L \xrightarrow{w_L} \KW_L .
\]
\end{corollary}
\begin{proof}
Apply $\M^{\sym}_L$ to the Karoubi-localising fundamental exact sequence $\KK_{\hC} \Rightarrow \KGW \Rightarrow \KL$  (see \S\ref{subsection:KGW-recall}) and use the fact that $\M^{\sym}_L$ preserves colimits when the target is stable by Lemma~\ref{lemma:colimits}.
\end{proof}

\begin{corollary}[Wood sequence]
The maps $f_L\colon \KQ_L \to \KGL_L$ and $\KGL_L \simeq \Sig^{\infty}_{\Pone}\KGL_L \xrightarrow{\Sig^{\infty}_{\Pone} h_L} \Sig^{\infty}_{\Pone}\KQ_L$ fit into a fibre sequence
\[
\KQ_L \to \KGL_L \to \Sig^{\infty}_{\Pone}\KQ_L .
\]
\end{corollary}
\begin{proof}
Apply $\M^{\sym}_L$ to the Karoubi-localising Bott-Genauer sequence $\KGW \Rightarrow \KK \Rightarrow \KGW((-)\qshift{1})$
(see \S\ref{subsection:KGW-recall}) and use Bott periodicity (Example~\ref{example:bott}) to identify the fibre.
\end{proof}

\begin{remark}
It follows from the projective bundle formula of Theorem~\ref{theorem:projective-bundle-formula} and the Bott periodicity of Example~\ref{example:bott} that the motivic spectra $\Om_{\PP^2}\KQ_L = \fib[\KQ_L^{\PP^2} \to \KQ_L]$ and $\KGL_L$ coincide as $\Pone$-spectra in Nisnevich sheaves, and hence also as motivic spectra. Unwinding the definitions in Theorem~\ref{theorem:projective-bundle-formula}, we see that under this equivalence, the map
\[
\Om_{\PP^2}\KQ_L \to \Om_{\Pone}\KQ_L
\]
induced by a linear hyperplane inclusion $\Pone \hrar \PP^2$ coincides with the composite
\[
\KGL_L \stackrel[\simeq]{\beta}{\to} \Om_{\Pone}\KGL_L \xrightarrow{\Om_{\Pone}h_L} \Om_{\Pone}\KQ_L .
\]
It is known that the cofibre of $\Sig^{\infty}_{\Pone}\Pone \hrar \Sig^{\infty}_{\Pone}\PP^2$ in $\SH(\ZZ)$ is equivalent to $\Sig^{\infty}_{\Pone}(\Pone \wedge \Pone)$, and the boundary map $\Om\Sig^{\infty}_{\Pone}(\Pone \wedge \Pone) \to \Sig^{\infty}_{\Pone}\Pone$ of this cofibre sequence determines, after $\Pone$ delooping, a map
\[
\eta\colon \Sig^{\infty}_{\Pone}\Gm \to \SS ,
\]
often denoted by $\eta$. We hence conclude that the boundary map
\[
\KQ_L \to \Sig\Om_{\Pone}\KQ_L = \Om_{\Gm}\KQ_L
\]
of the Wood fibre sequence is induced by $\eta$.
\end{remark}

\subsection{Pullback invariance of hermitian K-theory}
\label{subsection:pullback-invariance}%

\begin{construction}
\label{construction:base-change}%
Given a map $f\colon T \to S$ of qcqs schemes, the commutative square of symmetric monoidal functors
\[
\begin{tikzcd}[column sep=40pt]
(\Sm_S\op)^{\otimes} \ar[d]\ar[rrr, "{T \times_S (-)}"] &&& (\Sm_T\op)\op \ar[d]\\
\Mod_S(\Catpsi)^{\otimes} \ar[rrr, "{(\Dperf(T),\Dual_T) \otimes_{(\Dperf(S),\Dual_S)} (-)}"]  &&& \Mod_T(\Catpsi)^{\otimes}.
\end{tikzcd}
\]
constructed at the end of \S\ref{subsection:more-derived} determines, through the constructions of \S\ref{subsection:nisnevich-invariants} and \S\ref{subsection:bott-periodicity}, a commutative square of symmetric monoidal functors
\[
\begin{tikzcd}[column sep=40pt]
\Spa_{\Pone}^{\Sig}(\Sh^{\Nis}(S,\Spa))^{\otimes} \ar[d, "{\wtl{\T}_S^{\otimes}}"']\ar[r, "f^*"] & \Spa_{\Pone}^{\Sig}(\Sh^{\Nis}(T,\Spa))^{\otimes} \ar[d, "{\wtl{\T}_T^{\otimes}}"]\\
\Fun^{\bloc}(\Mod_S(\Catpsi),\Spa)^{\otimes} \ar[r]  & \Fun^{\bloc}(\Mod_T(\Catpsi),\Spa)^{\otimes}
\end{tikzcd}
\]
where the bottom horizontal map is given by left Kan extension along $(\Dperf(T),\Dual_T) \otimes_{(\Dperf(S),\Dual_S)} (-)$ followed by the localisation functor
\[
\Lbloc\colon \Fun(\Mod_S^{\flat,\kappa}(\Catpsi), \E) \to \Fun^{\bloc}(\Mod_S^{\flat,\kappa}(\Catpsi), \E).
\]
Passing to right adjoints we obtain a commutative square of lax symmetric monoidal functors 
\begin{equation}
\label{equation:realization-base-change}%
\begin{tikzcd}[column sep=40pt]
\Spa_{\Pone}^{\Sig}(\Sh^{\Nis}(S,\Spa))^{\otimes} & \Spa_{\Pone}^{\Sig}(\Sh^{\Nis}(T,\Spa))^{\otimes} \ar[l, "f_*"']  \\
\Fun^{\bloc}(\Mod_S^{\flat,\kappa}(\Catpsi),\Spa)^{\otimes} \ar[u, "{\R_S}"]  & \Fun^{\bloc}(\Mod_T^{\flat,\kappa}(\Catpsi),\Spa)^{\otimes} \ar[u,"\R_T"']\ar[l]
\end{tikzcd}
\end{equation}
where the bottom horizontal arrow is now given by pre-composition with $(\Dperf(T),\Dual_T) \otimes_{(\Dperf(S),\Dual_S)} (-)$. 
Finally, note that any bounded localising functor 
$\G \colon \Mod_S^{\flat,\kappa}(\Catpsi) \to \Spa$
can also view as a functor on $\Mod_T^{\flat,\kappa}(\Catpsi)$ via restriction along the forgetful functor from $T$-modules to $S$-modules.
The commutativity of this square then determines a natural equivalence
\[
\R_S(\G;(\Dperf(T),\Dual_T)) \simeq f_*\R_T(\G) .
\]
\end{construction}

Now given a bounded localising functor 
$\G \colon \Mod_S^{\flat,\kappa}(\Catpsi) \to \Spa$
the duality preserving $S$-linear functor $f^*\colon (\Dperf(S),\Dual_S) \to (\Dperf(T),\Dual_T)$ together with the equivalence of Construction~\ref{construction:base-change} determine a map
\[
\R_S(\G) \to \R_S(\G;(\Dperf(T),\Dual_T)) \simeq f_*\R_T(\G),
\]
and consequently an adjoint map
\[
f^*\R_S(\G) \to \R_T(\G) .
\]
When $\G = \F|_{\Mod_S^{\flat,\kappa}(\Catpsi)}$ for some Karoubi localising functor $\F\colon \Catp \to \Spa$ we may rewrite this natural map as
\[
f^*\R^{\sym}_S(\F) \to \R^{\sym}_T(\F) .
\]
Applying the $\Aone$-localisation functor $\Loc_{\Aone}$ and its commutativity with pullbacks we finally get a natural map
\[
f^*\M^{\sym}_S(\F) \to \M^{\sym}_T(\F) .
\]
relating the motivic realization of $\F$ over $T$ with the pullback of the motivic realization of $\F$ over $S$.

Taking $\F=\KGW$ we write
\[
\eta_f\colon f^*\KQ_S \to \KQ_T
\]
for the resulting map. Our goal in this section is the prove the following:
\begin{proposition}
\label{proposition:KQ-stable-under-basechange}%
Let $f\colon T\to S$ be a map of quasi-compact quasi-separated schemes. Then the map
\[
f^*\KQ_S\to \KQ_T
\]
is an equivalence of motivic spectra.
\end{proposition}

Though the motivic spectrum $\KQ$ is built out of symmetric $\GW$-theory, the proof will require passing through genuine $\GW$-theory. For this, recall from Notation~\ref{notation:functor-to-sheaf} that for a functor $\F\colon \Catp \to \E$, an invertible perfect complex with $\Ct$-action $L \in \Picspace(S)^{\BC}$ and integers $n,m \in \ZZ$, we define $\F^{\ge m,[n]}_L\colon \Sm_S\op \to \E$ by the formula
\[
\F^{\ge m,[n]}_L(X) = \F\left(\Dperf(X),(\QF^{\geq m}_L)\qshift{n}\right).
\]
For $m\in\{-\infty,+\infty\}$ we replace the superscripts $\geq -\infty$ and $\geq +\infty$ with $\sym$ and $\qdr$, and for $L=\cO_S$ we use the subscript $S$ instead of $\cO_S$. In addition, when $n=0$ we often drop it from the notation.

Using the genuine refinement of the projective line formula (Proposition~\ref{proposition:beilinson}), we now observe that Bott periodicity
(Example~\ref{example:bott})
can be extended to the case of $\F^{\geq m,[n]}_L$, if one accepts to shift $m$:

\begin{proposition}[Genuine Bott periodicity]
\label{proposition:bott-periodicity}%
If $\F$ is additive, for any $k\in \NN$ and $m\in \ZZ$, there is a canonical equivalence 
\[
\Om_{\Pone}^{2k+i}\F^{\ge m}_L\simeq \F^{\ge m-k,[-i]}_{(-1)^k L} \ .
\]
\end{proposition}
\begin{proof}
By Proposition~\ref{proposition:beilinson} and the additivity of $\F$, we have
\[
\Om_{\Pone}\F^{\ge m}_L(X) \simeq \F^{\ge m,[-1]}_L \simeq \F^{\ge m-1}_{L[-1]}
\]
and hence
\[
\Om^k_{\Pone}\F^{\ge m}_L(X) \simeq \F^{\ge m,[-k]}_L \simeq \F^{\ge m-k}_{L[-k]} .
\]
Combining this with Ranicki periodicity we get
\[
\Om^{2k+i}_{\Pone}\F^{\ge m}_L(X) \simeq \F^{\ge m,[-2k-i]}_{L} \simeq \F^{\ge m-2k,[-i]}_{L[-2k]} \simeq \F^{\ge m-k,[-i]}_{(-1)^kL} ,
\]
as desired.
\end{proof}

\begin{proof}[Proof of Proposition~\ref{proposition:KQ-stable-under-basechange}]

To begin, note that by 2-out-of-3 it will suffice to prove the claim for $S=\spec(\ZZ)$. In particular, we may (and will) assume that $S$ is affine. Second, if $T \to S$ is smooth then the pullback functor $f^*$ on the level of presheaves is implemented by pre-composition with $\Sm_T \xrightarrow{f \circ (-)} \Sm_S$ and hence preserves Nisnevich sheaves and commutes with $\Om_{\Pone}$. The claim can hence be checked on the level of $\Pone$-spectrum objects in presheaves, where it amounts to unwinding the definitions. Covering $T$ by open affines and using again 2-out-of-3 we consequently may (and will) assume that $T$ is affine as well.

We now prove the claim for $f\colon T \to S$ a map of affine schemes. 
As $f_*$ commutes with post-composition with $\Om^{\infty}\colon\Spa \to \Grp_{\Einf}$ we have after passing to left adjoints a commutative square
\[
\begin{tikzcd}
\Spa_{\Pone}^{\Sig}(S,\Grp_{\Einf}) \ar[r,"f^*"]\ar[d, "\iota_*"] & \Spa_{\Pone}^{\Sig}(T,\Grp_{\Einf}) \ar[d,"i_*"] \\
\Spa_{\Pone}^{\Sig}(S,\Spa) \ar[r, "f^*"] & \Spa_{\Pone}^{\Sig}(T,\Spa) \ .
\end{tikzcd}
\]
Combining this with Remark~\ref{remark:either-way} we conclude that to prove the desired claim it is enough to show that the map
\[
f^*\R_S(\GWspace^{\sym}) \to \R_T(\GWspace^{\sym})
\]
is an equivalence. We note that on the left hand side, the pullback is implemented by taking the pointwise pullback of Nisnevich sheaves and then spectrify the resulting pre-spectrum.
We now claim that, if we let $\X \in \PSpa^{\Sig}_{\Pone}(\Sh^{\Nis}(T,\Grp_{\Einf}))$ denote the pullback of $\R_S(\GWspace^{\sym})$ as a $\Pone$-prespectrum, then the induced map
\[
\sh_{\infty}\Om^{\infty}_{\Pone}\X \to \sh_{\infty}\Om^{\infty}_{\Pone}\R_T(\GWspace^{\sym}) \simeq \R_T(\GWspace^{\sym})
\]
is an equivalence (see \S\ref{subsection:recollection-stabilization} for the notation $\sh_{\infty}\Om^{\infty}_{\Pone}$). This simultaneously implies that the $\Pone$-prespectrum object $\X$ is semi-stable and that its spectrification is $\R^{\sym}_T(\GWspace)$.

To prove our claim, note that by Example~\ref{example:bott} we have that for a finite set $I$ of size $i=|I|$, the $I$'th object in the symmetric $\Pone$-spectrum $\R_S(\GWspace^{\sym})$ is given by the Nisnevich sheaf $\GWspace^{\sym,[i]}_S$, and so the $I$'th object in $\X$ is given by $f^*\R_S(\GWspace^{\sym,[i]})$ (pullback of Nisnevich sheaves of $\Einf$-groups).
We need to show that for every $i \geq 0$ the map
\[
\colim_n \Om_{\Pone}^nf^*\GWspace_S^{\sym,[i+n]} \to \GWspace_T^{\sym,[i]}
\]
is an equivalence of Nisnevich sheaves on $T$.
By cofinality the colimit on $n$ can be taken in jumps of $4$, that is, we can write this map as
\[
\colim_n \Om_{\Pone}^{4n-i}f^*\GWspace_S^{\sym,[4n]} \to \GWspace_T^{\sym,[i]} .
\]
At the same time, by Ranicki periodicity we have that $\GWspace_S^{\sym,[4n]} = \GWspace_S^{\sym}$ and so we can rewrite this further as
\[
\colim_n \Om_{\Pone}^{4n-i}f^*\GWspace_S^{\sym} \to \GWspace_T^{\sym,[i]} .
\]
We now fix an $i \geq 0$ and prove that this map is an equivalence for $i$. 

Let $Q$ be the poset whose objects are pairs $(n,m)$ with $n,m$ non-negative integers equipped with the order relation such that $(n,m) \leq (n',m')$ if and only if $n \leq n'$ and $n+m \leq n'+m'$. In particular, the projection $Q \to \NN$ sending $(n,m)$ to $n$ is a cartesian fibration such that for each $n'$ the transition functor $Q_{n+1} = \NN  \to \NN = Q_{n}$ is given by $x \mapsto x+1$. 
Specifying a functor out of $Q$ then amounts to specifying for each $n$ a functor $f_n$ out of $Q_n = \NN$ together with natural transformations $\eta_n\colon f_{n-1}((-)+1) \Rightarrow f_{n}(-)$ for $n \geq 1$. In particular, we may consider the functor $\theta\colon Q \to \Sh^{\Nis}(S, \Grp_{\Einf})$ given by 
\[
\theta(n,m) = \Om^{4n-i}_{\Pone}f^*\GWspace^{\geq -2m}_S,
\]
where, using Proposition~\ref{proposition:bott-periodicity}, the maps
\[
\eta_n\colon \Om^{4(n-1)-i}_{\Pone}f^*\GWspace^{\ge -2((-)+1)}_S = \Om^{4(n-1)-i}_{\Pone}f^*\Om^{4}_{\Pone}\GWspace^{\ge -2(-)}_S \to \Om^{4n-i}_{\Pone}f^*\GWspace^{\ge -2(-)}_S
\]
are induced by the Beck-Chevalley map $f^* \circ \Om^4_{\Pone} \Rightarrow \Om^4_{\Pone} \circ f^*$. 
Now consider the commutative diagram in $\Sh^{\Nis}(T,\Grp_{\Einf})$ given by
\[
\begin{tikzcd}
\colim_n \Om^{4n-i}_{\Pone}f^*\GWspace_S^{\ge0} \ar[r]\ar[d, "\simeq"'] & \colim_n \Om^{4n-i}_{\Pone}\GWspace_T^{\ge 0} \ar[d, "\simeq" ] \ar[r,equal] & \colim_n \GWspace_T^{\ge -2n,[i]} \ar[d, "\simeq" ] \\
\displaystyle\mathop{\colim}_{(n,m) \in Q} \Om^{4n-i}_{\Pone}f^*\GWspace^{\ge -2m}_S \ar[r] & \displaystyle\mathop{\colim}_{(n,m) \in Q}\Om^{4n-i}_{\Pone}\GWspace^{\ge -2m}_T\ar[r,equal] & \displaystyle\mathop{\colim}_{(n,m) \in Q}\GWspace^{\ge -2n-2m,[i]}_T \\
\ar[d,"\simeq"']\displaystyle\mathop{\colim}_{(n,m) \in \NN \times \NN} \Om^{4n-i}_{\Pone}f^*\GWspace^{\ge -2m}_S \ar[r]\ar[u,"\simeq"] &\displaystyle\mathop{\colim}_{(n,m) \in \NN \times \NN}\Om^{4n-i}_{\Pone}\GWspace^{\ge -2m}_T \ar[r,equal]\ar[u,"\simeq"'] & \displaystyle\mathop{\colim}_{(n,m) \in \NN \times \NN}\GWspace^{\ge -2n-2m,[i]}_T\ar[u,"\simeq"'] \ar[d,"\simeq"]\\
\colim_{n} \Om^{4n-i}_{\Pone}f^*\GWspace^{\sym}_S \ar[rr] && \GWspace^{\sym,[i]}_T 
\end{tikzcd}
\]
where the horizontal equivalences on the right column are by Proposition~\ref{proposition:bott-periodicity},
the vertical maps in the top row are induced by restriction to the sub-poset $\NN \times \{0\} = \{(n,m) \in Q | m=0\} \subseteq Q$, and the vertical maps in the middle row are induced by restriction along the poset map $\NN \times \NN \to Q$ which is the identity on objects. It is straightforward to verify that that both these functors are cofinal, so that the associated vertical maps are equivalences. In addition, the bottom vertical maps are equivalences as well: indeed, for every $j$ the sequence of Poincaré structures $(\QF^{\geq k})\qshift{j}$ converges to $(\QF^{\sym})\qshift{j}$ as $k \to -\infty$, while $\GWspace$ (as a functor on $\Catp$) and $\Om_{\Pone}$ (as a functor on $\Sh^{\Nis}(S,\Grp_{\Einf})$) both preserve filtered colimits (where we note that filtered colimits in Nisnevich sheaves are computed pointwise since the Nisnevich site is finitary). To show that the bottom horizontal map is an equivalence it will hence suffice to show that the top horizontal map is an equivalence. We have thus reduced to proving that the induced map
\[
f^*\GWspace_S^{\ge0}  \to \GWspace_T^{\ge0}
\]
is an equivalence of Nisnevich sheaves. Now as a Nisnevich sheaf valued in $\Einf$-groups, 
$\GWspace_S^{\ge0}$ is the group completion of the $\Mon_{\Einf}$-valued sheaf $\Vect^{\sym}_S$ of vector bundles equipped with perfect symmetric bilinear forms, and the same holds for $\GWspace_T^{\ge0}$; indeed, this statement can be verified on the level of stalks, and for affine schemes the group completion statement is proven in \cite{Hebestreit-Steimle}*{Theorem~A}.
Since $f^*$ commutes with group completion (again as above this follows from the fact that $f_*$ preserves the group-like property) it will now suffice to show that the map
\[
f^*\Vect^{\sym}_S\to \Vect^{\sym}_T
\]
is a motivic equivalence of $\Mon_{\Einf}$-valued sheaves. In fact, it is a Zariski-local equivalence: since the $\Einf$-valued presheaf $T' \mapsto \Vect^{\sym}(T')$ on $\Aff_S$ is represented by a smooth algebraic stack with affine diagonal it is left Kan extended from smooth affine $S$-schemes by \cite{deloop3}*{Proposition~A.0.4} (see also~\cite[Example A.0.6 (5)]{deloop3}), and so the above map is an equivalence when evaluated on affine $T$-schemes.
\end{proof}

\subsection{The Thom isomorphism}
\label{subsection:thom-isomorphism}%

Let $S$ be qcqs scheme and $p\colon V \to S$ a vector bundle over $S$ of rank $r$ with zero section $s\colon S \to V$.
Recall \cite{Morel-Voevodsky}*{Definition~2.16} that the Thom space $\Th(V)$ of $V$ is the quotient $V/V\smallsetminus\{0\}$, considered as a pointed motivic space over $S$ (that is, as the image of the corresponding pointed sheaf under the $\Aone$-localisation functor $\Sh^{\Nis}(S,\Sps_*) \to \Hinv(S,\Sps_*)$). 
Since $\SH(S)$ is tensored and cotensored over pointed motivic spaces over $S$ we may consider for every $E \in \SH(S)$ the tensor and cotensor 
\[
\Sig^V E = \Th(V) \otimes E \quad\text{and}\quad \Om^V E = E^{\Th(V)}
\]
in $\SH(S)$. We note that we can also describe $\Sig^VE$ and $\Om^VE$ by the formulas
\[
\Sig^V E = p_{\sharp}s_*E \quad\text{and}\quad \Om^VE = s^!p^*E
\]
using the six-functor formalism.

\begin{proposition}[Thom isomorphism]
\label{proposition:thom-iso}%
Let $S$ and $p\colon V \to S$ be as above and $L\in \Picspace(S)^{\BC}$ a perfect invertible complex with $\Ct$-action. 
Then there is a natural equivalence of motivic spectra over $S$
\[
\Om^{V}\KQ_L \simeq \KQ_{L \otimes \det V[-r]},
\]
and hence natural equivalences of spectra
\[
\map(\Sigma^\infty_{\Pone} \Th(V), \Sigma_{\Pone}^n\KQ_L) \simeq \map(\Sigma^\infty_{\Pone}S, \Sigma_{\Pone}^n\Om^{V}\KQ_L)\simeq \HGW^{\sym}(S,L \otimes \det V[n-r]).
\]
In particular, if $S$ is regular Noetherian of finite Krull dimension then
\[
\map(\Sigma^\infty_{\Pone} \Th(V), \Sigma_{\Pone}^n\KQ_L)\simeq \KGW^{\sym}(S,L \otimes \det V[n-r]).
\]
\end{proposition}
\begin{proof}
Let $\KGW^{\sym}\colon \Catpsi \to \Grp_{\Einf}$ be the bounded localising functor given by $\KGW^{\sym}(\C,\Dual) = \KGW(\C,\QF^{\sym}_{\Dual})$. As in Lemma~\ref{lemma:LAoneOmPone}, the $\Aone$-localisation functor
$\Loc_{\Aone}\colon \Spa^{\Sig}_{\Pone}(S,\Spa) \to \SH(S)$
commutes with taking cotensors by $\Th(V)$,
and so by the last part of Construction~\ref{construction:realization-cotensor} we thus have
\begin{align*} 
\Om^V\KQ_L = \Om^V\M^{\sym}_L(\KGW) = \Om^V\M_L(\KGW^{\sym}) =& \Om^V\Loc_{\Aone}\R_L(\KGW^{\sym}) \\
\simeq& \Loc_{\Aone}\big(\R_L(\KGW^{\sym})^{\Th(V)}\big) \\
\simeq& \Loc_{\Aone}\big(\R_L\big((\KGW^{\sym})^{\T_S(\Th(V))}\big)\big) \\
=& \M_L((\KGW^{\sym})^{\T_S(\Th(V))}).
\end{align*}
To prove the proposition it will hence suffice to prove a natural equivalence 
\[
\T_S(\Th(V)) \simeq \jbloc(\Dperf(S),\Dual_{\det V[-r]})
\]
in $\Mod^{\flat,\kappa}_{S}(\Catpsi)$,
see \S\ref{subsection:bott-periodicity} for the definition of the localised co-Yoneda functor $\jbloc$ .
In particular, we no longer need to keep track of $L$. In addition, using the equivalence $\Th(\cO_S \oplus V)\simeq \Sigma_{\Pone} \Th(V)$ from \cite{Morel-Voevodsky}*{Proposition~3.2.17(1+2)}
and the Bott periodicity statement of Example~\ref{example:bott} we may assume without loss of generality that $r$ is odd.

Let $W = \cO_S \oplus V$ be the direct sum of $V$ and a trivial line bundle. Then there is a natural equivalence of motivic spaces $\Th(V) \simeq \PP_S(W^\vee)/\PP_S(V^\vee)$, see \cite{Morel-Voevodsky}*{Proposition~3.2.17(3)}. 
As in \S\ref{subsection:pbf}, let $\cA\big(\frac{r-1}{2},\frac{-r+1}{2}\big)$ be the stable subcategory of $\Dperf(\PP_S(W^\vee))$ generated by 
\[
\cA\big(\frac{r-1}{2}\big)\cup\cA\big(\frac{r-2}{2}-1\big)\cup \cdots \cup \cA\big(\frac{-r+1}{2}\big),
\]
where $\cA(i)$ is the image of the fully-faithful functor $p^*(-) \otimes \cO(i)\colon \Dperf(S) \to \Dperf(\PP_S(W^\vee))$, so that $\cA\big(\frac{r-1}{2},\frac{-r+1}{2}\big)$ is invariant under the duality $\Dual_{\PP_S(W^{\vee})}$ by Lemma~\ref{lemma:duality-shift}. Since the restriction functor
\[
i^*\colon\PP_S(W^\vee) \to \PP_S(V^{\vee})
\]
sends $\cO(i)$ to $\cO(i)$, \cite{Khan-blow-up}*{Theorem~3.3~(ii)} implies that $i^*\cA\big(\frac{r-1}{2},\frac{-r+1}{2}\big)$ generates $\Dperf(\PP_S(V^{\vee}))$ as a stable $\infty$-category and 
so Lemma~\ref{lemma:proj-bundle-II} and Example~\ref{example:split-is-flat} together imply that the composite duality preserving functor
\[
\big(\cA\big(\frac{r-1}{2},\frac{-r+1}{2}\big),\Dual_{\PP_S(W^\vee)}\big) \hrar \big(\Dperf(\PP_S(W^\vee)), \Dual_{\PP_S(W^\vee)}\big) \xrightarrow{i^*} \big(\Dperf(\PP_S(V^{\vee}), \Dual_{\PP_S(V^\vee)}\big)
\]
is inverted by any bounded localising functor on $\Mod^{\flat,\kappa}_{S}(\Catpsi)$. 
At the same time, by Proposition~\ref{proposition:proj-bundle-I} the inclusion on the left participates in a split Poincaré-Verdier sequence
\[
\big(\cA\big(\frac{r-1}{2},\frac{-r+1}{2}\big),\Dual_{\PP_S(W^\vee)}\big)\to \big(\Dperf(\PP_S(W^\vee)),\Dual_{\PP_S(W^\vee)}\big)\xrightarrow{p_*(-\otimes\cO(-\frac{r+1}{2}))} \big(\Dperf(S),\Dual_{\det V^\vee[-r]}\big).
\]
We hence obtain the desired equivalence as a composite
\begin{align*}
\jbloc\big(\Dperf(S),\Dual_{\det V^\vee[-r]}\big) \simeq& \fib\big[\jbloc\big(\Dperf(\PP_S(W^\vee)),\Dual_{\PP_S(W^\vee)}\big) 
\to \jbloc\big(\cA\big(\frac{r-1}{2},\frac{-r+1}{2}\big),\Dual_{\PP_S(W^\vee)}\big)\big] \\
\simeq& \cof\big[\jbloc\big(\Dperf(\PP_S(V^{\vee}), \Dual_{\PP_S(V^\vee)}\big)
\to \jbloc\big(\Dperf(\PP_S(W^\vee), \Dual^{\sym}_{\PP_S(W^\vee)}\big)\big] \\
\simeq& \wtl{\T}_S(\Sig^{\infty}_{\Pone}(\PP_S(W^\vee)/\PP_S(V^\vee)))\\
\simeq& \wtl{\T}_S(\Th(V)).
\end{align*}
\end{proof}

\subsection{Purity of the hermitian K-theory spectrum}

In their work~\cite{deglise-fundamental}, Déglise, Jin and Khan constructed what they call a \emph{system of fundamental classes} in motivic homotopy theory. Given a fixed quasi-compact quasi-separated base scheme $S$, these consists of a compatible collection of maps
\[
\eta_f\colon \SS_{X} \to \Th(T_f)^{-1} \otimes f^!\SS_{Y}
\]
for every smoothable lci map $f\colon X \to Y$ of (separated finite-type) $S$-schemes with tangent complex $T_f \in \Dperf(X)$, where $\SS_{X}$ and $\SS_{Y}$ are the sphere spectra in $\SH(X)$ and $\SH(Y)$, that is, the symmetric monoidal units thereof, respectively. 
Here, smoothable lci means that $f$ can be factored as a composite $X \xrightarrow{i} Y'\xrightarrow{p} Y$ where $i$ is a regular embedding and $p$ is smooth. In this case, the tangent complex $T_f$ is perfect and sits in an exact sequence
\[
T_f \to i^*T_p \to N_X Y'
\]
where $T_p$ is the relative tangent bundle of $p$ and $N_XY'$ is the normal bundle of $X$ in $Y'$.
The associated Thom object $\Th(T_f) \in \SH(X)$ is then by definition the ratio $\Th(i^*T_p) \otimes \Th(N_X Y')^{-1}$ (which does not depend on the choice of factorization). 
Given a motivic spectrum $E \in \SH(S)$ and a smoothable lci map $f\colon X \to Y$ one may then define an $E$-valued variant $\eta^E_f$ of $\eta_f$ as the composite
\[
E_X \xrightarrow{E_X \otimes \eta_{f}} E_X \otimes \Th(T_f)^{-1} \otimes f^!\SS_Y \to \Th(T_f)^{-1} \otimes f^!E_Y ,
\]
where $E_X$ and $E_Y$ denote the pullback of $E$ to $X$ and $Y$, respectively.
The maps $\eta^E_f$ are by construction natural in $E$. In addition, the satisfy two types of compatibility conditions:
\begin{enumerate}
\item
\label{item:composition}%
Composition: if $X \xrightarrow{f} Y \xrightarrow{g} Z$ are a composable pair of smoothable lci morphisms among separated finite type $S$-schemes then $h := g\circ f$ is smoothable lci with $\Th(T_{h}) \simeq \Th(T_f) \otimes f^*\Th(T_g)$ and $\eta^E_{g \circ f}$ is given by the composite
\begin{align*}
E_{X} \xrightarrow{\eta^E_f} \Th(T_f)^{-1} \otimes f^!E_{Y} \xrightarrow{\Th(T_f)^{-1} \otimes f^!(\eta^E_{g})} & \Th(T_f)^{-1} \otimes f^!(\Th(T_g) \otimes g^!E_Z) \\
=& \Th(T_f)^{-1} \otimes f^*\Th(g) \otimes f^!g^!E_Z \\
=& \Th(T_h)^{-1} \otimes h^!E_{Z} 
\end{align*}
where the second equivalence is because $\Th(T_f)$ is tensor invertible. 
\item
\label{item:tor}%
Transverse base change: for every Tor-independent square 
\[
\begin{tikzcd}
X \ar[r, "f"]\ar[d, "g"] & Y \ar[d, "g'"] \\
W \ar[r, "f'"] & Z 
\end{tikzcd}
\]
among separated finite type $S$-schemes for which the horizontal maps are smoothable lci the maps $\eta^E_f$ and $\eta^E_{f'}$ fit into a commutative diagram
\[
\begin{tikzcd}
g^*E_W \ar[d,equal]\ar[r,"{g^*\eta^E_{f'}}"] &  g^*\Th(T_{f'})^{-1} \otimes g^*(f')^!E_Z   \\
g^*E_W \ar[d,equal]\ar[r] &  \Th(T_{f})^{-1} \otimes g^*(f')^!E_Z  \ar[u,"\simeq"]\ar[d] \\
E_X \ar[r, "\eta^E_f"] & \Th(T_f)^{-1} \otimes f^!E_Y 
\end{tikzcd}
\]
where the right vertical arrows are induced by the map $T_f \to g^*T_{f'}$ (which is an equivalence for Tor independent squares) and the Beck Chevalley map $g^*(f')E_Z \to f^!(g')^*E_Z = f^!E_Y$.
\end{enumerate}

Given a smoothable lci map $X \to Y$ among separated finite type $S$-schemes and a motivic spectrum $E \in \SH(S)$, we say that $E$ is \defi{pure at $f$} if the map $\eta^E_f$ is an equivalence. We say that $E$ is \defi{absolutely pure} if it is pure at $f$ for every smoothable lci morphism of finite type separated $S$-schemes
\[
\begin{tikzcd}
X \ar[rr]\ar[dr] && Y \ar[dl] \\
& S &
\end{tikzcd}
\]
such that $X$ and $Y$ are both regular. We note that for any $E$, the collection $P_E$ of maps at which $E$ is pure satisfies the following properties:
\begin{enumerate}
\item
\label{item:smooth}%
The collection $P_E$ contains all smooth maps (see~\cite[Example 2.3.4]{deglise-fundamental}).
\item
\label{item:compose-cancel}%
If $X \xrightarrow{f} Y \xrightarrow{g} Z$ are a pair of composable smoothable lci morphisms among separated finite type $S$-schemes and $g$ belongs to $P_E$ then $f$ belongs to $P_E$ if and only if $g \circ f$ belongs to $P_E$. This follows immediately from Property~\ref{item:composition} concerning compatibility with composition. Combined with the previous point this means in particular that any smoothable lci morphism between \emph{smooth} $S$-schemes is in $P_E$.
\item
\label{item:local}%
If $f\colon X \to Y$ and $Y = \cup_i U_i$ is a Zariski open covering then $f$ belongs to $P_E$ if and only if the base change $f_i\colon V_i := X \times_Y U_i \to U_i$ belongs to $P_E$ for every $i$. To see this, note that, on the one hand, equivalences in $\SH(X)$ are detected locally, and on the other hand, the square
\[
\begin{tikzcd}
V_i \ar[r, "f_i"]\ar[d] & U_i \ar[d] \\
X \ar[r, "f"] & Y 
\end{tikzcd}
\]
is Tor independent and the Beck-Chevalley map $f^!(-)|_{V_i} \Rightarrow f_i^!((-)|_{U_i})$ is an equivalence, so that Property~\ref{item:tor} implies that $\eta^E_{f_i} = (\eta^E_f)|_{V_i}$. 
\end{enumerate}

Combining the above we deduce the following:
\begin{lemma}
\label{lemma:reduction}%
A motivic spectrum $E \in \SH(S)$ is pure if and only if it is pure at $f$ whenever 
\[
\begin{tikzcd}
\spec(A/a) \ar[rr, "f"]\ar[dr] && \spec(A) \ar[dl] \\
& S &
\end{tikzcd}
\]
is a closed embedding of regular affine schemes of finite type over $S$ corresponding to a principal ideal $(a) \subseteq A$ generated by a non-zero divisor $a$. 
\end{lemma}
\begin{proof}
The only if direction is clear since the map $\spec(A/a) \to \spec(A)$ appearing in the lemma is a closed embedding among regular schemes and hence in particular smoothable lci among such. 
For the if direction, note first that the closure properties~\ref{item:smooth} and~\ref{item:compose-cancel} above together imply that $E$ is absolutely pure if and only if it is pure at every (automatically regular) closed embedding of regular separated finite type $S$-schemes. Let now $f\colon X \to Y$ be such a closed embedding. Since $f$ is regular there exists an open cover by affines $Y = \cup_i \spec(A_i)$ and for each $i$ a regular sequence $a_{1,i},...,a_{n_i,i} \in A_i$ such that $X \times_Y \spec(A_i) = \spec(A_i/(a_{1,i},...,a_{n_i,i}))$. In addition, since $Y$ is regular we have that each $\spec(A_i)$ and each $\spec(A_i/(a_{1,i},...,a_{n_i,i}))$ are regular. By~\ref{item:local} we then have that $E$ is absolutely pure if and only if it is pure at closed embeddings of the form $\spec(A/(a_{1},...,a_{n})) \to \spec(A)$ where $A$ is a regular ring and $a_1,...,a_n \in A$ is a regular sequence. Using again the closure under composition Property~\ref{item:compose-cancel} it will suffice to prove this for the case where the regular sequence has length one, and so the desired result follows.
\end{proof}

\begin{theorem}[Purity]
\label{theorem:purity}%
Let $S$ be a Noetherian scheme of finite Krull dimension. Then the motivic spectrum $\KQ_S \in \SH(S)$ constructed in \S\ref{subsection:construction-KQ} is absolutely pure.
\end{theorem}

In order to prove Theorem~\ref{theorem:purity} we first express the fundamental classes $\eta^{\KQ}_S$ in Poincaré categorical terms. This is done via the following construction.

\begin{construction}
\label{construction:fundamental-pushforward}%
Let $A$ be a regular ring, $a \in A$ a non-zero divisor and $r\colon \spec(A/a) \to A$ the associated regular embedding. Since the ideal $(a)$ is principal and non-zero we have an equivalence $r^!A[1] \simeq A/a$. The push-forward Poincaré functor of Corollary~\ref{corollary:lci-pushforward} for $L=A[1]$ then yields the following $A$-linear duality preserving functor 
\[
(r_*,\eta)\colon (\Dperf(A/a),\Dual_{A/a}) \to (\Dperf(A),\Dual_{A[1]}) .
\]
Applying the pre-motivic realization functor $\R^{\sym}_A(\KGW;-) := \R^{\sym}_{\spec(A)}(\KGW;-)$ and using Construction~\ref{construction:base-change} and Remark~\ref{remark:regular-base} we then obtain an induced map
\[
\mu_{r}\colon r_*\KQ_{\spec(A/a)} = \R^{\sym}_{A}(\KGW;(\Dperf(A/a),\Dual_{A/a})) \to \R^{\sym}_{A}(\KGW;(\Dperf(A),\Dual_{A[1]})) = \KQ_{A[1]}.
\]
\end{construction}

\begin{remark}
\label{remark:compatibility}%
Let us take a minute to consider the compatibility of Construction~\ref{construction:fundamental-pushforward} with respect to a ring maps $B \to A$ sending a non-zero divisor element $b \in B$ to a non-zero divisor element $a \in A$. We note that the condition of being non-zero divisors implies that the multiplication maps $a\colon A \to A$ and $b\colon B \to B$ are injective and the map
\[
A \otimes^L_B B/b \to A/a
\]
is an equivalence (where on the left we have the derived tensor product in $\Dperf(B)$).
In particular, the cartesian square of affine schemes
\[
\begin{tikzcd}
\spec(B/b)\ar[r, "s"] & \spec(B)  \\
\spec(A/a) \ar[r, "r"]\ar[u, "g"] & \spec(A) \ar[u, "f"']
\end{tikzcd}
\]
is Tor-independent and hence determines a commutative square of $B$-linear stable $\infty$-categories with duality
\[
\begin{tikzcd}
(\Dperf(B/b),\Dual_{B/b}) \ar[r, "s_*"]\ar[d,"g^*"'] & (\Dperf(B),\Dual_{B[1]})\ar[d,"f^*"]  \\
(\Dperf(A/a),\Dual_{A/a}) \ar[r, "r_*"] & (\Dperf(A),\Dual_{A[1]})  \ .
\end{tikzcd}
\]
Applying the motivic realization functor $\R^{\sym}_B(\KGW;-)$ we then obtain a commutative square
\[
\begin{tikzcd}
&s_*\KQ_{B/b} \ar[dl,dotted]\ar[r,"\mu_s"]\ar[dd] & \KQ_{B[1]} \ar[dd]  \\
s_*g_*\KQ_{A/a}\ar[dr,dotted, "\simeq"] &&\\
&f_*r_*\KQ_{A/a} \ar[r, "f_*\mu_r"] & f_*\KQ_{A[1]} 
\end{tikzcd}
\]
in $\SH(B)$, which transforms via the adjunction $f^*\dashv f_*$ into a commutative square
\[
\begin{tikzcd}
&f^*s_*\KQ_{B/b} \ar[dl,dotted, "\simeq"']\ar[r, "f^*\mu_s"]\ar[dd, "\simeq"] & f^*\KQ_{B[1]} \ar[dd, "\simeq"] \\
r_*g^*\KQ_{B/b} \ar[dr,dotted,"\simeq"] && \\
&r_*\KQ_{A/a} \ar[r,"\mu_r"] & \KQ_{A[1]} 
\end{tikzcd}
\]
whose vertical maps are equivalences by Proposition~\ref{proposition:KQ-stable-under-basechange} and proper base change (for motivic spectra). In summary, the outcome of this procedure can be viewed as a natural homotopy
\[
\mu_r \sim f^*\mu_s
\]
of maps $r_*\KQ_{\spec(A/a)} \to \KQ_{A[1]}$.
\end{remark}

\begin{proposition}
\label{proposition:pushforward-is-fundamental}%
For any regular embedding of the form $r\colon \spec(A/a) \to \spec(A)$ associated to a regular ring $A$ (of finite type over $S$) and non-zero divisor $a \in A$, the map 
\[
\mu_{r}\colon r_*\KQ_{\spec(A/a)} \to \KQ_{A[1]},
\]
of Construction~\ref{construction:fundamental-pushforward} corresponds, under the adjunction $r_* \dashv r^!$ and the equivalence $\KQ_{A[1]} = \Sig_{\Pone}\KQ_A$ arising from Bott periodicity (Example~\ref{example:bott}), to the fundamental class map
\[
\eta^{\KQ}_{r} \colon \KQ_{\spec(A/a)} \to \Th(T_r)^{-1} \otimes r^!\KQ_A = \Sig_{\Pone}r^!\KQ_A = r^!(\Sig_{\Pone}\KQ_A) ,
\]
where for the first equivalence we have used the fact that the normal bundle of the regular embedding $r\colon \spec(A/a) \to \spec(A)$ is equipped with a distinguished trivialization.
\end{proposition}

\begin{proof}
Without loss of generality, we may assume that $S =\spec(A)$.
The element $a \in A$ determines an $A$-algebra homomorphism $A[t] \to A$ sending $t$ to $a$. Since $t \in A[t]$ is never a zero divisor we obtain, as described in Construction~\ref{construction:fundamental-pushforward} and Remark~\ref{remark:compatibility}, a Tor independent square of affine schemes
\[
\begin{tikzcd}
\spec(A/a)\ar[r, "r"]\ar[d, "g"] & \spec(A) \ar[d, "f"] \\
\spec(A[t]/t) \ar[r, "s"] & \spec(A[t])
\end{tikzcd}
\]
and a compatibility homotopy $\mu_r \sim f^*\mu_s$ of maps $r_*\KQ_{\spec(A/a)} \to \KQ_{A[1]}$. As the first part of the proof, we will now show that the maps $\eta^{\KQ}_r$ and $\eta^{\KQ}_s$ satisfy a similar type of compatibility, so that the desired claim for $r$ follows from the claim for $s$ (a case we will then establish in the second part of the proof).

To establish this first part, note that the construction of the fundamental classes is also compatible with Tor independent squares, as recalled above. In particular, in the case at hand we get that $\eta^{\KQ}_r$ is given as the composite
\[
\KQ_{A/a} = g^*\KQ_{A[t]/t} \xrightarrow{g^*\eta^{\KQ}_s} g^*\Sig_{\Pone} s^!\KQ_{A[t]} = \Sig_{\Pone} g^*s^!\KQ_{A[t]} \to \Sig_{\Pone} r^!f^*\KQ_{A[t]} = r^!\Sig_{\Pone}\KQ_A.
\]
We thus get that the adjoint of $\eta^{\KQ}_r$ along $r_* \dashv r^!$ is given as the composite
\begin{equation}
\label{equation:composite-1}%
\begin{split}
r_*\KQ_{A/a} = r_*g^*\KQ_{A[t]/t} & \xrightarrow{r_*g^*\eta^{\KQ}_s} r_*g^*\Sig_{\Pone} s^!\KQ_{A[t]} = \Sig_{\Pone}r_*g^*s^!\KQ_{A[t]} \\
 & \xrightarrow{\phantom{r_*g^*\eta_{\KQ}}} \Sig_{\Pone}r_*r^!f^*\KQ_{A[t]} \longrightarrow \Sig_{\Pone}f^*\KQ_{A[t]} = \Sig_{\Pone}\KQ_A .
\end{split}
\end{equation}
On the other hand, the adjoint of $\eta^{\KQ}_s$ along $s_* \dashv s^!$ is given by the composite
\begin{equation*}
s_*\KQ_{A[t]/t} \xrightarrow{s_*\eta^{\KQ}_s} s_*\Sig_{\Pone} s^!\KQ_{A[t]} = \Sig_{\Pone}s_*s^!\KQ_{A[t]} \longrightarrow \Sig_{\Pone}\KQ_{A[t]} ,
\end{equation*} 
and if we now apply $f^*$ to this we get the composite
\begin{equation}
\label{equation:composite-2}%
\begin{split}
r_*\KQ_{A/a} = r_*g^*\KQ_{A[t]/t} = f^*s_*\KQ_{A[t]/t} & \xrightarrow{f^*s_*\eta^{\KQ}_s} f^*s_*\Sig_{\Pone} s^!\KQ_{A[t]} = \Sig_{\Pone}f^*s_*s^!\KQ_{A[t]} \\
& \xrightarrow{\phantom{f^*s_*\eta_{\KQ}}} \Sig_{\Pone}f^*\KQ_{A[t]} = \Sig_{\Pone}\KQ_A 
\end{split}
\end{equation} 
We claim that the composite~\eqref{equation:composite-1} is homotopic to the composite~\eqref{equation:composite-2}. Unwinding both constructions, this amounts to constructing a commuting homotopy in the diagram
\[
\begin{tikzcd}
r_*g^*s^!\ar[d,Rightarrow] & f^*s_*s^! \ar[d,Rightarrow]\ar[l,Rightarrow, "\simeq"']\\
r_*r^!f^*\ar[r,Rightarrow] & f^* \ .
\end{tikzcd}
\]
Here, the two arrows with target $f^*$ are induced by the counits of $s_* \dashv s^!$ and $r_* \dashv r^!$, respectively, while the two other arrows are induced by the Beck-Chevalley transformations $f^*s_* \Rightarrow r_*g^*$ and $g^*s^! \Rightarrow r^!f^*$. These two Beck-Chevalley transformations are by construction the mates of one another, and so we may write either one as a composite involving the other and (co)units. In particular, we may extend the above a-priori non-commutative diagram to a larger diagram as
\[
\begin{tikzcd}
r_*r^!r_*g^*s^! \ar[d,Rightarrow] & r_*g^*s^!\ar[d,Rightarrow]\ar[l,Rightarrow] & f^*s_*s^!  \ar[d,Rightarrow]\ar[l,Rightarrow, "\simeq"']\\
r_*r^!f^*s_*s^!\ar[r,Rightarrow] & r_*r^!f^*\ar[r,Rightarrow] & f^* \ .
\end{tikzcd}
\]
where the left half commutes. We then observe that the external part of the diagram commutes as well, yielding a commuting homotopy for the right half. To summarize so far, we have thus shown that the adjoint of $\eta^{\KQ}_r$ along $r_*\dashv r^!$ is homotopic to the image under $f^*$ of the adjoint of $\eta^{\KQ}_s$ along $s_* \dashv s^!$. Otherwise put, these adjoints satisfy the same compatibility as $\mu_r$ and $\mu_s$. It will hence suffice to show that $\mu_s$ is homotopic to the adjoint along $s_* \dashv s^!$ of $\eta^{\KQ}_s$. In other words, we may replace the pair $(A,A/a)$ with the pair $(A[t],A)$ and prove the claim for the regular embedding $s\colon \spec(A) \to \spec(A[t])$ of $A$-schemes. 

We now proceed to the second part of the proof, where we establish the desired claim for $s$. 
Write $p\colon \spec(A[t]) \to \spec(A)$ for the structure map and $j\colon U \to \spec(A[t])$ for the inclusion of the complement of $s\colon \spec(A) \to \spec(A[t])$. Then we have an exact sequence of functors
\[
s_*s^! \Rightarrow \id \Rightarrow j_*j^*
\]
which induces an equivalence
\[
s^! \simeq p_*\fib[\id\Rightarrow j_*j^*] \simeq \fib[p_*\Rightarrow p_*j_*j^*]
\]
of functors $\SH(A) \to \SH(A)$, under which the adjoint $E \to s^!E'$ of a given map $s_*E \to E'$ can be described as 
\[
E = p_*s_*E = \fib[p_*s_*E \to p_*j_*j^*s_*E] \to  \fib[p_*E' \to p_*j_*j^*E'] ,
\]
using the fact that $j^*s_*E = 0$. Taking $E' = p^*\KQ_{A[1]}$ we thus obtain an equivalence
\[
s^!p^*\KQ_{A[1]} = \fib[p_*p^*\KQ_{A[1]} \to  p_*j_*j^*p^*\KQ_{A[1]}] = \fib[\KQ_{A[1]}^{\Aone} \to \KQ_{A[1]}^{\Gm}] = \Om^A\KQ_{A[1]}
\]
where $\Om^A$ stands for cotensor by 
the Thom object of the trivial line bundle over $\spec(A)$ (see \S\ref{subsection:thom-isomorphism}). By the definition of $\mu_s\colon s_*\KQ_A \to \KQ_{(A[t])[1]} = p^*\KQ_{A[1]}$ (see Proposition~\ref{proposition:KQ-stable-under-basechange} for the last equivalence) and the above, the map  
\[
\KQ_A \to \fib[\KQ_{A[1]}^{\Aone} \to \KQ^{\Gm}_{A[1]}] = \Om^A\KQ_{A[1]}
\]
adjoint to $\mu_s$ then coincides under the equivalence 
$s^!p^*\KQ_{A[1]} = \Om^A\KQ_{A[1]}$ with the map induced by applying the functor $\R^{\sym}_A(\KGW;-)$ to the duality preserving $A$-linear functor
\[
s_*\colon (\Dperf(A),\Dual_A) \to (\Dperf_0(A[t]),\Sig\Dual_{A[t]}),
\]
where $\Dperf_0(A[t])$ is the kernel in the bounded Karoubi sequence 
\[
(\Dperf_0(A[t]),\Sig\Dual_{A[t]}) \to (\Dperf(A[t]),\Sig\Dual_{A[t]}) \to (\Dperf(A[t,t^{-1}],\Sig\Dual_{A[t,t^{-1}]})
\]
of Proposition~\ref{proposition:open-is-bounded}
and we have used Construction~\ref{construction:base-change}, Proposition~\ref{proposition:KQ-stable-under-basechange}, and Remark~\ref{remark:realization-exact} to identify the terms. Next, we note that the square
\[
\begin{tikzcd}
\Gm \ar[r]\ar[d] & \Aone \ar[d] \\
\Pone \setminus \{0\} \ar[r] & \Pone,
\end{tikzcd}
\]
considered as a pushout square of representable Nisnevich sheaves,
induces a natural equivalence 
\[
\Om_{\Pone} = \fib[E^{\Pone} \to E] \simeq \fib[E^{\Pone} \to E^{\Pone \setminus \{0\}}] \xrightarrow{\simeq} \fib[E^{\Aone} \to E^{\Gm}] =  \Om^A(E)
\]
for $E \in \SH(A)$. By open base change (Proposition~\ref{proposition:open-base-chage}) the push-forward-pullback square
\[
\begin{tikzcd}
\Dperf(A)\ar[d,"s_*"] & \Dperf(A)\ar[l,equal] \ar[d,"(i_0)_*"] \\
\Dperf(\Aone_A)  &  \Dperf(\Pone_A) \ar[l]
\end{tikzcd}
\]
commutes, and so we conclude that, under the composed equivalence 
\[
s^!p^*\KQ_{A[1]} \simeq \Om^A\KQ_{A[1]} \simeq \Om_{\Pone}\KQ_{A[1]}
\]
the adjoint of $\KQ_A \to \Om_{\Pone}\KQ_{A[1]}$ of
$\mu_s$ is given by applying the functor $\R^{\sym}_A(\KGW;-)$ to the duality preserving $A$-linear functor
\[
(i_{0})_*\colon (\Dperf(A),\Dual_A) \to \big(\Dperf_0(\Pone_A),\Sig\Dual_{\Pone_{A}}\big),
\]
where $\Dperf_0(\Pone_A)$ is 
the kernel in the bounded Karoubi sequence 
\[
\big(\Dperf_0(\Pone_A),\Sig\Dual_{\Pone_A}\big) \to \big(\Dperf(\Pone_A),\Sig\Dual_{\Pone_A}\big) \to \big(\Dperf(\Pone_A \setminus \{0\}),\Sig\Dual_{\Pone_A \setminus \{0\}}\big) .
\]

Passing now to the side of fundamental classes, unwinding their construction we see that under the equivalence $s^!p^*\KQ_{A} \simeq \Om_{\Pone}\KQ_A$ the map
\[
\eta^{\KQ}_s\colon \KQ_A \to  \Sig_{\Pone}s^!p^*\KQ_{A} \simeq \Sig_{\Pone}\Om_{\Pone}\KQ_A
\]
is the inverse of the counit exhibiting $\Sig_{\Pone}$ and $\Om_{\Pone}$ as adjoint-inverse to one another. 
To finish the proof it will now suffice to check that the map
\[
\KQ_A = \R^{\sym}_A(\KGW;(\Dperf(A),\Dual_A)) \to \R^{\sym}_A\big(\KGW;\big(\Dperf_0(\Pone_A),\Sig\Dual_{\Pone_A}\big)\big) = \Om_{\Pone}\KQ_{A[1]}
\]
induced by $(i_0)_*$ is inverse to the equivalence $\Om_{\Pone}\KQ_{A[1]} \simeq \KQ_{A}$ issued from the Bott periodicity 
of Example~\ref{example:bott}.
Recall the latter is obtained from the bounded Karoubi sequence
\[
\begin{tikzcd}
(\Dperf(A),\Dual_A)\ar[r,"q^*"] & \big(\Dperf(\Pone_A),\Dual_{\Pone_A}\big) \ar[rr,"{q_*((-) \otimes \cO(-1))}"]  && (\Dperf(A),\Dual_{A[-1]})
\end{tikzcd}
\]
of Construction~\ref{construction:projective-line}. Unwinding the definitions, under the identification $\Om_{\Pone}\KQ_{A[1]} = \R^{\sym}_A\big(\KGW;\big(\Dperf_0(\Pone_A),\Sig\Dual_{\Pone_A}\big)\big)$ the Bott periodicity equivalence $\Om_{\Pone}\KQ_{A[1]} \simeq \KQ_{A}$ is induced on $\R^{\sym}_A(\KGW;-)$ by the composed functor 
\[
\big(\Dperf_0(\Pone_A),\Sig\Dual_{\Pone_A}\big) \subseteq \big(\Dperf(\Pone_A),\Sig\Dual_{\Pone_A}\big) \xrightarrow{q_*((-) \otimes \cO(-1))}  (\Dperf(A),\Dual_{A}) .
\]
The desired result now follows form the fact that the composite $q_*((-) \otimes \cO(-1)) \circ (i_0)_*$
is homotopic to the identity as a duality preserving functor by Lemma~\ref{lemma:composite-pf}, since $q \circ i_0 = \id$ and
$i_0^*\cO(-1) \simeq \cO_X$. 
\end{proof}

\begin{proof}[Proof of Theorem~\ref{theorem:purity}]
By Lemma~\ref{lemma:reduction} and Proposition~\ref{proposition:pushforward-is-fundamental} it will suffice to show that for every regular embedding of the form $r\colon \spec(A/a) \to \spec(A)$ for $A$ a regular Noetherian ring of finite Krull dimension, the map 
\[
\mu_{r}\colon r_*\KQ_{\spec(A/a)} \to \KQ_{A[1]},
\]
is adjoint to an equivalence $\KQ_{\spec(A/a)} \to r^!(\KQ_{A[1]})$. Since $r_*$ is fully-faithful, this is equivalent to saying that the functor $r^!$ sends $\mu_r$ to an equivalence, and also equivalent to saying that the functor $r_*r^!$ sends $\mu_r$ to an equivalence. Let $U := \spec(A[a^{-1}]) \subseteq \spec(A)$ be the open complement of $\spec(A/a)$ in $\spec(A)$ and write $j\colon U \hrar \spec(A)$ for the inclusion. Then we have an exact sequence of functors
\[
r_*r^! \Rightarrow \id \Rightarrow j_*j^* ,
\]
on $\SH(A)$,
and so to prove that $r_*r^!$ sends $\mu_r$ to an equivalence it will be enough to prove that the sequence
\[
r_*\KQ_{\spec(A/a)} \xrightarrow{\mu_r} \KQ_{A[1]} \to j_*j^*\KQ_{A[1]} = j_*\KQ_{U[1]},
\]
whose composite admits a unique null homotopy since $j^*r_*\KQ_{\spec(A/a)} = 0$, is exact, where we have made implicit use of the equivalence $j^*\KQ_{A[1]} = \KQ_{U[1]}$ of Proposition~\ref{proposition:KQ-stable-under-basechange}. 
By Corollary~\ref{corollary:values-of-KQ-regular} what we now need to verify is that for every smooth $A$-scheme $X$ and every integer $n$ the resulting sequence
\[
\GW^{\sym}(Z,\cO_Z[n-1]) \to \GW^{\sym}(X,\cO_X[n]) \to \GW^{\sym}(V,\cO_V[n])
\]
is exact, where $Z := X \times_{\spec(A)} \spec(A/a)$ and $V := X \times_{\spec(A)} \spec(A[a^{-1}])$. Indeed, this is Dévissage for symmetric $\GW$, see Theorem~\ref{theorem:global-devissage}.
\end{proof}

\begin{appendices}
\renewcommand{\appendixtocname}{{\bf Appendices}}
\addappheadtotoc

\section{Derived categories of schemes}
\label{section:appendix}%

In this appendix, we summarize the $\infty$-categorical approach to derived categories of schemes. The principal exports which are use in the body of the text are the results of \S\ref{subsection:perfect} concerning perfect complexes, and in particular Corollary~\ref{corollary:nisnevich-karoubi}, which says that the formation of perfect derived categories sends Nisnevich squares to Karoubi squares. The results of \S\ref{subsection:quasi-perfect} also play an important part in \S\ref{subsection:pushforward}.

\subsection{Grothendieck abelian categories and their derived counterparts}

Recall that a Grothendieck abelian category is an abelian category which is presentable and in which the class of monomorphisms is closed under filtered colimits. Any Grothendieck abelian category $\cA$ has enough injectives, and one may endow the category $\Ch(\cA)$ of unbounded (homologically graded) complexes in $\cA$ with a model structure in which the weak equivalences are quasi-isomorphisms of complexes, the cofibrations are the levelwise monomorphisms, and every fibrant object is levelwise injective (though in general not every levelwise injective complex is fibrant). We write $\Der(\cA)$ for the underlying $\infty$-category of this model structure, so that $\Der(\cA)$ is by definition the localisation of $\Ch(\cA)$ by the collection of quasi-isomorphisms. It can also be realized more explicitly as the dg-nerve of the full subcategory of fibrant objects, see~\cite{HA}*{\S 1.3.5}. It is always a stable presentable $\infty$-category, and admits a t-structure which is compatible with filtered colimits, see~\cite{HA}*{Proposition 1.3.5.21}.

The injective model structure is convenient for many purposes, but its functorial dependence on $\A$ is rather weak: an adjunction $f^*\dashv f_*$ of Grothendieck abelian categories induces a Quillen adjunction on complexes equipped with injective model structures if and only if the left adjoint $f^*$ is exact, which is a rather strong condition. In the context of the present appendix, the Grothendieck abelian categories we are interested in here are the categories $\Mod(X)$ of $\cO_X$-module sheaves over a scheme $X$. Furthermore, for any morphism $f\colon X \to Y$ of schemes, we would like to consider the associated pullback-push-forward adjunction $f^*\dashv f_*$, and so we would like that the induced adjunction on the level complexes be a Quillen adjunction. This will not work in general for the injective model structure, since $f^*$ is usually not exact. To surmount this issue one can work with the associated \emph{flat} model structure. Recall that if $\A$ is a symmetric monoidal Grothendieck abelian category then an object $x \in \A$ is called flat if the functor $(-) \otimes x$ is exact. Let us then say that $\A$ is model-flat if $\Ch(\A)$ admits a combinatorial model structure whose weak equivalences are the quasi-isomorphisms and the cofibrant objects are those which consist of flat objects. We note that if such a model structure exists then it is unique, since model structures are determined by the weak equivalences and cofibrant objects. In addition, the symmetric monoidal structure on $\A$ induces a symmetric monoidal structure on $\Ch(\A)$, and this symmetric monoidal structure is always compatible (essentially by design) with the flat model structure.  
Given a symmetric monoidal left adjoint functor $f\colon \A \to \cB$ between model-flat Grothendieck abelian categories, if $f$ sends monomorphisms with flat cokernel to monomorphisms with flat cokernel, then the induced functor $\Ch(\A) \to \Ch(\cB)$ is a symmetric monoidal left Quillen functor with respect to the flat model structure on both sides. We may call such $f$ model-flat left functors. Finally, the model category $\Ch(\A)$ is stable and comes equipped with a built-in t-structure. This t-structure is multiplicative, that is, the collection of connective objects contains the unit and is closed under tensor products, and for any model-flat left adjoint functor $f\colon \A \to \cB$ the induced left Quillen functor $\Ch(\A) \to \Ch(\cB)$ preserves connective objects (that is, it is right t-exact).
The formation of flat model structures then assembles to form a functor
\begin{equation}
\label{equation:gro-to-modcat}%
\SMGro^{\flat} \to \SMModCat_{\st,\tstruc} \quad\quad \A \mapsto (\Ch(\A),\Cof^{\flat},\Fib^{\flat},\QIso)
\end{equation}
where $\SMGro^{\flat}$ is the category whose objects are model-flat symmetric monoidal Grothendieck abelian categories and model-flat symmetric monoidal left functors between them and $\SMModCat_{\st,\tstruc}$ is the category of symmetric monoidal stable combinatorial model categories equipped with multiplicative t-structures, and whose morphisms are right t-exact symmetric monoidal left Quillen functors.

Unwinding the definitions, the category $\SMModCat_{\st,\tstruc}$ can be identified with the category commutative algebra objects in the colored operad $\ModCat^{\otimes}_{\st,\tstruc}$ whose multi-maps $(\M_1,\ldots,\M_n) \to \N$ are the left Quillen multi-functors $f\colon \M_1 \times \cdots \times \M_n \to \N$ such that $f(x_1,\ldots,x_n)$ is connective whenever $x_i \in \M_i$ is connective for every $i=1,\ldots,n$. Here, the null operations $() \to \N$ are by definition the cofibrant connective objects. 

On the other hand, recall that the underlying $\infty$-category of any combinatorial model category is presentable, and that Quillen multi-functors induce multi-functors which preserve colimits in each variable separately. In addition, the underlying $\infty$-category of a stable model category is stable, and the notion of a t-structure on a stable model category or on its associated $\infty$-category is the same, since it can be defined purely in terms of the associated (triangulated) homotopy category.
In particular, the functor which inverts weak equivalences induces a map of $\infty$-operads
\begin{equation}
\label{equation:model-to-infinity}%
\ModCat^{\otimes}_{\st,\tstruc} \to (\PrL_{\st,\tstruc})^{\otimes} ,
\end{equation}
where, on the right hand side, we have the $\infty$-operad whose objects are the stable presentable $\infty$-categories with t-structures $(\C,\C_{\geq 0},\C_{\leq 0})$ and whose multi-maps 
\[
((\C_1,(\C_1)_{\geq 0},(\C_1)_{\leq 0}),\ldots,(\C_n,(\C_n)_{\geq 0},(\C_n)_{\leq 0})) \to (\D,\D_{\geq 0},\D_{\leq 0})
\]
are the functors $\C_1 \times \cdots \times \C_n \to \D$ which preserve colimits in each variable separately and that send $(\C_1)_{\geq 0} \times \cdots \times (\C_n)_{\geq 0}$ to $\D_{\geq 0}$. We note that without the t-structure part, this is exactly the underlying $\infty$-operad of the symmetric monoidal category of stable presentable $\infty$-categories, equipped with the (restriction to stable $\infty$-categories) of the Lurie tensor product of presentable $\infty$-categories. In fact, it is not hard to verify that the $\infty$-operad $(\PrL_{\st,\tstruc})^{\otimes}$ is also a symmetric monoidal $\infty$-category, where the tensor product of $(\C,\C_{\geq 0},\C_{\leq 0})$ and $(\D,\D_{\geq 0},\D_{\leq 0})$ is given by $\C \otimes \D$, equipped with the unique t-structure for which the connective objects form the smallest full subcategory containing the image of $\C_{\geq 0} \otimes \D_{\geq 0}$ and closed under colimits and extensions. We also note that commutative algebra objects in $(\PrL_{\st,\tstruc})^{\otimes}$ correspond to stable presentably symmetric monoidal $\infty$-categories $\C$ equipped with a multiplicative t-structure $(\C_{\geq 0},\C_{\leq 0})$, that is, a t-structure such that $\C_{\geq 0}$ contains the unit and is closed under tensor products.

Composing~\eqref{equation:gro-to-modcat} and the functor on commutative algebra objects induced from~\eqref{equation:model-to-infinity} we now obtain a functor
\begin{equation}
\label{equation:gro-to-pr}%
\SMGro^{\flat} \to \CAlg(\PrL_{\st,\tstruc}) \quad\quad \A \mapsto \Ch(\A)[\QIso^{-1}] 
\end{equation}
encoding the formation of derived $\infty$-categories of model-flat symmetric monoidal Grothendieck abelian categories.

\subsection{Derived categories of $\cO_X$-modules}

The relevance of the above discussion to the present context is that if $X$ is any scheme then the symmetric monoidal Grothendieck abelian category $\Mod(X)$ of $\cO$-modules is model-flat, and for any morphism of schemes $f\colon X \to Y$, the pullback functor $f^*$ is symmetric monoidal and model-flat~\cite{recktenwald2019}. 
Associating to $X$ the Grothendieck abelian category $\Mod(X)$ and post-composing with the formation of derived $\infty$-categories~\eqref{equation:gro-to-pr}, we obtain a functor
\begin{equation}
\label{equation:functor-der}%
\Sch\op \to \CAlg(\PrL_{\st,\tstruc}) \quad\quad X \mapsto (\Der(X),\Der(X)_{\geq 0},\Der(X)_{\leq 0}) 
\end{equation}
with $\Der(X) = \Ch(\Mod(X))[\QIso^{-1}]$ with associated t-structure as above.

\begin{remark}
By~\cite{SAG}*{Corollary 2.1.2.3} the $\infty$-category $\Der(X)$ is naturally equivalent to the full subcategory of $\Mod_{\cO_S}(\Sh(X_{\et},\Spa))$ spanned by those of $\cO_X$-module étale spectral sheaves whose underline sheaf is hypercomplete. It seems likely that the whole $\infty$-category of $\cO_X$-module spectral sheaves coincides with the \emph{unseparated derived category} of $\Mod(X)$, which can be modeled as the dg-nerve of all complexes of injectives (as opposed to only fibrant complexes), see~\cite[\S C.5.0]{SAG}. This subtlety however becomes irrelevant below when passing to quasi-coherent sheaves, since these are automatically hypercomplete, see~\cite{SAG}*{Proposition 2.2.6.1, Corollary 2.2.6.2}. 
\end{remark}

\begin{proposition}[Open base change]
\label{proposition:open-base-chage}%
If
\[
\begin{tikzcd}
U \ar[r,"i"]\ar[d,"f"'] & X \ar[d,"g"] \\
V \ar[r,"j"] & Y 
\end{tikzcd}
\]
is a cartesian square of schemes whose horizontal arrows are open embeddings then the Beck-Chevalley transformation $j^*g_* \Rightarrow i^*f_*$ is an equivalence of functors $\Der(X) \to\Der(V)$.
\end{proposition}
\begin{proof}
For an open embedding $i\colon U \hrar X$ the pullback functor $i^*\colon \Mod(X) \to \Mod(U)$ has an exact left adjoint, and hence $i^*$ preserves fibrant complexes of module sheaves (with respect to the injective model structure). Using this, the claim reduces to showing to the 1-categorical statement about the corresponding Beck-Chevalley transformation on the level of module sheaves, which is straightforward.
\end{proof}

\begin{proposition}[Zariski descent]
\label{proposition:descent}%
Let $X$ be a scheme equipped with a finite open covering $X = \cup_{i=1}^n U_i$ by opens. For every non-empty subsets $S \subseteq \{1,\ldots,n\}$ let us denote by $U_S := \cap_{i \in S} U_i$ and $j_S\colon U_S \hrar X$ the associated inclusion.
Then the functor
\begin{equation}
\label{equation:descent}%
\Der(X) \to \lim_{\emptyset \neq S \subseteq \{1,\ldots,n\}} \Der(U_S) \quad\quad M \mapsto \{j_S^*M\}_S 
\end{equation}
is an equivalence of $\infty$-categories.
\end{proposition}
\begin{proof}
The collection of push-forward functors $j_{S*}$ assemble to give a right adjoint
\[
\lim_{\emptyset \neq S \subseteq \{1,\ldots,n\}} \Der(U_S) \to \Der(X) \quad\quad \{M_S\}_S \mapsto \lim_{\emptyset \neq S \subseteq \{1,\ldots,n\}}j_{S*}M_S.
\]
Since the functors $j_S^*$ are jointly conservative, to show that the adjunction we get is an equivalence it will suffice to show that its counit is an equivalence. For a given family $\{M_S\}_S$, the component of the counit at a given $\emptyset \neq T \subseteq \{1,..,n\}$ is the map
\[
j^*_{T}\lim_{\emptyset \neq S \subseteq \{1,\ldots,n\}}j_{S*}M_S \to M_T
\]
adjoint to the projection $\lim_{\emptyset \neq S \subseteq \{1,\ldots,n\}}j_{S*}M_S \to j_{T*}U_T$. Now since $j^*_T$ preserves finite limits we can rewrite this map as the composite
\[
\lim_{\emptyset \neq S \subseteq \{1,\ldots,n\}}j^*_Tj^{\vphantom{*}}_{S*}M_S \to j^*_{T}j^{\vphantom{*}}_{T*}M_T \to M_T.
\]
By open base change (Proposition~\ref{proposition:open-base-chage}) the second map is an equivalence and the first is the projection
\[
\lim_{\emptyset \neq S \subseteq \{1,\ldots,n\}}j_{T,S*}M_{S \cup T} = \lim_{\emptyset \neq S \subseteq \{1,\ldots,n\}}j^{\vphantom{*}}_{T,S*} j^*_{T,S}M_S \to M_T 
\]
where $j_{S,T}\colon U_{S \cup T} \to U_T$ is the open inclusion of the intersection $U_{S \cup T} = U_S \cap U_T$. Now the claim that this last projection is an equivalence amounts to saying that the punctured cube 
\begin{equation}
\label{equation:formula}%
S \mapsto j_{S,T*}M_{S \cup T} \quad\quad \emptyset \neq S \subseteq \{1,\ldots,n\}
\end{equation}
in $\Der(U_T)$ can be extended to a cartesian cube whose edge relating the limit corner $\emptyset$ to $T$ is an equivalence. Indeed, the formula~\eqref{equation:formula} simply makes sense for any $S \subseteq \{1,\ldots,n\}$, including the empty one, and the resulting $n$-cube has the property that it sends any $S \subseteq S'$ such that $S \cup T = S' \cup T$ to an equivalence. Such an $n$-cube is always cartesian (see, e.g.,~\cite{HA}*{Proposition 6.1.1.13}).
\end{proof}

\begin{remark}
\label{remark:zariski}%
As explained in the proof of Proposition~\ref{proposition:descent}, the functor~\eqref{proposition:descent} is part of an adjunction, and since~\eqref{proposition:descent} is an equivalence the associated unit map is an equivalence. Explicitly, this unit is the map
\[
M \to \lim_{\emptyset \neq S \subseteq \{1,\ldots,n\}} j^{\phantom{*}}_{S*}j^*_SM,
\]
and so get in particular, that every $M \in \Der(X)$ admits a canonical presentation as a finite limit of complexes pushed forward from the various $U_S$'s.
\end{remark}

Given a scheme $X$ and a qcqs open subset $U \hrar X$ with complement $Z = X \setminus U$ we denote by $\Der_Z(X) \subseteq \Der(X)$ the full subcategory spanned by those $M \in \Der_U(X)$ such that $M|_{U} = 0$. 

\begin{remark}
\label{remark:pushforward-support}%
One often encounters situations where several open subsets are considered simultaneously. For this, it is convenient to point out that if $X$ is a scheme and $U, V \subseteq X$ two qcqs open subsets with closed complements $Z,Y \subseteq X$, respectively, then the restriction functor
\[
j^*\colon \Der(X) \to \Der(V)
\]
along the inclusion $j\colon V \hrar X$ sends $\Der_Z(X)$ to $\Der_{V \cap Z}(V)$, and by open base change (Proposition~\ref{proposition:open-base-chage}) the push-forward functor 
\[
j_*\colon \Der(V) \to \Der(X)
\]
sends $\Der_{V \cap Z}(V)$ to $\Der_Z(X)$. We conclude that $j^* \dashv j_*$ restricts to an adjunction
\[
j^*\colon \Der_Z(X) \adj \Der_{V \cap Z}(V) \cocolon j_*
\]
\end{remark}

\begin{corollary}[Excision]
\label{corollary:excision}%
Let $X$ be a scheme, $U \subseteq X$ an open subscheme with closed complement $Z = X \setminus U$ and $V$ an open subscheme which contains $Z$. Then the functor
\[
\Der_Z(X) \to \Der_Z(V),
\]
induced by restriction, is an equivalence.
\end{corollary}
\begin{proof}
This is the induced functor on the horizontal fibres of the square
\[
\begin{tikzcd}
\Der(X) \ar[r]\ar[d] & \Der(U) \ar[d] \\
\Der(V) \ar[r] & \Der(V \cap U) \ ,
\end{tikzcd}
\]
which is cartesian by Proposition~\ref{proposition:descent}.
\end{proof}

\begin{remark}
\label{remark:zariski-descent-support}%
Suppose given a scheme $X$, a closed subscheme $Z \subseteq X$, with complement $U = X \setminus Z$ and a finite open covering $X = \cup_{i=1}^{n} X_i$. Write $U_i = U \cap X_i$ and $Z_i = Z \cap X_i$. We may then apply Zariski descent to $\Der^{\qc}_Z(X)$ by placing it in the commutative diagram
\[
\begin{tikzcd}
\Der_Z(X) \ar[r]\ar[d] & \Der(X) \ar[r]\ar[d, "\simeq"] & \Der(U) \ar[d, "\simeq"] \\
\lim_{\emptyset \neq S \subseteq \{1,\ldots,n\}}\Der_{Z_i}(X_i) \ar[r] & \lim_{\emptyset \neq S \subseteq \{1,\ldots,n\}}\Der(X_i) \ar[r] & \lim_{\emptyset \neq S \subseteq \{1,\ldots,n\}}\Der(U_i)   \ .
\end{tikzcd}
\]
where the rows are fibre sequences and the middle and right most vertical maps are equivalences by Proposition~\ref{proposition:descent}. We may then conclude that the left most vertical map is an equivalence as well. In other words, $\infty$-categories of $\cO_X$-module complexes with support also satisfy Zariski descent.
\end{remark}

\subsection{Quasi-coherent complexes}
\label{subsection:quasi-coherent}%

For a scheme $X$, and a sheaf $M$ on $X$ we write $\Gamma_X(M) = M(X)$ for the global sections of $M$. In particular $\Gamma(X) = \Gamma_X(\cO_X)$ is the ring of global regular functions on $X$, and if $M$ is an $\cO_X$-module sheaf then $\Gamma_X(M)$ is a $\Gamma(X)$-module. We may consequently view $\Gamma_X$ as a functor
\[
\Gamma_X\colon \Mod(X) \to \Mod(\Gamma(X)) .
\]
As such, it is compatible with push-forwards, in the sense that for a map $f\colon X \to Y$ then the square
\[
\begin{tikzcd}
\Mod(X) \ar[r,"\Gamma_X"]\ar[d] & \Mod(\Gamma(X)) \ar[d] \\
\Mod(Y) \ar[r,"\Gamma_Y"] & \Mod(\Gamma(Y))  \ ,
\end{tikzcd}
\]
in which the left vertical map is given by push-forward and the right vertical map by the forgetful functor (which simply restricts the action along $\Gamma(Y) \to \Gamma(X)$), is commutative. The functor $\Gamma_X$ admits a left adjoint
\[
\Lam_X \colon \Mod(\Gamma(X)) \to\Mod(X)
\]
where for an $A$-module $M$ the sheaf $\Lam_X(M)$ is such that for any $a \in A$, the sections of $\Lam_X(M)$ over $\spec(A[f^{-1}])$ are given by $M[f^{-1}] = M \otimes_A A[f^{-1}]$. 
We note that the fact that $\Gamma_X$ is compatible with push-forward implies that $\Lam_X$ is compatible with pullback. When $X =\spec(A)$ is affine the left adjoint
\[
\Lam_{\spec(A)}\colon \Mod(A) \to \Mod(\spec(A))
\]
is fully-faithful. An $\cO_{\spec(A)}$-module is said to be \defi{quasi-coherent} if it is in the essential image of $\Lam_{\spec(A)}$. For a general scheme $X$, we say that an $\cO_X$-module is quasi-coherent if its restriction along every map of the form $\spec(A) \to X$ is quasi-coherent in the above sense. The compatibility of $\Lam_X$ with pullbacks implies that this definition indeed extends the definition for affine schemes.
We say that a complex $M \in \Der(X)$ is \defi{quasi-coherent} if all its cohomology sheaves are quasi-coherent, and we write $\Der^{\qc}(X) \subseteq \Der(X)$ for the full subcategory spanned by the quasi-coherent complexes. This is a stable subcategory which is closed under colimits.

\begin{remark}
\label{remark:local}%
Being quasi-coherent is a local property. In particular, a complex $M \in \Der(X)$ is quasi-coherent if and only if its restriction to each open subset $U \subseteq X$ is quasi-coherent, and if and only if there exists some open cover by $X = \cup_i U_i$ such that each $M|_{U_i}$ is quasi-coherent. 
\end{remark}

\begin{remark}
\label{remark:affine}%
The adjunction $\Lam_X \dashv \Gamma_X$ of Grothendieck abelian categories induces a Quillen adjunction 
\[
\Ch(\Mod(\Gamma(X))) \adj \Ch(\Mod(X)) ,
\]
on the model categorical level (where we can either endow both with the flat model structure or with the injective one),
and hence an adjunction
\[
\Der(\Gamma(X)) \adj\Der(X)
\]
on the level of derived $\infty$-categories. The essential image of the left adjoint is then $\Der(\Mod^\qc(A)) \simeq \Der^{\qc}(\spec(A))$ (see Remark~\ref{remark:t-structure}(\ref{item:Bocksted-Neemann}) below). 
\end{remark}

\begin{remark}
\label{remark:t-structure}%
The $\infty$-subcategory $\Der^{\qc}(X)$ inherits from $\Der(X)$ a t-structure such that the inclusion $\Der^{\qc}(X) \subseteq \Der(X)$ is t-exact. This t-structure has the following properties:
\begin{enumerate}
\item
The heart of $\Der^{\qc}(X)$ is the ordinary category $\Mod^\qc(X)$ of quasi-coherent $\cO_X$-modules. 
\item
Pullbacks along flat maps (and in particular open inclusions) are t-exact. 
\item
When $X = \spec(A)$ is affine the t-structure on $\Der^{\qc}(X) = \Der(A)$ coincides with the usual t-structure on $A$-complexes.
\item
If $X$ is qcqs then the t-structure on $\Der^{\qc}(X)$ is left complete (this follows by Zariski descent for quasi-coherent complexes and the fact that the t-structure on $\Der^{\qc}(\spec(A)) = \Der(A)$ is left complete for any ring $A$).  
\item
\label{item:Bocksted-Neemann}%
When $X$ is quasi-compact and separated the canonical functor $\Der(\Mod^\qc(X)) \to \Der^{\qc}(X)$ is an equivalence (this is the Bökstedt-Neeman theorem). 
\end{enumerate}
\end{remark}

For any map of schemes $f\colon X \to Y$ the pullback functor
\[
f^*\colon \Der(Y) \to \Der(X)
\]
sends quasi-coherent complexes to quasi-coherent complexes. Indeed, by Remark~\ref{remark:local} it suffices to prove this in the case where both $X$ and $Y$ are affine, in which case it is a consequence of Serre's affine theorem (see Remark~\ref{remark:affine}). See also~\cite{lipman-grothendieck-duality}*{Proposition 3.9.1} for a more homological algebra proof. 

We shall now restrict our attention from general to quasi-compact quasi-coherent schemes (qcqs for short). In this setting, the push-forward functor also preserves quasi-coherent complexes, see, e.g., \cite{lipman-grothendieck-duality}*{Proposition 3.9.2}). For the convenience of the reader we provide a proof below in the present language.

\begin{proposition}
\label{proposition:pushforward}%
Let $f\colon X \to Y$ be a qcqs morphism of schemes (e.g., $f$ is any map between qcqs schemes). Then 
\[
f_*\colon \Der(X) \to \Der(Y)
\]
sends quasi-coherent complexes to quasi-coherent complexes.
\end{proposition}

In particular, for $f$ qcqs the pullback-push-forward adjunction restricts to an adjunction
\[
f^*\colon \Der^{\qc}(Y) \adj \Der^{\qc}(X) \cocolon f_* .
\]
Since the full subcategory $\Der^{\qc}(X) \subseteq \Der(X)$ of quasi-coherent sheaves is closed under tensor products (\cite{stacks-project}*{\href{https://stacks.math.columbia.edu/tag/01CE}{Tag 01CE}}) it inherits from $\Der(X)$ its symmetric monoidal structure such that all pullback functors are symmetric monoidal. In light of Remark~\ref{remark:t-structure} and the discussion preceding it, we get that the functor~\eqref{equation:functor-der} refines to a functor
\begin{equation}
\label{equation:functor-derqc}%
\Sch\op \to \CAlg(\PrL_{\st,\tstruc}) \quad\quad X \mapsto (\Derqc(X),\Derqc(X)_{\geq 0},\Derqc(X)_{\leq 0}) .
\end{equation}

\begin{corollary}
\label{corollary:descent-qc}%
When all schemes involved are qcqs, the properties of open base change (Proposition~\ref{proposition:open-base-chage}), Zariski descent (Proposition~\ref{proposition:descent}) and Excision (Corollary~\ref{corollary:excision}) hold verbatim if one replaces everywhere $\Der(-)$ by $\Der^{\qc}(-)$. 
\end{corollary}

The proof of Proposition~\ref{proposition:pushforward} uses the following type of argument, which we later use again several times:
\begin{lemma}
\label{lemma:useful-argument}%
Suppose that $P$ is a property of schemes which is local in the following sense: a scheme $X$ which admits a finite covering $X = \cup_{i=1}^{n} X_i$ has property $P$ if $X_S = \cap_{i \in S} X_i$ has property $P$ for every $\emptyset \neq S \subseteq \{1,\ldots,n\}$. Then every qcqs scheme has property $P$ if every affine scheme has property $P$.
\end{lemma}
\begin{proof}
Any qcqs scheme $X$ admits a finite open covering $X = \cup_i X_i$ such that each $X_i$ is affine and each intersection $X_S$ is quasi-compact. If $X$ is furthermore \emph{separated}, then the open affines in $X$ are closed under intersection and hence each $X_S$ is actually affine. We conclude that each quasi-compact separated scheme has property $P$. Now if $X$ is quasi-compact and quasi-separated and $X = \cup_i X_i$ is a finite affine open covering then each $X_i$ is in particular separated and hence each $X_S$ is also separated, since being separated is inherited to open subsets. Hence each $X_S$ is quasi-compact and separated, so has property $P$ by the above. We conclude that $X$ has property $P$ as well.
\end{proof}

\begin{proof}[Proof of Proposition~\ref{proposition:pushforward}]
Being a quasi-coherent complex on $Y$ is a property that we can test by restricting to all affine open subschemes of $Y$. By open base change we may thus assume without loss of generality that $Y$ is affine. Let us say that a scheme $X$ is \defi{good} if the statement of the proposition holds for $X$ for any qcqs map $f\colon X \to Y$ with affine target. We note that if $X$ admits a qcqs map to an affine scheme then $X$ itself is qcqs, and on the other hand, if $X$ is qcqs then every map to an affine scheme is qcqs. Our goal is now to show that all qcqs schemes are good. Since the collection of quasi-coherent sheaves on $Y$ is closed under finite limits, Remark~\ref{remark:zariski} implies the following: if $X$ admits a finite open covering $X = \cup_{i=1}^{n} X_i$ such that for every $\emptyset \neq S \subseteq \{1,\ldots,n\}$ the scheme $X_S = \cap_{i \in S} X_i$ is good then $X$ itself is good. Invoking Lemma~\ref{lemma:useful-argument} it will hence suffice to show that any affine scheme is good. 
For this, we will show that for any map of commutative rings $A \to B$ (associated to a map of affine schemes $f\colon \spec(B) \to \spec(A)$), and every $B$-module complex $M \in \Der(B)$ the Beck-Chevalley map 
\begin{equation}
\label{equation:beck-chevalley}%
\Lam_{\spec(A)}(M) \to f_*\Lam_{\spec(B)}(M)
\end{equation}
is an equivalence in $\Mod(\spec(A))$ (where $\Lam_{\spec(A)}$ and $\Lam_{\spec(B)}$ are the (derived) left adjoints of $\Gamma_{\spec(A)}$ and $\Gamma_{\spec(B)}$, respectively, see discussion above). Now $\Lam_{\spec(A)},\Lam_{\spec(B)}$ and the forgetful functor $\Der(B) \to \Der(A)$ are induced on derived categories by exact functors of abelian categories, while the push-forward forward functor $\Mod(\spec(A)) \to \Mod(\spec(B))$ is exact on quasi-coherent sheaves by Serre's vanishing theorem. Passing to homotopy sheaves, it will hence suffice to show that the above Beck-Chevalley map is an isomorphism on the level of sheaves. For this, we can check that the map induces an equivalence on sections over opens of the form $\spec(A[g^{-1}])$ for $g \in A$. Then $\spec(B) \times_{\spec(A)} \spec(A[g^{-1}]) = \spec(B \otimes_A A[g^{-1}])$ and the map induced by~\eqref{equation:beck-chevalley} on sections over $\spec(A[g^{-1}])$ is just the isomorphism
\[
M \otimes_A A[g^{-1}] \xrightarrow{\cong} M \otimes_B B \otimes A[g^{-1}] . \qedhere
\]
\end{proof}

In the setting of quasi-coherent sheaves, the open base change property of Proposition~\ref{proposition:open-base-chage} can be extended to \emph{flat base change}:
\begin{proposition}[Flat base change]
\label{proposition:flat-base-change}%
Let
\[
\begin{tikzcd}
Z \ar[r,"p"]\ar[d,"f"'] & X \ar[d,"g"] \\
W \ar[r,"q"] & Y 
\end{tikzcd}
\]
be a cartesian square of qcqs schemes such that either $f$ or $q$ is flat.
Then the Beck-Chevalley transformation $q^*g_* \Rightarrow p^*f_*$ is an equivalence of functors $\Der(X) \to\Der(W)$. 
\end{proposition}
\begin{proof}
We reduce to the case where all schemes in the above square are affine as follows. First, we reduce to the case where both $W$ and $Y$ are affine. For this, let us pick an open affine covering $Y = \cup_{i=1}^{n} Y_i$. Let us write $X_i = X \times_Y Y_i$, $W_i = W \times_Y Y_i$ and $Z_i = Z \times_Y Y_i$. For each $i = 1,\ldots,n$ let us pick an open affine covering $W_i = \cup_{j=1}^{m} W_{i,j}$ and write $Z_{i,j} = W_{i,j} \times_{Y_i} X_i$.
For a given $M \in \Der^{\qc}(Y)$, the Beck-Chevalley map $\eta\colon q^*g_*(M) \Rightarrow p^*f_*(M)$ is an equivalence in $\Der^{\qc}(W)$ if and only if $\eta|_{W_{i,j}}$ is an equivalence $\Der^{\qc}(W_{i,j})$ for every $i=1,\ldots,n$. But by open base change we have that $\eta|_{W_{i,j}}$ is the Beck-Chevalley transformation $p_i^*f_{i*}(M|_{X_i}) \to (p^*_i)f_{i*}(M|_{X_{i}})$ for the cartesian square
\[
\begin{tikzcd}
Z_{i,j} \ar[r,"p_{i,j}"]\ar[d,"f_{i,j}"'] & X_i \ar[d,"g_i"] \\
W_{i,j} \ar[r,"q_{i,j}"] & Y_i \ .
\end{tikzcd}
\]
In addition, if the map $q\colon W \to Y$ was flat then the map $q_{i,j}\colon W_{i,j} \to Y_i$ is flat as well, being the composite of the open embedding $W_{i,j} \to W_i$ and the base change $W_i \to Y_i$ of $q$ along $Y_i \subseteq Y$.
We may hence assume without loss of generality that $W$ and $Y$ are affine. Let us therefore write $Y = \spec(A)$ and $W = \spec(B)$, so that we have a square of the form
\[
\begin{tikzcd}
X_B \ar[r,"p"]\ar[d,"f"'] & X \ar[d,"g"] \\
\spec(B) \ar[r,"q"] & \spec(A) 
\end{tikzcd}
\]
Let us now say that $X$ is \defi{good} if the flat base change statement holds for any square as above whenever either $\spec(B) \to \spec(A)$ or $X \to \spec(A)$ are flat. We claim that if $X$ admits a finite open covering $X = \cup_i X_i$ such that for any $\emptyset \neq S \subseteq \{1,\ldots,n\}$ the open subscheme $X_S = \cap_{i \in S} X_i$ is good then $X$ itself is good. To see this, let $M \in \Der^{\qc}(X)$ be an quasi-coherent complex and write $M_S = M|_{X_S}$. By Remark~\ref{remark:zariski} the map
\[
M \to \lim_{\emptyset \neq S \subseteq \{1,\ldots,n\}} j^{\phantom{*}}_{S*}M_S
\]
is an equivalence, where $j_S\colon X_S \hrar X$ is the inclusion. Since the functors $q^*g_*$ both commutes with finite limits we consequently have
\[
q^*g_*M = \lim_{\emptyset \neq S \subseteq \{1,\ldots,n\}} q^*g_*j^{\phantom{*}}_{S*}M_S.
\]
Let $X_{S,B} = X_S \times_X X_B$ and $p_{S,B}\colon X_{S,B} \to X_S$ and $j_{S,B}\colon X_{S,B} \to X_B$ the two associated maps.
Applying Remark~\ref{remark:zariski} to $X_B$ we get that the map
\[
p^*M \to \lim_{\emptyset \neq S \subseteq \{1,\ldots,n\}} j^{\phantom{*}}_{S,B*}M_{S,B} = \lim_{\emptyset \neq S \subseteq \{1,\ldots,n\}} j^{\phantom{*}}_{S,B*}p^*_{S,B}M_{S}
\]
is also an equivalence, where $M_{S,B} = (p^*M)_{X_{S,B}}$, and hence
\[
f_*p^*M = \lim_{\emptyset \neq S \subseteq \{1,\ldots,n\}} f_*j^{\phantom{*}}_{S,B*}p^*_{S,B}M_{S}.
\]
The Beck-Chevalley transformation for $q^*g_*M \to f_*p^*M$ can then be identified with the map
\[
\lim_{\emptyset \neq S \subseteq \{1,\ldots,n\}} q^*g_*j^{\phantom{*}}_{S*}M_S \to \lim_{\emptyset \neq S \subseteq \{1,\ldots,n\}} f_*j^{\phantom{*}}_{S,B*}p^*_{S,B}M_{S}
\]
induced on limits by the Beck-Chevalley transformation of $M_S \in \Der^{\qc}(X_S)$ associated to the cartesian square
\[
\begin{tikzcd}
X_{S,B} \ar[r,"p"]\ar[d,"f"'] & X_S \ar[d,"g"] \\
\spec(B) \ar[r,"q"] & \spec(A) \ .
\end{tikzcd}
\]
We also note that if $X \to \spec(A)$ is flat then $X_S \to \spec(A)$ is flat as well. We conclude that if each $X_S$ is good then $X$ is good. By Lemma~\ref{lemma:useful-argument} it will hence suffice to show that any affine scheme is good.
In particular, we need to show that the Beck-Chevalley transformation is an equivalence for squares of the form
\[
\begin{tikzcd}
\spec(B \otimes_A C) \ar[r,"p"]\ar[d,"f"'] & \spec(C) \ar[d,"g"] \\
\spec(B) \ar[r,"q"] & \spec(A) \ .
\end{tikzcd}
\]
where either $B$ or $C$ are flat over $A$.
Explicitly, we need to show that in this case for every $C$-module $M$, the map
\[
q^*g_*M = B \otimes^L_A M \to (B \otimes_A C) \otimes^L_C M = f_*p^*M
\]
is an equivalence. Since both sides preserves colimits it is enough to check this for $M=C$, where it becomes the map
\[
B \otimes^L_A C \to B \otimes_A C
\]
form the derived to the underived tensor product over $A$. This map is indeed an equivalence whenever either $B$ or $C$ are flat.
\end{proof}

\subsection{Flat excision}

Our goal in the present subsection is to prove the following proposition, which we will use to establish Corollary~\ref{corollary:nisnevich-karoubi} below.

\begin{proposition}[Flat excision]
\label{proposition:flat-square}%
Let $X$ be a qcqs scheme, $f\colon V \to X$ a flat qcqs map and $Z \subseteq X$ a finitely presented closed subscheme such the induced map $V \times_X Z \to Z$ is an isomorphism. Set $U = X \setminus Z$ and $W = V \times_X W$. Then the square
\[
\begin{tikzcd}
\Der^{\qc}(X) \ar[r]\ar[d] & \Der^{\qc}(V) \ar[d] \\
\Der^{\qc}(U) \ar[r] & \Der^{\qc}(W)
\end{tikzcd}
\]
is cartesian.
\end{proposition}

\begin{remark}\
\begin{enumerate}
\item
In Proposition~\ref{proposition:flat-square}, the condition that $Z$ is finitely presented implies that $U$ (and hence $W$) are quasi-compact, and hence qcqs.
\item
On the other hand, if $U \subseteq X$ is any qcqs open subscheme then we can always choose a finitely presented subscheme structure on the closed complement $Z$ of $U$. The assumption that $f\colon V \to X$ is flat insures that the condition that $V \times_X Z \to Z$ is an isomorphism does not depend on the choice of a finitely presented subscheme structure on $Z$.
\end{enumerate}
\end{remark}

The proof of Proposition~\ref{proposition:flat-square} will require a few preliminaries. We first establish the following more general observation:

\begin{lemma}
\label{lemma:flat-square-criterion}%
Let $X$ be a qcqs scheme $f\colon V\to X$ a flat qcqs map and $Z \hrar X$ a closed subscheme with qcqs complement
$j\colon U \hrar X$. Set $W = f^{-1}(U)$ and $Y = f^{-1}(Z)$. Assume that the map $U \amalg V \to X$ is surjective. Then the following statements are equivalent:
\begin{enumerate}
\item
\label{item:flat-square-cartesian}%
The square
\[
\begin{tikzcd}
\Der^{\qc}(X) \ar[r, "f^*"]\ar[d, "j^*"] & \Der^{\qc}(V) \ar[d, "j^*"] \\
\Der^{\qc}(U) \ar[r, "f^*"] & \Der^{\qc}(W)
\end{tikzcd}
\]
is cartesian, where, by abuse of notation, we denote the base change of $f$ again by $f$ and the base change of $j$ again by $j$.
\item
\label{item:excision-equivalence}%
The functor $\Der^{\qc}_Z(X) \to \Der^{\qc}_Y(V)$ induced by $f^*$ is an equivalence.
\end{enumerate}
In addition, of $f$ is affine then these conditions are also equivalent to:
\begin{enumerate}[resume*]
\item
\label{item:excision-fully-faithful}%
The functor $\Der^{\qc}_Z(X) \to \Der^{\qc}_Y(V)$ induced by $f^*$ is fully-faithful.
\end{enumerate}
\end{lemma}
\begin{proof}
We note that by definition, the functor appearing in \ref{item:excision-equivalence} is just the functor induced on vertical fibres by the square in \ref{item:flat-square-cartesian}. Hence $\ref{item:flat-square-cartesian} \Rightarrow \ref{item:excision-equivalence}$ and the implication in the other direction follows from~\cite{9-authors-II}*{Lemma~1.5.3} since $j^*$ is a (right split) Verdier projection and the square is right adjointable by flat base change.
On the other hand, if $f$ is affine then the functor $\Der^{\qc}_Y(V) \to \Der^{\qc}_Z(X)$ induced by $f_*$ is conservative (since $f_*$ is), and hence the adjunction $\Der^{\qc}_Z(X) \adj \Der^{\qc}_Y(V)$ is an equivalence if and only if its left adjoint (induced by $f^*$) is fully-faithful. 
\end{proof}

\begin{lemma}
\label{lemma:noetherian-principle}%
Let $X$ be a scheme and $i\colon Z \hrar X$ a finitely presented closed subscheme with open complement $U \subseteq X$. If $M \in \Der^{\qc}(X)$ is such that $M$ restricts to $0$ on both $U$ and $Z$ then $M=0$.
\end{lemma}
\begin{proof}
Testing the triviality of $M$ can be done locally, so we may as well assume that $X = \spec(A)$ for $A$ a ring. In this case $Z$ is affine as well, and corresponds to $\spec(A/I)$ for some finitely generated ideal $I = \<f_1,\ldots,f_n\> \subseteq A$. If $I=0$ then $Z=X$ and the claim is trivially true, so we may assume $n \geq 1$ with each $f_j \neq 0$.

Let $\E \subseteq \Der(A)$ be the full (stable) subcategory spanned by those $A$-complexes $N$ such that $N \otimes_A M = 0$. For $i=1,\ldots,n$ let
$E_j = \cof[f_j \colon A \to A] \in \Der(A)$
and $E = E_1 \otimes_A \cdots \otimes_A E_n$. 
We then claim that $\E$ contains $E$. To see this, we note that $\E$ contains $i_*N$ for every $N \in \Der(A/I)$, since $M \otimes i_*N = i_*(i^*M \otimes N) = 0$. Now since $E$ is connective and $n$-truncated we may prove by induction on $m$ that $\tau_{\leq m}E$ is in $\E$. At each point in the induction (including the base), it will suffice to show that $\pi_mE[m]$ is in $\E$. Unwinding the definitions, we see that each of the $A$-modules $\pi_m E[m]$ admits a filtration whose successive quotients are obtained from $A$ by iteratively taking either the kernel or cokernel of $f_i$ for each $i$; these successive quotients are in particular $A/I$-modules, and so we conclude that $\E$ contains $E$.

We now prove by descending induction on $j$ that $\E$ contains $E_1 \otimes_A \cdots \otimes_A E_j$ for every $j=0,\ldots,n$ (so that, for $j=0$ we get that $\E$ contains $A$ itself and hence $M=0$). Indeed, let $M'= M \otimes_A E_1 \otimes_A \cdots \otimes_A E_{j-1}$ and suppose that $M' \otimes_A E_j = \cof[f_j \colon M' \to M'] = 0$. Then $f_j$ acts invertibly on $M'$, and so $M'$ is pushed forward from $U_j = \spec(A[f_j^{-1}]) \subseteq \spec(A)$. But $U_j$ is contained in $U$, and so $M'|_{U_j} = 0$, so that $M' =0$. The proof is hence finished.
\end{proof}

\begin{proof}[Proof of Proposition~\ref{proposition:flat-square}]
Since $f\colon V \to X$ is flat and both $X$ and $V$ are qcqs we may find open coverings $X = \cup_{i=1}^{n} X_i$ and $V = \cup_{i=1}^{n} V_i$ such that for each $i=1,\ldots,n$ the schemes $X_i$ and $V_i$ are affine, $f(V_i) \subseteq X_i$ and the induced map $V_i \to X_i$ is flat. Write $Z_i = V_i \times_X Z$, which we can view as an open subset of $Z_i$ (since $V \times_X Z \to Z$ is an isomorphism by assumption). Set $Y_i = (Z \cap X_i) \setminus Z_i$, so that $Y_i$ can be considered as a closed subscheme of $Z$. We note that $Z_i$ is affine, being a closed subset of the affine scheme $V_i$. In particular, $Z_i$ is quasi-compact. We may consequently find a finite collection of affine opens $X_{i,1},\ldots,X_{1,m} \subseteq X_i \setminus Y_i$ such that $Z_i \subseteq \cup_j X_{1,j}$. Set $V_{i,j} = V_i \times_{X_j} X_{i,j}$. Then $V_{i,j} = V_i \times_{X_j} X_{i,j}$ is a fibre product of affine schemes, and is hence affine. Furthermore, the map $V_{i,j} \to X_{i,j}$ is a base change of the map $V_{i} \to X_{i}$, and is hence flat. Finally, by construction we have that $V_{i,j} \times_X Z \to X_{i,j} \times_X Z$ is an isomorphism. 
Let $X^{\circ} = \cup_{i,j} X_{i,j}$ and $V^{\circ} = \cup_{i,j} V_{i,j}$. We note that $X^{\circ}$ contains $Z$ and $V^{\circ} = V \times_X X^{\circ}$ contains $V \times_X Z \cong Z$. Let $\I = \{1,\ldots,n\} \times \{1,\ldots,m\}$ be the set of all pairs of indices $(i,j)$ with $i \in \{1,\ldots,n\}$ and $j \in \{1,\ldots,m\}$. For each $\emptyset \neq S \subseteq \I$ let us denote by $X_S = \cap_{(i,j) \in S} X_{i,j}$ and $V_S = \cap_{(i,j) \in S} V_{i,j}$. We may then consider the commutative diagram
\[
\begin{tikzcd}
\Der^{\qc}_{Z}(X) \ar[r]\ar[d,"\simeq"] & \Der^{\qc}_{V \times_X Z}(V) \ar[d, "\simeq"] \\
\Der^{\qc}_{Z}(X^{\circ}) \ar[r]\ar[d, "\simeq"] & \Der^{\qc}_{V \times_X Z}(V^{\circ}) \ar[d, "\simeq"] \\
\lim_{\emptyset \neq S \subseteq \I} \Der^{\qc}_{X_S \times_X Z}(X_S) \ar[r] & \lim_{\emptyset \neq S \subseteq \I} \Der^{\qc}_{V_S \times_X Z}(V_S)
\end{tikzcd}
\]
where the vertical arrows are equivalences by Zariski descent (see Remark~\ref{remark:zariski-descent-support}) and excision (Corollary~\ref{corollary:excision}). To prove Proposition~\ref{proposition:flat-square} for $X,V$ and $Z$ it will hence suffice to so for each of $X_{i,j}, V_{i,j}$ and $X_{i,j} \times_X Z$. Since each $X_{i,j}$ is affine we have that each $W_S$ is separated, and hence it will suffice to prove Proposition~\ref{proposition:flat-square} under the addition assumption that $X$ and $V$ are separated. On the other hand, if $X$ and $V$ are separated then each of the $X_{i,j}$ and $V_{i,j}$ are affine, and hence it will suffice to prove Proposition~\ref{proposition:flat-square} under the additional assumption that $X$ and $V$ are affine.

Let us then write $X = \spec(A)$ and $V =\spec(B)$ where $A$ is a noetherian ring and $B$ is a flat $A$-algebra. Let $I \subseteq A$ be the ideal corresponding to the closed set $Z = \spec(A) \setminus U$. Then condition that $\spec(B) \setminus f^{-1}(U) \to Z$ is an isomorphism can then be rephrased as saying that the map $A/I \to B \otimes_A A/I$ is an isomorphism. This means in particular that $\cof[A \to B] \otimes_A A/I = 0$, so that $\cof[A \to B]$ restricts to zero on $Z = \spec(A/I)$.
Using criterion \ref{item:excision-fully-faithful} of Lemma~\ref{lemma:flat-square-criterion}, what we need to show in this case is that the functor
\[
\Der^{\qc}_{Z}(\spec(A)) \to \Der^{\qc}_{Z}(\spec(B))
\]
is fully-faithful. Concretely, this means that if $M \in \Der^{\qc}(\spec(A)) = \Der(A)$ is such that $j^*M = 0$ then the map $M \to M \otimes_A B$ is an equivalence. Equivalently, the object $C = \cof[M \to M \otimes_A B] = M \otimes_A \cof[A \to B]$ is zero. Indeed, $C$ is a tensor product of an object $M$ which restricts to $0$ on $U$ and an object $\cof[A \to B]$ which restricts to zero on $Z$. Since restriction functors are symmetric monoidal we conclude that $C$ restricts to $0$ on both $U$ and $Z$, and hence $C = 0$ by Lemma~\ref{lemma:noetherian-principle}.
\end{proof}

\subsection{Perfect complexes}
\label{subsection:perfect}%

We now proceed to discuss compact objects in $\Der^{\qc}(X)$, which are our primary objects of interest. We begin with the following useful statement, which follows relatively directly form Zariski descent:

\begin{proposition}
\label{proposition:pushforward-colimits}%
Let $f\colon X \to Y$ be a map of qcqs schemes. Then the push-forward functor $f_*\colon \Der^{\qc}(X) \to \Der^{\qc}(Y)$ preserves filtered colimits (and hence all colimits).
\end{proposition}
\begin{proof}
A right adjoint whose target is a compactly generated category preserves filtered colimits if and only if its left adjoint preserves compact objects.
Since equivalences of sheaves can be tested locally and pullback functors preserve colimit we may use open base change to reduce to the case where $Y$ is affine. In this case, $\cO_Y$ is a compact generator of $\Der^{\qc}(Y)$, hence it will suffice to show that $f^*\cO_Y = \cO_X$ is a compact object of $\Der^{\qc}(X)$. Let us say that a qcqs scheme $X$ is \defi{good} if $\cO_X$ is compact in $\Der^{\qc}(X)$. Now by Zariski descent (see Corollary~\ref{corollary:descent-qc}) we get that if we can find a finite covering $X = \cup_{i=1}^{n}U_i$ by qcqs schemes such that for each $\emptyset \neq S \subseteq \{1,\ldots,n\}$ the intersection $U_S = \cap_{i \in S} U_i$ is good then $X$ itself is good. On the other hand, we know that affine schemes are good. The desired result hence follows from Lemma~\ref{lemma:useful-argument}.
\end{proof}

\begin{corollary}
\label{corollary:pullback-compact}%
Let $f\colon X \to Y$ be a map of qcqs schemes. Then the pullback functor $f^*\colon \Der^{\qc}(X) \to \Der^{\qc}(Y)$ preserves compact objects.
\end{corollary}

\begin{lemma}
\label{lemma:perfect-char}%
\label{lemma:perfect}%
Let $X$ be a qcqs scheme and $M \in \Der^{\qc}(X)$ be a quasi-coherent complex. Then the following conditions are equivalent:
\begin{enumerate}
\item
\label{item:compact-object}%
$M$ is a compact object of $\Der^{\qc}(X)$.
\item
\label{item:locally-compact-object}%
For every affine $\spec(A) \subseteq X$ the restricted complex $M|_{\spec(A)}$ is compact in $\Der^{\qc}(\spec(A))$.
\item
\label{item:locally-finite-cover-compact-object}%
There exists a finite covering $X = \cup_{i=1}^{n} \spec(A_i)$ by affines such that for every $i=1,\ldots,n$ the restricted complex $M|_{\spec(A_i)}$ is compact in $\Der^{\qc}(\spec(A_i))$ for every $i=1,\ldots,n$.
\item
\label{item:locally-bounded-projective}%
There exists a finite covering $X = \cup_{i=1}^{n} \spec(A_i)$ by affines such that for every $i=1,\ldots,n$ the $A_i$-module complex $M(\spec(A)) \in \Der(A_i)$ is quasi-isomorphic to a bounded complex of projective $A_i$-modules.
\end{enumerate}
\end{lemma}
\begin{proof}
The implication $\ref{item:compact-object} \Rightarrow \ref{item:locally-compact-object}$ follows from pullback functors preserving compact objects (Corollary~\ref{corollary:pullback-compact}), and the implication $\ref{item:locally-compact-object} \Rightarrow \ref{item:locally-finite-cover-compact-object}$ is because any qcqs $X$ admits a finite covering by affines. In addition, \ref{item:locally-finite-cover-compact-object} and \ref{item:locally-bounded-projective} are equivalent by the affine Serre theorem (Remark~\ref{remark:affine}) and the standard description of compact objects in module categories. 
We now show that $\ref{item:locally-finite-cover-compact-object} \Rightarrow \ref{item:compact-object}$. For this, let us fix a finite covering $X = \cup_{i=1}^{n} U_i$ by affine schemes and an $M \in \Der^{\qc}(X)$ whose restriction to each $U_i$ is compact. For every $\emptyset \neq S \subseteq \{1,\ldots,n\}$ let $U_S = \cap_{i \in S} U_i$. Since $M|_{U_i}$ is compact in $\Der(U_i)$ for every $i$ we get from Corollary~\ref{corollary:pullback-compact} that $M|_{U_S}$ is compact in $\Der^{\qc}(U_S)$ for every $\emptyset \neq S \subseteq \{1,\ldots,n\}$. But by Zariski descent we that
\[
\Der^{\qc}(X) = \lim_{\emptyset \neq S \subseteq \{1,\ldots,n\}} \Der^{\qc}(U_S)
\]
and so $M$ is compact in $\Der^{\qc}(X)$.
\end{proof}

\begin{definition}
A quasi-coherent complex $M \in \Der^{\qc}(X)$ is \defi{perfect} if it satisfies the equivalent conditions of Lemma~\ref{lemma:perfect}. We denote by $\Dperf(X) \subseteq \Der^{\qc}(X)$ the full subcategory spanned by the perfect complexes. For any map $f\colon X \to Y$ between qcqs schemes the pullback functor restricts to a functor $f^*\colon \Dperf(Y) \to \Dperf(X)$, which we denote identically. 
\end{definition}

By the fourth characterization in Lemma~\ref{lemma:perfect-char} we see that $\Dperf(X)$ is closed in $\Der^{\qc}(X)$ under tensor products and so $\Dperf(X)$ inherits from $\Der(X)$ its symmetric monoidal structure, and the functor $X \mapsto \Derqc(X)^{\otimes}$ discussed in \S\ref{subsection:quasi-coherent} refines to a functor
\[
\Sch\op \to \CAlg(\Catx) \quad\quad X \mapsto \Dperf(X)^{\otimes}.
\]
Similarly, taking into account the standard t-structure on $\Derqc(X)$ we obtained a refinement of~\eqref{equation:functor-derqc} to a functor
\[
\Sch\op \to \CAlg(\Catxt) \quad\quad X \mapsto (\Dperf(X), \Derqc(X)_{\geq 0},\Derqc(X)_{\leq 0})
\]
where $(\Catxt)^{\otimes} = (\Catx)^{\otimes} \times_{(\PrL_{\st})^{\otimes}} (\PrL_{\st,\tstruc})^{\otimes}$ is the symmetric monoidal $\infty$-category of oriented stable $\infty$-categories,
see Definition~\ref{definition:catext}.

\begin{remark}
\label{remark:perfect-local}%
Observing the equivalent conditions of Lemma~\ref{lemma:perfect} that the property of being perfect is local in the following sense: for any open covering $X = \cup_i U_i$ we have that a given $M \in \Derqc(X)$ is perfect if and only if $M|_{U_i} \in \Derqc(U_i)$ is perfect for every $i$.
\end{remark}

\begin{remark}
\label{remark:perfect-connective}%
If $X$ is qcqs then any $M \in \Derqc(X)_{\geq 0}$ can be written as a filtered colimits of objects in $\Dperf(X)_{\geq 0} = \Dperf(X) \cap \Derqc(X)_{\geq 0}$, see~\cite{SAG}*{9.6.1.2}. It then follows that $\Derqc(X)_{\geq 0} = \Ind(\Dperf(X)_{\geq 0})$ in this case.
\end{remark}

Now let $X$ be a qcqs scheme and $U \subseteq X$ a qcqs open subset.
We will say that $M \in \Der^{\qc}_Z(X)$ is perfect if it is perfect when considered as an object of $\Der^{\qc}(X)$. We then similarly denote by $\Dperf_Z(X) \subseteq \Der^{\qc}_Z(X)$ the full subcategory spanned by the perfect complexes. In particular, we have a fibre sequence
\begin{equation}
\label{equation:localisation}%
\Dperf_Z(X) \to \Dperf(X) \to \Dperf(U) 
\end{equation}
of stable $\infty$-categories. More generally, if $V \subseteq X$ is another qcqs open subset with closed complement $Y = X \setminus V$, then we have a fibre sequence
\begin{equation}
\label{equation:localisation-2}%
\Dperf_{Z \cap Y}(X) \to \Dperf_{Y}(X) \to \Dperf_{Y \cap U}(U) 
\end{equation}

\begin{remark}
\label{remark:inherited-perfect}%
Since being perfect is a local property the Zariski descent and excision properties (see Corollary~\ref{corollary:descent-qc}) are inherited by perfect complexes. In particular, if $X = V \cup U$ with $V,U$ qcqs and $Z := X \setminus U \subseteq V$ then the square
\[
\begin{tikzcd}
\Dperf(X) \ar[r]\ar[d] & \Dperf(U) \ar[d] \\
\Dperf(V) \ar[r] & \Dperf(U \cap V) 
\end{tikzcd}
\]
is cartesian and the functor $\Dperf_Z(X) \to \Dperf_Z(V)$ is an equivalence.
\end{remark}

\begin{theorem}
\label{theorem:perfect-generators}%
Let $X$ be a qcqs scheme. Then the following holds:
\begin{enumerate}
\item
\label{item:perfect-generation-with-support}%
For every closed subscheme $Z \subseteq X$ with qcqs complement $U$ we have that $\Der^{\qc}_Z(X)$ is generated under filtered colimits by $\Dperf_Z(X)$.
\item
\label{item:karoubi-sequence-intersection-closed}%
For every pair of closed subschemes $Z,Y \subseteq X$ with qcqs complements $U,V$ respectively the sequence~\eqref{equation:localisation-2} is a Karoubi sequence.
\end{enumerate}
\end{theorem}

\begin{corollary}
\label{corollary:nisnevich-karoubi}%
Let $X$ be a qcqs scheme, $f\colon V \to X$ a qcqs étale map and $U \subseteq X$ a qcqs open subset such that the induced map $V \setminus f^{-1}(U) \to X \setminus U$ is an isomorphism. Then 
\[
\begin{tikzcd}
\Dperf(X) \ar[r]\ar[d] & \Dperf(V) \ar[d] \\
\Dperf(U) \ar[r] & \Dperf(f^{-1}(U))
\end{tikzcd}
\]
is a Karoubi square of stable $\infty$-categories.
\end{corollary}
\begin{proof}
By Theorem~\ref{theorem:perfect-generators} the vertical arrows are Karoubi projections, and hence it will suffice to show that the square is cartesian. Since the analogous square for quasi-coherent complexes is cartesian by Proposition~\ref{proposition:flat-square}, it will suffice to show that a quasi-coherent complex $M \in \Der^{\qc}(X)$ is perfect if and only if its restriction to $U, V$ and $f^{-1}(U)$ is perfect. Indeed, the only if direction is because pullback functors preserve perfect complexes, and on the other hand, if the restrictions of $M$ to $U,V$ and $f^{-1}(U)$ are all perfect and the image of $M$ under the equivalence 
\[
\Der^{\qc}(X) \tosimeq \Der^{\qc}(U) \times_{\Der^{\qc}(f^{-1}(U))} \Der^{\qc}(V)
\]
is compact, and hence $M$ itself is compact.  
\end{proof}

We prove Theorem~\ref{theorem:perfect-generators} in several steps.

\begin{lemma}
\label{lemma:i-implies-ii}%
For a fixed qcqs $X$, Claim~\ref{item:perfect-generation-with-support} of Theorem~\ref{theorem:perfect-generators} implies Claim~\ref{item:karoubi-sequence-intersection-closed}.
\end{lemma}
\begin{proof}
Assume \ref{item:perfect-generation-with-support} holds for $X$ and let us prove that \ref{item:karoubi-sequence-intersection-closed} holds as well. Since~\eqref{equation:localisation-2} is a fibre sequence it will suffice to show that its second map is a Karoubi projection. By excision for perfect complexes (see Remark~\ref{remark:inherited-perfect}) we may identify the last term in~\eqref{equation:localisation-2} as $\Dperf_{U \cap Y}(U) = \Dperf_{U \cap Y}(U \cup V) = \Dperf_{(U \cup V) \cap Y}(U \cup V)$. Replacing $U$ with $U \cup V$ we may hence assume without loss of generality that $V \subseteq U$ and $Z \subseteq Y$. Now in the commutative diagram
\[
\begin{tikzcd}
\Dperf_{Y}(X) \ar[r]\ar[d] & \Dperf_{Y \cap U}(U) \ar[d]\ar[r] & 0 \ar[d] \\
\Dperf(X) \ar[r] & \Dperf(U) \ar[r] & \Dperf(V)
\end{tikzcd}
\]
the right square and external rectangle are cartesian, and hence the left square is cartesian as well. We conclude that the functor $\Dperf_{Y}(X) \to \Dperf_{Y \cap U}(U)$ is base changed from $\Dperf(X) \to \Dperf(U)$. Since Karoubi projections are closed under base change (\cite{9-authors-II}*{Lemma~A.3.10}) it will suffice to show that $\Dperf(X) \to \Dperf(U)$ is a Karoubi projection. Since we know \ref{item:perfect-generation-with-support} holds for $X$ we know that $\Der^{\qc}(X)$ and $\Der^{\qc}_Z(X)$ are generated under filtered colimits by $\Dperf(X)$ and $\Dperf_Z(X)$. This also implies that $\Der^{\qc}(U)$ is generated under filtered colimits by $\Dperf(U)$: indeed, given $M \in \Der^{\qc}(U)$ we can write $j_*M$ as a filtered colimits of perfect complexes over $X$, and thus obtain a similar presentation over $U$ by applying the colimit preserving functor $j^*$. In particular, applying the functor $\Ind$ to the sequence 
\[
\Dperf_Z(X) \to \Dperf(X) \to \Dperf(U)
\]
yields the sequence
\[
\Der^{\qc}_Z(X) \to \Der^{\qc}(X) \xrightarrow{j^*} \Der^{\qc}(U).
\]
The last sequence is a right split Verdier sequence (of large stable $\infty$-categories) since the functor $j^*$ has a fully-faithful right adjoint $j_*$. It then follows that the former sequence is a Karoubi sequence by~\cite{9-authors-II}*{Theorem~A.3.11}.
\end{proof}

\begin{lemma}
\label{lemma:affine}%
Theorem~\ref{theorem:perfect-generators} holds whenever $X= \spec(A)$ is affine.
\end{lemma}
\begin{proof}
By Lemma~\ref{lemma:i-implies-ii} it will suffice to show \ref{item:perfect-generation-with-support} in the case where $X = \spec(A)$ is affine.
If $U$ is empty then $\Der^{\qc}_Z(X) = \Der^{\qc}(X) = \Der(A)$ and we know that $\Der(A)$ is generated under filtered colimits by $\Dperf(A)$. We may hence assume that $U$ is not empty. For brevity let us write $\Der_Z(A)$ and $\Dperf_Z(A)$ for $\Der^{\qc}_Z(\spec(A))$ and $\Dperf_Z(\spec(A))$, respectively. Let $I \subseteq A$ be the ideal of elements vanishing on $Z$. Then $U$ can be covered by the opens $\spec(A[f^{-1}])$ for $f \in I$, and since $U$ is quasi-compact we may choose a finite subset $\{f_1,\ldots,f_n\} \in I$ such that $\spec(A[f^{-1}_1]),\ldots,\spec(A[f^{-1}])$ cover $U$. Let $j_i\colon \spec(A[f_i]) \hrar \spec(A)$ be the inclusion. For each $i=1,\ldots,n$ let 
\[
E_i = \cof[A \xrightarrow{f_i} A] \in \Dperf(A)
\]
and let $E=\otimes_i E_i$. We note that $E$ is a perfect complex. In addition, we have that $j^*_iE_i =0$ for every $i$ and hence $j^*E = 0$. We conclude that $E \in \Dperf_Z(A)$. 

Let us write $\Hom_A(-,-)$ for the mapping complex between two $A$-module complexes. We claim that for an $A$-module $M$, if $\Hom_A(E,M) = 0$ then $M=0$. In other words, we claim that the functor 
\[
\Hom_A(E,-)\colon \Der_Z(A) \to \Der_Z(A)
\]
has trivial kernel. This implies that via standard arguments that $\Der_Z(A)$ is generated under filtered colimits by its smallest stable subcategory containing $K$. By definition, $\Hom_A(E,-)$ is the composite of the functors 
\[
\Hom_A(E_i,-) \colon \Der_Z(A) \to \Der_Z(A)
\]
for $i=1,\ldots,n$ (to avoid confusion, note that the $E_i$'s do not belong to $\Der_Z(A)$, but $\Der_Z(A)$ is closed in $\Der(A)$ under exponentiation with any perfect complex). It will hence suffice to show that each of the functors $\Hom_A(E_i,-)$ has trivial kernel in $\Der_Z(A)$. Indeed, if $\Hom_A(E_i,M) =0$ then $f_i$ acts invertibly on $M$, and hence 
\[
M = M[f_i^{-1}] = (j_i)_*j^*_iM .
\]
But $j^*_iM = 0$ (because $j^*_M = 0$), and so we conclude that $M=0$.
\end{proof}

\begin{lemma}
\label{lemma:induction}%
Suppose that the qcqs scheme $X$ can be covered by two qcqs open subsets $X = X' \cup X''$. If Theorem~\ref{theorem:perfect-generators} holds for $X'$ and for $X''$ then it holds for $X$.
\end{lemma}
\begin{proof}
By Lemma~\ref{lemma:i-implies-ii} it will suffice to prove that \ref{item:perfect-generation-with-support} of Theorem~\ref{theorem:perfect-generators} holds for $X$. Let $Z \subseteq X$ be a closed subset of qcqs complements $U$. We need to show that $\Der^{\qc}_Z(X)$ is generated under filtered colimits by $\Dperf_Z(X)$. Write $Z' = Z \cap X'$, $Z'' = Z \cap X''$, $Y' = X \setminus X'$ and $Y'' = X \setminus X''$. Consider the commutative diagram
\[
\begin{tikzcd}
\Dperf_{Z \cap Y''}(X) \ar[r]\ar[d, "\simeq"] & \Dperf_Z(X) \ar[r]\ar[d] & \Dperf_{Z''}(X'') \ar[d] \\
\Dperf_{Z' \cap Y''}(X') \ar[r] & \Dperf_{Z'}(X') \ar[r] & \Dperf_{Z' \cap Z''}(X' \cap X'')
\end{tikzcd}
\]
Here, the rows are fibre sequences and the right square is a pullback square by Zariski descent (see Remark~\ref{remark:zariski-descent-support}).  
Now since Theorem~\ref{theorem:perfect-generators} is assumed known for $X'$ the bottom row is a Karoubi sequence and hence the top row is a Karoubi sequence by~\cite{9-authors-II}*{Lemma~A.3.10}. Now consider the diagram
\[
\begin{tikzcd}
\Ind(\Dperf_{Z \cap Y''}(X)) \ar[r]\ar[d,"\simeq"] & \Ind(\Dperf_Z(X)) \ar[r]\ar[d] & \Ind(\Dperf_{Z''}(X'')) \ar[d,"\simeq"] \\
\Der^{\qc}_{Z \cap Y''}(X) \ar[r] & \Der^{\qc}_Z(X) \ar[r] & \Der^{\qc}_{Z''}(X'') 
\end{tikzcd}
\]
where the right vertical map is an equivalence since we assume Theorem~\ref{theorem:perfect-generators} for $X''$ and the left vertical map is an equivalence since $\Der^{\qc}_{Z \cap Y''}(X) = \Der^{\qc}_{Z' \cap Y''}(X')$ by excision and we assume Theorem~\ref{theorem:perfect-generators} for $X'$. Here the bottom row is a right split Verdier sequence since the pullback functor $\Der^{\qc}_Z(X) \to \Der^{\qc}_{Z''}(X'')$ has a fully-faithful right adjoint given by push-forward (see Remark~\ref{remark:pushforward-support}), and the top row is a right split Verdier sequence since it is obtained by applying $\Ind$ to a Karoubi sequence, see~\cite{NS}*{Proposition I.3.5}. Now the middle vertical map is at least fully-faithful, since every perfect complex is compact in $\Der^{\qc}_Z(X)$ (since it is compact in $\Der^{\qc}(X)$). This implies that the right adjoint of the bottom Verdier inclusion gives a right adjoint to the top Verdier inclusion when restricted to $\Ind(\Dperf_Z(X))$. In other words, the left square is right adjointable, in the sense that it remains commutative when replacing the two horizontal Verdier inclusions by their right adjoints. 
Since both rows are right split Verdier sequences, the counit of the Verdier inclusion and the unit of the Verdier projection form a fibre sequence. It follows that the right square is right adjointable as well, in the sense that it remains commutative when replacing the two horizontal Verdier projections by their right adjoints. 
It then follows that the essential image of the middle vertical functor contain the essential image of $\Der^{\qc}_{Z''}(X'')$ via push-forward. Since it also contains $\Der^{\qc}_{Z \cap Y''}(X)$ it must contain all of $\Der^{\qc}_Z(X)$, and so the proof is complete.
\end{proof}

\begin{proof}[Proof of Theorem~\ref{theorem:perfect-generators}]
Since $X$ is qcqs it admits a finite covering $X = \cup_{i=1}^{n} U_i$ such that each $U_i$ is affine. We then argue by induction on the minimal number $n$ of affine opens needed to cover $X$. If $n=1$ then $X$ is affine and the claim is covered by Lemma~\ref{lemma:affine}. Otherwise, $X$ can be covered by two open subsets each of which can be covered by less than $n$ affines, and hence the claim follows from the induction hypothesis via Lemma~\ref{lemma:induction}.
\end{proof}

\begin{corollary}
\label{corollary:limit-open-karoubi}%
Let $X$ be a qcqs scheme and $p \in X$ a (not necessarily closed) point. Then the restriction functor
\[
\Dperf(X) \to \Dperf(\cO_{X,p})
\]
is a Karoubi projection.
\end{corollary}

The proof of Corollary~\ref{corollary:limit-open-karoubi} will require the following lemma:

\begin{lemma}
\label{lemma:inverse-system}%
Let $\I$ be a filtered poset and $\{U_{\alp}\}_{\alp \in \I\op}$ an $\I\op$-indexed diagram of qcqs schemes such that each of the maps $p_{\alp,\beta}\colon U_{\bet} \to U_{\alp}$ in the diagram is affine. Let $U = \lim_{\alp \in \I} U_\alpha$ and $p_{\alp}\colon U \to U_{\alp}$ the canonical maps. Then the functor $\colim_{\alp \in \I}\Dperf(U_{\alp}) \to \Dperf(U)$ induced by $(M,\alp) \mapsto p^*_{\alp}M$ is an equivalence.
\end{lemma}
\begin{proof}
The condition that each of the maps in the diagram $\{U_{\alp}\}$ are affine implies that each of the maps $p_{\alp}\colon U \to U_{\alp}$ is affine. In particular, $U$ is qcqs.
Without loss of generality we may assume that $\I$ has an terminal element $\alp_0 \in \I$ (otherwise, pick any $\alp \in \I$ and replace $\I$ by the cofinal subposet $\{\bet \in \I | \bet \geq \alp\}$. Since $U_{\alp_0}$ is qcqs we may cover it by a finite collection of open affine subschemes. Since each of the $p_{\alp,\bet}\colon U_{\bet} \to U_{\alp}$ is affine this covering determines an open affine covering of each $U_{\bet}$, and similarly an open affine covering of $U$. Applying Zariski descent for perfect complexes (Remark~\ref{remark:inherited-perfect}) and using the commutation of filtered colimits with finite limits we may reduce to the case where all the $U_{\alp}$'s are affine. Transforming the problem to the setting of commutative rings, we find a filtered diagram $\{A_{\alp}\}_{\alp \in \I}$ of rings, and we need to show that  the functor
\[
\colim_{\alp} \Dperf(A_{\alp}) \to \Dperf(A) \quad\quad (M,\alp) \mapsto M \otimes_{A_{\alp}} A
\]
is an equivalence, where $A = \colim_{\alp} A_{\alp}$. We first show that this functor is fully-faithful. For this, it will suffice to show that for any $\alp \in \I$ and any $M,M'\in \Dperf(A_{\alp})$ the map
\[
\colim_{\bet \geq \alp} \map_{A_{\bet}}(A_{\bet} \otimes_{A_{\alp}} M,A_{\bet} \otimes_{A_{\alp}} M) \to \map_{A}(A \otimes_{A_{\alp}} M,A \otimes_{A_{\alp}} M)
\]
is an equivalence. By adjunction, we can identify this map with the map
\[
\colim_{\bet \geq \alp} \map_{A}(M,A_{\bet} \otimes_{A_{\alp}} M) \to \map_{A_{\alp}}(M,A \otimes_{A_{\alp}} M) ,
\]
and since $M$ is perfect as an $A_{\alp}$-module we can also take $\map_A(M,-)$ out of the colimit, so that it will suffice to show that the map
\[
\colim_{\bet \geq \alp}  A_{\bet} \otimes_A M \to A \otimes_A M
\]
is an equivalence. Indeed, this is clear since $A = \colim_{\bet}A_{\bet} = \colim_{\bet \geq \alp} A_{\bet}$ and $(-) \otimes_A M$ preserves colimits.

We now show that this functor is essentially surjective. Since we already have that the functor in question is fully-faithful it will suffice to show that its image generates all of $\Dperf(A)$ under finite colimits and desuspensions. In particular, it will suffice to show that every finitely generated projective $A$-module $N$ is of the form $N = M \otimes_{A_{\alp}} A$ for some finitely generated projective $A_{\alp}$-module. Indeed, $N$ is a retract of $A^n$ for some $n$, and the coefficients of the associated idempotent $(n \times n)$-matrix must be in the image of $A_{\alp} \to A$ for some sufficiently large $\alp$.
\end{proof}

\begin{proof}[Proof of Corollary~\ref{corollary:limit-open-karoubi}]
We may write $\spec(\cO_{x,p}) = \lim_{\alp} U_{\alp}$, where the $U_{\alp}$ ranges over all affine open neighborhoods of $p$. Since each of the functor $\Dperf(X) \to \Dperf(U_{\alp})$ is a Karoubi projection by Theorem~\ref{theorem:perfect-generators} and the collection of Karoubi projections is closed under colimits, we get from Lemma~\ref{lemma:inverse-system} that
\[
\Dperf(X) \to \Dperf(\cO_{X,p}) = \colim_{\alp} \Dperf(U_{\alp})
\]
is a Karoubi projection as well.
\end{proof}

\subsection{Quasi-perfect maps}
\label{subsection:quasi-perfect}%

Recall from Proposition~\ref{proposition:pushforward-colimits} that for any map $f\colon X \to Y$ of qcqs schemes, the push-forward functor $f_*\colon \Derqc(X) \to \Derqc(Y)$ preserves colimits. Since $\Derqc(X)$ is presentable this means that $f_*$ admits a right adjoint $f^!\colon \Derqc(Y) \to \Derqc(X)$.  

\begin{remark}[Flat base change for push-forward and upper shriek]
\label{remark:flat-base-change-shriek}%
Given a cartesian square
\begin{equation}
\label{equation:flat-square}%
\begin{tikzcd}
V \ar[r,"i"]\ar[d,"f"'] & X \ar[d,"g"] \\
U \ar[r,"j"] & Y 
\end{tikzcd}
\end{equation}
of qcqs schemes with $j$ and $i$ flat, we have by flat base change (Proposition~\ref{proposition:flat-base-change}) an equivalence 
$j^*g_* \tosimeq f_*i^*$. Replacing all functors by their right adjoints we get an equivalence 
\[
g^!j_* \tosimeq i_*f^!
\]
Otherwise put, since the square of push-forward functors associated to~\eqref{equation:shriek-base-change} is horizontally left adjointable by flat base change and all functors have right adjoints, it follows formally that it is also vertically right adjointable, and we obtain a commutative square
\begin{equation}
\label{equation:shriek-base-change}%
\begin{tikzcd}
\Derqc(V) \ar[r,"i_*"] & \Derqc(X)  \\
\Derqc(U) \ar[r,"j_*"]\ar[u,"f^!"] & \Derqc(Y) \ar[u,"g^!"'] 
\end{tikzcd}
\end{equation} 
\end{remark}

Let $f\colon X \to Y$ be a map of qcqs schemes. Then $f^*\colon \Derqc(Y) \to \Derqc(X)$ is symmetric monoidal, and hence can in particular be considered as a $\Derqc(Y)$-linear functor (where $\Derqc(Y)$ acts on $\Derqc(X)$ by restricting $\Derqc(X)$'s self action via $f^*$). This $\Derqc(Y)$-linear structure induces a lax $\Derqc(Y)$-linear structure on the right adjoint $f_*$, which the projection formula asserts is actually an honest $\Derqc(Y)$-linear structure. This last $\Derqc(Y)$-linear structure then induces a lax $\Derqc(Y)$-linear structure on its further right adjoint $f^!\colon \Derqc(Y) \to \Derqc(X)$. Concretely, given a map $f\colon X \to Y$ of qcqs schemes and $M,N \in \Derqc(Y)$, the map 
\[
f_*(f^*(M) \otimes f^!(N)) \simeq M \otimes f_*f^!(N) \to M \otimes N
\]
obtained using the projection formula and the counit of $f_* \dashv f^!$ determines a natural map
\[
\tau\colon f^*(M) \otimes f^!(N) \to f^!(M \otimes N) ,
\]
encoding a lax $\Derqc(Y)$-linear structure on $f^!\colon \Derqc(Y) \to \Derqc(X)$.

\begin{lemma}
\label{lemma:eta-equivalence}%
The map $\tau$ is an equivalence if $M$ is perfect. If $f$ is quasi-perfect then $\tau$ is an equivalence for any $M$.
\end{lemma}
\begin{proof}
The first claim follows from Lemma~\ref{lemma:rune-magic-lemma}, by considering the underlying lax $\Dperf(Y)$-linear structure on $f^!$. If $f$ is quasi-perfect then $f^!$ preserves colimits and so the domain and target of $\tau$ both depend on $M$ in a colimit preserving manner. Since $\Dperf(Y)$ generates $\Derqc(Y)$ under colimits the second part follows from the first part.
\end{proof}

\begin{definition}
\label{definition:quasi-perfect}%
A map $f\colon X \to Y$ is said to be \defi{quasi-perfect} if $f_*$ preserves perfect complexes, or, equivalently, if $f^!$ preserves colimits. 
\end{definition}

Clearly quasi-perfect maps are closed under composition. They are also closed under base change against open embeddings:

\begin{lemma}
\label{lemma:flat-quasi-perfect}%
Consider a cartesian square of qcqs schemes as in~\eqref{equation:flat-square}, and suppose in addition that $i$ and $j$ are open embeddings. If $g$ is quasi-perfect then $f$ is quasi-perfect and the square~\eqref{equation:shriek-base-change} is left adjointable, that is, the associated mate transformation 
\[
\eta\colon i^*g^!  \Rightarrow f^!j^*
\]
is an equivalence. 
\end{lemma}
\begin{proof}
To see that $f$ is quasi-perfect let $M \in \Derqc(U)$ be a perfect complex. 
Since $i$ is an open embedding we have that every perfect complex $M \in \Dperf(U)$ is a retract of one of the form $i^*N$ for $N \in \Dperf(X)$. Now since $g$ is quasi-perfect we have that $f_*i^* \simeq j^*g_*\colon \Dperf(X) \to \Dperf(V)$ preserves perfect complexes and hence $f$ is quasi-perfect as well.

We now prove that $\eta$ is an equivalence. For this, note that any $M \in \Derqc(Y)$ sits in a fibre sequence of the form
\[
N \to M \to j_*j^*M
\]
with $N$ supported on the complement of $U$. It will hence suffice to show that $\eta$ is an equivalence on objects which are either in the image of $j_*$ or in the kernel of $j^*$. If $M$ is in the kernel of $j^*$ then $f^!j^*M = 0$ and 
\[
i^*g^!M = i^*(g^*M \otimes g^!\cO_Y) = i^*g^*M \otimes i^*g^!\cO_Y = j^*f^*M \otimes i^*g^!\cO_Y = 0
\]
by Lemma~\ref{lemma:eta-equivalence}. Now assume that $M$ is in the image of $j_*$. Consider the commutative diagram
\[
\begin{tikzcd}
\Derqc(V)\ar[d, "f_*"]  & \Derqc(X)\ar[d, "g_*"]\ar[l, "i^*"'] & \Derqc(V)\ar[d, "f_*"]\ar[l, "i_*"'] \\
\Derqc(U) & \Derqc(Y)\ar[l, "j^*"'] & \Derqc(U) \ar[l, "j_*"'] 
\end{tikzcd}
\]
obtained by gluing the square encoding the functoriality of push-forward for the square~\eqref{equation:flat-square} with its horizontal mate (which is itself a commutative square by open base change). Now in the external rectangle the horizontal arrows are equivalences (since $i_*$ and $j_*$ are fully-faithful) and the vertical arrows have right adjoints and hence this rectangle is automatically vertically right adjointable. In addition, the right square is vertically right adjointable by Remark~\ref{remark:flat-base-change-shriek}. It then follows that the vertical right mate transformation of the left square is an equivalence on objects in the image of $j_*$. But this vertical mate is the same as the horizontal left mate of the corresponding square of right adjoints~\eqref{equation:shriek-base-change}, which is $\eta$.
\end{proof}

We deduce that being quasi-perfect is a local property on the codomain:
\begin{corollary}
\label{corollary:quasi-perfect-local}%
Let $f\colon X \to Y$ be a map of qcqs schemes, $Y = \cup_i V_i$ an open covering of $Y$. Then $f$ is quasi-perfect if and only if its base change $f_i\colon U_i = X \times_Y V_i \to V_i$ to $V_i$ is quasi-perfect for every $i$. In addition, when these equivalent conditions hold we have that $(f^!M|)_{U_i} \simeq f_i^!(M|_{V_i})$ for every $i$.
\end{corollary}
\begin{proof}
If $f$ is quasi-perfect then each $f_i$ is quasi-perfect by the first part of Lemma~\ref{lemma:flat-quasi-perfect}. On the other hand, if each $f_i$ is quasi-perfect and $M \in \Dperf(X)$ is a perfect complex then $(f_*M)|_{V_i} \simeq (f_i)_*(M|_{U_i})$ is perfect for every $i$ and so $f_*M$ is perfect  (see Remark~\ref{remark:perfect-local}). The last part now follows from the second part of Lemma~\ref{lemma:flat-quasi-perfect}. 
\end{proof}

\section{The Zariski and Nisnevich topologies}
\label{section:finitary}%

Throughout this section, we fix a presentable $\infty$-category $\A$. 

\subsection{Extensive $\infty$-categories and coverages}

Recall that an $\infty$-category $\C$ is said to be \defi{extensive} if it admits coproducts and the functor
\[
\C_{/x} \times \C_{/y} \to \C_{/x \amalg y} \quad\quad ([a \to x],[b \to y]) \mapsto \Big[a \amalg b \to x \amalg y \Big]
\]
is an equivalence for every $x,y \in \C$. On extensive $\infty$-categories one can construct Grothendieck topologies from a certain data known as a \emph{coverage}, which is a collection of morphisms which behave like covering morphisms. More precisely, a collection of morphisms in an extensive $\infty$-category $\C$ is called a coverage if it is closed under composition, base change and coproducts (that is, if $x \to y$ and $z \to w$ are in the collection, then so is $x \amalg z \to y \amalg w$). One then refers to the morphisms in the coverage as \defi{epimorphisms}. To a coverage one may associate a Grothendieck topology by declaring that a sieve on $y$ is covering if and only if it contains a finite collection $\{x_i \to y\}_{i=1,\ldots,n}$ such that $\amalg_i x_i \to y$ is an epimorphism, see~\cite[Proposition 3.2.1]{SAG}. This Grothendieck topology is finitary by construction. In addition, as proven in~\cite[Proposition 3.3.1]{SAG}, a presheaf $\F\colon \C\op \to \A$ is a sheaf with respect to this Grothendieck topology if and only if it satisfies the following two conditions:
\begin{enumerate}
\item
$\F$ sends finite coproducts in $\C$ to products in $\A$.
\item
For every epimorphism $p\colon x \to y$ in the coverage, $\F$ satisfies Čech descent for $p$, that is, the induced map
\[
\F \to \Tot \F(U_\bullet(p))
\]
is an equivalence. Here $U_\bullet(f)$ is the Čech resolution of $p$, that is, the simplicial object whose $n$-simplices consist of the $(n+1)$-fold fibre product $x \times_y \cdots \times_y x$.
\end{enumerate}

We record the following result:
\begin{lemma}
\label{lemma:descent-square}%
Let $\C$ be an extensive $\infty$-category equipped with a coverage and let 
\begin{equation}
\label{equation:square-epi}%
\begin{tikzcd}
x \ar[r, "p'"]\ar[d, "q'"'] & y \ar[d, "q"] \\
z \ar[r,"p"] & w 
\end{tikzcd}
\end{equation}
be a commutative square all of whose arrows are epimorphisms. Let $\F\colon \C\op \to \A$ be a presheaf which satisfies Čech descent for the morphisms $U_n(p') \to U_n(p)$ and $U_m(q') \to U_m(q)$ for every $n,m \geq 0$.
Then $\F$ satisfies Čech descent for $q$ if and only if it satisfies Čech descent for $p$.
\end{lemma}
\begin{proof}
Let $\Del_+ = \Del \cup \{\emptyset\} = \Del^{\triangleleft}$ be the extended simplex category, so functors $\Del_+\op \to \C$ correspond to augmented simplicial object and the formation of Čech resolutions (including their augmentation) can be identified with right Kan extension along the functor $\iota\colon \Del^1 \to \Del_+\op$ picking the arrow $[0] \leftarrow \emptyset$ of $\Del_+\op$. Considering the square~\eqref{equation:square-epi} as a functor $\rho\colon \Del^1 \times \Del^1 \to \C$, we may right Kan extend it along $\iota \times \iota \colon \Del^1 \times \Del^1 \to \Del_+\op \times \Del_+\op$ to obtain a functor
\[
W\colon \Del_+\op \times \Del_+\op \to \C .
\]
We then have that $W_{\bullet,\emptyset}$ is the Čech resolution of $p$, $W_{\emptyset,\bullet}$ is the Čech resolution of $q$, $W_{\bullet,0}$ is the Čech resolution of $p'$ and $W_{0,\bullet}$ is the Čech resolution of $q'$. Similarly, $W_{\bullet,m}$ is the Čech construction of $U_m(q') \to U_m(q)$ and $W_{n,\bullet}$ is the Čech resolution of $U_n(p') \to U_n(p)$. Our assumptions hence imply that in the commutative diagram
\[
\begin{tikzcd}
& \F(w) \ar[dr]\ar[dl] & \\
\displaystyle\mathop{\lim}_{n \in \Del}\F(U_n(p)) \ar[dr, "\simeq"'] &&  \displaystyle\mathop{\lim}_{m \in \Del}\F(U_m(q)) \ar[dl, "\simeq"] \\
&  \displaystyle\mathop{\lim}_{(n,m)\in \Del \times \Del} \F(W_{n,m}) &
\end{tikzcd}
\]
the two bottom arrows are equivalences. We then conclude that the top left arrow is an equivalence if and only if the top right arrow is one, as desired.
\end{proof}

\begin{corollary}
\label{corollary:epi-refinement}%
Let $\C$ be an extensive $\infty$-category equipped with a coverage and let $z \xrightarrow{g} y \xrightarrow{f} w$ be a composable sequence of maps such that $f$ and $f \circ g$ are epis. 
Let $\F\colon \C\op \to \A$ be a presheaf. If $\F$ satisfies Čech descent with respect to $f \circ g$ and all of its base changes, then $\F$ also satisfies Čech descent with respect to $f$.
\end{corollary}
\begin{proof}
Apply Lemma~\ref{lemma:descent-square} to the square
\[
\begin{tikzcd}
z \times_w y \ar[r]\ar[d] & y \ar[d, "f"] \\
z \ar[r,"f \circ g"]\ar[ur,dashed,"g"] & w 
\end{tikzcd}
\]
and use the fact that any presheaf satisfies Čech descent with respect to split epis (that is, those that admit a section), because the Čech resolution extends in this case to a split simplicial object.
\end{proof}

\subsection{The Zariski and Nisnevich coverages}

Let $S$ be a qcqs scheme. The category $\qSch_{/S}$ of qcqs schemes over $S$ is extensive: indeed, this follows from the fact that the category of all schemes is extensive and the collection of qcqs schemes is closed under fibre products and finite coproducts. We may hence use the formalism of coverages to define Grothendieck topologies on $\qSch$. As we will see below, both the Zariski and Nisnevich topologies can be obtained in this manner. Our main motivation for taking up this point of view is that it allows for a simple argument establishing Proposition~\ref{proposition:reduce-affine} below.

\begin{definition}
\label{definition:epi}%
We will say that a map $p\colon V \to X$ of qcqs $S$-schemes is
\begin{enumerate}
\item
A \defi{Zariski epi} if $p$ is étale and there exists a finite open cover $Y = \cup_i U_i$ such that for every $i=1,\ldots,n$ there exists a dotted lift 
\[
\begin{tikzcd}
& V \ar[d, "p"] \\
U_i  \ar[ur,dotted] \ar[r] & Y
\end{tikzcd}
\]
as indicated.
\item
A \defi{Nisnevich epi} if $p$ is étale and there exists a finite sequence of open subsets $\emptyset = U_0 \subseteq U_1 \subseteq \cdots \subseteq U_n = X$ of $X$ such that for every $i=1,\ldots,n$ there exists a dotted lift 
\[
\begin{tikzcd}
& V \ar[d, "p"] \\
U_i \setminus U_{i-1} \ar[ur,dotted] \ar[r] & X
\end{tikzcd}
\]
as indicated. 
\end{enumerate}
\end{definition}

\begin{example}
If $X = \cup_i U_i$ is a finite open cover then $\coprod U_i \to X$ is a Zariski epi.
\end{example}

\begin{example}
Any Zariski epi is also a Nisnevich epi.
\end{example}

\begin{lemma}
\label{lemma:zariski-nisnevich-coverage}%
Both the collection of Zariski epis and the collection of Nisnevich epis satisfy the axioms of a coverage, that is, they are closed under composition, base change and coproducts.
\end{lemma}
\begin{proof}
We begin with the collection of Zariski epis. Here, closure under base change and disjoint union is straightforward. For composition, suppose that $V \to Y \to X$ is a composable sequence of Zariski epis. Then by definition we have a finite open cover $X = \cup_i U_i$ such that the base change $Y \times_X U_i \to U_i$ admits a section $s_i\colon U_i \to Y \times_X U_i$ for every $i$ and a finite open cover $Y = \cup W_j$ such that $V \times_Y W_j \to W_j$ admits a section for every $j$. Setting $Z_{i,j} := s_i^{-1}(W_j \times_X U_i)$ we obtain a finite open cover $X = \cup_{i,j} Z_{i,j}$ such that each of the base changes $V \times_{X} Z_{i,j} \to Z_{i,j}$ admits a section, as desired.

We now consider the collection of Nisnevich epis. Closure under base change remains straightforward. Now consider two Nisnevich epis $V \to X$ and $W \to Y$. We want to show that $V \amalg W \to X \amalg Y$ is a Nisnevich epi.
\end{proof}

\begin{remark}
If $k$ is a field then the only non-empty open subset of $\spec(k)$ is $\spec(k)$ itself, and hence any Nisnevich epi $V \to \spec(k)$ admits a section. It then follows that for any scheme $X$ if $p\colon V \to X$ is a Nisnevich epi then the induced map $p(k)\colon V(k) \to X(k)$ on $k$-points is surjective for any field $k$. If $X$ is Noetherian then the inverse implication holds as well. In other words, Nisnevich epis over Noetherian schemes can be characterized as those morphisms which are surjective on $k$-points for any field $k$. 
\end{remark}

Our next goal is to show that the Zariski and Nisnevich Grothendieck topologies on qcqs schemes coincides, respectively, with the Grothendieck topologies associated the Zariski and Nisnevich coverages, as above. We begin with the Zariski topology.

\begin{proposition}
\label{proposition:zariski-coverage}%
Let $\E \subseteq \Sch_{/X}$ be a sieve on qcqs scheme $X$. Then the following are equivalent:
\begin{enumerate}
\item
\label{item:covering-sieve}%
$\E$ is a covering sieve for the Zariski topology.
\item
\label{item:open-zariski-epi}%
$\E$ contains a finite collection of open embeddings $\{U_i \to X\}_{i=1,\ldots,n}$ such that the induced map $\coprod_i U_i \to X$ is a Zariski epi in the sense of Definition~\ref{definition:epi}. 
\item
\label{item:zariski-epi-finite}%
$\E$ contains a finite collection of maps $\{Z_i \to X\}_{i=1,\ldots,n}$ such that the induced map $\coprod_i Z_i \to X$ is a Zariski epi in the sense of Definition~\ref{definition:epi}. 
\end{enumerate}
\end{proposition}
\begin{proof}
If $\E$ is a Zariski covering sieve then by definition (and since $X$ is quasi-compact) it contains a finite collection of open embeddings $\{U_i \to X\}_{i=1,\ldots,n}$ which form an open covering of $X$. In this case clearly the induced map $\coprod_i U_i \to X$ is Zariski epi, so that $\ref{item:covering-sieve} \Rightarrow \ref{item:open-zariski-epi}$. Clearly also $\ref{item:open-zariski-epi} \Rightarrow \ref{item:zariski-epi-finite}$. We now prove that $\ref{item:zariski-epi-finite} \Rightarrow \ref{item:covering-sieve}$.

Suppose there exists a finite $\{Z_i \to X\}_{i=1,\ldots,n}$ such that the induced map $Z = \coprod_i Z_i \to X$ is a Zariski epi. Then there exists finite collection of open subschemes $U_1,\ldots,U_m \subseteq X$ such that $Z \times U_j \to U_j$ admits a section $s \colon U_j \to Z \times_X U_j$. Now each $Z_i$ is an open and closed subscheme of $Z$, and so the section $s$ determines a decomposition $U_j = \coprod_i U_{i,j}$ where $U_{i,j} = s^{-1}(Z_i \times_X U_j)$. Then each of the open embeddings $U_{i,j} \hrar X$ factors through $Z \to X$ and is hence contained in $\E$. We conclude that the sieve $\E$ contains an open covering of $X$ and is hence a covering sieve with respect to the Zariski topology.
\end{proof}

\begin{corollary}
\label{corollary:zariski-sheaf}%
Let $\F\colon (\qSch_{/S})\op \to \A$ be a presheaf. Then the following are equivalent:
\begin{enumerate}
\item
\label{item:F-Zar-sheaf}%
$\F$ is a sheaf with respect to the Zariski topology.
\item
\label{item:F-coproducts-to-products-zar-epi}%
$\F$ sends coproducts in $\qSch_{/S}$ to products in $\A$ and satisfies Čech descent with respect to any Zariski epi $p\colon X \to Y$.
\item
\label{item:F-coprod-to-prod-zar-coprod}%
$\F$ sends coproducts in $\qSch_{/S}$ to products in $\A$ and satisfies Čech descent with respect to any Zariski epi of the form $\coprod_{i=1}^{n} U_i \to X$, where $U_1,...,U_n$ are an open cover of $X$. 
\item
\label{item:finite-cover}%
$\F(\emptyset)$ is terminal in $\A$ and for every finite open covering $X = U_1 \cup \cdots \cup U_n$ the induced map
\[
\F(X) \to \lim_{\emptyset \neq S \subseteq \{1,\ldots,n\}} \F(\cap_{i \in S} U_i)
\]
is an equivalence. 
\item
\label{item:two-cover}%
$\F(\emptyset)$ is terminal in $\A$ and for every Zariski covering of the form $X = U_1 \cup U_2$ of qcqs $S$-schemes, the square
\[
\begin{tikzcd}
\F(X) \ar[r]\ar[d] & \F(U_1) \ar[d] \\
\F(U_2) \ar[r] & \F(U_1 \cap U_2)
\end{tikzcd}
\]
is cartesian.
\end{enumerate}
\end{corollary}
\begin{proof}
The equivalence of \ref{item:F-Zar-sheaf} and \ref{item:F-coproducts-to-products-zar-epi} follows from Proposition~\ref{proposition:zariski-coverage} and the characterization of~\cite[Proposition A.3.3.1]{SAG}. The equivalence of \ref{item:F-coproducts-to-products-zar-epi} and \ref{item:F-coprod-to-prod-zar-coprod} follows from Corollary~\ref{corollary:epi-refinement}, since any Zariski epi $Z \to X$ refines to a Zariski epi of the form $\coprod_{i=1}^{n} U_i \to X$. The equivalence of \ref{item:F-coprod-to-prod-zar-coprod} and \ref{item:finite-cover} is standard and follows from the fact the forgetful functor from the simplex category of $\cosk_0(\{1,...,n\})$ to the poset of non-empty subsets in $\{1,...,n\}$ is coinitial. Finally, the equivalence of \ref{item:finite-cover} and \ref{item:two-cover} follows in a standard manner by induction.
\end{proof}

As for the Nisnevich topology, it is often considered in the literature exclusively for Noetherian schemes. Since we do not want to make this assumption at this point, we will work with the following definition in the general qcqs case (see~\cite{SAG}):

\begin{definition}
\label{definition:nisnevich}%
Let $X$ be a qcqs scheme. A collection of maps $\{V_{\alp} \to X\}_{\alp}$ is said to be a Nisnevich covering if the following holds:
\begin{enumerate}
\item
\label{item:all-etale-maps}%
Each $V_{\alp} \to X$ is étale.
\item
\label{item:lift-exists}%
There exists a finite sequence of open subsets $\emptyset = U_0 \subseteq U_1 \subseteq \cdots \subseteq U_n = X$ of $X$, such that for every $i=1,\ldots,n$ there exists an index $\alp_i$ such that the dotted lift 
\[
\begin{tikzcd}
& V_{\alp_i} \ar[d] \\
U_i \setminus U_{i-1} \ar[ur,dotted] \ar[r] & X
\end{tikzcd}
\]
exists.
\end{enumerate}
\end{definition}
The \defi{Nisnevich topology} on $\qSch_{/S}$ is then by definition the Grothendieck topology generated by Nisnevich coverings. In other words, a sieve on $X \in \qSch_{/S}$ is a covering sieve if and only if it contains a collection of morphisms $V_{\alp} \to X$ which form a Nisnevich covering. We note that every Zariski covering of a qcqs scheme is also a Nisnevich covering, and so the Nisnevich topology on $\qSch_{/S}$ is finer than the Zariski topology.

\begin{remark}
\label{remark:affine-nisnevich}%
In Definition~\ref{definition:nisnevich}, if $\emptyset = U_0 \subseteq U_1 \subseteq \cdots \subseteq U_n = X$ is a finite sequence of opens satisfying \ref{item:lift-exists} of that definition, then any refinement of this sequence (that is, a sequence obtained by factoring each inclusion $U_i \subseteq U_{i+1}$ into a finite sequence $U_i = U^0_{i} \subseteq U^1_{i} \subseteq \cdots \subseteq U^m_{i}=U_{i+1}$) will again satisfy \ref{item:lift-exists}. It hence follows that if $\{V_{\alp} \to X\}$ is a Nisnevich covering such that $X = \spec(A)$ is affine then we may assume (by passing to a refinement) that the open subsets $\emptyset = U_0 \subseteq U_1 \subseteq \cdots \subseteq U_n = X$ are of the form $U_i = \spec(A[f_i^{-1}])$ for $f_i \in A$. In particular, a Nisnevich covering among affine schemes in this sense coincides with the notion appearing in~\cite[Definition B.4.1.1]{SAG}.
\end{remark}

\begin{remark}
\label{remark:nisnevich-finitary}%
The Nisnevich topology is again finitary: indeed, if $\{V_{\alp} \to X\}_{\alp}$ is a Nisnevich covering then $\{V_{\alp_i} \to X\}_{i=1,\ldots,n}$ is again a Nisnevich covering, where the $\alp_i$'s are determined by the finite sequence $\emptyset = U_0 \subseteq U_1 \subseteq \cdots \subseteq U_n = X$ as in \ref{item:lift-exists} of Definition~\ref{definition:nisnevich}. 
\end{remark}

\begin{proposition}
\label{proposition:nisnevich-epi}%
Let $\E \subseteq \Sch_{/X}$ be a sieve on $X$. Then $\E$ is covering with respect to the Nisnevich topology if and only if it contains a finite collection of étale maps $\{U_i \to X\}_{i=1,\ldots,n}$ such that the induced map $\coprod_i U_i \to X$ is a Nisnevich epi in the sense of Definition~\ref{definition:epi}. In particular, the Nisnevich topology on $\Sch_{/X}$ coincides with the one induced by the coverage of Nisnevich epis.
\end{proposition}
\begin{proof}
If $\E$ is a Nisnevich covering sieve then by Remark~\ref{remark:nisnevich-finitary} it contains a finite collection of étale maps $\{U_j \to X\}_{j=1,\ldots,m}$ which form a Nisnevich covering of $X$. In this case clearly the induced map $\coprod_j U_j \to X$ is Nisnevich epi. Conversely, suppose that $\E$ contains a finite $\{V_j \to X\}_{j=1,\ldots,m}$ of étale maps such that the induced map $V = \coprod_j V_j \to X$ is a Nisnevich epi.
Then there exists finite sequence $\emptyset = U_0 \subseteq U_1 \subseteq \cdots \subseteq U_n = X$ of open subschemes of $X$ such that for every $i=1,\ldots,n$ there exists a dotted lift 
\[
\begin{tikzcd}
& V \ar[d] \\
Z_i \ar[ur,"s_i"] \ar[r] & X
\end{tikzcd}
\]
where $Z_i := U_i \setminus U_{i-1}$ is the reduced complement of $U_{i-1}$ in $U_i$. Now since each $V_j$ is open and closed in $V$, its pre-image $Z_{i,j} := s_i^{-1}(V_j)$ is open and closed in $Z_i$, yielding a disjoint decomposition $Z_i = \coprod_{j=1}^{m} Z_{i,j}$. For $i=0,...,n-1$ and $j=0,...,m$ let us then set $U_{i,j} := U_i \cup \bigcup_{j'=1,...,j} Z_{i+1,j}$, so that, when lexicographically ordered, the $U_{i,j}$ form an ascending sequence of open subschemes
\[
\emptyset = U_{0,0} \subseteq ... \subseteq U_{0,m} = U_{1,0} \subseteq ... \subseteq U_{1,m} = U_{2,0} \subseteq ... \subseteq U_{n,m} = X
\]
such that that reduced complement of each open in the subsequent one is either empty or one of the $Z_{i,j}$'s. In particular, for each such complement $Z'$ there exists a $j=1,\ldots,m$ and a lift of the form
\[
\begin{tikzcd}
& V_j \ar[d] \\
Z' \ar[ur] \ar[r] & X,
\end{tikzcd}
\]
and so we conclude that $V_1,...,V_m$ form a Nisnevich covering of $X$, as desired.  
\end{proof}

By~\cite[Proposition A.3.3.1]{SAG} we then conclude:

\begin{corollary}
\label{corollary:nisnevich-sheaf}%
A presheaf $\F\colon (\qSch_{/S})\op \to \A$ is a Nisnevich sheaf if and only if it sends coproducts in $\qSch_{/S}$ to products in $\A$ and satisfies Čech descent for Nisnevich epis.
\end{corollary}

\subsection{Characterization of Nisnevich sheaves}

Let $\qcsSch_{/S} \subseteq \qSch_{/S}$ be the full subcategory spanned by the morphisms $X \to S$ such that $X$ is separated, and let $\Aff_{/S} \subseteq \qcsSch_{/S}$ be the full subcategory therein spanned by those $X \to S$ such that $X$ is affine. Both of these full subcategories are closed under direct sums and fibre products are hence themselves extensive. We then define the Nisnevich topologies on $\qcsSch_{/S}$ and $\Aff_{/S}$ to be those associated to the coverage of Nisnevich epis among separated and affine schemes, respectively. By the same argument as in the proof of Proposition~\ref{proposition:nisnevich-epi} we have that the resulting Grothendieck topologies on $\qcsSch_{/S}$ and $\Aff_{/S}$ coincide with the ones generated by the Nisnevich coverings among separated and affine schemes.

\begin{proposition}
\label{proposition:reduce-affine}%
Let $S$ be a qcqs scheme and let $\F\colon (\qSch_{/S})^{\op} \to \A$ be a functor valued in some presentable $\infty$-category $\A$. Then the following are equivalent:
\begin{enumerate}
\item
\label{item:nis-sheaf}%
$\F$ is a Nisnevich sheaf.
\item
\label{item:zar-sheaf-qcs}%
$\F$ is a Zariski sheaf and the restriction of $\F$ to $(\qcsSch_{/S})\op \subseteq (\qSch_{/S})\op$ is a Nisnevich sheaf.
\item
\label{item:zar-sheaf-aff}%
$\F$ is a Zariski sheaf and the restriction of $\F$ to $(\Aff_{/S})\op \subseteq (\qSch_{/S})\op$ is a Nisnevich sheaf.
\end{enumerate}
\end{proposition}

The proof of Proposition~\ref{proposition:reduce-affine} will require the following lemma:

\begin{lemma}
\label{lemma:epi-to-pullback}%
Let
\[
\begin{tikzcd}
W \ar[r]\ar[d] & Y \ar[d] \\
Z \ar[r] & X
\end{tikzcd}
\]
be a pullback square of qcqs schemes all of whose legs are Nisnevich epis, and let $p\colon W' \to W$ be a map of qcqs schemes. Then $p$ is a Nisnevich epi if and only if the composites $W' \to W \to Z$ and $W' \to W \to Y$ are Nisnevich epis. 
\end{lemma}
\begin{proof}
The only if direction follows from the fact that Nisnevich epis are closed under composition. To prove the if the direction, let $\emptyset = U_0 \subseteq U_1 \subseteq \cdots \subseteq U_n= Z$ and $\emptyset = V_0 \subseteq V_1 \subseteq \cdots \subseteq V_m=Y$ be ascending sequences of open subschemes such that each of the inclusions $U_i\setminus U_{i-1} \hrar Z$ and $V_j \setminus V_{j-1} \hrar Y$ lifts to $W'$. We construct a sequence of open subschemes 
\[
W_{1,1} \subseteq W_{1,2} \subseteq \cdots \subseteq W_{1,m} \subseteq W_{2,1} \subseteq \cdots \subseteq W_{1,m} \subseteq \cdots \subseteq W_{n,m} = W
\]
inductively as follows. First we set $W_{1,1} = U_1 \times_X V_1$. 
Let $1 \leq i \leq n$ and $1 \leq j \leq m$ be such that $W_{i',j'}$ has been defined for every $(i',j')$ such that either $i' < i$ or $i'=i$ and $j'<j$, that is, all $(i',j')$ lexicographically smaller than $(i,j)$. Let $(i',j')$ be the largest element smaller than $(i,j)$ in the lexicographical order. Then we define $W_{i,j}$ by
\[
W_{i,j} = W_{i',j'} \cup (U_i \times_X V_j) .
\]
We now observe that by construction, both $U_{i-1} \times_X V_j$ and $U_i \times V_{j-1}$ are contained in $W_{i',j'}$, and hence $W_{i,j} \setminus W_{i',j'}$ is contained in $(U_i\setminus U_{i-1}) \times_X (V_j \setminus V_{j-1})$ for every lexicographically consecutive pair $(i',j') < (i,j)$ in $\{1,\ldots,n\} \times \{1,\ldots,m\}$ (this is also true if $i$ or $j$ are $1$, since $U_0=V_0 = \emptyset$)  It then follows that the inclusion $W_{1,1} \hrar W$, as well as each of the inclusions $W_{i,j} \setminus W_{i',j'}$ for every lexicographically consecutive pair $(i',j') < (i,j)$ in $\{1,\ldots,n\} \times \{1,\ldots,m\}$, lifts to $W'$, and so $W'\to W$ is a Nisnevich epi, as desired.
\end{proof}

\begin{proof}[Proof of Proposition~\ref{proposition:reduce-affine}]
Assume first that $\F\colon (\qSch_{/S})\op \to \A$ is a Nisnevich sheaf. By Corollary~\ref{corollary:nisnevich-sheaf} $\F$ satisfies \v{c}ech descent for Nisnevich epis. Since any Zariski epi is a Nisnevich epi $\F$ is also a Zariski sheaf. In addition, the restriction of $\F$ to $\qcsSch_{/S}$ satisfies Čech descent for Nisnevich epis among separated and affine schemes and so \ref{item:nis-sheaf} $\Rightarrow$ \ref{item:zar-sheaf-qcs}, and similarly \ref{item:zar-sheaf-qcs} $\Rightarrow$ \ref{item:zar-sheaf-aff} for the same reason. 

We now prove that \ref{item:zar-sheaf-aff} $\Rightarrow$ \ref{item:zar-sheaf-qcs} $\Rightarrow$ \ref{item:nis-sheaf}.
Suppose that $\F$ is a Zariski sheaf whose restriction to $(\Aff_{/S})\op$ is a Nisnevich sheaf. We prove that the restriction of $\F$ to $\qcsSch_{/S}$ is also a Nisnevich sheaf. Now since $\F$ is a Zariski sheaf it takes coproducts of schemes to products in $\A$ by Corollary~\ref{corollary:zariski-sheaf}. It will hence suffice to show that it satisfies Čech descent with respect to Nisnevich epis among separated schemes. Let $q\colon V \to X$ be a Nisnevich epi of quasi-compact $S$-schemes such that $X$ and $V$ are separated. Choose a finite open covering $X = \cup_{i=1,\ldots,n} U_i$ by affines and write $U = \coprod_i U_i$, so that $p\colon U \to X$ is a Zariski epi and $U$ is affine (and in particular quasi-compact and separated). Let $Y = U \times_X V$. Being a fibre product of quasi-compact and separated schemes, $Y$ is also quasi-compact separated. Let us now choose again a finite open covering $Y = \cup_{j=1,..,m} W_i$, and set $W = \coprod_j W_j$, so that $W \to Y$ is a Zariski epi and $W$ is affine. Since Nisnevich epis are closed under base change and composition, we have that the composed map $W \to Y \to U$ is a Nisnevich epi, and similarly, the composed map $W \to Y \to V$ is a Zariski epi. Consider the commutative square
\begin{equation}
\label{equation:square-for-lemma}%
\begin{tikzcd}
W \ar[r, "p'"]\ar[d, "q'"'] & V \ar[d, "q"] \\
U \ar[r, "p"] & X
\end{tikzcd}
\end{equation}
whose vertical arrows are Nisnevich epis and whose horizontal arrows are Nisnevich epis. We claim that this square satisfies the assumptions of Lemma~\ref{lemma:descent-square}. For this, first note that
since Nisnevich epis are closed under base change and composition we have that all the maps in Čech resolutions of $p$ and $p'$ are Zariski epis. Combined with Lemma~\ref{lemma:epi-to-pullback}, this implies that each of the maps $U_n(p') \to U_n(p)$ is a Nisnevich epi. At the same time, $U_n(p)$ decomposes as a finite disjoint union of the form
\[
U_n(p) = \coprod_{i_1,\ldots,i_{n+1}}[ U_{i_1} \cap \cdots \cap U_{i_{n+1}}]
\]
and each $U_{i_1} \cap \cdots \cap U_{i_{n+1}}$ is affine, being the intersection of affine open subschemes inside a separated scheme.
By assumption we thus have that $\F$ satisfies Čech descent with respect to each $U_n(p') \to U_n(p)$. We now need to show that $\F$ also satisfies Čech descent with respect to each $U_m(q') \to U_m(q)$. Now we claim that each of these maps is a Zariski epi. To see this, let us factor the above square as the pasting of two squares as follows
\[
\begin{tikzcd}
W \ar[r, "p''"]\ar[d,"q'"'] & Y \ar[r, "p'"]\ar[d, "q''"'] & V \ar[d, "q"] \\
U \ar[r,equal] & U \ar[r, "p"] & X
\end{tikzcd}
\]
We claim that the maps $U_m(q') \to U_m(q'')$ and $U_m(q'') \to U_m(q)$ are both Zariski epis (and hence the composite is such as well). For the latter, by construction this is the map
\[
W \times_U \cdots \times_U W = \coprod_{i_1,\ldots,i_{m+1}}[W_{i_1} \times_U \cdots \times_U W_{i_{m+1}}] \to Y \times_U \cdots \times_U Y
\]
which is a Zariski epi since each of the maps 
\[
W_{i_1} \times_U \cdots \times_U W_{i_{m+1}} \to Y \times_U \cdots \times_U Y
\]
is an open embedding.

We have thus shown that the square~\eqref{equation:square-for-lemma} satisfies the assumption of Lemma~\ref{equation:square-for-lemma}, and hence we may conclude that $\F$ satisfies Čech descent with respect to $p$ if and only if it satisfies Čech descent with respect to $q$. Since $p$ is a Zariski epi $\F$ satisfies descent with respect to $p$ and hence we conclude that $\F$ satisfies Čech descent with respect to $q$.

This shows that \ref{item:zar-sheaf-aff} $\Rightarrow$ \ref{item:zar-sheaf-qcs}. The proof that \ref{item:zar-sheaf-qcs} $\Rightarrow$ \ref{item:nis-sheaf} proceeds in exactly the same manner, where we note that the only place we used the assumption that $X$ is separated was in order to conclude that each $U_{i_1} \cap \cdots \cap U_{i_{n+1}}$ is affine, being the intersection of affine open subschemes inside a separated scheme. In order to prove \ref{item:zar-sheaf-qcs} $\Rightarrow$ \ref{item:nis-sheaf} we no longer need each $U_{i_1} \cap \cdots \cap U_{i_{n+1}}$ to be affine, just separated. But this automatic from the fact that these are open subschemes of the affine (and in particular separated) scheme $U_{i_1}$, and being separated is a property which is inherited by open subschemes. 
\end{proof}

\end{appendices}

\end{document}

%% file: packages.tex
\usepackage[utf8]{inputenc}         
\usepackage[T1]{fontenc}            

\usepackage{a4wide}                 
\usepackage{amsmath, amssymb}       
\usepackage{amsthm}                 
\usepackage{appendix}               

\usepackage{tikz}                   
\usepackage{tikz-cd}                
\tikzcdset{ 
arrow style=tikz
}

\usepackage{stackrel}               

\usepackage{color}                  
\usepackage[normalem]{ulem}         
\usepackage{enumitem}               
\setlist[itemize]{leftmargin=*}
\setlist[enumerate]{leftmargin=*,label=(\arabic*),ref=(\arabic*)}
\newlist{subenumerate}{enumerate}{2}
\setlist[subenumerate]{leftmargin=*,label=\alph*),ref=\alph*)}

\definecolor{darkblue}{rgb}{0,0,0.6} 
\usepackage[ocgcolorlinks,colorlinks=true, citecolor=darkblue, filecolor=darkblue, linkcolor=darkblue, urlcolor=darkblue,pagebackref]{hyperref}         

\usepackage[alphabetic,initials]{amsrefs}

\usepackage{bbm}                    
\usepackage{stix}                   
\usepackage{eucal}                  
\DeclareFontFamily{T1}{cbgreek}{}
\DeclareFontShape{T1}{cbgreek}{m}{n}{<-6>  grmn0500 <6-7> grmn0600 <7-8> grmn0700 <8-9> grmn0800 <9-10> grmn0900 <10-12> grmn1000 <12-17> grmn1200 <17-> grmn1728}{}
\DeclareSymbolFont{quadratics}{T1}{cbgreek}{m}{n}
\DeclareMathSymbol{\qoppa}{\mathord}{quadratics}{19}
\DeclareMathSymbol{\Qoppa}{\mathord}{quadratics}{21}
\DeclareUnicodeCharacter{221E}{$\infty$}

\usepackage{comment}                

%% file: theorem-styles.tex
\newtheorem{theorem}{Theorem}[subsection]
\newtheorem{proposition}[theorem]{Proposition}
\newtheorem{lemma}[theorem]{Lemma}
\newtheorem{corollary}[theorem]{Corollary}

\newtheorem{assumption}[theorem]{Assumption}

\theoremstyle{definition}
\newtheorem{definition}[theorem]{Definition}
\newtheorem{notation}[theorem]{Notation}
\newtheorem{construction}[theorem]{Construction}

\theoremstyle{remark}
\newtheorem{remark}[theorem]{Remark}
\newtheorem{example}[theorem]{Example}

\numberwithin{equation}{subsection}

%% file: commands.tex
\newcommand{\defi}[1]{\emph{#1}}                 

\newcommand{\spacefont}{\mathcal}                

\newcommand{\alp}{\alpha}
\newcommand{\bet}{\beta}

\newcommand{\Del}{\Delta}

\newcommand{\Lam}{\Lambda}
\newcommand{\Om}{\Omega}

\newcommand{\Sig}{\Sigma}
\newcommand{\sig}{\sigma}
\newcommand{\vphi}{\varphi}
\newcommand{\eps}{\epsilon}

\newcommand{\hrar}{\hookrightarrow}

\newcommand{\wtl}{\widetilde}

\newcommand{\ovl}{\overline}
\newcommand{\tosimeq}{\xrightarrow{\,\smash{\raisebox{-0.65ex}{\ensuremath{\scriptstyle\simeq}}}\,}}

\newcommand{\adj}{\mathbin{
\begin{tikzpicture}[baseline,thick] 
\coordinate (source) at (0ex,.5ex);
\coordinate (target) at (3ex,.5ex);
\draw[->] ([yshift=1ex]source) -- ([yshift=1ex]target); 
\draw[->] ([yshift=-.5ex]target) -- ([yshift=-.5ex]source);
\node at (1.5ex,.8ex) {$\scriptscriptstyle \perp$};
\end{tikzpicture}\ %
}}

\renewcommand{\>}{\rangle}

\newcommand{\cocolon}{\nobreak \mskip6mu plus1mu \mathpunct{}\nonscript\mkern-\thinmuskip {:}\mskip2mu \relax}

\newcommand{\NN}{\mathbb{N}}   	            
\newcommand{\RR}{\mathbb{R}}               
\newcommand{\CC}{\mathbb{C}}               
\newcommand{\ZZ}{\mathbb{Z}}               
\newcommand{\FF}{\mathbb{F}}               

\newcommand{\im}{\mathrm{im}}              

\newcommand{\Psh}{\operatorname{Psh}}           

\newcommand{\fm}{\mathfrak{m}}                  

\newcommand{\id}{\mathrm{id}}                	

\newcommand{\op}{^\mathrm{op}}                  
\newcommand{\mop}{\mathrm{op}}                  

\newcommand{\Hom}{\operatorname{Hom}}           

\newcommand{\Map}{\operatorname{Map}}           

\newcommand{\fib}{\operatorname{fib}}           
\newcommand{\cof}{\operatorname{cof}}           

\newcommand{\colim}{\mathop{\mathrm{colim}}}    

\newcommand{\core}{\mathrm{Cr}}                 

\newcommand{\LKan}{\operatorname{Lan}}          

\newcommand{\Pro}{\operatorname{Pro}}           
\newcommand{\Ind}{\operatorname{Ind}}           

\newcommand{\Ar}{\mathrm{Ar}}			

\newcommand{\lax}{\mathrm{lax}}			

\newcommand{\I}{\mathcal{I}}               	

\newcommand{\Spa}{{\mathcal{S}\mathrm{p}}}	
\newcommand{\PSpa}{{\mathcal {PS}\mathrm{p}}}	
\newcommand{\Sps}{\mathcal S}			

\newcommand{\map}{\mathrm{map}}        	       	

\newcommand{\Ominfty}{\Omega^{\infty}}         	

\renewcommand{\SS}{\mathbb{S}}                 	

\newcommand{\GEM}{\mathrm{H}}                   

\newcommand{\cosk}{\mathrm{cosk}}		

\newcommand{\B}{\mathrm{B}}                     

\newcommand{\h}{\mathrm{h}}                    	
\newcommand{\Ct}{\mathrm{C_2}}                 	
\newcommand{\BC}{\B \Ct}                      	
\newcommand{\hC}{{\h\Ct}}                      	
\newcommand{\tC}{{\mathrm{t}\Ct}}              	

\newcommand{\N}{\mathrm{N}}                    	

\newcommand{\cP}{\mathcal{P}}			

\newcommand{\Fin}{\mathrm{Fin}}			
\newcommand{\Inj}{\mathrm{Inj}}			

\newcommand{\Ass}{\mathrm{Ass}}			
\newcommand{\Com}{\mathrm{Com}}       		
\newcommand{\Env}{\mathrm{Env}}			

\newcommand{\EE}{\mathbb{E}}			

\newcommand{\Einf}{\EE_\infty}         		

\newcommand{\Alg}{\mathrm{Alg}}                 
\newcommand{\CAlg}{\operatorname{CAlg}}	 	

\newcommand{\Mon}{\operatorname{Mon}}           
\newcommand{\Grp}{\operatorname{Grp}}           

\newcommand{\MCom}{\mathrm{MCom}}		

\newcommand{\E}{\mathcal{E}}               	

\newcommand{\C}{\mathcal{C}}                	
\newcommand{\D}{\mathcal{D}}               	

\newcommand{\stable}{\mathrm{st}} 		
\newcommand{\semi}{\mathrm{ss}} 		

\newcommand{\sh}{\mathrm{sh}}			

\newcommand{\one}{\mathbb{1}} 			

\newcommand{\Zotimes}{\mathcal{Z}^\otimes}	

\newcommand{\A}{\mathcal{A}}               	

\newcommand{\cB}{\mathcal{B}}              	

\newcommand{\st}{\mathrm{st}}			

\newcommand{\ModCat}{\mathrm{ModCat}}		
\newcommand{\Cof}{\mathcal{C}\mathrm{of}}	
\newcommand{\Fib}{\mathcal{F}\mathrm{ib}} 	
\newcommand{\QIso}{\mathrm{QIso}} 		

\newcommand{\M}{\mathcal{M}}               	

\newcommand{\Cat}{\mathrm{Cat}}               	

\newcommand{\PrL}{\mathrm{Pr}^\mathrm{L}}	
\newcommand{\PrR}{\mathrm{Pr}^\mathrm{R}}	

\newcommand{\Inv}{\mathrm{Inv}}                 

\newcommand{\CAT}{\mathrm{CAT}}                 
\newcommand{\ex}{\mathrm{ex}}                   
\newcommand{\Catx}{\Cat^\ex}     		
\newcommand{\tstruc}{\mathrm{t}}		

\newcommand{\Catxt}{\Cat^\ex_\tstruc}     	
\newcommand{\Catxi}{\Cat^\ex_{\natural}}     

\newcommand{\Catr}{\mathrm{Rig}}

\newcommand{\Catexr}{\mathrm{Rig}^{\ex}}   	
\newcommand{\Catexri}{\mathrm{Rig}^{\ex}_{\natural}}   	

\newcommand{\SMGro}{\mathrm{SMGro}}		
\newcommand{\SMModCat}{\mathrm{SMModCat}} 	

\newcommand{\Cath}{\Cat^\mathrm{h}}     	
\newcommand{\Catps}{\Cat^\mathrm{ps}}   	
\newcommand{\Catpst}{\Cat^{\mathrm{ps}}_{\tstruc}}
\newcommand{\Catpsi}{\Cat^\mathrm{ps}_{\natural}}   
\newcommand{\Catp}{\Cat^{\mathrm p}}    	
\newcommand{\Catpt}{\Catp_\tstruc}    		
\newcommand{\Catsb}{\mathrm{Cat}^{\mathrm{sb}}} 

\newcommand{\Catb}{\Cat^{\mathrm{b}}} 		
\newcommand{\Catpb}{\Cat^{\mathrm{pb}}}		

\newcommand{\Fun}{\operatorname{Fun}}           

\newcommand{\bloc}{\mathrm{bloc}} 		

\newcommand{\oplax}{\mathrm{oplax}}		

\newcommand{\seq}{\mathrm{seq}}			

\newcommand{\Funx}{\operatorname{Fun^{ex}}}     
\newcommand{\Funq}{\operatorname{Fun^q}}        
\newcommand{\Funs}{\operatorname{Fun^s}}        

\newcommand{\kloc}{\mathrm{kloc}}               

\newcommand{\jbloc}{\mathrm{j}_{\mathrm{bloc}}} 

\newcommand{\lin}{\mathrm{lin}} 		%
\newcommand{\Llin}{\mathrm{L}^\lin}		

\newcommand{\bil}{\mathrm{bil}}
\newcommand{\Funbil}{\Fun^{\bil}}		
\newcommand{\Lbil}{\mathrm{L}^\bil}		

\newcommand{\Lqdr}{\mathrm{L}^\qdr}		

\newcommand{\QF}{\Qoppa}                   	
\newcommand{\QFD}{\Phi}                    	
\newcommand{\QFE}{\Psi}                    	

\newcommand{\Bil}{\mathrm{B}}			

\newcommand{\Psubone}{\mathrm{P}_1}             

\newcommand{\LT}{\mathcal{E}}			

\newcommand{\Dual}{\mathrm{D}}              	

\newcommand{\qshift}[1]{^{[#1]}}                

\newcommand{\sym}{\mathrm{s}}                   
\newcommand{\qdr}{\mathrm{q}}                   
\newcommand{\gs}{\mathrm{gs}}                   
\newcommand{\gq}{\mathrm{gq}}                   
\newcommand{\gev}{\mathrm{ge}}                  

\newcommand{\hyp}{\mathrm{hyp}}                 

\newcommand{\Sym}{\mathrm{Sym}}			

\newcommand{\Hyp}{\operatorname{Hyp}}           
\newcommand{\Met}{\operatorname{Met}}           
\newcommand{\Lag}{\mathcal{L}}             	

\newcommand{\Poinc}{\mathrm{Pn}}                

\newcommand{\cob}{w}	                       	

\newcommand{\T}{\mathcal{T}}               	

\newcommand{\F}{\mathcal{F}}               	
\newcommand{\G}{\mathcal{G}}               	

\newcommand{\Pic}{\mathrm{Pic}}			

\newcommand{\Picspace}{{\spacefont Pic}}        

\newcommand{\K}{\operatorname K}           	
\newcommand{\KK}{\mathbb K}                     

\newcommand{\Q}{\mathrm{Q}}   			

\renewcommand{\L}{\operatorname L}              
\newcommand{\KL}{\mathbb{L}}                    
\newcommand{\W}{\operatorname W}                

\newcommand{\GWspace}{\operatorname{\spacefont{GW}}} 
\newcommand{\GW}{\operatorname{GW}}             
\newcommand{\KGW}{\mathbb{GW}}                  

\newcommand{\fgt}{\mathrm{f\hspace{-.2ex}gt}}                	

\newcommand{\cl}{\mathrm{cl}}			

\newcommand{\cK}{\mathcal{K}} 			

\newcommand{\HGW}{\mathbb{HGW}}			
\newcommand{\HL}{\mathbb{HL}}			
\newcommand{\HK}{\mathbb{HK}}			

\newcommand{\Mod}{\operatorname{Mod}}           

\newcommand{\Vect}{\mathrm{Vect}}               

\newcommand{\Proj}{\operatorname{Proj}}         

\newcommand{\inflation}{\mathrm{inf}}		

\newcommand{\spec}{\mathrm{spec}}              	
\newcommand{\cO}{\mathcal{O}}              	

\newcommand{\Sch}{\mathrm{Sch}}			

\newcommand{\sm}{\mathrm{sm}}			

\newcommand{\Sm}{\mathrm{Sm}}                   
\newcommand{\Aff}{\mathrm{Aff}}			
\newcommand{\aff}{\mathrm{aff}}			

\newcommand{\qcs}{\mathrm{qcs}}			
\newcommand{\qc}{\mathrm{qc}}			

\newcommand{\qSch}{\mathrm{qSch}}		
\newcommand{\qSchPic}{\mathrm{qSchPic}}		
\newcommand{\qcsSch}{\Sch^\qcs}			

\newcommand{\Pone}{\mathbb{P}^1}		
\newcommand{\Aa}{\mathbb{A}}    		
\newcommand{\Aone}{\Aa^1}    			
\newcommand{\PP}{\mathbb{P}}                     
\newcommand{\Gm}{\mathbb{G}_m}                   

\newcommand{\GL}{\mathrm{GL}}                    
\newcommand{\Orth}{\mathrm{O}}			 

\newcommand{\Sh}{\mathrm{Sh}}			

\newcommand{\Coh}{\mathrm{Coh}}			

\newcommand{\rk}{\mathrm{rk}}        		

\newcommand{\et}{\text{ét}}                     
\newcommand{\Nis}{\mathrm{nis}} 		

\newcommand{\Tot}{\mathrm{Tot}}			

\newcommand{\X}{\mathcal{X}}               	

\newcommand{\Ext}{\operatorname{Ext}}		

\newcommand{\Ch}{\operatorname{Ch}}            	

\newcommand{\Der}{\mathcal D}                   
\newcommand{\perf}{\mathrm{p}}                  
\newcommand{\Dperf}{\Der^\perf}         	
\newcommand{\Derqc}{\Der^{\mathrm{qc}}}    	

\newcommand{\cA}{\mathcal{A}}                  	

\newcommand{\Loc}{\mathrm{Loc}}			
\newcommand{\SH}{\operatorname{SH}}             
\newcommand{\LAone}{\Loc_{\Aone}} 		
\newcommand{\Lbloc}{\Loc_{\mathrm{b}}} 		

\newcommand{\Sone}{\mathrm{S}^1}		

\newcommand{\OmPone}{\Omega_{\mathbb{P}^1}}	

\newcommand{\Hinv}{\mathrm{H}}			

\newcommand{\Th}{\mathrm{Th}}		        

\newcommand{\Y}{\mathcal{Y}} 			
\newcommand{\Z}{\mathcal{Z}}			

\newcommand{\R}{\mathcal{R}}		   	

\newcommand{\KQ}{\operatorname{KQ}}             
\newcommand{\KW}{\operatorname{KW}}             
\newcommand{\KGL}{\operatorname{KGL}}           

%% file: toc-tweak.tex
\makeatletter

\renewcommand{\tocsection}[3]{%
\indentlabel{\@ifnotempty{#2}{\parbox[b]{3ex}{\bfseries\ignorespaces#1 #2}}}\bfseries#3} 

\renewcommand{\tocsubsection}[3]{%
\indentlabel{\@ifnotempty{#2}{\hspace{1.6em}\parbox[b]{5ex}{\ignorespaces#1 #2}}}#3}

\renewcommand{\tocsubsubsection}[3]{%
\indentlabel{\@ifnotempty{#2}{\hspace{3.9em}\parbox[b]{5ex}{\ignorespaces#1 #2}}}#3}

\makeatother

\DeclareRobustCommand{\SkipTocEntry}[5]{}

%% file: bib-format.tex
\BibSpec{collection.article}{%
    +{}  {\PrintAuthors}                {author}
    +{,} { \textit}                     {title}
    +{.} { }                            {part}
    +{:} { \textit}                     {subtitle}
    +{,} { \PrintContributions}         {contribution}
    +{,} { \PrintConference}            {conference}
    +{}  {\PrintBook}                   {book}
    +{,} { }                            {booktitle}
    +{,} { }                            {series}
    +{,} { \voltext}                    {volume}
    +{,} { }                            {publisher}
    +{,} { }                            {organization}
    +{,} { }                            {address}
    +{,} { \PrintDateB}                 {date}
    +{,} { \eprintpages}                {pages}
    +{,} { }                            {status}
    +{,} { available at \eprint}        {eprint}
    +{}  { \parenthesize}               {language}
    +{}  { \PrintTranslation}           {translation}
    +{;} { \PrintReprint}               {reprint}
    +{.} { }                            {note}
    +{.} {}                             {transition}
    +{}  {\SentenceSpace \PrintReviews} {review}
}

\BibSpec{article}{%
    +{}  {\PrintAuthors}                {author}
    +{.} { \textit}                     {title}
    +{.} { }                            {part}
    +{:} { \textit}                     {subtitle}
    +{,} { \PrintContributions}         {contribution}
    +{.} { \PrintPartials}              {partial}
    +{,} { }                            {journal}
    +{}  { \textbf}                     {volume}
    +{} { \PrintDate}                  {date}
    +{,} { \issuetext}                  {number}
    +{,} { \eprintpages}                {pages}
    +{,} { }                            {status}
    +{,} { available at \eprint}        {eprint}
    +{}  { \parenthesize}               {language}
    +{}  { \PrintTranslation}           {translation}
    +{;} { \PrintReprint}               {reprint}
    +{.} { }                            {note}
    +{.} {}                             {transition}
    +{}  {\SentenceSpace \PrintReviews} {review}
}